\documentclass[11pt,twoside]{article}

\usepackage[utf8]{inputenc}
\usepackage[T1]{fontenc}
\usepackage{microtype}

\usepackage[
  a4paper,
  margin=1in
]{geometry}

\usepackage{amsmath}
\usepackage{amssymb}
\usepackage{amsfonts}
\usepackage{amsthm}
\usepackage{mathtools}
\usepackage{dsfont}
\usepackage{mathrsfs}
\usepackage{nicefrac}
\usepackage{graphicx}
\usepackage{booktabs}
\usepackage{xcolor}
\usepackage{comment}
\usepackage{caption}
\usepackage{url}
\usepackage{hyperref}
\usepackage{natbib}
\usepackage{authblk}

\theoremstyle{plain}
\newtheorem{theorem}{Theorem}[section]
\newtheorem{proposition}[theorem]{Proposition}
\newtheorem{lemma}[theorem]{Lemma}
\newtheorem{corollary}[theorem]{Corollary}
\theoremstyle{definition}

\theoremstyle{remark}
\newtheorem{remark}[theorem]{Remark}

\newcommand{\D}{\mathcal{D}}
\newcommand{\U}{\mathcal{U}}
\newcommand{\X}{\mathcal{X}}

\newcommand{\bR}{\mathbb{R}}

\DeclareMathOperator{\argmin}{argmin}

\theoremstyle{plain}
\newtheorem{thm}{Theorem}

\newtheorem{defn}[thm]{Definition}

\newcommand{\Var}[0]{\text{Var}}

\newcommand{\calF}[0]{\mathcal{F}}
\newcommand{\calG}[0]{\mathcal{G}}

\newcommand{\E}[0]{\mathbb{E}}
\newcommand{\R}[0]{\mathbb{R}}

\newcommand{\Z}[0]{\mathbb{Z}}

\newcommand{\rank}[0]{\text{rank}}

\renewcommand{\leq}{\leqslant} 
\renewcommand{\geq}{\geqslant}

\numberwithin{equation}{section}

\title{\Large\bfseries Energy-Guided Learning for Heat Equation with Noisy Source: Optimal Estimation and Generalization Bounds}

\author[1,*]{Zixiang Zhou}
\author[1,2,3,*]{Xiaohui Chen}
\author[2,1,*]{Shanghua Teng}

\affil[1]{Department of Mathematics,
University of Southern California}

\affil[2]{Thomas Lord Department of Computer Science,
University of Southern California}

\affil[3]{Computing and Mathematical Sciences,
California Institute of Technology}

\affil[*]{Email: 
\texttt{\{zzhou108, xiaohuic, shanghua\}@usc.edu}}

\date{}

\begin{document}

\maketitle

\begin{abstract}

Data and the governing partial differential equations (PDEs) of continuous-time dynamical systems provide rich structure for statistical learning. A central challenge is how to integrate data with PDE constraints in a principled way to improve generalization. In this paper, we characterize the statistical estimation and generalization properties of learning solutions to the heat equation with a noisy source function from finite, randomly sampled data. We first derive a novel variational formulation, which we characterize as the minimizer of loss function for an energy regularized and ellipticity debiased Sobolev regression. By formulating the learning task as a supervised problem over space–time domains under physically meaningful norms, our approach goes beyond the traditional supervised learning framework by effectively leveraging \emph{auxiliary data}—samples of the source function of the heat equation—to design a \emph{physics-informed energy regularization}. Our analysis reveals that learning in linear parabolic PDEs exhibits intrinsic anisotropy between spatial and temporal regularity, in sharp contrast to the static elliptic problems, leading to fundamentally different estimation and generalization behaviors. Specifically, we derive non-asymptotically rate-optimal estimation errors, as well as generalization error, which explicitly capture the interplay among temporal sampling, spatial resolution, and solution smoothness. Our numerical experiments demonstrate substantial orders of magnitude of reduction in testing errors against various physics-informed neural network methods on various parabolic and hyperbolic PDEs.


\end{abstract}

\section{Introduction}

Continuous-time dynamical systems governed by partial differential equations (PDEs) arise across scientific and engineering fields involving the study of time evolution behaviors of physical phenomena ranging from fluid dynamics to quantum mechanics~\citep{Batchelor_2000,gasiorowicz2007quantum}. Due to their instability and nonlinearity, learning solutions of time-dependent PDEs from noisy data is a central yet challenging problem in scientific machine learning~\citep{brunton2024promising, couairon1996global}. Classical numerical methods such as finite element and Galerkin methods provide highly accurate solvers but depend critically on the availability of a suitable computational mesh in high dimensional domain~\citep{persson2005mesh,ceniceros01, bern2000mesh}. This limitation motivates data-driven approaches that aim to bypass explicit discretization.

A recent popular mesh-free paradigm is the physics-informed neural networks (PINNs), which enforce PDE constraints via \emph{pointwise residual minimization} to extrapolate~\citep{PINN_Karniadakis2021}. Despite empirical successes in several applications~\citep{en16052343,doi:10.1126/science.aaw4741,HU2024112495,HAGHIGHAT2021113741,MO2022127739,PhysRevLett.132.010801,10.1145/3592979.3593412,JMLR:v25:23-0597}, this formulation evaluates differential operators at collocation points and therefore implicitly assumes sufficient smoothness of the solution. As a result, PINNs are fundamentally tied to strong solutions and can fail in regimes involving nonsmooth dynamics or weak solutions, with documented instability and training pathologies~\citep{Cuomo2022_PINNreview,Krishnapriyan2021_PINN-failure-modes,wang2022is}.


Variational approaches, exemplified by deep Ritz method~\citep{yu18}, instead minimize an energy functional whose minimizer characterized by the Euler-Lagrange equation coincides with the PDE solution. This provides global control in Sobolev spaces and enables learning of \emph{weak solutions} with provable guarantees~\citep{lustatistical, Yuling_Jiao_2022, Chenguang_Duan_2022}. However, such formulations crucially rely on the existence of an exact variational principle and are therefore limited to \emph{static} elliptic PDEs. For time-dependent systems governed by parabolic/hyperbolic PDEs, no corresponding \textit{exact} space-time energy functional exists, preventing a direct extension of deep Ritz method.

In this work, we focus on the statistical learning problem for weak solution of the heat equation
\begin{equation}
\left\{
\begin{aligned}
    \partial_t u(x,t) - \Delta u(x,t) &= f(x,t), && (x,t) \in \Omega \times (0,T], \\
    u(x,t) &= 0, && (x,t) \in \partial \Omega \times [0,T], \\
    u(x,0) &= u_0(x), && x \in \Omega,
\end{aligned}
\right.
\label{eq:heat_pde}
\end{equation}
where $f$ is an \textbf{\emph{unknown}} source function in a closed, bounded spatial domain $\Omega \subset \bR^d$ over a finite time horizon $[0, T]$, $u_0$ is a known initial condition, and the above equation satisfies the zero Dirichlet boundary condition. When we study the impact of the heat source $f$ on the solution $u$, solving the PDE~\eqref{eq:heat_pde} in an appropriate sense is often a well-posed forward problem with mild boundary and initial conditions~\citep{evans}. Thus without loss of generality, we assume that the true PDE solution of~\eqref{eq:heat_pde} is unique and identifiable, which is denoted by $u^* := u^*(x, t)$ throughout the paper.


In reality,  neither the source function $f$ nor the true solution $u^{*}$ is directly observed over the domain $\Omega\times [0,T]$. 
Instead, we are given observations of these functions at discrete time points $t_1, \dots, t_n \in [0, T]$. 
At each time $t_j$, we observe $m$ spatial points $X_{1j}, \dots, X_{mj} \in \Omega$. In this paper, we work with noisy data of the source function and its corresponding solution discretized at the same spatio-temporal points generated by the following sampling model
\begin{equation}
\label{regression_model}
S = \{(X_{ij}, t_j, f_{ij}, u^{*}_{ij})\}_{i=1,j=1}^{m,n},
\end{equation}
where 
\begin{align}
\label{regression_model_2}
f_{ij} = f(X_{ij}, t_j) + \varepsilon_{ij}, \quad \varepsilon_{ij} \stackrel{\text{i.i.d.}}{\sim} \mathcal{N}(0,1), \quad u_{ij}^{*} = u^{*}(X_{ij}, t_j).
\end{align}
Interestingly, we may conceptually partition the resulting training dataset $S$ into two components:
 (1) {\bf Labeled Data}
(i.e., standard regression data for the target $u^*$):
$S^{labeled} = 
\{(x_{ij}, t_j,  u^{*}_{ij})\}_{i=1,j=1}^{m,n}$, 
and (2) {\bf Auxiliary Data} (consisting of additional (noisy)
observations of the source term): $S^{aux} = 
\{(x_{ij}, t_j,  f_{ij})\}_{i=1,j=1}^{m,n}$.
The observational model~\eqref{regression_model} and~\eqref{regression_model_2} we consider here aligns with many real-world scenarios. In physical systems governed by PDEs, $f(x,t)$ is often unknown or only indirectly observed through sensors \citep{JRCannon_1986, arridge2019solving, hasanouglu2021introduction,iglesias2013ensemble,nickl2024bernstein}, which are typically noisy, while quantities represented by the solution $u^{*}(x,t)$, such as temperature in the reaction-diffusion equation, are more readily measurable. Furthermore, the governing PDE imposes additional structure on the learning problem. Characterizing how to exploit the structural richness of the linear parabolic PDEs in learning their dynamics is a fundamental problem in this area.

\subsection{Our Contribution}

In this paper, we address the fundamental challenge of incorporating PDE constraints into data-driven learning and demonstrate improved generalization and learning efficiency.
We develop an extended supervised learning framework for learning the dynamics of the heat equation, a classical linear parabolic PDEs. Moving beyond the traditional regression approach based solely on labeled solution data $\{x_{ij}, t_j, u^*(x_{ij}, t_j)\}$, which aims to learn a mapping $(x, t) \mapsto u^*(x,t)$, our framework incorporates \emph{energy-guided learning with auxiliary data} $S^{aux}$. 

Our key idea is to use noisy observations of the source function $f$ 
to construct an approximate energy functional via an \emph{elliptic regularization} of the dynamical system~\eqref{eq:heat_pde}, resulting in a PDE learner with improved regularity properties.
By coupling these auxiliary observations with labeled solution data, we obtain a penalized estimator that leverages both statistical supervision and the underlying physical structure.
Conceptually, this approach corresponds to regression in Sobolev spaces under energy-based regularization.

On the theoretical front, we show that this combination is necessary: purely supervised estimator (with labeled solution data) and the approximate Ritz method (with auxiliary data) are both suboptimal in sample complexity when the degree of spatial smoothness of the solution is below a threshold, whereas our energy-guided method with a suitable balance between data-fitting and physics constraint simultaneously achieves sharp estimation convergence rate, excess risk and generalization error bounds. Next, we turn to a more detailed summary of these contributions. 

\textbf{A new variational formulation.} We introduce a new variational formulation for the heat equation. The idea is to characterize its solution as the minimizer of an energy-regularized, ellipticity-debiased Sobolev regression objective. Moreover, the derived total loss function with data fidelity and physics prior enjoys a stability estimate. This physics-informed energy-regularized formulation naturally integrates labeled and auxiliary data, which together provide global control in Sobolev spaces and enable efficient learning of weak solutions that least-PDE-residual methods (e.g., PINNs) fail to handle. Even though our theoretical analysis is limited to the heat equation with unknown source, we expect that the proposed energy-guided approach with elliptic regularization technique can be adapted to more general linear parabolic and hyperbolic PDEs (e.g., with techniques in~\cite{article}). Towards this goal, we show some strong empirical performance of the proposed method against various PINNs when applied to those equations, cf. Section~\ref{more_experiments} and Appendix~\ref{sec:lp-method-comparison}.


\textbf{Optimal estimation and generalization error bounds.} We establish a sharp statistical estimation and generalization error analysis for learning solution $u^{*}$ to the heat equation. Our statistical estimation error exhibits an interesting smoothness-aware adaptation of the proposed energy-guided approach. On one hand, if the solution of the heat equation is spatially smooth (i.e., the subcritical regime $q > d/2$ in Theorem~\ref{thm_3.1}), then labeled data alone can achieve the optimal rate of convergence (cf. Theorem~\ref{minimax_lower_bound}). On the other hand, if the spatial smoothness of the heat equation is below the threshold $q \leq d/2$, then convergence rates depend on the balance of data fidelity and energy regularization. When the regularization parameters is chosen optimally, the rate attains the minimax lower bound. In addition, we also derive a fast generalization error bound via a delicate localization technique.

\textbf{Performance Improvement.}
To corroborate our theoretical findings, we provide substantial experimental results to see the improvements on generalization\ref{sec:controlled-parabolic-regularization}. Our approach reduces the test error of 50d Hamilton Jacobi equation by up to $\textbf{79.5}\%- \textbf{89.4}\%$ relative to strong baselines in physical learning literature.


    
Overall, this work contributes to the growing body of literature in addressing the limitations of PINNs \citep{Krishnapriyan2021_PINN-failure-modes, WANG2022110768, wang2022is}. Our energy-guided approach is particularly strong at capturing transient physical phenomena when initial data $u_0^{*}$ exhibit singularities, while PINNs fail to model them. By characterizing these failure modes, we offer the perspective that integrating labeled and auxiliary data is a useful way to improve stability and convergence of neural solvers in dynamical systems.

\subsection{Relevant Work}
\textbf{Physics-Informed Machine Learning.}
Despite the empirical advances of numerous PINN variants across diverse applications~\citep{McClenny2023SelfAdaptivePINN,Shukla2020UltrasoundPINN,Patel2022ThermoPINN,Pfau2020AbInitio,doi:10.1137/22M1522504,kharazmi2021hp,wang2022is}, a rigorous theoretical understanding of machine learning based PDEs solvers remains largely vacant. In recent works, \cite{lustatistical,Chenguang_Duan_2022,Yuling_Jiao_2022} have established minimax optimal convergence rates for elliptic PDEs comparing the performance of the Deep Ritz Method (DRM) and PINNs. Another line of work focuses on the green function hidden in differential operator mapping $u$ to $f$ for elliptic and parabolic PDEs~\citep{doi:10.1073/pnas.2303904120,NEGI2024116779,fdef534d8b094cc4aed95f77c4fb9244}.

\textbf{Variational Methods and Energy Regularization.} In calculus of variation literature, variational methods to solve elliptic PDEs can be extended to solve not only elliptic PDEs but also parabolic and even hyperbolic ones \citep{article}. There are extensive works \citep{Mielke11, doi:10.1142/S0218202525400019,  Audrito_2023, AKAGI2025113431, AKAGI20112541, ROSSI20191} in this direction to establish approximate solution from regularized PDEs with various boundary conditions, which are relatively easier to tackle with compared to original PDEs problem. In our case, the penalized PDE, which we are concerned with, is called the telegraph heat equation, a hyperbolic modification of heat equation \citep{dehghan2008numerical, orsingher2004time, goldstein1951diffusion}. 

\textbf{Semi-Supervised Learning in Physical Systems.}
In many real-world scenarios, the source function $f(x,t)$ is either unknown or indirectly observed through noisy sensors, whereas the solution $u^*(x,t)$ is more readily observable. We formulate this learning task as a semi-supervised problem where the solution labels $u^*$ provide rich empirical information while 
the auxiliary information 
$\{f(x,t)\}$
enables an energy-based regularizer that enforces the underlying physical constraints. Similar to our approach, \cite{Cui2024ASL, 10.5555/2789272.2912113} employ partially observed labels to learn solutions to complex systems such as minimal surface equations and Hamilton-Jacobi equations.

\section{Method}
\label{section_2}


\subsection{Design of Unbiased and Stable Loss Function}

There are two key ingredients to build a loss function whose minimization can lead to \textit{efficient} learning of the solution function to the heat equation~\eqref{eq:heat_pde} from the data sample~\eqref{regression_model}. First, viewing the heat equation as an input-output system $f \mapsto u$, we can determine the solution $u$ from the source function $f$, which we have the auxiliary data access to. It is well-known that solution to the time-independent Poisson equation $\Delta u = f$ is the minimizer of the Dirichlet energy functional $I(u) = \int_{\Omega} [{1\over2}|\nabla u(x)|^2 - f(x) u(x)] dx$. For the heat equation, such variational formulation does not exist and we adopt the following weighted energy functional (WEF) as a proxy, defined for $\epsilon > 0$,
\begin{equation}
\label{energy_functional}
    I_{\epsilon}(u):= \int_{0}^{T} e^{-\frac{t}{\epsilon}} \int_{\Omega} \left[ \dfrac{1}{2} |u_{t}|^{2}(x,t) +  \frac{1}{2\epsilon} |\nabla u(x,t)|^{2} -\frac{1}{\epsilon}  f(x,t)u(x,t) \right] dxdt.
\end{equation}
Our motivation of WEF comes from observing that $I_{\epsilon}(u)$ is \textit{strongly convex} in $u$ and we can identify its first-order optimality condition via the Euler-Lagrange equation given by the hyperbolic PDE
\begin{equation}
\label{eqn:euler_lagragian_elliptic_regularization}
    -\epsilon \, u_{tt} + u_{t} - \Delta u = f \text{ for } (x,t) \in \Omega \times (0,T),
\end{equation}
cf. Appendix~\ref{derivation of Euler-Langragian} for derivation. In literature,~\eqref{eqn:euler_lagragian_elliptic_regularization} is called the \emph{elliptic regularization} of the heat equation~\citep{Lions09}, which originated in the study of geometric evolution equations~\citep{Ilmanen}. When $\epsilon$ is small, the relaxed solution $u^*_{\epsilon}$ of~\eqref{eqn:euler_lagragian_elliptic_regularization} enjoys the uniform-in-time approximation for the heat equation solution $u^*$~\citep{Mielke11,ROSSI20191} in the sense that: for all $t > 0$, $u^{*}_{\epsilon_{n}} \to u^{*}$ in $L^{2}(\Omega)$ for a subsequence $\epsilon_{n}\to 0$.

Second, from a statistical learning perspective, we can construct a data-fitting term using the features $(x, t)$ to predict the solution $u^*(x, t)$. It turns out the appropriate loss metric to compensate the WEF curvature is the Sobolev geometry in the solution or label space (cf. Proposition~\ref{prop:stability} below for precise statement in terms of stability). Specifically, we choose 
\begin{equation}
\label{eqn:Sobolev_metric}
    D(u,u^{*}):= ||u-u^{*}||_{H^{1}_{t}H^{1}_{x}}^{2},
\end{equation}
where we recall $H^{1}_{t}H^{1}_{x} = H^{1}_{t}(0,T;H^{1}_{x}(\Omega))$. Note that this $H^{1}_{t}H^{1}_{x}$ geometry has close connection to the Sobolev training in literature~\citep{10.5555/3294996.3295182,lu2022sobolev}. Combining the two pieces of information together, we propose the population total loss defined as
\begin{equation}
L_{\epsilon,\beta}(u):= D(u,u^{*})+\beta  \big[I_{\epsilon}(u) + M_{\epsilon}
(u, u^{*})\big],
\label{loss_function}
\end{equation}
where $\beta > 0$ is a tuning parameter to balance data fidelity and physics constraint, and 
\begin{equation}
\label{eqn:model_correction_term}
M_\epsilon(u,u^*)
:=\int_0^T\!\int_\Omega
e^{-t/\epsilon}u(x,t)u_{tt}^*(x,t)\,dx\,dt
-e^{-T/\epsilon}\int_\Omega
u(x,T)u_t^*(x,T)\,dx.
\end{equation}
is an extra bias correction term due to the ellipticity offset from~\eqref{eqn:euler_lagragian_elliptic_regularization}. With debiasing, we can show that the (exact) heat equation is the minimizer of the total loss function $L_{\epsilon,\beta}(u)$.

\begin{theorem}[Minimizer of the total loss function $L_{\epsilon,\beta}(u)$]
\label{euler_lagrangian_for_L}
Let $s \geq 1$ and $p \in [0, s]$. Then, the solution $u^*$ to heat equation in $(\ref{eq:heat_pde})$ is the minimizer of the total loss function $L_{\epsilon,\beta}(u)$ over $\U_0 := H^{p}_{t}(0,T;H_{x}^{2s+2-2p}(\Omega)) \cap L^{2}_{t}(0,T;H^{1}_{0,x}(\Omega)) \cap \{u : \Omega \times [0,T] \to \mathbb{R} | u(x,0)=u^{*}_{0}(x)\}$, i.e., $u^* = \arg\min_{u \in \U_0} L_{\epsilon,\beta}(u)$ for any $\epsilon, \beta > 0$.
\end{theorem}
We remark that $p$ is the temporal regularity parameter and $s$ is the total space-time smoothness parameter. To summarize, our learning objective has three components: (i) a data-fitting term $D(u,u^{*})$ that provides the distance between a candidate solution function $u$ and the true label $u^{*}$, (ii) a WEF term $I_{\epsilon}(u)$ that navigates the underlying PDE dynamics during training, and (iii) a model debiasing term $M_{\epsilon}(u, u^{*})$ that has a negligible impact in the sample version. Detailed derivations and discussions on our loss function can be found in Appendix~\ref{motivation_of_loss_function}. Moreover, we can prove that the total loss function $L_{\epsilon,\beta}(u)$ has a good lower and upper curvature around its minimizer $u^*$.


\begin{proposition}[Stability estimate of the total loss function $L_{\epsilon,\beta}(u)$\label{prop:stability}]
Let $u^{*}$ be the true solution to the heat equation defined in (\ref{eq:heat_pde}). Then for any $u\in H^{1}_{t}H^{1}_{x}$, we have
\begin{equation*}
M_{1}(\beta, \epsilon)    ||u - u^{*}||_{H^{1}_{t}H^{1}_{x}}^{2} \leq L_{\epsilon,\beta}(u) - L_{\epsilon,\beta}(u^{*}) \leq M_{2}(\epsilon,\beta) ||u - u^{*}||_{H^{1}_{t}H^{1}_{x}}^{2},
\end{equation*}
where 
\begin{equation}
M_{1}(\beta,\epsilon):= \frac{1}{3} + \min( \frac{1}{3}, \frac{\beta e^{-\frac{\epsilon}{T}}}{C(d,\epsilon, \Omega)},\frac{\beta e^{-\frac{\epsilon}{T}}}{4} ), \qquad
M_{2}(\beta, \epsilon):= 1+\frac{\beta}{2} + \frac{\beta}{2\epsilon},
\end{equation}
and $C(d,\epsilon, \Omega) > 0$ is a constant depending on $d,\epsilon, $ and $\Omega$.
\end{proposition}

We highlight that stability is a central structure to establish error estimation in PDE solutions in terms of the input data (e.g., initial and boundary conditions, source function) and it is also a cornerstone of statistical learning theory to prove sharp rate of convergences in estimation (cf. Theorem~\ref{thm_3.1}) and generalization error bounds (cf. Theorem~\ref{thm_3.4}).

\subsection{Empirical Loss Minimization} Given the data $\{(X_{ij}, t_j, f_{ij}, u^{*}_{ij})\}_{i=1,j=1}^{m,n}$ and their computed derivatives $
\{\nabla u^*,\partial_tu^*,\nabla\partial_tu^*
,u_{tt}^{*}\}$, it is natural to minimize the empirical version of the total loss in~\eqref{loss_function} as
\begin{equation}
\label{empirical_loss_function}
\hat{u}_{\theta} = \underset{u\in \mathcal{U}_J}{\arg\min}\,\, \left\{ \hat{L}_{\epsilon,\beta}(u):= \frac{1}{nm} \bigg[ \hat{D}(u,u^{*})+\beta\big( \hat{I}_{\epsilon}(u) +  \hat{M}_{\epsilon}(u, u^{*}) \big)\bigg] \right\},
\end{equation}
where $\mathcal{U}_J$ is an appropriately chosen function space such as parameterized neural networks or an orthogonal subspace spanned by the truncated local basis in~\eqref{parametrized_function_class}. For the explicit formula of the empirical loss function terms, please refer to Appendices~\ref{motivation_of_loss_function}.

Note that the empirical loss minimization procedure~\eqref{empirical_loss_function} involves data from the \emph{noiseless} solution and its derivative through numerical computation (e.g., calculated by the Python \texttt{findiff} library with high-order accuracy). This is a consequence of Sobolev norm underlying  \eqref{empirical_loss_function}. The way to compute these derivatives follow previous works in physical informed learning literature \cite{long2019pdenet2,rudy2017data,baydin2018automatic}. In our theory (Section~\ref{sec:theory}), the solution data and its numerically computed derivatives are treated as noiseless and consistent with the statistical frameworks 
for learning elliptic equations under Sobolev-type losses~\citep{lu2022sobolev}.
Since derivative computation is known to amplify the noise, we complement our experiments with ablation studies that evaluate the sensitivity of the proposed estimator to noise in both the solution and derivative data.  In practice, our experiments also show using only solution labels can achieve significant performance improvement. 

\section{Theoretical Results}
\label{sec:theory}

For convenience of our theoretical analysis, we use rescaled regularization parameter $\lambda = \beta / (nm) \label{lambda_beta_scale}$ in the empirical total loss function in~\eqref{empirical_loss_function} and consider a subspace in the span of truncated local basis in space and time
\begin{equation}
\label{parametrized_function_class}
\mathcal{U}_{J}:= \bigg\{u \in \mathcal{U}_0 \,\bigg{|} \, \left\| \frac{d^{p}u}{dt^{p}} \right\|_{L^{2}_{t}(0,T;H^{2s+2-2p}_{x}(\Omega))}^{2}\leq R \bigg\} \cap V_{J,J^{\prime}},
\end{equation}
where 
\begin{equation}
\label{V_class}
V_{J,J^{\prime}}: = \bigg\{u : \Omega \times [0,T] \to \mathbb{R} \,\bigg{|} \,u(x,t) = \sum_{|\xi|_{\infty} \leq J, \tau \leq J^{\prime}} a_{\tau,\xi}^{2} \phi_{\tau,\xi}^{2}(x,t) \bigg\}.
\end{equation}
Note that one can view local basis as decoupled basis acting on space and time dimension, respectively $\phi_{\tau,\xi}(x,t):= \phi_{\xi}(x) \phi_{\tau}(t)$ to satisfy certain properties. In our analysis, we choose the local basis to be wavelet expansions that capture the local information in the dynamics. For more details about these function space and construction of estimator $\hat{u}_{J}$ and local basis $\phi(x,t)$, please refer to Appendices~\ref{sec:main_assumptions} and~\ref{sec:truncated_fourier}. In our numerical experiments in Section~\ref{numerical_experiments}, we use simple multi-layer perceptrons (MLPs) to compare with other standard neural solvers such as PINNs.

To derive theoretical results, we work with the standard statistical sampling model in which the space-time point locations \(\{(X_{ij}, t_j)\}_{i=1,j=1}^{m,n}\) are sampled independently and uniformly from distribution $P_{0}$, where $P_{0}$ is the product Lebesgue measure on $\Omega\times[0,T]$. 

\subsection{Estimation Error Bounds}

We denote $N:=nm$ as the total sample size. Our first main result is an upper bound for the empirical loss minimization procedure of the energy-guided estimator in~\eqref{empirical_loss_function}.

\begin{theorem}[Convergence rates of energy-guided estimator]
\label{thm_3.1}
Under Assumptions 1-3 in Appendix~\ref{sec:main_assumptions} and the sampling model~\eqref{regression_model} and~\eqref{regression_model_2}. We have the following bounds on the excess risk.


\label{upper_bound}
\textbf{Case I: supercritical regime $q\leq\frac{d}{2}$.} The error is further divided into three sub-regimes:
\begin{equation}
\E||\hat{u}_{J}- u^{*}||_{H^{1}_{t}H^{1}_{x}}^2 \lesssim_{\Omega_T,R}
\begin{cases}
\Bigl(\frac{\log (1+C_{2}N)}{N}\Bigr)^{\frac{4s-2}{4s+2d+2}}, \hspace{1.3cm}\text{if }\, 0\leq \lambda \leq \frac{e^{T/\epsilon}}{N}\\\\
\Bigl(\frac{\log (1+C_{2}N)e^{T/\epsilon}}{N^{2}\lambda}\Bigr)^{\frac{4s-2}{4s+2d+2}}
, \hspace{0.7cm}\text{ if } \frac{e^{T/\epsilon}}{N}<\lambda\leq\lambda^{*}
\\\\ \frac{\lambda}{e^{T/\epsilon}} ,
\,\,\hspace{3.4cm}\text{ if } \lambda^{*}\leq\lambda<\infty
\end{cases},
\end{equation}
where \begin{equation}\lambda^{*}:=(\sqrt{\log (1+C_{2}N)}/{N})^{\frac{4s-2}{4s+d}}e^{T/\epsilon}.
\end{equation}

\textbf{Case II: subcritical regime $q>\frac{d}{2}$.} For all $\beta\geq0$, we have the following error bound
\begin{equation}
\mathbb{E}||\hat{u}_{J}- u^{*}||_{H^{1}_{t}H^{1}_{x}}^2\lesssim_{\Omega_T,R}
\Bigl(\frac{\log (1+C_{2}N)}{N}\Bigr)^{\frac{4s-2}{4s+d}}.
\end{equation}
\end{theorem}

\begin{remark}
The relationship between $n$ and $m$ captures the parabolic scaling in the heat equation, where one temporal derivative is balanced against two spatial derivatives. This relation is also implicitly reflected in our proof using parabolic multi-level truncation $(J, J^{2})$ for space-time representation in the wavelet space. More precisely, we need $n\approx m^{2/d}$ in Proposition~\ref{prop:parabolic-balance} for Theorem~\ref{thm_3.1} to hold.
\end{remark}

Depending on the regularity of the ground-truth solution $u^{*}$, it is interesting to see from Theorem~\ref{thm_3.1} that there are two regimes for the statistical convergence rate. In the subcritical regime, the prescribed regularity is sufficiently strong relative to the dimension for global Sobolev control to yield suitable metric 
entropy bounds. Consequently, the estimation problem behaves essentially as a standard nonparametric regression problem over the corresponding Sobolev class, and the energy-guidance term does not improve the leading-order statistical rate. In contrast, in the supercritical regime, the same prescribed regularity is insufficient, relative to the dimension, for such global Sobolev arguments to provide a sharp characterization of the statistical complexity. A more localized analysis, based on wavelet representations and the PDE energy structure, is then required. In this regime, energy guidance can exploit the local dynamical information encoded by the equation and lead to a sharper rate of convergence. In addition, it turns out that by properly choosing the tuning parameter $\lambda^*$, the rate of convergence for energy-guided learning with auxiliary data is minimax optimal up to a logarithmic factor, justified by Theorem~\ref{minimax_lower_bound} below.




\begin{theorem}[Lower bound on estimation error]
\label{minimax_lower_bound}
Fix $\epsilon>0$,\,\,$s\geq 1$ and $2p+q=2s+2$ for $0\leq p \leq s$. Under Assumptions 1-4 in Appendix~\ref{sec:main_assumptions}, we have the following result holds
\begin{equation}
\label{eqn:minimax_lower_bound}
 \inf_{\hat{u}(x,t)}\sup_{u^{*}(f)\in \U} \E \bigg[ ||\hat{u}(x,t) -u^{*}(f)||_{H^{1}_{t}H^{1}_{x}}^{2}\bigg] \gtrsim N^{-{\frac{4s-2}{4s+d}}}.
\end{equation}
\end{theorem}

Our minimax lower bound under the space-time Sobolev norm is a non-standard rate. However,~\eqref{eqn:minimax_lower_bound} relates to the standard nonparametric rate $n^{-\frac{2s}{2s+d-2}}$ for estimating spatially H\"older smooth functions under the $L_{2}$ evaluation metric~\citep{Tsybakov:1315296}. To interpret the difference, our obtained rate suffers regularity loss by degree of three: one from spatial derivative and other two from temporal derivative, witnessing one versus two space-time regularity trade-off by using the $H^{1}_{t}H^{1}_{x}$ metric. 

\subsection{Excess Risk and Generalization Error Bounds}
There are two commonly used performance measures to evaluate the statistical and generalization capability of a learned model: (1) the \textit{excess risk} $\Delta\mathcal{E}_{S}:=L_{\epsilon,\beta}(\hat{u}_{J})-L_{\epsilon,\beta}(u^{*})$, reflecting the difference between the expected losses between a learned model and the best possible model (i.e., oracle); (2) the \textit{generalization error} $\Delta\mathcal{E}_{S,gen} :=\hat{L}_{\epsilon,\beta}(\hat{u}_{J})-L_{\epsilon,\beta}(\hat{u}_{J})$, quantifying the gap between the training and testing loss for the \emph{learned model} on the training data. Bounds on excess risk and generalization error both measure the quality of prediction on unseen testing data. These two quantities are related and we remark that the excess risk bound is also a key step to establish the convergence rates of energy-guided estimator in Theorem~\ref{thm_3.1}.



\begin{theorem}[Excess risk bound]
\label{thm_3.3}
Suppose the assumptions in Theorem~\ref{thm_3.1} hold. For the empirical minimizer of the total loss function $\hat{u}_J$, we choose $J= J_{1}^{*}$ in \eqref{J_*1} if $p\leq \frac{d}{2}$ and $J = J_{2}^{*}$ in \eqref{J_*2} if $q > \frac{d}{2}$. Then the excess risk satisfies the following expectation bound over the training data distribution
\begin{equation}
\mathbb{E}\,\,[\Delta \mathcal{E}_{S}] \lesssim  \frac{e^{T/\epsilon}}{\epsilon}\Bigl(\frac{\sqrt{\log (1+C_{2}N)}}{N}\Bigr)^{\frac{4s-2}{4s+d}},
\end{equation}
where the constant of $\lesssim$ depends only on $\Omega_T$ and $R$.
\end{theorem}

\begin{theorem}[Generalization error bound]
\label{thm_3.4}
Suppose the assumptions in Theorem~\ref{thm_3.1} hold, the ground truth solution $u^{*}$ to the PDE in (\ref{eq:heat_pde}), and $\hat{u}_{J}$ be empirical minimizer of the total loss function in~\eqref{empirical_loss_function} on the training data. We further assume $q> \frac{d}{2}+1$ and $p\geq 2$, then the following result holds 
\begin{align}
\E[\Delta \mathcal{E}_{S,gen}]  \lesssim J^{\frac{d+2} {2}}\cdot \log{J} \cdot \frac{\log(N)}{N},
\end{align}
where the constant of $\lesssim$ depends only on $\epsilon, d, q,R,T, \Omega$, and $F$.

\begin{remark}
\cite{pmlr-v134-lu21a} proved the slow generalization error using a one-layer neural network with $m$ neurons for elliptic PDEs, and \cite{lustatistical} improved this bound to fast generalization error with $J$-truncated Fourier series instead. Their generalization bounds also suffer
from polynomial growth of model parameters in terms of $m$ or $J$. Our Theorem \ref{thm_3.4} achieves the fast rate of generalization error in terms of the effective sample size $n m$. Note that the upper bound on excessive risk in Theorem~\ref{thm_3.3} incorporates approximation error, while Theorem \ref{thm_3.4} only focus on statistical errors. We also impose a slightly stronger assumption $q> \frac{d}{2}+1$ to enforce higher spatial regularity so the Sobolev embedding theorem is applied to enhance local integrability.
\end{remark}
\end{theorem}

\section{Numerical Experiments}
\label{numerical_experiments}

In this section, we present some interesting examples using our energy-regularized approach for readers to understand how our approach works. The comprehensive experiment report can be found in \ref{more_experiments}. We report using our energy-regularized approach to solve four parabolic linear PDEs and one nonlinear parabolic PDEs in various dimensions up to $d=50$. In appendix \ref{sec:lp-method-comparison}, we show our methods achieve dominant accuracy advantage in $L_2$ relative in comparison with various PINNs baseline. Furtheremore, we show our energy regularization effects beats purely sobolev regression using similar sample-adaptive regularization tuning in appendix \ref{sec:controlled-parabolic-regularization}. Since we ideally assume solutions label to be noiseless and compute derivatives from solution labels, we also provide noise ablation on how additional noise injection will degrade our prediction; our ablation results show our method is robust against mild noisy solution labels and computed derivatives as well. To see the extension of our method, we provide straightforward extension to nonlinear hyperbolic PDEs in appendix \ref{extension_to_hyperbolic}. Lastly, we also present parameter sensitivity analysis (see appendix \ref{sec:parameter-sensitivity-analysis}) for readers to understand how PDEs regularization parameter $\epsilon$, $\beta$, and sample complexity $N$ dynamically entangles.

\subsection{Heat Equation}

We first consider the heat equation in the domain $\Omega \times [-1,1]^{d}$ with a zero forcing term for $d=1,3$. Here we choose the initial condition $u_0(x) = \log\log (1+|x|^{-1})\in H^{1}(\Omega)$ which has a severe singularity around $|x|=0$. In this setting, strong solution does not exist and one only has weak solution in an appropriate Sobolev space. In practice, since $u_0$ cannot be numerically computed due to singularity, we perturb around $|x|=0$ by adding a small perturbation $10^{-6}$, leading to a modified initial condition $u_0(x) = \log\log (1+(|x|^2 + 10^{-6})^{-1/2})$. For a fair comparison, we make the total sample size roughly the same between a vanilla PINN and our method, because they have distinct sampling schemes. For the 1D heat equation, Figure \ref{figure_1} shows that our solution is very accurate in the entire phase of space and time, while the PINN solution fails. Moreover, Figure~\ref{figure_2} shows that PINN prioritizes learning snapshots around the singularity $|x|=0$ at initial time and does not keep track of the overall dynamics; that is, PINN does not capture physical dynamics when $t$ is away from zero. Our explanation is that the penalizing mechanism in the PINN loss design leads to the imbalance of initial condition and PDE residuals, causing the failure in capturing dynamics at larger $t>0$, since one needs to allocate different penalty weights for learning interior solutions, boundary solutions, or initial solutions. As a singularity is present at $t=0$, PINNs heavily enforce residuals to initial time points instead of overall dynamics, as shown in Figure~\ref{figure_2} (top row).

\begin{figure}[ht]
\centering\includegraphics[width=\textwidth]{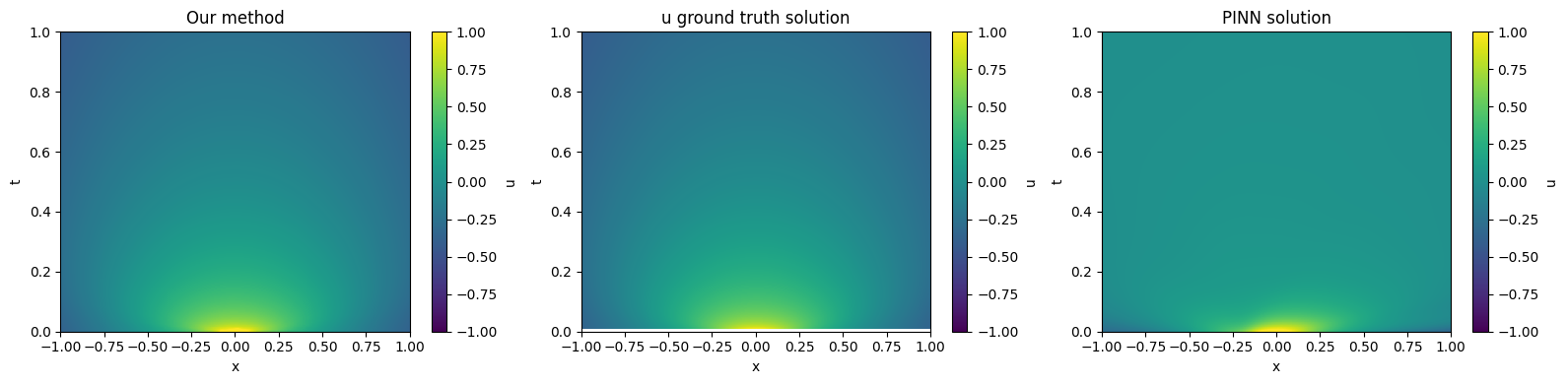}
    \caption{Comparison of solutions from our method and PINN for a 1-D heat equation.}
    \label{figure_1}
\end{figure}

\begin{figure}[ht]
\centering
\includegraphics[width=\textwidth]{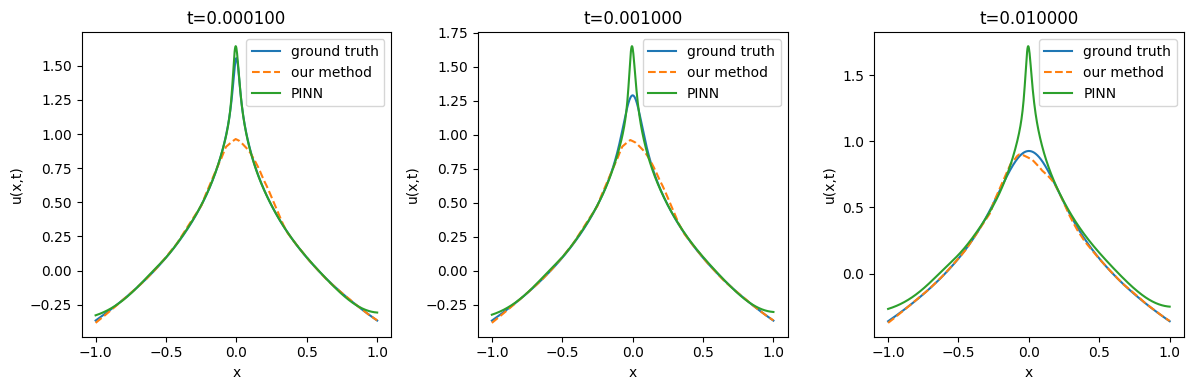}
\includegraphics[width=\textwidth]{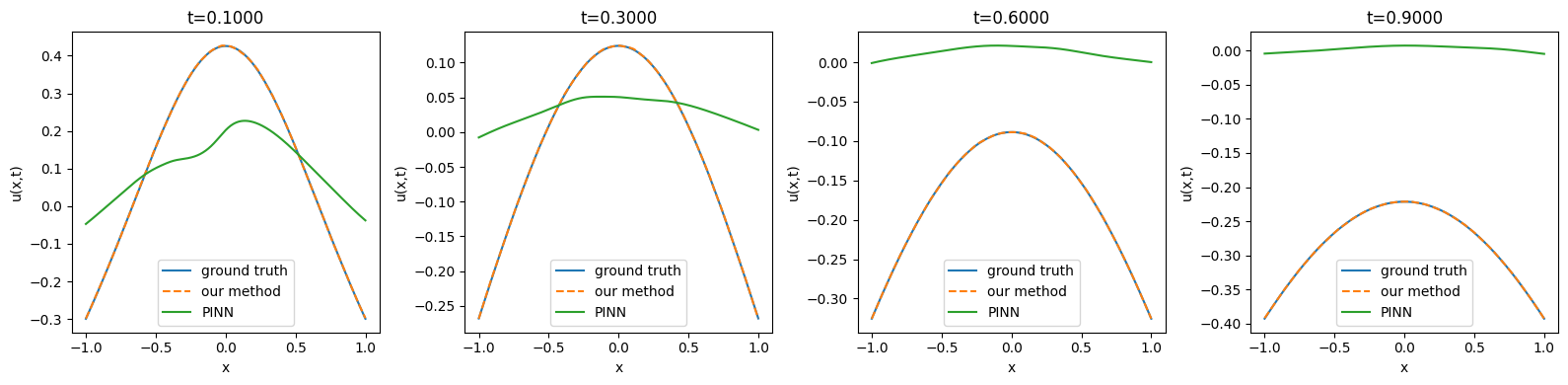}
    \caption{Top row: Near $t\approx0$ snapshot comparison of solutions from our method and a PINN for a 1-D solution. Bottom row: zoom-in snapshot comparison of solutions from our method and a PINN different times.}
    \label{figure_2}
    \vspace{-0.2in}
\end{figure}

We observe similar phenomena in higher-dimensional experimental settings. Figure \ref{fig_4} shows a 3D scenario in which PINN fails to solve the heat equation. In the literature, several approaches have been proposed to improve PINNs by rebalancing loss penalty weights to handle singularities formed in shock waves and to enhance estimation accuracy for weak solutions. Although these alternative methods improve upon the baseline PINN benchmark~\citep{doi:10.1137/22M1522504,kharazmi2021hp,McClenny2023SelfAdaptivePINN,wang2022is}, they remain underperformed relative to our approach by several orders of magnitudes in testing errors in Appendix~\ref{A.7}. Given that strong solutions are generally easier to approximate, we anticipate that those competing methods will correspondingly fall short of the performance of our model in more complex scenarios where only weak solutions exist, e.g., with singularity initial conditions.

\begin{figure}[ht]
\centering\includegraphics[width=0.32\textwidth]{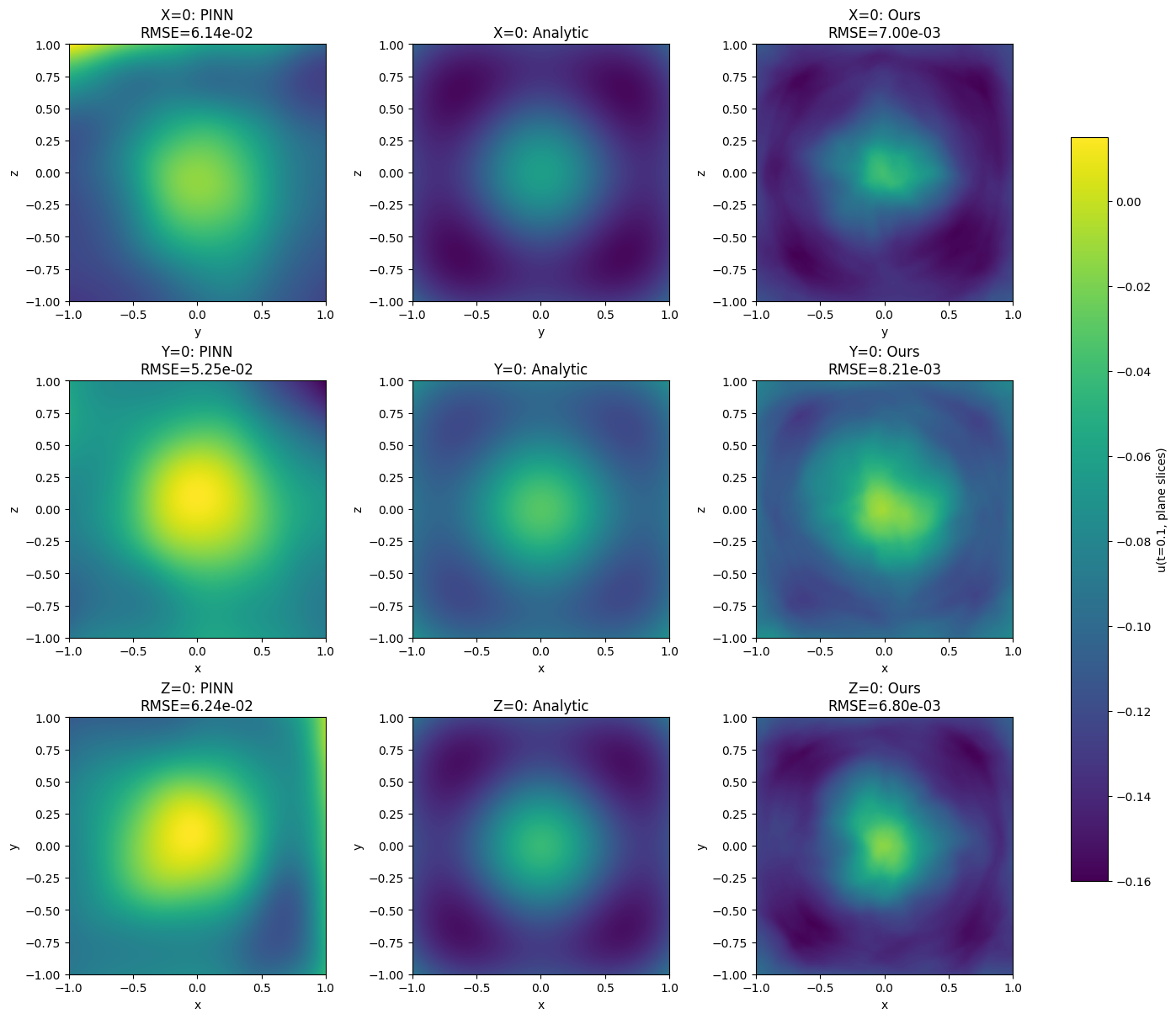}
\centering\includegraphics[width=0.32\textwidth]{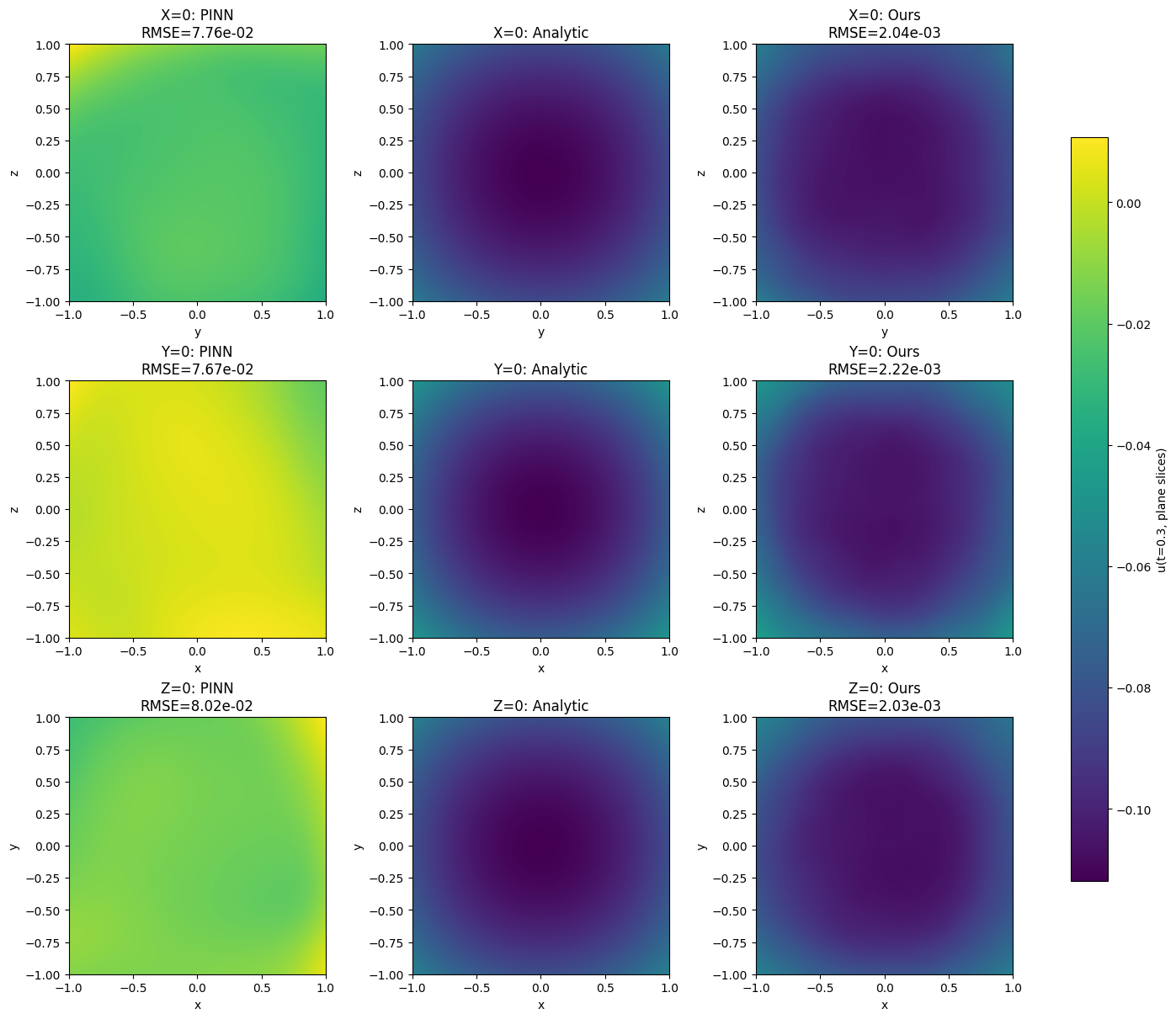}
\centering\includegraphics[width=0.32\textwidth]{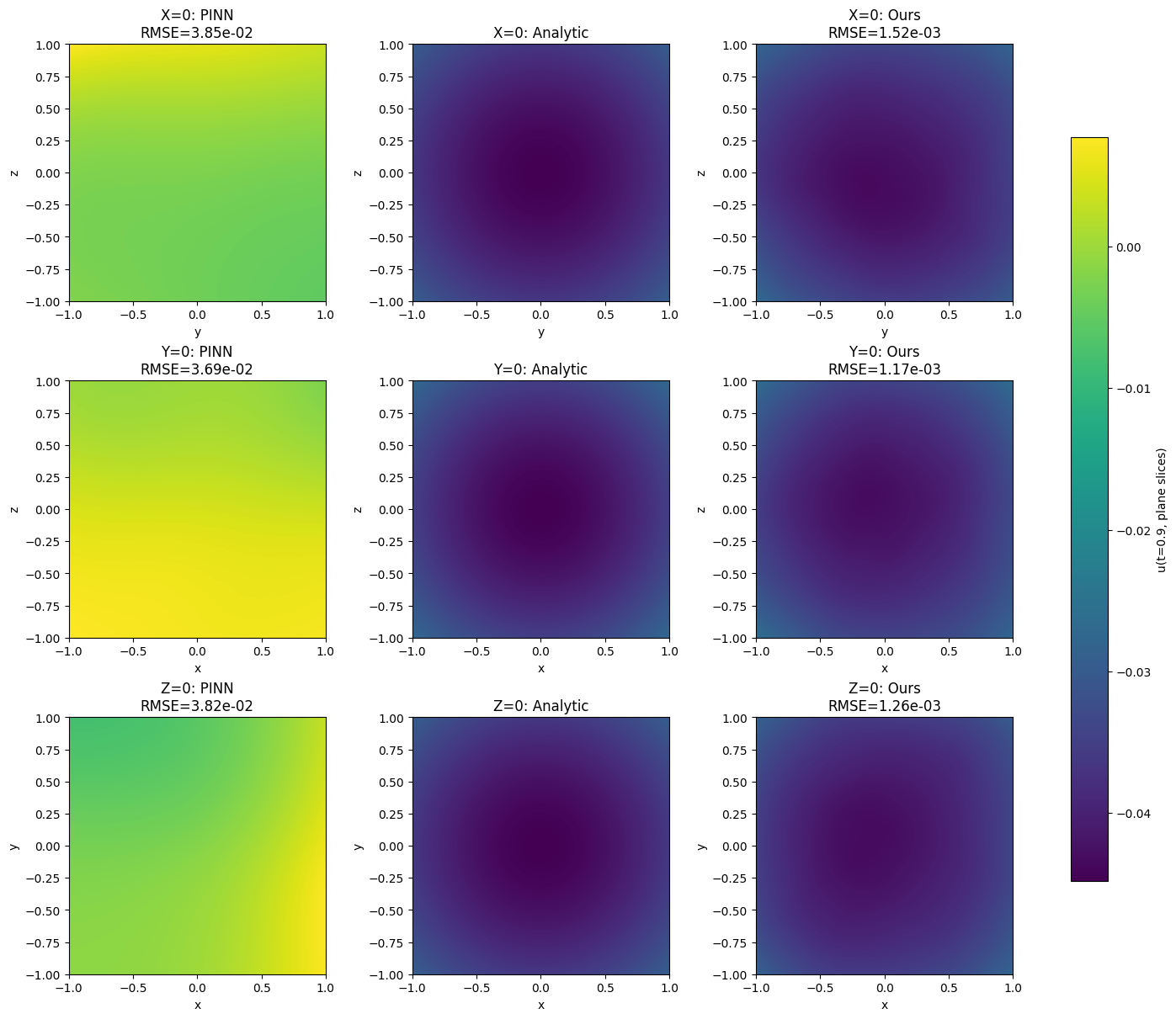}
 \caption{3D heat equation example. Left plot is 3D solution projection onto two coordinates at $t=0$. Middle plot is 3D solution projection onto two coordinates at $t=0.3$. Right plot is 3D solution projection onto two coordinates at $t=0.9$.}
 \label{fig_4}
 \end{figure}

\subsection{3D Fokker-Planck Equation}
\label{subsec:fokker_planck}

In this experiment, we show the applicability of our energy-guided method for more general linear parabolic PDEs. Specifically, we consider the Fokker-Planck equation $u_t + \nabla \cdot \big(b\,u\big) - \nabla\cdot \big(D\,\nabla u\big)=0$ with standard Gaussian initial condition $u(x,0) = (2\pi)^{-1/2} e^{-|x|^2}$. In this setting, we use a constant drift coefficient $b=-1$ and diffusion coefficient $D=0.25$. Throughout the time evolution, our solution roughly preserves mass conversation inherited by physics law in the Fokker-Planck equation. The learned solutions of our method compared with the ground truth solutions are shown in Figure~\ref{fig:fokker_planck}.

\begin{figure}[ht]
\centering\includegraphics[width=0.5\textwidth]{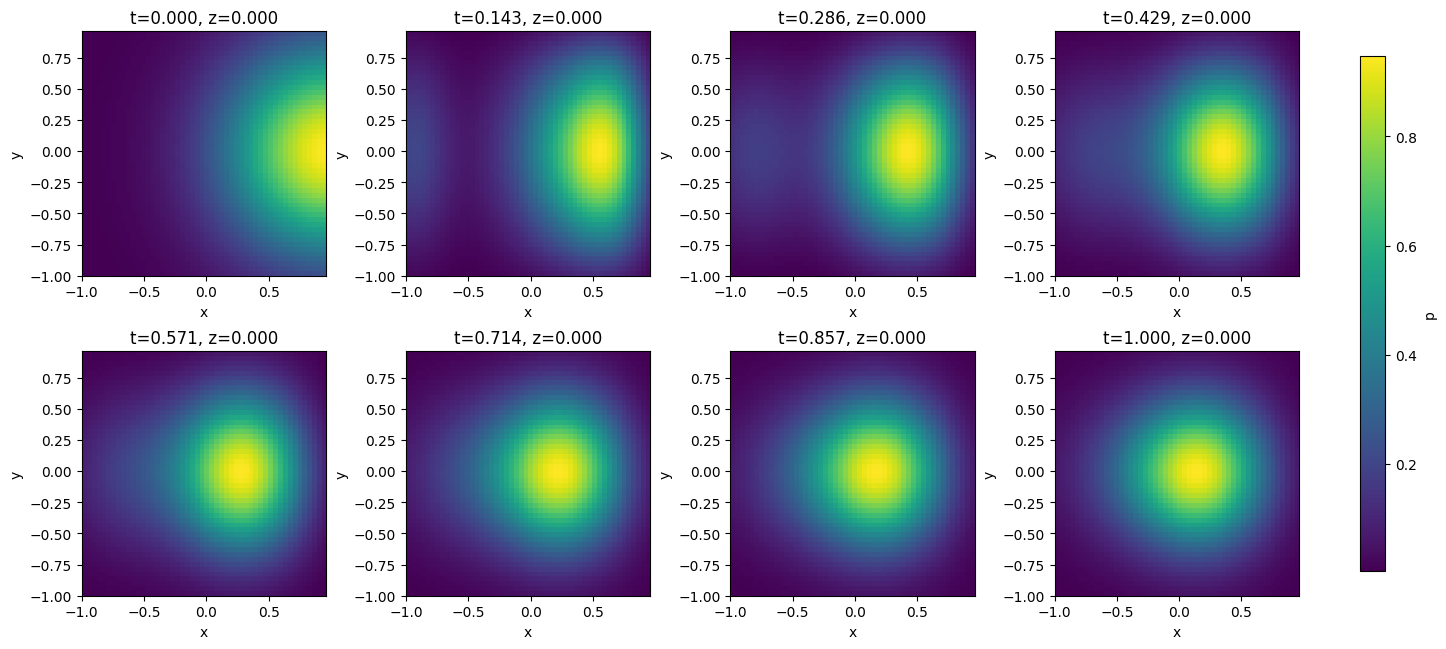}\includegraphics[width=0.5\textwidth]{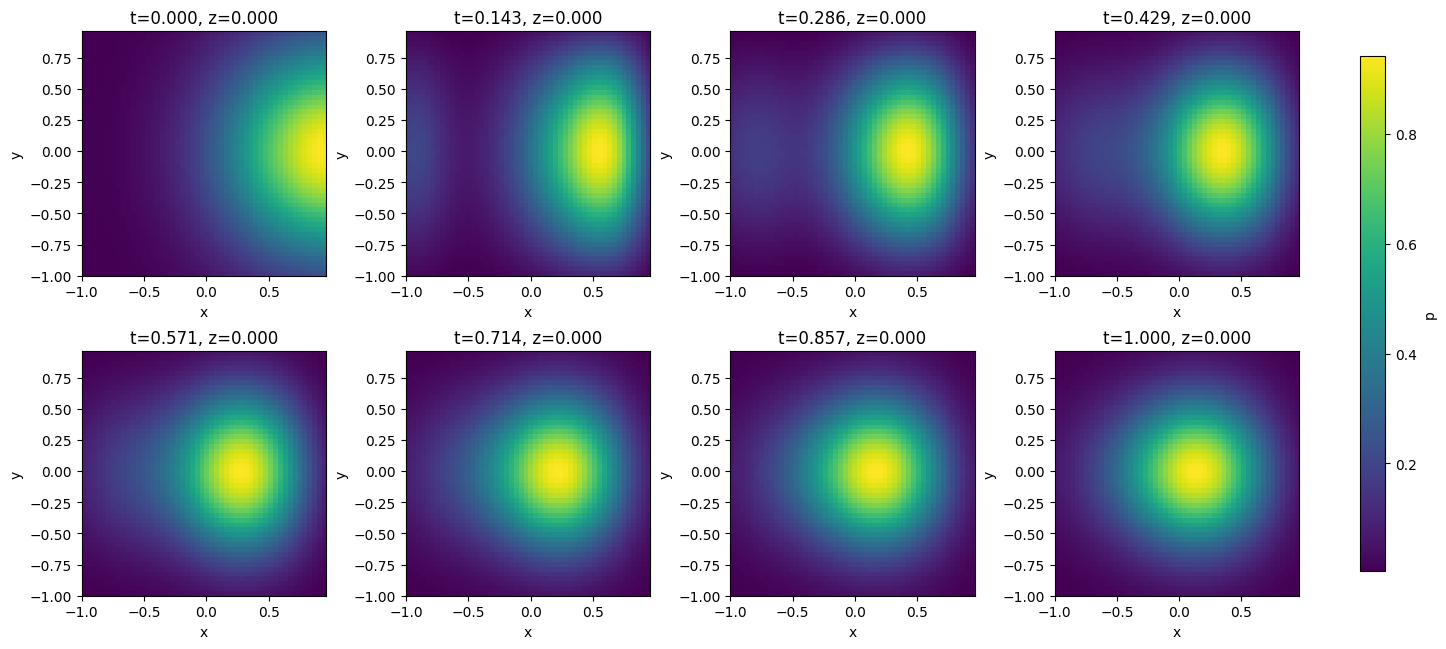}
    \caption{3D Fokker-Planck equation example. Left plot is ground-truth solution projection onto x,y coordinates over time evolution. Right plot is our solution projection onto x,y coordinates over time evolution.}
    \label{fig:fokker_planck}
\end{figure}

\section{Discussion}
 Our work studies variational energy regularization for data-driven
approximation of parabolic PDEs. Though our theoretical sheds some light in the guidance of energy regularization in physical learning, we admit our framework could be better implemented in suitable architecture including optimizer and neural network.Current popular architecture in time-dependent physical dynamic prediction only uses MLPs where temporal information is not well-represented in architecture realization. We believe networks with separate spatial and temporal representations can better capture parabolic regularity than standard MLPs. Current research on foundation models on  dynamics system \cite{mccabe2025walrus} is worth investigation to understand how these architectural choices interact with energy
regularization to affect approximation error, optimization, and generalization remains an important direction.

\bibliographystyle{plainnat}
\bibliography{ML4PDE}

\newpage


\appendix

\begin{center}
    \textbf{\Huge Appendix}
\end{center}

\section{Notations and Assumptions}

It this section, we collect the notations and the main assumptions that are used in this paper.

\subsection{Notations}
\label{sec:notations}
Throughout the paper, we let the domain be closed bounded set $\Omega \in \mathbb{R}^{d}$ with Lipschitz boundary $\partial \Omega$ over finite time horizon $[0,T]$, and spatial-temporal domain be $\Omega_{T}:=\Omega \times [0,T]$. we denote implicit regularization parameter $0<\epsilon<1$, total smoothness parameter $s$, temporal regularity parameter $p$ and spatial regularity parameter $q$ such that  $s\geq 1$ and $2p+q=2s+2$ and $0\leq p \leq s$, and the size of Sobolev norm of $u$ $R>1$, and denote the functional spaces $\mathcal{U}:=\mathcal{U}(R)$, $\mathcal{F}:=\mathcal{F}(R)$, and $H^{1}_{t}H^{1}_{x}:=H^{1}_{t}(0,T;H^{1}_{x}(\Omega))$. Unless we make additional specifications, we follow the following notation:
\begin{align*}
d &: \text{spatial dimension},\\
\Omega &: \text{closed bounded domain with Lipschitz boundary $\partial \Omega$} \\
\epsilon &: \text{implicit regularization parameter for the regularized PDEs},\\
\epsilon_1 &: \text{radius of net for construction of covering number}
\\
J &: \text{the truncation level of the local basis representation},\\
n &: \text{the number of temporal samples}, \qquad
m : \text{the number of spatial samples}\\
N &: \text{ the number of total sample size}\\
s &: \text{the total smoothness}, \qquad q: \text{the spatial smoothness}, \qquad p: \text{the temporal smoothness }\\
\beta &: \text{explicit regularization parameter for the energy regularizer } I_{\epsilon}(u),\\
\rho &: \text{the radius of localized balls},\\
L_{\epsilon,\beta}(u) &: \text{our loss function } \epsilon \text{ and } \beta,\\
\hat{L}_{\epsilon,\beta}(u) &: \text{the empirical counterpart of } L_{\epsilon,\beta}(u),\\
I_{\epsilon}(u) &: \text{weighted energy functional},\\
\hat{I}_{\epsilon}(u) &: \text{the empirical counterpart of } I_{\epsilon}(u),\\
H_t^p H_x^q&: H^p(0,T;H^q(\Omega)), \qquad H_t^1 H_x^1: H^1(0,T;H^1(\Omega)),\\
u^*&: \text{the ground truth solution satisfying the PDE problem},\\
u &: \text{estimator},\\
\hat{u} &: \text{empirical estimator of } u^*,\\
\mathcal U &: \text{the search space of models/estimators},\\
\mathcal U_J &: \text{the search space of models/estimators with local basis representation at truncation level } J,\\
\hat{u}_J &:=
\arg\min_{u \in \mathcal U_J} \widehat{L}_{\epsilon,\beta}(u),\\
u_J &:=
\arg\min_{u \in \mathcal U_J} L_{\epsilon,\beta}(u).
\end{align*}

\subsection{Main Assumptions}
\label{sec:main_assumptions}

To follow up discussions on assumptions of this paper, we make a few comments on parameters used in our paper.  $\epsilon$ is an implicit regularization parameter that serves to regularize PDEs model. The relation between $p$, $q$, and $s$ is well-known as linear parabolic equations always gain more one more derivative as time evolves. The bounded Sobolev norm $R$ is used to control the growth of our estimator $\hat{u}_{J}$ in (\ref{parametrized_function_class}).\\

\textbf{A1}\,(Regularity Condition) 
 Let $\mathcal{F}(R)$ be the functional space of forcing function $f: \Omega_{T} \to \mathbb{R}$ and $u^{*}_{0}:=u^{*}(x,0)$ be initial condition such that 
\begin{align}
&(1)\label{assumption_1}\,\, u^{*}_{0}\in H^{2s+1}(\Omega),\\
&(2)\label{assumption_2}\,\,||\frac{d^{p}f}{dt^{p}}||_{L^{2}_{t}(0,T;H^{2s-2p}_{x}(
\Omega))}^{2}\leq R,\,\, 0\leq p \leq s.
\end{align}
\\
\textbf{A2}\,(Compatibility Condition)\label{assumption_A2} Let $u^{*}_{0}$ be initial value of linear parabolic PDEs and $f$ be right hand side of the PDEs, then $u_{0}^{*}$ and $f$ need to satisfy the following conditions
\begin{align}&(3)\label{assumption_3}=
\begin{cases}
\,\,g_{0}:=u^{*}_{0}\in H^{1}_{0}(\Omega),\,\, g_{1}:=f(0,x) - Lg_{0} \in H^{1}_{0}(\Omega),\\
\,\,,....,\,g_{p-1}:= \displaystyle\frac{d^{p-1}}{dt^{p-1}}f(0,x) - Lg_{p-1} \in H^{1}_{0}(\Omega).
\end{cases}
\end{align}
\textbf{A3}\,(Search Space)\label{assumption_A3} Let $\mathcal{U}(R)$ be the set of all functions $u:\Omega_{T}\to \R$ such that 
\begin{align}
&(4)\,\, u(x,t) \in H^{p}_{t}(0,T;H_{x}^{2s+2-2p}(\Omega)) \cap L^{2}_{t}(0,T;H^{1}_{0,x}(\Omega)), \,\,u(x,0) =u^{*}_{0}(x)\,\, ,0\leq p \leq s \label{functional_space},\\
&(5)\,\, ||\frac{d^{p}u}{dt^{p}}||_{L^{2}_{t}(0,T;H^{2s+2-2p}_{x}(\Omega))}^{2}\leq R,\,\, 0\leq p \leq s, \\
&(6)\,\, u^{*} \in \U(R).
\end{align}
\\
\textbf{A4}\,(Invertibility Condition) \label{assumption_A4} Moreover, we require $u^{*}$ to satisfy
\begin{align}
&(7)\,\, u^{*} \in \U(R) \cap \{u \mid u_{t} +Lu\in \mathcal{F}(R)  \}.
\end{align}
\begin{remark}
(1)-(3) are essential assumptions needed in the the parabolic regularity theory \ref{thm_3.1} in \cite[Chapter 7, Section 7.1, Theorem 6]{evans} to establish higher order regularity result.; (5) is equivalent to $||f||_{H^{p}_{t}(0,T;H^{2s+2-2p}_{x}(\Omega))}^{2}\lesssim  R$, which guarantees that the approximation errors used in theorem \ref{thm_3.1} are controllable, and (6) guarantees learning framework is realizable; (7) is sufficient condition to guarantee for $u^{*}\in \U(R)$, there always exists data $f$ such that we can construct an estimator $\hat{u}$ in minimax lower bound.
\end{remark}

\subsection{Functional Space and Norms}
\label{functional_spaces_and_norms}
In this section, we summarize the key functional spaces used in our analysis, with particular attention to their relationship with Fourier series representations. A more detailed treatment of these topics can be found in \cite[Chapter 5]{evans} and \cite{MR0450957}. \\

Let $\Omega \subseteq \mathbb{R}^{n}$ be an open set with nonempty interior and integers $\alpha,\beta\geq 0$, we denote $W^{\alpha, \beta}$ as the set of the Sobolev function $f$ with $\alpha$ degree of smoothness and $\beta$ degree of integrability. We define the norms as 
\begin{equation}
||f||_{W^{\alpha,\beta}(\Omega)}:= \sum_{|k|\leq \alpha} ||\partial^{k}f||_{L^{\beta}(\Omega)},
\label{def_of_sobolev}
\end{equation}
where the derivative $\partial^{k}$ is defined as weak derivative in the distributional sense and detailed discussion of weak derivative is in \cite[Chapter 3, Section 1]{stein2011functional}.\\

Sometimes, an alternative but equivalent definition is used:
\begin{equation}
\|f\|_{W^{\alpha,\beta}(\Omega)} := \|f\|_{L^\beta(\Omega)} + \sum_{|k| = \alpha} \|\partial^k f\|_{L^\beta(\Omega)}.
\end{equation}
In this work, we adopt the full norm in \eqref{def_of_sobolev}. In particular, we denote
\begin{equation}
H^\alpha(\Omega) := W^{\alpha,2}(\Omega).
\end{equation}
We also define the negative-order Sobolev space $H^{-1}(\Omega)$ as the dual of $H^1_0(\Omega)$, which satisfies the following Gelfand triple:
\begin{equation}
H^1_0(\Omega) \subseteq L^2(\Omega) \subseteq H^{-1}(\Omega).
\end{equation}

Next, we define functional spaces involving time. Here, we consider function $f$ as a mapping from $t\in[0,T]$ to Banach Space $\X$.

\subsubsection{\textbf{Spaces involving time}}\label{time_dependent_space}
Let $\X$ denote a real Banach space, with the norm $||\cdot||$.
\begin{defn} The space 
\begin{equation}
L^{\beta}(0,T;\X)
\end{equation}
consists of all measurable functions $f: [0,T] \rightarrow \X$ with \\
\begin{equation}
||f||_{L^{\beta}(0,T;\X)} := \begin{cases} \bigg(\displaystyle\int_{0}^{T} ||f(t)||^{\beta}\bigg)^{\frac{1}{\beta}}<\infty, \quad 1\leq \beta<\infty,\,\quad\quad \\\\
\underset{0\leq t\leq T}{ess\,\sup} ||f(t)||<\infty, \quad\quad \beta=\infty
\end{cases} 
\end{equation}
\end{defn}

Once we have defined $L^{\beta}(0,T;\X)$, we can define the time-dependent Sobolev space
\begin{defn} The Sobolev space
\begin{equation}
W^{\alpha,\beta}(0,T;\X)
\end{equation}
consists of all functions $f\in L^{\beta}(0,T;\X)$ such that weak derivative $\tfrac{\partial^{k} u}{\partial t^{k}}$ exists in the distributional sense and belongs to $L^{\beta}(0,T;\X)$. Furthermore, 
\begin{equation}
||u||_{W^{\alpha,\beta}(0,T;\X)}:=\begin{cases}
\bigg(\displaystyle\int_{0}^{T} \sum_{k=0}^{\alpha}||\tfrac{\partial^{k}u(t)}{\partial t^{k}}(t,\cdot)||^{\beta}dt\bigg)^{1/\beta},\quad 1\leq \beta<\infty\\\\\\
\underset{0\leq t\leq T}{ess\,\sup} \sum_{k=0}^{\alpha}||\tfrac{\partial^{k}u(t)}{\partial t^{k}}(t,\cdot)||^{\beta}dt, \quad \beta =\infty 
\end{cases}.
\label{time_dependent_sobolev_norm}
\end{equation}
We write $H^{\alpha}(0,T;\X) = W^{\alpha,2}(0,T;\X)$. In particular, let $\X = H^{\alpha^{\prime}}(\Omega)$,\,$\alpha^{\prime}>0$, we abbreviate 
\begin{equation}
H^{\alpha}_{t}H^{\alpha^{\prime}}_{x}=H^{\alpha}(0,T;H^{\alpha^{\prime}}(\Omega)).
\end{equation}
\end{defn}

\subsubsection{\textbf{Sobolev Extension and Fourier Characterization}}
A very useful tool to handle function space defined on compact domains is using extension operators that preserve the norms. The following theorem ensures the existence of such an extension operator that is the same among all Sobolev spaces. This allows us to characterize Sobolev functions
on compact domains via the Fourier coeﬃcients of their extensions, a crucial step to derive approximation errors in our proof.
\begin{theorem}[Lipschitz Domain Extension Theorem] 
\label{Sobolev_extension_theorem}
{\cite[Theorem 5.4, when $\alpha$=1]{evans}\cite[Theorem 1.105]{triebel2006theory}} Let $\Omega\subseteq \R^{d}$ be a bounded Lipschitz domain and the smoothness parameter $\alpha>0$. Then, there exists a linear extension operator 
\begin{equation}
Ext: H^{\alpha}(\Omega) \rightarrow H^{\alpha}(\R^{d})
\end{equation}
such that for any $f\in H^{\alpha}(\Omega)$,
\begin{align}
(i)& Ext\,f \,|_{\Omega} =f ,\,\, \text{a.s} \,\, \text{in} \,\, \Omega, \\
(ii)& ||Ext \,f||_{W^{\alpha,2}(\R^{d})}\leq C(\Omega)||f||_{W^{\alpha,2}(\Omega)},
\end{align}
where the constant $C$ depending only on $\Omega$.
\end{theorem}

If $f(x)$ admits representation in finitely many orthornormal basis $\phi_{\xi}(x)$,
\begin{equation}
f:= \sum_{\xi=0}^{\infty} a_{\xi} \phi_{\xi}, \quad \{\phi_{\xi}\}\,\, \text{orthonormal basis in } L^{2}(\R^{d}).
\end{equation}
By orthogonality of $L_2$ basis $\phi_{\xi}$,  $f$ in Sobolev norm can be written 
\begin{equation}
||f||_{H^{\alpha}(\R^{d})}^{2}\lesssim \sum_{\xi=0}^{\infty} (1+|\xi|^{\alpha})^{2} a_{\xi}^{2}.
\label{sobolev_norm_in_fourier_transform}
\end{equation}

    Moreover, we need Theorem \ref{Sobolev_extension_theorem} to establish space-time extension lemma so that we can extend $u\in H^{1}_{t}([0,T];H^{1}_{x}(\Omega))$ to $ ext\,u\in H^{1}_{t}(\R;H^{1}_{x}(\R^{d}))$. Similar ideas that leverage approximation analysis on larger or full space-time domain and then restrict analytical results to space-time cylinder $[0,T]\times \Omega$ are common in PDEs and numerical analysis literature. see \cite{dong2025lpestimatesconormalderivativeproblem,Dong2021SobolevEF} \cite[Section 1.3]{Amann2019}

\begin{lemma}[Space--time extension]
\label{space_time_extension}
Let $\Omega \subset \mathbb{R}^d$ be a bounded Lipschitz domain, $T>0$, and smoothness parameters $\alpha_{1},\alpha_{2}\geq0$.
There exists a bounded linear operator
\begin{equation}
  Ext : H_t^{\alpha_{2}}\bigl(0,T; H_x^{\alpha_{1}}(\Omega)\bigr)
  \longrightarrow
H_t^{\alpha_{2}}\bigl(\mathbb{R}; H_x^{\alpha_{1}}(\mathbb{R}^d)\bigr)
\end{equation}
such that for every $u \in H_t^1(0,T;H_x^1(\Omega))$:\\\\
$(i)\,\, Ext\, u(t,x) = u(t,x)$ for a.e.\ $(t,x) \in (0,T)\times\Omega$,
\\\\
(ii) There exists a constant $C = C(\Omega,T) > 0$ with
\begin{equation}
    ||Ext \,\,u\|_{H_t^{\alpha_{2}}(\mathbb{R};H_x^{\alpha_{1}}(\mathbb{R}^d))}
    \le C\,\|u\|_{H_t^{\alpha_{2}}(0,T;H_x^{\alpha_{1}}(\Omega))}.
\end{equation}
\end{lemma}

\begin{proof}[Proof of Lemma~\ref{space_time_extension}]
We first extend spatial domain $\Omega$ to $\R^{d}$.
Since $\Omega$ is a bounded Lipschitz domain, the
extension Theorem \ref{Sobolev_extension_theorem} provides a bounded linear operator
\begin{equation}
  E_x : H_x^{\alpha_{1}}(\Omega) \longrightarrow H_x^{\alpha_{1}}(\mathbb{R}^d)
\end{equation}
such that for any $t\in[0,T]$, (i) and (ii) of Theorem \ref{Sobolev_extension_theorem} hold for all $u(\cdot,t ) \in H^{\alpha_{1}}(\Omega)$.

For $u \in H_t^{\alpha_{2}}(0,T;H_x^{\alpha_{1}}(\Omega))$,  we can define
\begin{equation}
  (\hat{E}_x u)(\cdot,t) := E_x\bigl(u(\cdot,t)\bigr),
  \quad \text{for a.e.\ } t \in (0,T).
\end{equation}
Since $E_x$ is bounded operator along with Fubini's theorem, we obtain for every temporal derivative index $\alpha_{2}\geq 0$
\begin{equation}
  \int_0^T \|\partial_t^{\alpha_{2}} (\hat{E}_x u)(t)\|_{H^{\alpha_{1}}(\mathbb{R}^d)}^2 \, dt
  \le C_\Omega^2 \int_0^T \|\partial_t^{\alpha_{2}} u(t)\|_{H^{\alpha_{1}}(\Omega)}^2 \, dt.
\end{equation}
This implies $\hat{E}_x u \in H^{\alpha_{2}}(0,T;H^{\alpha_{1}}(\mathbb{R}^d))$. Thus, the operator $ \hat{E}_x : H_t^{\alpha_{2}}(0,T;H_x^{\alpha_{1}}(\Omega))\to H_t^{\alpha_{2}}(0,T;H_x^{\alpha_{1}}(\mathbb{R}^d))$ is bounded with
  
\begin{equation}
  \|\hat{E}_x u\|_{H_t^{\alpha_{2}}(0,T;H_x^{\alpha_{1}}(\mathbb{R}^d))}
  \le C_\Omega\,\|u\|_{H_t^{\alpha_{2}}(0,T;H_x^{\alpha_{1}}(\Omega))}.
\end{equation}
Moreover, $(\widetilde E_x u)(t,x) = u(t,x)$, for a.e.\ $(t,x)\in [0,T]\times\Omega$.

Next we apply Sobolev extension theorem again to the interval $(0,T)\subset\mathbb{R}$, with a bounded linear $E_t:\, H^{\alpha_{2}}(0,T) \to H^{\alpha_{2}}(\mathbb{R})$ satisfying (i) and (ii) of Theorem ~\ref{Sobolev_extension_theorem} as well.

Let $X$ be a separable Hilbert space and let
$\{e_k\}_{k\ge 1}$ be an orthonormal basis of $X$.
For $u \in H^{\alpha_{2}}(0,T;X)$, define scalar functions
$u_k(t) := \langle u(t),e_k\rangle_X \in H^{\alpha_{2}}([0,T])$.

Define $\hat{E}_t u : \mathbb{R}\to X$ by
\begin{equation}
  (\hat{E}_t u)(t) := \sum_{k=1}^\infty (E_t u_k)(t) \, e_k,
\end{equation}
and similarly, for every temporal derivative index $\alpha_{2}\geq 0$, we obtain
\begin{equation}
  \|\hat{E}_t \partial_{t}^{\alpha_{2}}u\|_{L^2(\mathbb{R};X)}^2
  = \sum_{k=1}^\infty \|E_t \partial_{t}^{\alpha_{2}}u_k\|_{L^2(\mathbb{R})}^2
  \le C_T^2 \sum_{k=1}^\infty \|\partial_{t}^{\alpha_{2}}u_k\|_{L^2(0,T)}^2
  = C_T^2 \|\partial_{t}^{\alpha_{2}}u\|_{L^2(0,T;X)}^2.
\end{equation}
Moreover, for a.e.\ $t\in[0,T]$, we have
\begin{equation}
  (\hat{E}_t u)(t)
  = \sum_{k=1}^\infty (E_t u_k)(t)\,e_k
  = \sum_{k=1}^\infty u_k(t)\,e_k
  = u(t),
\end{equation}

When $X = H_x^1(\mathbb{R}^d)$, $\hat{E}_t : H_t^{\alpha_{2}}\bigl(0,T;H_x^{\alpha_{1}}(\mathbb{R}^d)\bigr)
  \to H_t^{\alpha_{2}}\bigl(\mathbb{R};H_x^{\alpha_{1}}(\mathbb{R}^d)\bigr)$
is bounded and extends functions in time. Finally, we define
\begin{equation}
  Ext := \hat{E}_t \circ \hat{E}_x.
\end{equation}
For $u \in H_t^1(0,T;H_x^1(\Omega))$, we have
\begin{equation}
  Ext\,\, u = \hat{E}_t\bigl(\hat{E}_x u\bigr)
  \in H_t^{\alpha_{2}}\bigl(\mathbb{R};H_x^{\alpha_{1}}(\mathbb{R}^d)\bigr),
\end{equation}
and for a.e.\ $(t,x)\in[0,T]\times\Omega$,
\begin{equation}
  Ext\, u(t,x)
  = \hat{E}_t(\hat{E}_x u)(t,x)
  = (\hat{E}_x u)(t,x)
  = u(t,x),
\end{equation}
so (i) holds. Moreover,
\begin{equation}
  \| Ext \,u\|_{H_t^{\alpha_{2}}(\mathbb{R};H_x^{\alpha_{1}}(\mathbb{R}^d))}
  \le C_T \|\hat{E}_x u\|_{H_t^{\alpha_{2}}(0,T;H_x^{\alpha_{1}}(\mathbb{R}^d))}
  \le C_T C_\Omega \|u\|_{H_t^{\alpha_{2}}(0,T;H_x^{\alpha_{1}}(\Omega))},
\end{equation}
which yields (ii) with $C(\Omega,T) := C_T\cdot C_\Omega$.
\end{proof}

\subsubsection{\textbf{Inner Product in $H^1_t H^1_x$}}
The Hilbert space structure of $H^1_t H^1_x := H^1(0,T;H^1(\Omega))$ plays a crucial role in our analysis. We use the norm
\begin{equation}
\|u\|_{H^1_t H^1_x} := \left( \int_0^T \|u(\cdot,t)\|_{H^1(\Omega)}^2 + \left\| \frac{\partial u}{\partial t}(\cdot,t) \right\|_{H^1(\Omega)}^2 \,dt \right)^{1/2}.
\end{equation}
The associated inner product is
\begin{equation}
\langle u, v \rangle_{H^1_t H^1_x} := \int_0^T \int_\Omega uv + \nabla u \cdot \nabla v + \partial_t u \, \partial_t v + \nabla(\partial_t u) \cdot \nabla(\partial_t v) \,dx\,dt.
\label{inner_product}
\end{equation}
We also define the $L^2(\Omega_T)$ inner product as
\begin{equation}
\langle u, v \rangle_{L^2(\Omega_T)} := \int_0^T \int_\Omega uv\,dx\,dt,
\label{L_2_inner_product}
\end{equation}
so that the $H^1_t H^1_x$ inner product can be rewritten as
\begin{equation}
\langle u, v \rangle_{H^1_t H^1_x} = \langle u, v \rangle_{L^2(\Omega_T)} + \langle \nabla u, \nabla v \rangle_{L^2(\Omega_T)} + \langle \partial_t u, \partial_t v \rangle_{L^2(\Omega_T)} + \langle \nabla(\partial_t u), \nabla(\partial_t v) \rangle_{L^2(\Omega_T)}.
\end{equation}

\section{Background for Linear Parabolic PDEs}
\label{sec:background_parabolicPDE}

In this section, we provide overview of linear parabolic PDEs theory.

\subsection{Dirichlet Boundary Problem}
Given $f\in L^{2}(\Omega_{T})$ and $g\in L^{2}(\Omega)$, the Dirichlet boundary value problem for linear parabolic PDEs is defined as
\begin{equation}
\begin{cases}
\,\, u_{t} +Lu =f, \quad in\,\, \Omega\times (0,T],\\
\quad\quad\quad u=0,\quad on\,\, \partial\Omega \times [0,T],\\
\quad\quad\quad u=g, \quad on \,\, \Omega\times \{t=0\},
\end{cases}
\label{linear_parabolic_pdes}
\end{equation}\\
where $L$ is an elliptic operator 
\begin{equation}
L:= -\sum_{i,j=1}^{n} a^{ij}(x,t) u_{x_{i}x_{j}} + \sum_{i=1}^{n} b^{i}(x,t)u_{x_{i}} + c(x,t)u,\,\,a^{ij}, b^{i}, c\in L^{\infty}(\Omega_{T}).
\label{elliptic_operator}
\end{equation}
Furthermore, the stronger \textit{uniform ellipticity} condition requires that there exists some constant $\theta>0$ such that
\begin{equation}
\sum_{i,j=1}^{n}a^{ij}(x,t) \xi_{i}\xi_{j} \geq \theta |\xi|^{2} \text{ for all } (x,t) \in \Omega_{T}, \xi\in \R^{d}.
\end{equation}
Because solutions to \eqref{linear_parabolic_pdes} may lack of classical smoothness, particularly when the forcing term \( f \) or initial condition \( g \) have singularities, it is natural to study solutions in the weak sense.

\subsection{Weak Solution}
We say $u$ is a weak solution to (\ref{linear_parabolic_pdes}) if 
\begin{align}
&(i)\quad u \in L^{2}(0,T;H^{1}_{0}(\Omega)), \text{ with} \,\,u^{\prime} \in L^{2}(0,T;H^{-1}(\Omega)),\\
&(ii)\quad (u^{\prime},v) + B[u,v;t] = (f,v), \text{ for all test function } v\in H^{1}_{0}(\Omega) \text{ and } a.e \text{ time } 0\leq t\leq T, \\
&(iii)\quad u(0)=g,
\label{weak_solution}
\end{align}
where $B[\cdot, \cdot ]$ is the time-dependent bilinear form
\begin{equation}
B[u,v;t]:= \int_{\Omega} \sum_{i,j=1}^{n} a^{ij}(\cdot, t) u_{x_{i}}u_{x_{j}} + \sum_{i=1}^{n} b^{i}(\cdot,t) u_{x_{i}} v + c(\cdot,t)uv dx
\end{equation}
for any $u,v\in H^{1}_{0}(\Omega)$  and a.e $0\leq t\leq T$.

Next we present a classical result on the regularity of solutions to linear parabolic PDEs. This result highlights the smoothing property of parabolic equations: the solution \( u \) gains one additional regularity in space compared to the initial data.

\subsection{Regularity Theory}
\begin{theorem}[Higher Regularity Theory for Linear Parabolic PDEs]. Assume
\begin{equation}
g\in H^{2m+1}(\Omega), \frac{d^{p}}{dt^{p}}f\in L^{2}(0,T;H^{2m-2p}(\Omega))\quad (p=0,...,m).
\end{equation}
Suppose also that the following $p$-th order compatibility conditions hold: 
\begin{equation}
\begin{cases}
\,\,g_{0}:=g\in H^{1}_{0}(\Omega),\,\, g_{1}:=f(0,x) - Lg_{0} \in H^{1}_{0}(\Omega),\\\\\\
\,\,,....,\,g_{m}:= \displaystyle\frac{d^{p-1}}{dt^{p-1}}f(0,x) - Lg_{p-1} \in H^{1}_{0}(\Omega).
\end{cases}
\label{compatibility_condition}
\end{equation}
Then \begin{equation}
    \frac{d^{p}u}{dt^{p}}\in L^{2}(0,T;H^{2m+2-2p}) \text{ holds for solution } u \text{ to PDEs } (\ref{linear_parabolic_pdes}).
\end{equation}
\label{higher_regulairty_thm}
\end{theorem}

\begin{proof}[Proof of Theorem~\ref{higher_regulairty_thm}]
    The proof is standard and follows from Theorem 6 in Chapter 7.1 of~\citep{evans}.
\end{proof}

\begin{theorem}[Infinite differentiability]. Assume $g\in C^{\infty}(\bar{\Omega}), f\in C^{\infty}(\bar{\Omega}_{T})$ and compatibility condition $(\ref{compatibility_condition})$ hold for $m=0,1,2,3,...$\,. Then the parabolic initial/boundary-value problem $(\ref{linear_parabolic_pdes})$ has a unique solution 
\begin{equation}
u\in C^{\infty}(\bar{\Omega}_{T}).
\end{equation}
\label{infinite_regularity_theory}
\end{theorem}

\begin{proof}[Proof of Theorem~\ref{infinite_regularity_theory}]
    See Theorem 7 in Chapter 7 of~\citep{evans}.
\end{proof}

We conclude this section by an important characterization of solution to linear parabolic PDEs, which is essential to construction of our lower bound and space-time trade-off in approximation error analysis.

\subsection{Duhamel Principle}

\begin{theorem}[Duhamel’s Principle for General Linear Parabolic PDE]
Let \( \Omega \subseteq \mathbb{R}^n \) be a domain with boundary $\partial \Omega$, and let  $L$ is an elliptic operator defined in (\ref{elliptic_operator})
\label{Duhamel_principle}
\begin{equation}
\begin{cases}
\quad u_t(x,t) - L u(x,t) = f(x,t), & (x,t) \in \Omega \times (0,T],\\
\quad u(x,t) = 0, &(x,t) \in \partial \Omega\times[0,T], \\
\quad u(x,0) = 0, & x \in \Omega.
\end{cases}
\end{equation}
Then the solution to the inhomogeneous problem is given by the Duhamel formula
\begin{equation}
u(x,t) = \int_{0}^{t} (P^s f)(x,\,s)\, \mathrm{d}s = \int_{0}^{t} v(x,\,s)\, \mathrm{d}s.
\end{equation}
where \(P^s\) is the solution operator of the homogeneous problem with $s>0$ 
 \begin{equation}
\begin{cases}
\quad v_t(x,t) - Lv(x,t) = 0, & (x,t)\in \Omega \times (s,T],\\
\quad v(x,t) = 0 , &  (x,t) \in \partial \Omega \times [s,T],\\
\quad v(x,s) = f(x,s), & x\in \Omega.
\end{cases}
\end{equation}
\end{theorem}

\begin{proof}[Proof of Theorem~\ref{Duhamel_principle}]
    See \cite[Section 2.4]{UMAROV20125217} and \cite[Chapter 5, Section 1]{john1982partial}.
\end{proof}

An immediate corollary followed by the Duhamel principle is the solution to non-zero initial value problem with Dirichlet boundary. Duhamel Principle tells us that the construction of solution $u$ to linear parabolic equation in our setting is enough to look at the homogeneous solution $w$ generated by initial data $g(x)$ when $t=0$ and $v$ generated by forcing function $f(x,t)$ when $t>0$. 

\begin{corollary}
Assume the conditions in Theorem~\ref{Duhamel_principle} holds. Then the solution $u$ to the PDE problem (\ref{linear_parabolic_pdes}) is given by 
\begin{equation}
u(x,t) = g(x) + \int_{0}^{t} w(x,\,s)\, \mathrm{d}s
\end{equation}
where $w(x,t)$ is the solution to homogeneous initial value problem with Dirichlet boundary condition
\begin{equation}
\begin{cases}
\quad w_t(x,t) - Lw(x,t) = 0, & (x,t)\in \Omega \times (0,T],\\
\quad w(x,t) = 0 , &  (x,t) \in \partial \Omega \times [0,T],\\
\quad w(x,0) = g(x).
\end{cases}
\end{equation}
\end{corollary}

\begin{corollary}[Scaling Invariant Property]
\label{scaling_invariant}
Let $L=\triangle$ and $\alpha>0$. $u$ satisfies the scaling invariant property such that $u(\frac{x}{\alpha}, \frac{t}{\alpha^{2}})$ is a solution to the PDE problem (\ref{linear_parabolic_pdes}) if and only if $u(x,t)$ is a solution to the PDE problem (\ref{linear_parabolic_pdes}).
\end{corollary}

\begin{proof}[Proof of Corollary~\ref{scaling_invariant}]
This proof follows directly from the fact that $w(\frac{x}{\alpha}, \frac{t}{\alpha^{2}})$  and $v(\frac{x}{\alpha}, \frac{t}{\alpha^{2}})$ solves $u_{t}-\triangle u =0$ if and only if $w(x,t)$ and $v(x,t)$ solves  $u_{t}-\triangle u =0$.
Therefore, it is immediately followed by direct computation that $u(x,t)$ also enjoys this scaling invariance property.
\end{proof}

\section{Derivation for Our Loss Function}
\label{motivation_of_loss_function}

Our loss function is motivated by solving linear parabolic PDEs using its corresponding convex optimization formulation. Let $f \in \calF$, $g\in \calG$ be right-hand side and boundary value of the PDE and define elliptic operator [see Chapter 6 in \cite{evans}]
\begin{equation}
L:= -\sum_{i,j=1}^{n} a^{ij}(x,t) u_{x_{i}x_{j}} + \sum_{i=1}^{n} b^{i}(x,t)u_{x_{i}} + c(x,t)u.
\end{equation}
Furthermore, there exists some constant $\theta>0$ such that
\begin{equation}
\sum_{i,j=1}^{n}a^{ij}(x,t) \xi_{i}\xi_{j} \geq \theta |\xi|^{2} \text{ for all } (x,t) \in \Omega_{T}, \xi\in \R^{d}.
\end{equation}
Solving the boundary valued problem of linear parabolic PDEs is defined as finding $u^{*}\in \U$ satisfying the following constraint
\begin{equation}
\begin{cases}
\quad u^{*}_{t} +Lu^{*} =f \quad in\,\, \Omega \times (0,T],\\
\quad\,\, u^{*}(x,t)=0 \,\,\,\quad on\,\, \partial\Omega \times (0,T],\\
\quad\,\,
u^{*}(x,0)=g \quad \,\,on \,\, \Omega.
\end{cases}
\end{equation}
Note $\calF$ and $\calG$ and $\U$ are three distinctive functional spaces where differentiability and integrability could be totally different.
Its relaxation is given as the equation below, and we leave the detailed discussion to Section C.1 and C.2
\begin{equation}
\,\, -\epsilon u^{*}_{tt}+u^{*}_{t} +Lu^{*} =f \quad in\,\, \Omega\times (0,T].
\end{equation}
The solution to the relaxed problem enjoys the following uniform in time approximation property. \cite[Theorem 1]{Mielke11} \cite[Section 7, Theorem 7.2]{Lions09}
\begin{equation}
\text{ As } \epsilon_{n}\to 0,\,\, u^{*}_{\epsilon_{n}} \to u^{*} \,\, in \,\, L^{2}(\Omega)\,\,, \forall\,t>0.
\end{equation}
This approximation theory allows us to solve linear parabolic PDEs in its relaxed version, and surprisingly the variational structure underlying the relaxed problem can be formulated as convex optimization problem for some strongly convex energy functional $I^{\text{parabolic}}_{\epsilon}(u)$ in some funtional space $\U_{parabolic}$
\begin{equation}
\underset{u\in \U_{parabolic}}{\min} \,\,I^{\text{parabolic}}_{\epsilon}(u).
\label{relaxed_parabolic_pdes}
\end{equation}

However, $I_{\epsilon}^{parabolic}(u)$, directly minimizing $I_{\epsilon}^{parabolic}$ is not applicable for serious drawbacks. First, $I_{\epsilon}^{parabolic}(u)$, though strongly convex,  is degenerate when $\epsilon$ approaches to $0$, leading to vanishing lower curvature bound in derived stability estimate. Therefore, we introduce additional data-fitting function $D(u,u^{*})$ over Sobolev space to improve lower curvature bound in derived stability estimate.

\subsection{Euler-Lagrange equation for the weighted energy functional~\eqref{energy_functional}}

In this section, we derive the Euler–Lagrange equations associated with the weighted energy functional \( I_\epsilon(u) \). As shown in Theorem~\ref{euler_lagrangian_for_I}, minimizing \( I_\epsilon(u) \) is equivalent to solving initial value problem of a regularized heat equation with zero Dirichlet condition in equation . The well-posedness ensures that the minimizer of \( I_\epsilon(u) \) is unique. Building upon this, we also show that the minimizer of the loss functional \( L_{\epsilon,\beta}(u) \) is the solution to the (exact) heat equation \eqref{eq:heat_pde}, as shown in Theorem~\ref{euler_lagrangian_for_L}.

We incorporate regularity theory from linear parabolic PDEs in accordance with hyperbolic PDEs. Assume that forcing function $f\in \mathcal{F}$ and initial data $u_{0}^{*} \in H^{2s+1}(\Omega) \cap H^{1}_{0}(\Omega)$ satisfy the compatibility condition in (\ref{compatibility_condition}), we consider the minimizer $u$ living in the search space $\U$ 
\begin{equation}
\U:= \bigg\{ u(x,t) \in H^{p}_{t}(0,T;H_{x}^{2s+2-2p}(\Omega)) \cap L^{2}_{t}(0,T;H^{1}_{0}(\Omega))\,\,\bigg|\,\, u(x,0) =u^{*}_{0}(x)\bigg\}.
\end{equation}
Moreover,  we denote $u^{*}\in \U$ as ground truth solution that solves initial value problem for heat equation with zero Dirichlet boundary in $(\ref{eq:heat_pde})$. Our search functional space $\U$, forcing function $f$, and initial data $u^{*}_{0}$ satisfy functional setting where Euler-Lagrange equation has unique minimizer. In literature, it is often required though with some variation that $f \in L^{2}(0,T;L^2(\Omega))$ and $u_{0}^{*}\in H^{1}_{0}(\Omega)$ and $u\in L^{2}(0,T;H^1_{0}(\Omega))$. See \cite{Mielke11,refId0, doi:10.1142/S0218202525400019}.

\subsection{Euler-Lagrange Equation for the weighted energy functional $I_{\epsilon}(u)$}\label{derivation of Euler-Langragian}

The following Theorem~\ref{euler_lagrangian_for_I} gives the first-order optimality condition for the weighted energy functional $I_{\epsilon}(u)$ in~\eqref{energy_functional}.

\begin{theorem}[Minimizer of the WEF $I_{\epsilon}(u)$]
There exists a unique minimizer $u\in \mathcal{U}$ of $I_{\epsilon}(u)$ such that $u$ is a solution to linear hyperbolic PDEs 
\label{euler_lagrangian_for_I}
\begin{equation}\label{euler-lagragian*}
\begin{cases}
&-\epsilon u_{tt} + u_{t} - \triangle u = f,
\quad (x,t) \in \Omega \times (0,T],\\
&u(x,t) = 0,\quad (x,t) \in \partial\Omega \times [0,T],\\
&u(x,0) = u_{0}^{*}(x),\quad x\in \Omega.
\end{cases}
\end{equation}
\end{theorem}

\begin{remark}
    In \cite[Section 7.2]{Lions09}, equation~\eqref{euler-lagragian*} is called elliptic regularization of heat equation. It is shown there as well that regularized solution $u$ satisfies classical energy estimate for linear parabolic PDEs theory so that the regularized solution $u_{\epsilon}$ is unique. The general theory of solution to the above equations is discussed in \cite[Chapter 2]{roger13}. 
\end{remark}

\begin{proof}[Proof of Theorem~\ref{euler_lagrangian_for_I}]
Recall the definition of the weighted energy functional (WEF) :

\begin{equation}I_{\epsilon}(u) = \int_{0}^{T}  e^{-\tfrac{t}{\epsilon}} \int_{\Omega} \left[ \dfrac{1}{2} |u_{t}|^{2} +  \frac{1}{2\epsilon} |\nabla u|^{2} - \frac{1}{\epsilon}f u \right] dx dt.\end{equation}

Note that the Euler-Lagrangian equation depends on the definition of test function. Here we choose test function $v \in C^{\infty}_{c}([0,T]\times\Omega).$ 
For any test function $v \in C^{\infty}_{c}([0,T]\times\Omega),$ we know $
v(\cdot,0) = v(\cdot,T) =0
.$
Then,
\begin{align}
I_{\epsilon}(u+sv) = \int_{0}^{T}  e^{-\tfrac{t}{\epsilon}} \int_{\Omega} \dfrac{1}{2} \big|u_{t}+sv_{t}\big|^{2} + \dfrac{1}{2\epsilon} \big| \nabla u+s \nabla v\big|^{2} - \dfrac{1}{\epsilon}f \big(u+sv\big)\hspace{0.5mm}dx dt. \end{align}
We apply Fubini theorem and integration by parts with respect to time $t$:
\begin{align}
\dfrac{d}{ds} I_{\epsilon}(u+sv)\bigg|_{s=0} &= \int_{0}^{T}  e^{-\tfrac{t}{\epsilon}} \int_{\Omega}  u_{t} v_{t} + uv - \dfrac{1}{\epsilon} \triangle uv - \dfrac{1}{\epsilon}fv\hspace{0.5mm}dx dt
\\ &=\int_{0}^{T} e^{-\tfrac{t}{\epsilon}} \int_{\Omega}  \dfrac{1}{\epsilon} u_{t}v  -\dfrac{1}{\epsilon} \triangle uv - \dfrac{1}{\epsilon}fv-  vu_{tt} dxdt\\
&\quad\quad -e^{-\tfrac{T}{\epsilon}} \int_{\Omega}  f(x,T)v(x,T)dx + u_{t}(x,0)v(x,0)dx.
\end{align}
Making appropriate substitutions and matching terms from the above, we obtain\\
\begin{align}
\dfrac{d}{ds} I_{\epsilon}(u+sv)\bigg|_{s=0} &= e^{-\tfrac{T}{\epsilon}} \int_{\Omega} \big( -\triangle u(x,T) + u_{t}(x,T) -f(x,T) \big)v(x,T)dx\\
&\quad\quad+ \int_{0}^{T} \int_{\Omega} e^{-\tfrac{t}{\epsilon}} \big(-u_{tt} + \dfrac{1}{\epsilon} u_{t}  - \dfrac{1}{\epsilon} \triangle u - \dfrac{1}{\epsilon}f\big)vdx dt\\
&\quad\quad+  \int_{\Omega} - u_{t}(x,0) v(x,0)dx\\
&= \int_{0}^{T} \int_{\Omega} e^{-\tfrac{t}{\epsilon}} \big(-u_{tt} + \dfrac{1}{\epsilon} u_{t}  - \dfrac{1}{\epsilon} \triangle u - \dfrac{1}{\epsilon}f\big)vdx dt.
\label{derivative_of_I}
\end{align}
We set $\frac{d}{ds} I_{\epsilon}(u+sv) = 0$ and obtain the corresponding Euler-Lagrangian equations for the functional weighted energy, and the boundary condition is satisfied because $u\in L^{2}_{t}H^{1}_{0,x}$.
\begin{equation}
\begin{cases}
&-\epsilon u_{tt} + u_{t} - \triangle u = f,
\quad (x,t) \in \Omega \times (0,T],\\
&u(x,t) = 0,\quad (x,t) \in \partial\Omega \times [0,T],\\
&u(x,0) = u_{0}^{*}(x), \quad x\in \Omega.
\end{cases}
\end{equation}
\end{proof}

\subsection{Minimizer of the Total Loss Function $L_{\epsilon,\beta}(u)$}

Next, we characterize the optimality condition for the total loss function $L_{\epsilon,\beta}(u)$ in~\eqref{loss_function}. Recall that Theorem~\ref{euler_lagrangian_for_L} states any minimizer $u \in \U$ the total loss function \texorpdfstring{$L_{\epsilon,\beta}(u)$}\, is the solution to heat equation in $(\ref{eq:heat_pde})$. Below we provide the proof.

\begin{proof}[Proof of Theorem~\ref{euler_lagrangian_for_L}]
Recall that
\begin{equation}
L_{\epsilon,\beta}(u): = 
||u - u^{*}||_{H^{1}_{t}(0,T;H^{1}_{x}(\Omega))}^{2}  + \beta \bigg[I_{\epsilon} (u) + \int_{0}^{T}\int_{\Omega} e^{-\frac{t}{\epsilon}}u^{*}_{tt}u dx dt - \int_{\Omega} e^{-\frac{T}{\epsilon}}u^{*}_{t}(x,T)u(x,T) dx\bigg].
\end{equation}
According to the inner product in $H^{1}_{t}(0,T;H^{1}_{x}(\Omega))$, we readily compute first variation of loss function $L_{\epsilon,\beta}(u)$. For any test function $v\in C^{\infty}_{c}(\Omega \times [0,T])$,
\begin{align}
\dfrac{d}{ds} L_{\epsilon,\beta}(u+sv)\bigg|_{s=0} &= 2\langle u -u^{*},v\rangle_{H^{1}_{t}H^{1}_{x}}+ \beta\int_{0}^{T} \int_{\Omega} e^{-\tfrac{t}{\epsilon}} \big(-u_{tt} + \dfrac{1}{\epsilon} u_{t}  - \dfrac{1}{\epsilon} \triangle u - \dfrac{1}{\epsilon}f\big)vdx dt \\
&+\beta \int_{0}^{T} \int_{\Omega} e^{-\tfrac{t}{\epsilon}}u_{tt}^{*}v dxdt-\beta\int_{\Omega} u_{t}(x,0) v(x,0)dx- \beta\int_{\Omega} e^{-\frac{T}{\epsilon}}u^{*}_{t}(x,T)u(x,T) dx\\
&= 2\langle u -u^{*},v\rangle_{H^{1}_{t}H^{1}_{x}} + \beta\int_{0}^{T} \int_{\Omega} e^{-\tfrac{t}{\epsilon}} \big(u^{*}_{tt}-u_{tt} + \dfrac{1}{\epsilon} u_{t}  - \dfrac{1}{\epsilon} \triangle u - \dfrac{1}{\epsilon}f\big)vdx dt\\
&= 2\int_{0}^{T}\int_{\Omega}(u-u^{*})v dxdt + 2\int_{0}^{T}\int_{\Omega}(\nabla u-
\nabla u^{*}) \nabla v dxdt\\
&\quad\quad + 2\int_{0}^{T}\int_{\Omega}(u_{t} -
u^{*}_{t}) v_{t} dxdt + 2\int_{0}^{T}\int_{\Omega}(\nabla u_{t}-
\nabla u^{*}_{t}) \nabla v_{t}dxdt\\
&\quad\quad + \beta\int_{0}^{T} \int_{\Omega} e^{-\tfrac{t}{\epsilon}} \big(u^{*}_{tt}-u_{tt} + \dfrac{1}{\epsilon} u_{t}  - \dfrac{1}{\epsilon} \triangle u - \dfrac{1}{\epsilon}f\big)vdx dt\\
&= \int_{0}^{T}\int_{\Omega} \bigg[2\bigg(u-u^{*} -\triangle u+ \triangle u^{*} - u_{tt}+u^{*}_{tt}+ \triangle u_{tt} - \triangle u^{*}_{tt}\bigg)\\
&\quad\quad\quad+\frac{\beta}{\epsilon} e^{-\tfrac{t}{\epsilon}}\bigg(\epsilon(u_{tt}^{*}-u_{tt}) +  u_{t} - \triangle u - f\bigg)\bigg]v dxdt.\end{align}
Second equality follows by $v(x,0)=0$ and $v(x,T)=0$; third equality follows by inner product in $H^1_t H^1_x$;last equality follows by integration by parts and vanishing boundary of $v$. Setting $\frac{d}{ds} L_{\epsilon,\beta}(u+sv) = 0$, we obtain 
\begin{equation}
2\bigg[u-u^{*} -\triangle u+ \triangle u^{*} - u_{tt}+u^{*}_{tt}+ \triangle u_{tt} - \triangle u^{*}_{tt}\bigg] +\frac{\beta}{\epsilon} e^{-\tfrac{t}{\epsilon}}\bigg(\epsilon(u_{tt}^{*}-u_{tt}) +  u_{t} - \triangle u - f\bigg) = 0.
\end{equation}

Let $w: = u-u^{*}$, we have $(x,t)$-a.s in $\Omega_{T}$ that
\begin{equation}
2\bigg[w - \triangle w - w_{tt} + \triangle w_{tt}\bigg]+\frac{\beta}{\epsilon} e^{-\tfrac{t}{\epsilon}}(-\epsilon w_{tt} +  w_{t} - \triangle w + u^{*}_{t} - \triangle u^{*}- f) = 0.
\end{equation}\\
Since $u^{*}$ is true solution to heat equation, $u_{t}^{*} - \triangle u^{*} -f=0$, and 
\begin{equation}
\tfrac{\beta}{\epsilon} e^{-\tfrac{t}{\epsilon}}(\epsilon w_{tt} -  w_{t} + \triangle w) = 2\bigg[w - \triangle w - w_{tt} + \triangle u_{tt}\bigg] \,\,,(x,t) \text{-}a.s \text{ in } \Omega_{T}.
\end{equation}
We denote inner product in $L_{2}(\Omega_{T})$
\begin{equation}
\langle u,v\rangle_{L^{2}(\Omega_{T})}: = \int_{0}^{T}\int_{\Omega}uv\,dxdt.
\end{equation}
Taking inner product with respect to $w$ in $L_2(\Omega_{T})$ on both sides and applying integration by parts, we have
\begin{equation}
\langle \tfrac{\beta}{\epsilon} e^{-\tfrac{t}{\epsilon}}(\epsilon w_{tt} -  w_{t} + \triangle w), w\rangle_{L_{2}(\Omega_{T})} = 2\langle w - \triangle w - w_{tt} + \triangle u_{tt}, w\rangle_{L^{2}(\Omega_{T})}.
\end{equation}
Making some rearrangement, we obtain 
\begin{align}
&0\geq\bigg[ (2+\beta e^{-\tfrac{T}{\epsilon}})\langle w_{t},w_{t} \rangle_{L_{2}(\Omega_{T})} + (2+\tfrac{\beta}{\epsilon} e^{-\tfrac{T}{\epsilon}}) \langle \nabla w, \nabla w\rangle_{L_{2}(\Omega_{T})}\\
&\quad\quad\quad + 2 \langle w,w \rangle_{L_{2}(\Omega_{T})} + 2\langle \nabla w_{t},\nabla w_{t} \rangle_{L_{2}(\Omega_{T})}\bigg].
\end{align}
Since $\beta$, $\langle w_{t}, w_{t}\rangle_{L_{2}(\Omega_{T})}$, $\langle \nabla w, \nabla w\rangle_{L_{2}(\Omega_{T})}$, $\langle  w,  w\rangle_{L_{2}(\Omega_{T})}$, and $\langle  \nabla w_{t},  \nabla w_{t}\rangle_{L_{2}(\Omega_{T})}$ are nonnegative, we have 
\begin{equation}
w_{t} = 0,\quad \nabla w = 0,\quad \nabla w_{t} = 0, \quad w=0,\quad(x,t)\,a.s \text{ in } \Omega_{T}.
\end{equation}
Since $w = u^{*} - u$, we have
\begin{equation}
u^{*} = u \,\,, (x,t) \text{-}a.s \text{ in } \Omega_{T}.
\end{equation}
Therefore, the minimizer to the total loss function $L_{\epsilon,\beta}(u)$ is the solution to heat equation.
\end{proof}

\subsection{Empirical Loss Function}

Our estimator $\hat{u}_{\theta}$ is defined as $\hat{u}_{\theta}:= \underset{u\in \mathcal{U}_J}{\arg\min}\,\,\hat{L}_{\epsilon,\beta}(u)$, where $\hat{L}_{\epsilon,\beta}(u)$ is empirical counterpart of $L_{\epsilon, \beta}(u)$. In particular, $\hat{L}_{\epsilon,\beta}(u)$ is defined as 
\begin{equation}
\hat{L}_{\epsilon,\beta}(u):= \frac{1}{nm} \hat{D}(u,u^{*})+\lambda \bigg( \hat{I}_{\epsilon}(u) + \hat{M}_{\epsilon}(u,u^{*}) \bigg)
\end{equation}
where
\begin{align}
\hat{D}(u,u^{*})&:= |\Omega_{T}| \sum_{i=1}^{n} \bigg(|u(X_{im},t_{m})-u^{*}(X_{im},t_{m})|^2 + |\partial_{t}u(X_{im},t_{m})-\partial_{t}u^{*}(X_{im},t_{m})|^2 \\
&\quad\quad + |\nabla u(X_{im},t_{m})-\nabla u^{*}(X_{im},t_{m})|^2 +|\nabla \partial_{t}u(X_{im},t_{m})-\nabla \partial_{t}u^{*}(X_{im},t_{m})|^2 \bigg),\\
\hat{I}_{\epsilon}(u):= & |\Omega_{T}|\sum_{j=1}^{m}  e^{-\tfrac{t_{j}}{\epsilon}} \sum_{i=1}^{n} \bigg( \dfrac{1}{2} |\partial_{t}u|^{2}(x_{ij},t_{j})+  \frac{1}{2\epsilon} |\nabla u(x_{ij},t_{j})|^{2} - \frac{1}{\epsilon}f(x_{ij},t_{j}) u(x_{ij},t_{j}) \bigg)\\
&+ m|\Omega|e^{-\tfrac{t_{m}}{\epsilon}}\sum_{i=1}^{n} \bigg(\frac{1}{2} |\nabla u(x_{im},t_{m})|^{2} - f(x_{im},t_{m}) u(x_{im},t_{m}) \bigg), \\
\hat{M}_{\epsilon}(u^{*})&:= |\Omega_{T}|\sum_{j=1}^{m}e^{-\frac{t_{j}}{\epsilon}}  \sum_{i=1}^{n} \partial_{tt}u^{*}(X_{ij},t_{j})u(X_{ij},t_{j}) - m|\Omega|e^{-\tfrac{t_{m}}{\epsilon}}\sum_{i=1}^{n} u(x_{im},t_{m}) u^{*}_{t}(x_{im},T_{m}).
\end{align}

We are given with the training data on space-time point clouds on which ground truth solution $u^{*}$ and forcing function $f$ are defined , namely $S:=\{(X_{ij},t_{j},u^{*}_{ij},\nabla u^{*}_{ij}, \partial_{t}u^{*}_{ij},\nabla \partial_{t} u_{ij}^{*},f_{ij})\}_{i=1,j=1}^{n,m}$, where derivative labels $\nabla u^{*}_{ij}, \partial_{t}u^{*}_{ij},\nabla \partial_{t} u_{ij}^{*}$ are obtained through numerical computation for $H^{1}_{t}H^{1}_{x}$ evaluation norm.

\section{Truncated Local Basis in Space and Time}
\label{sec:truncated_fourier}
  Among various representations for parametrizing $\hat{u}_{\theta}$,
  we choose $\hat{u}_{\theta}:=\hat{u}_{J}$, the projection over a $(J,J)$-truncated space $\mathcal{U}_{J}$ spanned by spatio-temporal local basis \ due to its approximation power in functional space $\mathcal{U}$ and more importantly its analytical simplicity. One might use other estimators, such as neural network RHKS, as long as the estimator's analytical approximation results are well-established.
\begin{equation}
\mathcal{U}_{J}: = \bigg\{\hat{u}_{J}(x,t)\,\bigg{|} \,\hat{u}_{J}(x,t) = \sum_{|\xi|_{\infty} \leq J, \tau \leq J} a_{\xi,\tau} \phi_{\xi,\tau}(x,t), \hat{u}_{J}\in \mathcal{U} \bigg\}.
\end{equation}

The essential idea of elliptic regularization is to view time as an additional dimension of space with extra penalty scheduler term $e^{-\frac{t}{\epsilon}}$. In this viewpoint, we can write local basis as direct product of two local basises $\phi_{\xi}(x)$ and $\phi_{\tau}(t)$ in space and time respectively

Our proof needs $\phi_{\tau}(t)$ and $\phi_{\xi}(x)$ to satisfy the following propertieis:
\begin{align}
&(1)\quad \phi_{\xi}(x) \text{ and }\phi_{\tau}(t) \text{ are orthonormal basis in } L^{2}(\mathbb{R}^{d}) \text{ and }  L^{2} (\mathbb{R}) \text{ respectively},\label{basis_property_1}
\\
&(2)\quad \phi_{\xi}(x) \text{ and } \phi_{\tau}(t) \text{ have compact support in } \mathbb{R}^{d} \text{ and in } \mathbb{R} \text{ respectively }.
\label{basis_property_2}
\end{align}
There are only few family of local bases satisfying $(1)$-$(2)$ while demonstrating nice approximation property in Sobolev space. Local basis representation like Legendre polynomial and Fourier representation will fail locally compact support property. For this representation, we construct wavelet bases using Dauchecies wavelets to establish approximation in Besov spaces whose special case is Sobolev space, following the definitions given in \cite{triebel2006theory}. In this paper, the reader can assume our local bases are wavelets, and the omission of thorough construction of such bases, however, will not jeopardize our argument as long as (1) and (2) are applied. For example, we can construct wavelet bases using Dauchecies wavelets to establish approximation in Besov spaces whose special case is Sobolev space, following the definitions given in \cite{triebel2006theory}[Section 3], and we refer the reader to this reference for further details on wavelets. Denote by $\psi_{\mathfrak{m}} \in C^r(\mathbb{R})$ and $\psi_{\mathfrak{F}} \in C^r(\mathbb{R})$ a wavelet on compact support  and scaling function, respectively. This implies that
\begin{equation}
    \psi^j_k = \begin{cases}
        \psi_{\mathfrak{F}}(x-k), & j=0, \, k \in \mathbb{Z}, \\
        2^{(j-1)/2}\psi_{\mathfrak{m}}(2^{j-1}x - k), & j \in \mathbb{N}, \, k \in \mathbb{Z}.
    \end{cases}
\end{equation}
is an orthonormal basis of $L^2(\mathbb{R})$.

To obtain a basis of $L^2(\mathbb{R}^d)$, for $j \in \mathbb{N}$, set
\begin{equation}
    G^j = \{\mathfrak{F}, \mathfrak{m}\}^d \setminus \{(\mathfrak{F}, \dots, \mathfrak{F})\}, \quad G^0 = \{(\mathfrak{F}, \dots, \mathfrak{F})\},
\end{equation}
and for $g \in G^j \cup G^0$,
\begin{equation}
    \Psi^g_k(x) = \prod_{i=1}^d \psi_{g_i}(x_i - k_i), \quad k \in \mathbb{Z}^d.
\end{equation}
This gives the orthonormal basis
\begin{equation}
    \Psi^{j,g}_k = \begin{cases}
        \Psi^g_k(x), & j=0, \, g \in G^0, \, k \in \mathbb{Z}^d, \\
        2^{(j-1)d/2}\Psi^g_k(2^{j-1}x), & j \in \mathbb{N}, \, g \in G^j, \, k \in \mathbb{Z}^d.
    \end{cases}
\end{equation}
Wavelet coefficients, defined as the expansion coefficients with respect to the above basis for $L^2(\mathbb{R}^d)$ functions, can be used to characterize the \textit{Besov spaces}, whose special cases are Sobolev spaces.

Let $1 \le p, q \le \infty$, $\alpha \ge 0$, and let the regularity of the above wavelets satisfy
\begin{equation}
    r > \alpha \vee \left( \frac{2d}{p} + \frac{d}{2} - s \right).
\end{equation}
If $f(x)$ has wavelet representation, we have
\begin{equation}
    f = \sum_{j=0}^\infty \sum_{g \in G^j} \sum_{k \in \mathbb{Z}^d} \gamma^{j,g}_k \Psi^{j,g}_k, \tag{B.1}
\end{equation}
we define the Sobolev norm as special case of Besov norm $H^{s}(\Omega)$ by
\begin{equation} 
    \|f\|_{H^{\alpha}(\R^{d})} = \|f\|_{B^\alpha_{2,2}}:= \|\gamma\|_{B^\alpha_{2,2}}
 := \left[ \sum_{j=0}^\infty 2^{2j\alpha } \sum_{g \in G^j} \sum_{k \in \mathbb{Z}^d} |\gamma^{j,g}_k|^2  \right]^{1/2}.
\end{equation}
In particular, by the orthonormality of the wavelets $\{ \Psi^{j,g}_k \}_{j,g,k}$,
\begin{equation}
    \|f\|_{B^0_{2,2}} = \|\gamma\|_{\ell^2} = \|f\|_{L^2(\mathbb{R}^d)}.
\end{equation}
With this construction, we can replace $\phi_{\tau}(x)$ with wavelet bases $\Psi^{j,g}_k$, so $\xi$ is indexed by both wavelet level and group $G$.
When we treat $\xi = 2^{j}$ and $|a_{\xi}|^{2} = \sum_{g \in G^j} \sum_{k \in \mathbb{Z}^d} |\gamma^{j,g}_k|^2 $, one can show the above Sobolev norm in terms of wavelet expansion is as same as standard Sobolev norm and moreover we can achieve the same estimate
\begin{equation}
\|f\|_{H^{\alpha}(\R^{d})}^{2}:=  \sum_{j=0}^\infty 2^{2j\alpha} \sum_{g \in G^j} \sum_{k \in \mathbb{Z}^d} |\gamma^{j,g}_k|^2  = \sum_{|\xi|= 2^{j}, j\geq 0} |\xi|^{2\alpha} a_{\xi}^{2} \leq \sum_{\xi\geq0}(1+ |\xi|^{\alpha})^{2} a_{\xi}^{2}.
\end{equation}
Following time-dependent definition of Sobolev space (see Chapter 5.10 of \cite{evans}), one can proceed to apply similar construction for basis $\phi_{\tau}(x)$ on time dimension. Since we treat time as an additional dimension of space,  we obtain
\begin{equation}
||f||_{H^{\alpha^{\prime}}(\R;H^{\alpha}(\R^{d}))}^{2}\lesssim  \sum_{\xi=0}^{\infty} (1+|\xi|^{\alpha})^{2}(1+\tau^{\alpha^{\prime}})^{2} a_{\xi,\tau}^{2},
\end{equation}
where $a_{\xi,\tau 
} = a_{\xi}a_{\tau}$.

\section{Proof of Main Results}
\label{sec:main_proofs}
In this Appendix, we provide the proofs of our main theoretical results in the paper.

\subsection{Stability of the total loss function $L_{\epsilon,\beta}(u)$ in~\eqref{loss_function}}

\begin{proof}[Proof of Proposition \ref{prop:stability}]
Recall
\begin{equation}
L_{\epsilon,\beta}(u): = 
||u - u^{*}||_{H^{1}_{t}(0,T;H^{1}_{x}(\Omega))}^{2}  + \beta \bigg( I_{\epsilon} (u) +  \int_{0}^{T}\int_{\Omega} e^{-\frac{t}{\epsilon}}u^{*}_{tt}u dx dt - \int_{\Omega} e^{-\frac{T}{\epsilon}}u^{*}_{t}(x,T)u(x,T) dxdt\bigg).
\end{equation}
We can write the excess risk as following
\begin{align}
L_{\epsilon,\beta}(u) - L_{\epsilon,\beta}(u^{*}) &= ||u - u^{*}||_{H^{1}_{t}(0,T;H^{1}_{x}(\Omega))}^{2}  + \beta \bigg( I_{\epsilon} (u) -  I_{\epsilon} (u^{*})  + \int_{0}^{T}\int_{\Omega} e^{-\frac{t}{\epsilon}}u^{*}_{tt}(u^{*}-u) dx dt\\
&\quad\quad+ \int_{\Omega} e^{-\frac{T}{\epsilon}}u^{*}_{t}(x,T)(u^{*}(x,T)-u(x,T)) dxdt \bigg).
\end{align}
\noindent \textbf{Step 1: } Compute $I_{\epsilon}(u) - I_{\epsilon}(u^{*}) + \int_{0}^{T}\int_{\Omega} e^{-\frac{t}{\epsilon}}u^{*}_{tt}(u^{*}-u) dx dt$.

Recall $f(x,t) = u^{*}_{t}(x,t) - \triangle u^{*}(x,t)$, $(x,t) \in  \Omega_{T}$.  Apply integration by parts, we get
\begin{align}
I_{\epsilon}(u) - I_{\epsilon}(u^{*}) &= \int_{0}^{T}  e^{-\tfrac{t}{\epsilon}} \int_{\Omega} \dfrac{1}{2} |u_{t}|^{2} + \dfrac{1}{2\epsilon} |\nabla u|^{2} - \dfrac{1}{\epsilon}f  u dx dt\\
&\quad\quad- \int_{0}^{T}  e^{-\tfrac{t}{\epsilon}} \int_{\Omega} \dfrac{1}{2} |u^{*}_{t}|^{2} + \dfrac{1}{2\epsilon} |\nabla u^{*}|^{2} - \dfrac{1}{\epsilon}f  u^{*} dxdt\\
&= \int_{0}^{T}  e^{-\tfrac{t}{\epsilon}} \int_{\Omega} \dfrac{1}{2} |u_{t}|^{2} + \dfrac{1}{2\epsilon} |\nabla u|^{2} - \dfrac{1}{\epsilon}(u^{*}_{t}- \triangle u^{*}) u dx dt\\
&\quad\quad+ \int_{0}^{T}  e^{-\tfrac{t}{\epsilon}} \int_{\Omega} -\dfrac{1}{2} |u^{*}_{t}|^{2} - \dfrac{1}{2\epsilon} |\nabla u^{*}|^{2} + \dfrac{1}{\epsilon}(u^{*}_{t} - \triangle u^{*}) u^{*} dxdt\\
&= \int_{0}^{T}  e^{-\tfrac{t}{\epsilon}} \int_{\Omega} \dfrac{1}{2} |u_{t}|^{2} + \dfrac{1}{2\epsilon} |\nabla u|^{2} - \dfrac{1}{\epsilon}u^{*}_{t} u - \frac{1}{\epsilon}\nabla u^{*}\nabla u dx dt\\
&\quad\quad+ \int_{0}^{T}  e^{-\tfrac{t}{\epsilon}} \int_{\Omega} -\dfrac{1}{2} |u_{t}^{*}|^{2} - \dfrac{1}{2\epsilon} |\nabla u^{*}|^{2} + \dfrac{1}{\epsilon}u^{*}_{t} u^{*} + \frac{1}{\epsilon} |\nabla u^{*}|^{2} dx dt.
\end{align}
By a second integration by parts to $u_{t}^{*}u$ and $|u^{*}_{t}|^2$, we get
\begin{align}
I_{\epsilon}(u) - I_{\epsilon}(u^{*}) &= \int_{0}^{T}\int_{\Omega}e^{-\frac{t}{\epsilon}}\bigg(\frac{1}{2}|u_{t}|^2  - \frac{1}{2}|u_{t}^{*}|^2 +  u_{tt}^{*} (u^{*} -u) +  |u_{t}^{*}|^{2}  - u_{t}^{*}u_{t} \bigg) dxdt\\
&\quad\quad+\int_{\Omega}e^{-\frac{t}{\epsilon}}u^{*}_{t}\bigg(u - u^{*}\bigg)   dx \bigg|_{t=0}^{T}+\int_{0}^{T}\int_{\Omega}\frac{1}{2\epsilon}e^{-\frac{t}{\epsilon}}|\nabla u - \nabla u^{*}|^{2} dx.
\end{align}
By definition of $\mathcal{U}$, $u(x,0) = u^{*}(x,0)$, it follows that
\begin{align}
I_{\epsilon}(u) - I_{\epsilon}(u^{*})&= \int_{0}^{T}\int_{\Omega}e^{-\frac{t}{\epsilon}} u_{tt}^{*} (u^{*}-u) dxdt
 + \int_{\Omega}e^{-\frac{T}{\epsilon}}u^{*}_{t}(x,T)\big(u(x,T) - u^{*}(x,T)\big)dx \\&\quad\quad+\int_{0}^{T}e^{-\frac{t}{\epsilon}}\int_{\Omega}\frac{1}{2\epsilon}|\nabla u - \nabla u^{*}|^{2} +\frac{1}{2}| u_{t} -u_{t}^{*}|^{2} dxdt.
\end{align}
Then, we obtain 
\begin{align}
I_{\epsilon} (u) - I_{\epsilon} (u^{*})& + \int_{0}^{T}\int_{\Omega} e^{-\frac{t}{\epsilon}}u^{*}_{tt}(u-u^{*}) dx dt + \int_{\Omega}e^{-\frac{T}{\epsilon}}u^{*}_{t}(x,T)\big(u^{*}(x,T) - u(x,T)\big)dx\\&=  \int_{0}^{T}\int_{\Omega}\frac{1}{2}e^{-\frac{t}{\epsilon}}| u_{t} -u_{t}^{*}|^{2} dxdt +\int_{0}^{T}\int_{\Omega}\frac{1}{2\epsilon}e^{-\frac{t}{\epsilon}}|\nabla u - \nabla u^{*}|^{2} dx. \label{integrals}
\end{align}
\textbf{Step 2: } Bound $I_{\epsilon} (u) - I_{\epsilon} (u^{*}) + \int_{0}^{T}\int_{\Omega} e^{-\frac{t}{\epsilon}}u^{*}_{tt}(u^{*}-u) dx dt + \int_{\Omega}e^{-\frac{T}{\epsilon}}u^{*}_{t}(x,T)\big(u^{*}(x,T) - u(x,T)\big)dx  $ from up and below.

To bound first integral in (\ref{integrals}), we may assume without loss of generality that $u(x,t)-u^{*}(x,t)$ is centered in the following sense
\begin{equation}
\int_{0}^{T} e^{-\frac{t}{\epsilon}}\bigg(u(x,t)- u^{*}(x,t)\bigg)dt = 0.
\end{equation}
According to weighted Poincar\'e inequality in Lemma~\ref{lem:weighted-poincare} with respect to the time variable, there exist constant $C(d,\epsilon, \Omega) > 0$ such that 
\begin{equation}
\frac{1}{C^{2}(d,\epsilon, \Omega)} \int_{0}^{T}\int_{\Omega} e^{-\frac{t}{\epsilon}} |u-u^{*}|^{2}dxdt  \leq  \int_{0}^{T}\int_{\Omega}e^{-\frac{t}{\epsilon}}| u_{t} -u_{t}^{*}|^{2} dxdt.
\end{equation}
Then it follows that
\begin{align}
 \int_{0}^{T}\int_{\Omega} \frac{1}{2} e^{-\frac{t}{\epsilon}}| u_{t} -u_{t}^{*}|^{2} &dxdt  =  \int_{0}^{T}\int_{\Omega} \frac{1}{4} e^{-\frac{t}{\epsilon}}| u_{t} -u_{t}^{*}|^{2} dxdt   +  \int_{0}^{T}\int_{\Omega} \frac{1}{4} e^{-\frac{t}{\epsilon}}| u_{t} -u_{t}^{*}|^{2} dxdt  \\
 &\geq \frac{1}{C^{2}(d,\epsilon, \Omega)} \int_{0}^{T}\int_{\Omega} e^{-\frac{t}{\epsilon}} |u-u^{*}|^{2}dxdt +  \int_{0}^{T}\int_{\Omega} \frac{1}{4} e^{-\frac{t}{\epsilon}}| u_{t} -u_{t}^{*}|^{2} dxdt \\
 &=  \min\Bigl\{\frac{e^{-\frac{T}{\epsilon}}}{C^{2}(d,\epsilon, \Omega)},\frac{e^{-\frac{T}{\epsilon}}}{4} \Bigr\} \int_{0}^{T}\int_{\Omega} \Bigl(|u-u^{*}|^2  + |u_{t}-u^{*}_{t}|^2\Bigr) dxdt\\
 &= \min\Bigl\{\frac{e^{-\frac{T}{\epsilon}}}{C^{2}(d,\epsilon, \Omega)},\frac{e^{-\frac{T}{\epsilon}}}{4} \Bigr\} ||u-u^{*}||_{H^{1}_{t}(0,T;L^{2}_{x}(\Omega))}^{2}.
\end{align}
We obtain
 \begin{equation}
 \int_{0}^{T}\int_{\Omega} \frac{1}{2} e^{-\frac{t}{\epsilon}}| u_{t} -u_{t}^{*}|^{2} dxdt  \leq \int_{0}^{T}\int_{\Omega} \frac{1}{2} e^{-\frac{t}{\epsilon}} \Bigl( | u_{t} -u_{t}^{*}|^{2} + |u - u^{*}|^{2} \Bigr)dxdt \leq \frac{1}{2}||u-u^{*}||_{H^{1}_{t}(0,T;L^{2}_{x}(\Omega))}^{2}.
 \label{u_t_bound}
 \end{equation}\\
To bound the second integral in (\ref{integrals})
, we again assume without loss of generality that $u(x,t) - u^{*}(x,t)$ is centered in the following sense
\begin{equation}
\int_{\Omega} u(x,t)- u^{*}(x,t)dx = 0.
\end{equation}
Let's call $P_{\epsilon}(u):= \int_{0}^{T}\int_{\Omega}\frac{1}{2\epsilon}e^{-\frac{t}{\epsilon}} \Bigl(|\nabla u - \nabla u^{*}|^2 + |u - u^{*}|^2 \Bigr) dxdt$.\\

According to standard Poincare inequality to space derivative, we obtain
\begin{align}
\int_{0}^{T}\int_{\Omega}\frac{1}{2\epsilon}e^{-\frac{t}{\epsilon}}|\nabla u - \nabla u^{*}|^{2} dx dt&\leq P_{\epsilon}(u)
\leq \frac{1}{2\epsilon} ||u-u^{*}||_{L^{2}_{t}(0,T;H^{1}_{x}(\Omega))}^{2},
\end{align}
and
\begin{align}
\int_{0}^{T}\int_{\Omega}\frac{1}{2\epsilon}e^{-\frac{t}{\epsilon}}|\nabla u - \nabla u^{*}|^{2} &dx dt\geq P_{\epsilon}(u)
\geq \frac{e^{-\frac{T}{\epsilon}}}{4\epsilon}||u-u^{*}||_{L^{2}_{t}(0,T;H^{1}_{x}(\Omega))}^{2}.
\end{align}
We obtain
\begin{equation}
\frac{e^{-\frac{T}{\epsilon}}}{4\epsilon}||u-u^{*}||_{L^{2}_{t}(0,T;H^{1}_{x}(\Omega))}^{2}\leq \int_{0}^{T}\int_{\Omega}\frac{1}{2\epsilon}e^{-\frac{t}{\epsilon}}|\nabla u - \nabla u^{*}|^{2} dx dt \leq \frac{1}{2\epsilon} ||u-u^{*}||_{L^{2}_{t}(0,T;H^{1}_{x}(\Omega))}^{2}.
\label{u_grad_bound}
\end{equation}
\textbf{Step 3: }Combine  (\ref{u_t_bound}) and (\ref{u_grad_bound}) together.

Let's denote 
\begin{align}E_{\epsilon, \beta}(u)&: = \beta (I_{\epsilon}(u) - I_{\epsilon}(u^{*})) + \beta\int_{0}^{T}\int_{\Omega} e^{-\frac{t}{\epsilon}}u^{*}_{tt}(u^{*}-u) dx dt\\&  + \int_{\Omega}e^{-\frac{T}{\epsilon}}u^{*}_{t}(x,T)\big(u^{*}(x,T) - u(x,T)\big)dx .
\end{align}
Then, we have
\begin{align}
 &E_{\epsilon, \beta}(u) \geq  \frac{\beta e^{-\frac{T}{\epsilon}}}{4\epsilon}||u-u^{*}||_{L^{2}_{t}(0,T;H^{1}_{x}(\Omega))}^{2} +\min\Bigl\{\frac{\beta e^{-\frac{\epsilon}{T}}}{C^{2}(d,\epsilon, \Omega)},\frac{\beta e^{-\frac{\epsilon}{T}}}{4} \Bigr\} ||u-u^{*}||_{H^{1}_{t}(0,T;L^{2}_{x}(\Omega))}^{2},\\
&E_{\epsilon, \beta}(u) \leq \frac{\beta}{2}||u-u^{*}||_{H^{1}_{t}(0,T;L^{2}_{x}(\Omega))}^{2}  + \frac{\beta}{2\epsilon}||u-u^{*}||_{L^{2}_{t}(0,T;H^{1}_{x}(\Omega))}^{2}.
\end{align}

\textbf{Step 4:} Embedding into $H^{1}_{t}([0,T];H^{1}_{x}(\Omega))$. 

By the definition of time-dependent Sobolev space (see [\citealp[Chapter 5]{evans}]) and standard $H^1$ embedding into $L^2$, we can absorb lower and upper bound into $||u - u^{*}||_{H^{1}_{t}([0,T];H^{1}_{x}(\Omega))}^{2}$. It is straightforward to establish upper bound
\begin{align}
||u - u^{*}||_{H^{1}_{t}([0,T];H^{1}_{x}(
\Omega))}^{2}+ E_{\epsilon, \beta}(u) \leq \Bigl( 1+ \frac{\beta}{2} + \frac{\beta}{2\epsilon} \Bigr) ||u - u^{*}||_{H^{1}_{t}([0,T];H^{1}_{x}(
\Omega))}^{2}.
\end{align}
To obtain lower bound, we have {
\begin{align}
||u - u^{*}||_{H^{1}_{t}(0,T;H^{1}_{x}(\Omega))}^{2}+ E_{\epsilon, \beta}(u) &\geq \frac{1}{3} ||u - u^{*}||_{H^{1}_{t}(0,T;H^{1}_{x}(\Omega))}^{2} + \frac{1}{3} ||u - u^{*}||_{H^{1}_{t}(0,T;H^{1}_{x}(\Omega))}^{2} \\
&+\frac{1}{3} ||u - u^{*}||_{H^{1}_{t}(0,T;H^{1}_{x}(\Omega))}^{2}+\frac{\beta e^{-\frac{T}{\epsilon}}}{4\epsilon}||u-u^{*}||_{L^{2}_{t}(0,T;H^{1}_{x}(\Omega))}^{2}\\
&+ \min\Bigl\{\frac{\beta e^{-\frac{\epsilon}{T}}}{C^{2}(d,\epsilon, \Omega)},\frac{\beta e^{-\frac{\epsilon}{T}}}{4} \Bigr\}||u-u^{*}||_{H^{1}_{t}(0,T;L^{2}_{x}(\Omega))}^{2}\\
&\geq \frac{1}{3} ||u - u^{*}||_{H^{1}_{t}(0,T;H^{1}_{x}(\Omega))}^{2} + \frac{1}{3} ||\nabla u - \nabla u^{*}||_{H^{1}_{t}(0,T;L^{2}_{x}(\Omega))}^{2}\\
&+ \min\Bigl\{\frac{\beta e^{-\frac{\epsilon}{T}}}{C^{2}(d,\epsilon, \Omega)},\frac{\beta e^{-\frac{\epsilon}{T}}}{4} \Bigr\}||u-u^{*}||_{H^{1}_{t}(0,T;L^{2}_{x}(\Omega))}^{2}\\
&+\frac{1}{3} ||u_{t} - u^{*}_{t}||_{L^{2}_{t}(0,T;H^{1}_{x}(\Omega))}^{2} +\frac{\beta e^{-\frac{T}{\epsilon}}}{4\epsilon}||u-u^{*}||_{L^{2}_{t}(0,T;H^{1}_{x}(\Omega))}^{2}
\\& \geq ( \frac{1}{3} + \min( \frac{1}{3}, \frac{\beta e^{-\frac{T}{\epsilon}}}{C^{2}(d,\epsilon, \Omega)},\frac{\beta e^{-\frac{T}{\epsilon}}}{4} )) ||u - u^{*}||_{H^{1}_{t}(0,T;H^{1}_{x}(\Omega))}^{2} .
\end{align}}
Finally, we obtain 
{
\begin{align}
( \frac{1}{3} + \min( \frac{1}{3}, \frac{\beta e^{-\frac{T}{\epsilon}}}{C^{2}(d,\epsilon, \Omega)},\frac{\beta e^{-\frac{T}{\epsilon}}}{4} )) ||u - u^{*}||_{H^{1}_{t}(0,T;H^{1}_{x}(\Omega))}^{2}\leq L_{\epsilon,\beta}(u) - L_{\epsilon,\beta}(u^{*})  
\end{align}
and
\begin{align}
 L_{\epsilon,\beta}(u) - L_{\epsilon,\beta}(u^{*})  \leq \Bigl( 1+ \frac{\beta}{2} + \frac{\beta}{2\epsilon} \Bigr) ||u - u^{*}||_{H^{1}_{t}(0,T;H^{1}_{x}(\Omega))}^{2}.
\end{align}}
\end{proof}

\subsection{Upper bound on estimation error}

Proving the upper bound on the estimation error is quite involved. So we outline the plan the proof and divide it into several steps.

To begin this section,  First, we need to prove high probability bound on excess risk in Proposition~\ref{high_probability_bound_on_excess_risk} and incorporate this bound using localization technique introduced in \cite{vdG87,vdG02} along with strong convexity of our loss function. Last step is to plug in approximation error \eqref{approximation_error} to establish classical bias-variance trade-off.

\begin{proposition}[High Probability Bound on Excess Risk]
\label{high_probability_bound_on_excess_risk}
Under Assumptions A1 - A4 in Appendix \ref{sec:main_assumptions} and define $\U_{J}$ as in (\ref{parametrized_function_class}). For any localization parameter $\rho>0$, define
\begin{equation}
\U_{J}(\rho):=\bigg\{u \in \U_{J} \,\bigg{|} \, S_{*,\epsilon}(u)\leq \rho\bigg\}.
\end{equation}\\
Then there exists constant $C_{1}(d,\Omega,T,\epsilon,\beta)$ such that with probability $1-e^{-s}$,\\
\begin{equation}
\underset{u_{J}\in \mathcal{\U_{J}(\rho)}}{\sup}  |\hat{S}_{*,\epsilon}(u_{J}) - S_{*,\epsilon}(u_{J})| \lesssim \psi(\rho,\tau,J) + \sqrt{\frac{2\rho B(\epsilon,\beta,\Omega,T,d)s}{nm}}  + \tau_{1}(\rho, J),
\end{equation}
where

\begin{align}
\psi(\rho,\tau, J)&:=\frac{1}{\sqrt{nm}} J^{\frac{d+2}{2}} \sqrt{\rho B(\epsilon,\beta,\Omega,T,d)\log(1+ \sqrt{\frac{C_{2} \rho} {B(\epsilon,\beta,\Omega,T,d)}}}) + C_{1}(d,\Omega,T, \epsilon, \beta)\sqrt{\frac{\rho }{nm}} 
\end{align}
and
\begin{equation}
\tau_{1}(\rho, J) := \begin{cases} \,\,\frac{J^{d+1}s}{nm} + \frac{RC(\Omega,T)}{nm} J^{2(d+2)}\log(1+ \sqrt{\frac{C_{2} \rho} {B(\epsilon,\beta,\Omega,T,d)}}),\quad q\leq\frac{d}{2}, \\\\
\,\,\frac{Rs}{nm}+\frac{RC(\Omega,T)}{nm} J^{d+2}\log(1+ \sqrt{\frac{C_{2} \rho} {B(\epsilon,\beta,\Omega,T,d)}}),\,\,\,\,\quad\quad\, q>\frac{d}{2}.\end{cases}
\end{equation}
$B(\epsilon, \beta, \Omega, T, d): = (\frac{1}{3}+\frac{\beta}{\beta_{0}})^{-1}\,, \beta_{0}:=\frac{1}{C^{2}(d,\epsilon,\Omega)e^{T/\epsilon}}$ and $C(d,\epsilon,\Omega)$\text{ is a Cheeger constant.}
\end{proposition}

\begin{proof}[Proof of Proposition~\ref{high_probability_bound_on_excess_risk}]
For simplicity of notation, we denote the excess risks as\\
\begin{equation}
S_{*}(u):= L_{\epsilon,\beta}(u) - L_{\epsilon,\beta}(u^{*}),\,\,\hat{S}_{*}(u): = \hat{L}_{\epsilon,\beta}(u) -\hat{L}_{\epsilon,\beta}(u^{*})
\end{equation}
and the space of localized functional class $\U_{J}$ as\\
\begin{equation}
\U_{J}(\rho):= \bigg\{u_{J}\in \U_{J} \,\,\bigg| S_{*,\epsilon}(u_{J})\leq \rho   \bigg\}.
\end{equation}

Our goal is to control excess risk $L_{\epsilon,\beta}(\hat{u}_{J}) - L_{\epsilon,\beta}(u^{*})$ for a parametrized estimator $\hat{u}_{J}\in \U_{J}$, where $\U_{J}$ is a class of functions that have Fourier series representations with truncation of the level $J$. First, we decompose the excess risk $L_{\epsilon,\beta}(\hat{u}_{J}) - L_{\epsilon,\beta}(u^{*})$ into telescoping sums, and we view the telescoping sums as a localized version of the empirical process and control empirical process for $u_{J}\in \U_{J}(\rho)$.\\

To start,  we center all functions appropriately and consider the extension of a restriction of the functions $u_{J}$ and $u^{*}$ to compact space $\Omega$ over finite time horizon $[0,T]$ , which
allows us to use Fourier expansions for harmonic analysis. Note that  $u_{J}$ and $u^{*}$ are defined in all of space-time $\R^{d}\times \R$, but we do not need to control their norms over the entire
space-time domain. The existence of such extension is mainly attributed to extension Theorem~\ref{Sobolev_extension_theorem}, which induces space-time expansion Lemma~\ref{space_time_extension}.

With notation of linear extension operator $ext$ defined in Theorem~\ref{space_time_extension}, we use the following notation 
\begin{equation}
ext\, u_{J} = u_{J},\,\, ext\, u^{*} = u^{*}.
\end{equation}

For any $u_{J} \in \U_{J}(p)$, by optimality of $\hat{u}_{J}$ we have
\begin{align}
& L_{\epsilon, \beta}(\hat{u}_{J}) - L_{\epsilon, \beta}(u_{J}) \\
=& L_{\epsilon,\beta}(\hat{u}_{J})- L_{\epsilon,\beta}(u^{*})  +   \hat{L}_{\epsilon,\beta}(u^{*})- \hat{L}_{\epsilon,\beta}(\hat{u}_{J})+\hat{L}_{\epsilon,\beta}(\hat{u}_{J}) -\hat{L}_{\epsilon,\beta}(u^{*}) - L_{\epsilon,\beta}(u_{J}) + L_{\epsilon,\beta}(u^{*})\\
=& L_{\epsilon,\beta}(\hat{u}_{J})- L_{\epsilon,\beta}(u^{*})   -   \hat{L}_{\epsilon,\beta}(\hat{u}_{J})  - \hat{L}_{\epsilon,\beta}(u^{*}) + \hat{L}_{\epsilon,\beta}(\hat{u}_{J}) -\hat{L}_{\epsilon,\beta} (u^{*}) -  L_{\epsilon,\beta}(u_{J}) - L_{\epsilon,\beta}(u^{*})\\
\leq& \bigg(S_{*}(\hat{u}_{J}) - \hat{S}_{*} (\hat{u}_{J})\bigg) + \bigg(\hat{S}_{*}(u_{J}) - S_{*}(u_{J})  \bigg).
\end{align}
From the above inequality, it is enough to bound $S_{*}(\hat{u}_{J}) - \hat{S}_{*}(u_{J})$ to control $L_{\epsilon,\beta}(\hat{u}_{J}) - L_{\epsilon,\beta}(u_{J})$. To proceed, we consider bounding a localized version of the above empirical process 
\begin{equation}
 \underset{u_{J}\in \U_{J}(\rho)}{\sup}S_{*}(u_{J}) - \hat{S}_{*} (u_{J}).
\end{equation}
We want to first control the expectation of the supremum. The expectation control has two components: the first component contains the bounding tail error arising from the truncation of the Fourier series at the truncation of level $J$ and the second part contains statistical error from truncated Fourier series itself. Once we obtain upper bound for the expectation, we are able to produce a high-probability bound via Talagrand's concentration inequality, cf. Theorem~\ref{Talagrand Inequality}.
\\\\
For convenience, let's write $h= ext\,(u_{J}- u^{*})$ and its local basis Representation as\\
\begin{align}
h &= \sum_{0\leq \tau, |\xi|_{\infty}\leq \infty} a_{\tau,\xi} \phi_{\tau,|\xi|_{\infty}}(x,t)\\
&= \sum_{0\leq \tau \leq J^{2},0\leq|\xi|_{\infty} \leq J} a_{\tau,\xi} \phi_{\tau,|\xi|_{\infty}}(x,t) + \sum_{\tau>J^{2},|\xi|_{\infty} >J} a_{\tau,\xi} \phi_{\tau,|\xi|_{\infty}}(x,t)\\
&= P_{J,J^{2}}h + P_{>J,>J^{2}}h.
\end{align}
Plugging local basis representation, we obtain 

\begin{align}
& \E \bigg[ \underset{u_{J}\in \U_{J}(\rho)}{\sup} |S_{*,\epsilon}(u_{J}) - \hat{S}_{*.\epsilon}(u_{J})| \bigg] \\
=& \E \bigg[ \underset{u_{J}\in \U_{J}(\rho)}{\sup} | L_{\epsilon,\beta}(u_{J}) -  \hat{L}_{\epsilon,\beta}(u_{J}) + \hat{L}_{\epsilon,\beta}(u^{*}) - L_{\epsilon,\beta}(u^{*})  | \bigg]\\
=&  \E \bigg[ \underset{u_{J}\in \U_{J}(\rho)}{\sup} | (L_{\epsilon,\beta} - \hat{L}_{\epsilon,\beta})(u_{J}-u^{*}) | \bigg] \\
\leq&  \underbrace{\E \bigg[ \underset{u_{J}\in \U_{J}(\rho)}{\sup} | (L_{\epsilon,\beta} - \hat{L}_{\epsilon,\beta})P_{J, J^{2}}h | \bigg]}_{I_{1}}+ \underbrace{\E \bigg[ \underset{u_{J}\in \U_{J}(\rho)}{\sup} | (L_{\epsilon,\beta} - \hat{L}_{\epsilon,\beta})P_{>J,>J^{2}}h | \bigg]}_{I_{2}}.
\end{align}
Here we want to apply Symmetrization technique to $I_{1}$ and $I_{2}$.  Since the sampling scheme \eqref{regression_model} and \eqref{regression_model_2} treats time dimension $t$ as another additional spatial dimension, we flatten $d+1$-d spatial-temporal training data matrix $\{ (x_{ij},t_{j} ) \}_{i=1,j=1}^{n,m}$ to $1$-d data vector $\{z_{i}\}_{i=1}^{nm}$ such that
\[
z_{1} = (x_{11},t_{1}),....,z_{m} = (x_{m1},t_{1}), z_{m+1} =(x_{12},t_{2}),.... ,z_{2m} = (x_{m2},t_{2}),....,z_{mn} = (x_{mn},t_{n}).
\]
In this way, standard symmetrization argument with $nm$ copies of independent Rademacher random variables $\sigma_{i}$ is applied
\begin{equation}
I_{1}\leq \underbrace{\frac{1}{nm}\E \bigg[ \underset{u_{J}\in \U_{J}(\rho)}{\sup}  \sum_{i=1}^{nm} \sigma_{i}P_{J,J^{2}}h(z_{i})\bigg]}_{T_{1}}, \,\,\,\, I_{2}\leq \underbrace{\frac{1}{nm}\E \bigg[ \underset{u_{J}\in \U_{J}(\rho)}{\sup}  \sum_{i=1}^{nm} \sigma_{i}P_{>J,>J^{2}}h(z_{i})\bigg]}_{T_{2}}.
\end{equation}
The rest of the subsection is devoted to first bound $T_{2}$ and then bound $T_{1}$.
\vspace{1cm}\\
\textbf{Control of $T_{2}$}:
\begin{equation}
T_{2}: = \frac{1}{nm}\E \bigg[ \underset{u_{J}\in \U_{J}(\rho)}{\sup}  \sum_{i=1}^{nm} \sigma_{i}P_{>J,>J^{2}}h(z_{i})\bigg] = \frac{1}{nm}\E \bigg[ \underset{u_{J}\in \U_{J}(\rho)}{\sup}  \sum_{i=1}^{nm} \sigma_{i} \sum_{\tau>J^{2},|\xi|_{\infty} >J} a_{\tau,\xi} \phi_{\tau,\xi}(x,t) \bigg].
\end{equation}
By Cauchy-Schwarz inequality, we obtain
\begin{align}
\sum_{i=1}^{nm} \sigma_{i} \sum_{\tau>J^{2},|\xi|_{\infty} >J} a_{\tau,\xi} \phi_{\tau,\xi}(x,t) \leq \underbrace{{\bigg(\sum_{i=1}^{nm} \sigma_{i}^{2}\bigg)^{\frac{1}{2}}}}_{T_{2,1}} \underbrace{\bigg(\sum_{i=1}^{nm} \bigl(\sum_{\tau>J^{2},|\xi|>J}a_{\tau,\xi} \phi_{\tau,\xi}(x,t)\bigr)^{2}\bigg)^{\frac{1}{2}}}_{T_{2,2}}\\
T_{2,2} = \bigg(\sum_{\tau>J^{2},|\xi|_{\infty} >J} a_{\tau,\xi}^{2} \phi_{\tau,\xi}^{2}(z_{i}) + \sum_{1\leq j\neq k\leq nm } a_{\tau,\xi}^{2}\phi_{\tau,\xi}(z_{j}) \phi_{\tau,\xi}(z_{k}) \bigg)^{\frac{1}{2}}.
\end{align}
By Jensen inequality, we can pass expectation over $\sigma$ inside $T_{2,1}$ and expectation over $z$ inside $T_{2,2}$. Note $\sigma_{i}$ and $\sigma_{j}$ are independent variables with $\E \sigma_{i} = \E\sigma_{j} = 0$, and $||\phi_{\tau,\xi}||_{L_{2}}^{2} =1$.
\begin{align}
\E_{\sigma,z}\bigg[\,\, T_{2,2}\,\,\bigg] &\leq \big(\sum_{\tau>J^{2},|\xi|_{\infty} >J} a_{\tau,\xi}^{2} ||\phi_{\tau,\xi}||_{L_{2}}^{2} + \sum_{\tau,\tau^{\prime}>J^{2}} \sum_{|\xi|_{\infty}, |\xi^{\prime}|_{\infty}>J}a_{\tau,\xi}a_{\tau^{\prime},\xi^{\prime}}  ||\phi_{\tau,\xi} \phi_{\tau^{\prime},\xi}(z_{k})||_{L_1} \big)^{\frac{1}{2}} \\
&= \big(\sum_{\tau>J^{2},|\xi|_{\infty} >J} a_{\tau,\xi}^{2} \big)^{\frac{1}{2}}\hspace{3.2cm} - \text{ Orthogonality of } \phi_{\tau,\xi}(z) \ \\
&\leq ||P_{>J, >J^{2}}h||_{H^{1}_{t}(\R;H^{1}_{x}(\R^{d}))}\\
&\leq C(\Omega,T) ||P_{>J,>J^{2}}h||_{H^{1}_{t}(0,T;H^{1}_{x}(\Omega))}\hspace{1.1cm} - \text{ space-time extension Lemma~\ref{space_time_extension} }\\
&\lesssim  RJ^{-(2p+q-3)}\sqrt{\rho}. \hspace{3.8cm} - \text{ stability estimate Proposition~\ref{prop:stability}}
\end{align}
Combining $T_{2,1}$ and $T_{2,2}$ all together, we obtain
\begin{equation}
T_{2} = T_{2,1} * \E_{\sigma,z}[T_{2,2}]\lesssim J^{-(2p+q-3)}\sqrt{\frac{\rho}{nm}} \label{T_2}.
\end{equation}
\textbf{Control of $T_{1}$}:\\
In this section, we stick with $h:= ext\,(u_{J} - u^{*})$. We use a chaining bound for metric entropy to bound $T_{1}$ using $L^2_{t}(\R;L^2_{x}(\R^{d}))$ bracketing number, and note $L^2_{t}(\R;L^2_{x}(\R^{d}))$ is equivalent to $L^2_{t,x}(\R\times \R^{d})$.  In particular, we construct a bracketing number for $J$ finitely many Fourier Coefficients $a_{\xi,\tau}$ in $L^{2}$ balls. \\\\
For $h_{1}, h_{2} \in \U_{J}(\rho)$,

\begin{align}
\big|\big|\,P_{J, J^{2}}h_{1} - P_{J,J^{2}}h_{2}\,\big|\big|_{L^{2}_{t}(0,T;L^{2}_{x}(
\Omega))}^{2} &= \big|\big| \sum_{1\leq |\xi|_{\infty}\leq J, 1\leq \tau \leq J^{2}} a^{(1)}_{\xi,\tau} \phi_{\xi,\tau}(x,t) - a^{(2)}_{\xi,\tau} \phi_{\xi,\tau}(x,t) \,\big|\big|_{L^{2}_{t}(0,T;L^{2}_{x}(
\Omega))}^{2}\\
&\lesssim 
\big|\big| \sum_{1\leq |\xi|_{\infty}\leq J, 1\leq \tau \leq J^{2}} (\,a^{(1)}_{\xi,\tau} - a^{(2)}_{\xi,\tau}\,) \phi_{\xi,\tau}(x,t) \,\,\big|\big|_{L^{2}_{t}(0,T;L^{2}_{x}(\Omega))}^{2} 
\\
&\leq \sum_{1\leq |\xi|_{\infty}\leq J, 1\leq \tau \leq J^{2}} \big|\,a^{(1)}_{\xi,\tau} - a^{(2)}_{\xi,\tau}\,\big|^{2}.
\end{align}
Hence we obtain 
\begin{equation}
\big|\big|\,P_{J}h_{1} - P_{J}h_{2}\,\big|\big|_{L^{2}_{t}(0,T;L^{2}_{x}(
\Omega))}^{2} \lesssim \sum_{1\leq |\xi|_{\infty} \leq J, 1\leq \tau \leq J^{2}} |a^{(1)}_{\xi,\tau} - a^{(2)}_{\xi,\tau}|^{2}. \label{packing-number-inequality}
\end{equation}
The inequality above suggests a control of finitely many spatial-temporal cofficients $a_{\xi,\tau}$ in $L^{2}$ ball of coefficient space is sufficient to construct an $\epsilon$-net for $h\in \U_{J}(\rho)$.\\\\
We want to construct $L^{2}$ packing number for the space of  finitely many Fourier coefficients $a_{\xi,\tau}$. Using standard $L^{2}$-packing for $L^{2}$ balls argument in Lemma \ref{lemma_F.6} and note the dimension of  coefficients are at most $J^{d+2}$, we have

\begin{align}
\log N(\U_{J}(\rho), ||\cdot||_{L^{2}_{t}(0,T;L^{2}_{x}(\Omega))},\epsilon)
&\lesssim  \log(\frac{3}{\delta})^{J^{d+2}}= J^{d+2} \log(\frac{1}{\epsilon}).
\end{align}

With the above packing number at hand, we're ready to apply Bernstein Chaining in Theorem~\ref{Bernstein Chaining}. Before applying Bernstein Chaining, we nee to check finite second moment condition and finite $L_{\infty}$ norm condition.\\

From the lower bound of stability estimates in Proposition~\ref{prop:stability} and Lemma~\ref{J_truncation_estimates}, we can verify second moment condition for $P_{J}h\in \U_{J}(\rho)$
\begin{align}
\label{second-moment-condition}
||P_{J}h||_{L^{2}_{t}(\R ;L^{2}_{x}( \R^{d}))} \leq ||h||_{L^{2}_{t}(\R ;L^{2}_{x}( \R^{d}))}\lesssim C(\Omega,T) \sqrt{\rho B(\epsilon,\beta,\Omega,T,d)},
\end{align}
where $B(\epsilon,\beta,\Omega,T,d):= (1+\frac{\beta}{\beta_{0}})^{-1}$, $\beta_{0} =  e^{T/\epsilon} C^{2}(d,\epsilon,\Omega)$, and $C(d,\epsilon, \Omega)$ is the Cheeger constant here.\\

To verify $L_\infty$ condition, we use Lemma~\ref{J_truncation_estimates} to obtain
\begin{equation}
||P_{J}h||_{\infty}
\lesssim 
\begin{cases}
 R, \,\,\,\quad q>\frac{d}{2},\\
 J^{d+2}, \quad q\leq\frac{d}{2}.
\end{cases}
\label{L_infty_condition}
\end{equation}
Next we will incorporate bracketing numbers into the Dudley integral theorem (cf. Theorem \ref{theorem_A.3} ) to get
\begin{align}
D_{[]}&:=D_{[ ]}\bigg(\sqrt{\rho B(\epsilon,\beta,\Omega,T,d)},\,\, \U_{J}(\rho), \,\,||\cdot||_{L^{2}_{t,x}(\mathcal{P})}\bigg)\\
&=\displaystyle\int_{0}^{\sqrt{\rho B(\epsilon,\beta,\Omega,T,d)}} \sqrt{logN_{[ ]}\bigg(\epsilon, \U_{J}(\rho), ||\cdot||_{L^{2}_{t,x}(\mathcal{P})}\bigg)} \,\, d\epsilon\\
&\lesssim \int_{0}^{\sqrt{ \rho B(\epsilon,\beta,\Omega,T,d)}} J^{\frac{d+2}{2}} \sqrt{ \log(1+ \frac{1}{\epsilon})}\,\, d\epsilon
\end{align}
and
\begin{equation}
D_{[ ]}\lesssim J^{(d+2)/2}\sqrt{\rho B(\epsilon,\beta,\Omega,T,d)\log(1+ \sqrt{\frac{C_{2} \beta_{0}}{\rho \beta}}}),
\end{equation}
where $C_{2}$ is some appropriately chosen constant.
\\\\
Plugging second moment condition (\ref{second-moment-condition}) and $L^{\infty}$ condition (\ref{L_infty_condition}) into Theorem  \ref{Bernstein Chaining}, we obtain 
\begin{equation}
T_{1}: \lesssim \frac{1}{\sqrt{nm}}D_{[]} \bigg(1+ \frac{D_{[]}}{\sqrt{nm}||P_{J}h||_{L^{2}_{t}([0,T];L^{2}_{x}(\Omega))}^{2}} ||P_{J}h||_{\infty}\bigg).
\end{equation}
We obtain 
\begin{equation}
T_{1}\lesssim \frac{1}{\sqrt{nm}} J^{\frac{d+2}{2}} \sqrt{\rho B(\epsilon,\beta,\Omega,T,d)\log(1+ \sqrt{\frac{C_{2} \rho} {B(\epsilon,\beta,\Omega,T,d)}}}) + \tau(\rho, J),\label{T_1}
\end{equation}
where 
\begin{equation}
\tau(\rho, J)\lesssim \begin{cases} \frac{R}{nm} J^{2(d+2)}\log(1+ \sqrt{\frac{C_{2} \rho} {B(\epsilon,\beta,\Omega,T,d)}}),\hspace{0.6cm} q\leq\frac{d}{2},\\
\frac{R}{nm} J^{d+2}\log(1+ \sqrt{\frac{C_{2} \rho} {B(\epsilon,\beta,\Omega,T,d)}}),\hspace{1cm} q>\frac{d}{2}.
\end{cases}
\end{equation}
\textbf{Final Statistical Estimation Error Bound: }
\hfill\\
From control of $T_{1}$ (\ref{T_1}) and control of $T_{2}$ (\ref{T_2}), we obtain\\

\begin{align}
T_{1}&\lesssim \frac{1}{\sqrt{nm}} J^{\frac{d+2}{2}} \sqrt{\rho B(\epsilon,\beta,\Omega,T,d)\log(1+ \sqrt{\frac{C_{2} \rho} {B(\epsilon,\beta,\Omega,T,d)}}}) + \tau(\rho, J),\\
T_{2}&\lesssim  C(d,\Omega,T, \epsilon, \beta)\sqrt{\frac{\rho}{nm}},\\
T_{1} + T_{2} &= \psi(\rho,\tau, J) + \tau(\rho,J).
\end{align}
\\
Combining the bound on $T_{1}$ and $T_{2}$ all together yields
\begin{equation} \E \bigg[ \underset{u_{J}\in \U, S_{*,\epsilon}(u_{J})\leq \rho}{\sup} |S_{*,\epsilon}(u_{J}) - \hat{S}_{*.\epsilon}(u_{J})| \bigg] \lesssim \psi(\rho,\tau, J) + \tau(\rho,J),
\end{equation}\\
where \begin{align}\psi(\rho,\tau, J)&:=\frac{1}{\sqrt{nm}} J^{\frac{d+2}{2}} \sqrt{\rho B(\epsilon,\beta,\Omega,T,d)\log(1+ \sqrt{\frac{C_{2} \rho} {B(\epsilon,\beta,\Omega,T,d)}}}) + C(d,\Omega,T, \epsilon, \beta)\sqrt{\frac{\rho }{nm}}. \end{align}
Finally, we apply Talagrand's concentration inequality Theorem \ref{Talagrand Inequality} to bound $I_{\epsilon}(h) - \hat{I}_{\epsilon}(h)$, which upper bounds excess risk. In order to apply the concentration inequality, we again need to verify the same second moment condition (\ref{second-moment-condition}) and $L_{\infty}$ condition (\ref{L_infty_condition}), which were verified earlier:
\\
\begin{equation}
||h||_{L^{2}_{t,x}(\R\times \R^{d})} \lesssim \sqrt{\rho B(\epsilon,\beta,\Omega,T,d)}, \,\,
||P_{J}h||_{\infty}
\lesssim 
\begin{cases}
 R, \hspace{0.9cm} q>\frac{d}{2},\\
 J^{d+2}, \hspace{0.4cm} q\leq\frac{d}{2}.
\end{cases}
\end{equation}
With probability $1-e^{-s}$,
\begin{equation}
\underset{u_{J}\in \mathcal{\U_{J}(\rho)}}{\sup}  |\hat{S}_{*,\epsilon}(u_{J}) - S_{*,\epsilon}(u_{J})| \lesssim \psi(\rho,\tau,J) + \sqrt{\rho B(\epsilon,\beta,\Omega,T,d)\cdot\frac{2s}{nm}}  + \tau_{1}(\rho, J),
\end{equation}
where 
\begin{equation}
\tau_{1}(\rho, J) := \begin{cases} \frac{J^{d+1}s}{nm} + \frac{R}{nm} J^{2(d+2)}\log(1+ \sqrt{\frac{C_{2} \rho} {B(\epsilon,\beta,\Omega,T,d)}}),\hspace{0.5cm} q\leq\frac{d}{2}, \\\\
\frac{Rs}{nm}+\frac{R}{nm} J^{d+2}\log(1+ \sqrt{\frac{C_{2} \rho} {B(\epsilon,\beta,\Omega,T,d)}}), \hspace{1.2cm} q>\frac{d}{2}.
\end{cases}
\end{equation}
\end{proof}

\begin{remark} we obtain $L_{\infty}$ bound on $h$ would be through the fact that $L^{\infty}$ norm agrees with $C^{0}$ norm on compact set $\Omega_{T}$. By embedding theorem \ref{lem:morrey_embedding}, we can use $H^{q}_{x}$ norm to bound $C^{0}_{x}$ and therefore $L^{\infty}_{x}$ when spatial regularity $ q> \frac{d}{2} $, while temporal domain lies in 1d and therefore it's naturally embedded into $L^{\infty}_{t}$.
However, this only works when spatial regularity $q$ is large enough against dimension $d$.
\end{remark}

Next we use localization technique to prove the following proposition.

\begin{proposition}[Localization]
\label{localization}
Under Assumptions A1- A4 in Appendix ~\ref{sec:main_assumptions} hold and define $\U_{J}$ as in (\ref{parametrized_function_class}). Then, there exist constant $C_{2}:=C_{2}(\Omega,T)$ such that estimator $\hat{u}_{J}$ defined in (\ref{empirical_loss_function}) satisfies with probability at least $1-e^{-s}$,\\\\
If $\,q<\frac{d}{2},$
\begin{align}
||\hat{u}_{J}- u^{*}||_{H^{1}_{t}(0,T;H^{1}_{x}(\Omega))}^2&\lesssim \bigg[1+(1+\frac{\beta}{2\epsilon}) B(\epsilon
\beta, \Omega,T,d)\bigg]||u_{J}-u^{*}||_{{H^{1}_{t}(0,T;H^{1}_{x}(\Omega))}}^{2}\\& + (1+ \frac{\beta}{2\epsilon} )B^{2}(\epsilon,
\beta, \Omega,T,d)\frac{J^{2(d+2)}}{nm}\log{(1+C_{2}nm)}\\
&+ \frac{J^{2(d+2)}s+1}{B(\epsilon,
\beta, \Omega,T,d)nm}.
\end{align}
Moreover, 
\begin{align}
\E||\hat{u}_{J}- u^{*}||_{H^{1}_{t}([0,T];H^{1}_{x}(\Omega))}^2&\lesssim \bigg[1+(1+\frac{\beta}{2\epsilon}) B(\epsilon,
\beta, \Omega,T,d)\bigg]||u_{J}-u^{*}||_{{H^{1}_{t}(0,T;H^{1}_{x}(\Omega))}}^{2}\\& + (1+ \frac{\beta}{2\epsilon} )B^{2}(\epsilon,
\beta, \Omega,T,d)\frac{J^{2(d+2)}}{nm}\log{(1+C_{2}nm)}\\
&+ \frac{1}{B(\epsilon,
\beta, \Omega,T,d)nm}.
\end{align}
On the other hand when $q\geq\frac{d}{2}$, with the probability $1-e^{-s}$,
\begin{align}
||\hat{u}_{J}- u^{*}||_{H^{1}_{t}(0,T;H^{1}_{x}(\Omega))}^2 &\lesssim  \bigg[1+(1+\frac{\beta}{2\epsilon}) B(\epsilon,
\beta, \Omega,T,d)\bigg]||u_{J}-u^{*}||_{{H^{1}_{t}(0,T;H^{1}_{x}(\Omega))}}^{2}\\& +  B(\epsilon,
\beta, \Omega,T,d)(1+\frac{\beta}{2\epsilon})\frac{J^{d+2}}{nm}\log{(1+C_{2}nm)}+ \frac{1+s}{nm}.
\end{align}
Moreover,
\begin{align}
\E||\hat{u}_{J}- u^{*}||_{H^{1}_{t}(0,T;H^{1}_{x}(\Omega))}^2 &\lesssim_{R}  \bigg[1+(1+\frac{\beta}{2\epsilon}) B(\epsilon,
\beta, \Omega,T,d)\bigg]||u_{J}-u^{*}||_{{H^{1}_{t}(0,T;H^{1}_{x}(\Omega))}}^{2}\\& +  B(\epsilon,
\beta, \Omega,T,d)(1+\frac{\beta}{2\epsilon})\frac{J^{d+2}}{nm}\log{(1+C_{2}nm)}+ \frac{1}{nm} 
\end{align}
\label{prop_8.2}
\end{proposition}

\begin{proof}[Proof of Proposition~\ref{localization}]
Let $\hat{u}_{J}:= \underset{u\in \U_{J}(\rho)}{\argmin}\, \hat{L}_{\epsilon,\beta} (u)$ and $u_{J}:= \underset{u\in \U_{J}(\rho)}{\argmin}\, L_{\epsilon,\beta} (u)$. \\\\
Fix $r>0$ to be determined later, we set 
\begin{equation}
\hat{u}_{s}:= s\hat{u}_{J} +(1-s) u_{J}, \,\,s = \frac{\sqrt{r}}{\sqrt{r}+ ||\hat{u}_{J} - u_{J} ||_{H^{1}_{t}(0,T;H^{1}_{x}(\Omega))}}.
\end{equation}
Since $s \in [0,1]$ and $\U_{J}$ is convex, we have $\hat{u}_{s}\in \U_{J}$ and $\hat{u}_{s}$ is localized in the following sense 
\begin{equation}
||\hat{u}_{s} - u_{J}||_{H^{1}_{t}(0,T;H^{1}_{x}(\Omega))}\leq \frac{\sqrt{r}|| \hat{u}_{s} - u_{J}||_{{H^{1}_{t}(0,T;H^{1}_{x}(\Omega))}}}{\sqrt{r}+ || \hat{u}_{s} - u_{J}||_{{H^{1}_{t}(0,T;H^{1}_{x}(\Omega))}}} \leq \sqrt{r}.
\end{equation}
Using stability estimate in Proposition \ref{prop:stability} and the triangle inequality, we therefore have $\hat{u}_{s}\in \U_{J}(\rho)$.
{\begin{align}
S_{*,\epsilon}(\hat{u}_{s}) \leq& \bigg(1+ \frac{\beta}{2\epsilon}\bigg) ||\,\hat{u}_{s} - u^{*}||_{{H^{1}_{t}(0,T;H^{1}_{x}(\Omega))}}^{2} 
\leq \bigg(1+ \frac{\beta}{2\epsilon}\bigg)  \bigg( r + ||\,u_{J} - u^{*}||_{{H^{1}_{t}(0,T;H^{1}_{x}(\Omega))}}^{2}\bigg):=\rho.
\label{r_definition}
\end{align}}
By Convexity of $\hat{S}_{*,\epsilon}$, we have 
\begin{equation}
\hat{S}_{*,\epsilon}(\hat{u}_{s})\leq  \hat{S}_{*,\epsilon}(u_{J}) + s\bigg(\hat{S}_{*,\epsilon}(\hat{u}_{J}) - \hat{S}_{*,\epsilon}(u_{J}) \bigg).
\end{equation}
By the definition of empirical minimizer $\hat{u}$, we have 
\begin{equation}
 \hat{L}_{\epsilon} (\hat{u}_{J}) \leq  \hat{L}(u_{J}).
\end{equation}
It follows  
\begin{equation}
\hat{S}_{*,\epsilon}(\hat{u}_{J}) - \hat{S}_{*,\epsilon}(u_{J}) = \hat{L}_{\epsilon} (\hat{u}_{J}) - \hat{L}(u_{J}) \leq 0.
\end{equation}
Therefore, the empirical excess risk of $\hat{u}_{s}$ is smaller than $u_{J}$
\begin{equation}
\hat{S}_{*,\epsilon}(\hat{u}_{s}) \leq \hat{S}_{*,\epsilon}(u_{J}).
\end{equation}
Next we want to bound excess risk $S_{*,\epsilon}$ using generaliztion error
\begin{align}
&\hat{S}_{*,\epsilon}(\hat{u}_{s}) -S_{*,\epsilon}(\hat{u}_{s}) + S_{*,\epsilon}(\hat{u}_{s})\leq \hat{S}_{*,\epsilon}(u_{J}) + S_{*,\epsilon}(u_{J}) -S_{*,\epsilon}(u_{J}),\\
& S_{*,\epsilon}(\hat{u}_{s})\leq  S_{*,\epsilon} (u_{J}) + \hat{S}_{*,\epsilon}(u_{J}) -S_{*,\epsilon}(u_{J}) +S_{*,\epsilon}(\hat{u}_{s}) - \hat{S}_{*,\epsilon}(\hat{u}_{s}),\\
&S_{*,\epsilon}(\hat{u}_{s})\leq S_{*,\epsilon}(u_{J}) + 2\underset{u\in \U_{J}}{sup} \big |S_{*,\epsilon}(u) - \hat{S}_{*,\epsilon}(u) \big|.
\end{align}
\label{generalization_bound}
Using lower bound of stability estimate in Proposition \ref{prop:stability} and Proposition \ref{high_probability_bound_on_excess_risk}, we obtain 
\begin{align}
||u_{J} - \hat{u}_{s}||^{2}_{H^{1}_{t}(0,T;H^{1}_{x}(\Omega))} &\leq 2||u_{J} - u^{*}||^{2}_{{H^{1}_{t}(0,T;H^{1}_{x}(\Omega))}} + 2||\hat{u}_{s} - u^{*}||^{2}_{{H^{1}_{t}(0,T;H^{1}_{x}(\Omega))}}\\
&\leq B(\epsilon,\beta,\Omega,T,d) \bigg[2S_{*,\epsilon}(u_{J}) + 2S_{*,\epsilon}(\hat{u}_{s}) - e^{-\tfrac{T}{\epsilon}} ||\nabla u(\cdot,T) - \nabla u^{*}(\cdot,T)||^{2}_{L^{2}(\Omega)}  \bigg]\\
&\leq 4B(\epsilon,\beta,\Omega,T,d) \bigg[ S_{*,\epsilon}(u_{J}) + \underset{u\in \U_{J}}{sup} ( S_{*,\epsilon}(u) - \hat{S}_{*,\epsilon}(u)) - E_{T}(\epsilon)\bigg], 
\end{align}
where we denote $E_{T}(\epsilon):= \frac{e^{-\tfrac{T}{\epsilon}}}{2} ||\nabla u(\cdot,T) - \nabla u^{*}(\cdot,T)||^{2}_{L^{2}(\Omega)}$.
\\\\
With $s= \frac{1}{2}$, we have $||\hat{u}_{s} - u_{J}||_{H^{1}_{t}([0,T];H^{1}_{x}(\Omega))}\geq \frac{\sqrt{r}}{2}$ iff $||\hat{u}_{J} - u_{J}||_{H^{1}_{t}([0,T];H^{1}_{x}(\Omega))}\geq \sqrt{r} $. Using upper and lower bound of stability estimate in  Proposition \ref{prop:stability} and the relation (\ref{r_definition}) between $\rho$ and $r$ leads to the following
{
\begin{align}
&\mathbb{P}\Bigl(||\hat{u}_{J} - u^{*}||_{H^{1}_{t}(0,T;H^{1}_{x}(\Omega))}^{2} \geq r + ||u_{J}- u^{*}||_{H^{1}_{t}(0,T;H^{1}_{x}(\Omega))}^{2}\Bigr)\\
&\leq \mathbb{P}\Bigl(||\hat{u}_{s} - u_{J}||_{H^{1}_{t}(0,T;H^{1}_{x}(\Omega))}^{2} \geq \frac{r}{4} \Bigr)\\
&\leq \mathbb{P}\bigg(\underset{u\in \U_{J}(\rho)}{\sup}  (S_{*,\epsilon}(u) - \hat{S}_{*,\epsilon}(u) ) \gtrsim \frac{r}{16B(\epsilon,\beta,\Omega,T,d)} - S_{*,\epsilon}(u_{J}) + E_{T}(\epsilon)   \bigg)\,\, \\
&\leq \mathbb{P}\Bigl(\underset{u\in \U_{J}(\rho)}{\sup}  (S_{*,\epsilon}(u) - \hat{S}_{*,\epsilon}(u) ) \gtrsim \frac{\rho - (1+\frac{\beta}{2\epsilon}) ||u_{J} - u^{*}||_{H^{1}_{t}(0,T;H^{1}_{x}(\Omega))}^{2}}{16(1+\frac{\beta}{2\epsilon})B(\epsilon,\beta,\Omega,T,d)} - S_{*,\epsilon}(u_{J}) + E_{T}(\epsilon) \Bigr)\,\,.
\end{align}
}
In observation of the high probability bound of the excess risk, we can use the previous result on probabilistic bound of $S_{*,\epsilon}(u) - \hat{S}_{*,\epsilon}(u)$. In order to apply this previous result, we want the following to hold:
\begin{align}
\frac{\rho }{16(1+\frac{\beta}{2\epsilon})B(\epsilon,\beta,\Omega,T,d)}  &\geq S_{*,\epsilon}(u_{J}) + \sqrt{\rho B(\epsilon, \beta, \Omega, T,d)\frac{2s}{nm}}  + \frac{||u_{J} - u^{*}||^{2}_{H^{1}_{t}([0,T];H^{1}_{x}(\Omega))}}{16B(\epsilon, \beta, \Omega,T,d)} \\
& \quad - E_{T}(\epsilon)+ 2 \psi(\rho,\tau,J)+ \tau_{1}(\rho,J).
\end{align}
In order to make the inequality hold, we substitute the upper bound of stability estimate in Proposition \ref{prop:stability}. The terminal term $E_{T}(\epsilon)$ will vanish due to the cancellation by the same term appearing in the upper bound. So it follows 
\begin{align}
\frac{\rho }{16(1+\frac{\beta}{2\epsilon})B(\epsilon,\beta,\Omega,T,d)}  &\geq (1+\frac{\beta}{2\epsilon})||u_{J} - u^{*}||^{2}_{H^{1}_{t}(0,T;H^{1}_{x}(\Omega))} \\
&+ 2 \psi(\rho,\tau,J)+\sqrt{\rho B(\epsilon, \beta, \Omega, T,d)\frac{2s}{nm}} + \tau_{1}(\rho, J).
\end{align}\\
The balance between $r$ and $J$ and $\beta$ is different for different regularity regime: $q\leq \frac{d}{2}$ and $q>\frac{d}{2}$. We first determine $r$ when $q< \frac{d}{2}$. According to the relation (\ref{r_definition}) between $r$ and $\rho$, we can choose $r$ to be 
\begin{align}
r &\gtrsim  (1+ \frac{\beta}{2\epsilon} )B^{2}(\epsilon,
\beta, \Omega,T,d)J^{2(d+2)}\frac{\log{(1+ C_{2}\sqrt{nm})}}{nm}\label{r_determined_1_p<d/2}\\
&\quad +  (1+\frac{\beta}{2\epsilon}) B(\epsilon,
\beta, \Omega,T,d)||u_{J}-u^{*}||_{{H^{1}_{t}(0,T;H^{1}_{x}(\Omega))}}^{2}+  \frac{J^{2(d+2)}s +1}{B(\epsilon,
\beta, \Omega,T,d)nm}
\label{r_determined_2_p<d/2}.
\end{align}
Applying probabilistic bound for excess risk $S_{*,\epsilon}(u)- \hat{S}_{*,\epsilon}(u)\,\, $ with $r$ in (\ref{r_determined_1_p<d/2}) and (\ref{r_determined_2_p<d/2}), we obtain with probability $1-e^{-s}$,

\begin{align}
||\hat{u}_{J}- u^{*}||_{H^{1}_{t}(0,T;H^{1}_{x}(\Omega))}^2 &\lesssim ||u_{J}- u^{*}||_{H^{1}_{t}(0,T;H^{1}_{x}(\Omega))}^2 + r \\
& \lesssim \bigg[1+(1+\frac{\beta}{2\epsilon}) B(\epsilon,
\beta, \Omega,T,d)\bigg]||u_{J}-u^{*}||_{{H^{1}_{t}(0,T;H^{1}_{x}(\Omega))}}^{2}\\& + (1+ \frac{\beta}{2\epsilon} )B^{2}(\epsilon,
\beta, \Omega,T,d)\frac{J^{2(d+2)}}{nm}\log{(1+C_{2}nm)}\\ &+ \frac{J^{2(d+2)}s}{B(\epsilon,
\beta, \Omega,T,d)nm} + \frac{1}{B(\epsilon,
\beta, \Omega,T,d)nm}.
\end{align}
Taking expectation over $s$, we obtain
\begin{align}
\E||\hat{u}_{J}- u^{*}||_{H^{1}_{t}(0,T;H^{1}_{x}(\Omega))}^2 &\lesssim (1+ \frac{\beta}{2\epsilon} )B^{2}(\epsilon,
\beta, \Omega,T,d)\frac{J^{2(d+2)}}{nm}\log{(1+C_{2}nm)}\\
&\quad + \bigg[1+(1+\frac{\beta}{2\epsilon}) B(\epsilon,
\beta, \Omega,T,d)\bigg]||u_{J}-u^{*}||_{{H^{1}_{t}(0,T;H^{1}_{x}(\Omega))}}^{2}\\
&+ \frac{1}{B(\epsilon,
\beta, \Omega,T,d)nm}
\label{excess_risk+upper_bound}.
\end{align}
\\
When $q\geq \frac{d}{2}$, we mimic previous step to choose different $r$ as 
\begin{align}
r &\gtrsim  B(\epsilon,
\beta, \Omega,T,d)(1+\frac{\beta}{2\epsilon})J^{d+2}\frac{\log{(1+ C_{2}\sqrt{nm})}}{nm}\label{r_determined_2_p>d/2_part_1}\\
&\quad +  (1+\frac{\beta}{2\epsilon}) B(\epsilon,
\beta, \Omega,T,d)||u_{J}-u^{*}||_{{H^{1}_{t}(0,T;H^{1}_{x}(\Omega))}}^{2}+  \frac{Rs +1}{B(\epsilon,
\beta, \Omega,T,d)nm}
\label{r_determined_2_p>d/2_part_2}.
\end{align}
Similarly, with $r$ in (\ref{r_determined_2_p>d/2_part_1}) and (\ref{r_determined_2_p>d/2_part_2}), we obtain with probability $1-e^{-s}$,

\begin{align}
||\hat{u}_{J}- u^{*}||_{H^{1}_{t}(0,T;H^{1}_{x}(\Omega))}^2 &\lesssim_{R}  \bigg[1+(1+\frac{\beta}{2\epsilon}) B(\epsilon,
\beta, \Omega,T,d)\bigg]||u_{J}-u^{*}||_{{H^{1}_{t}(0,T;H^{1}_{x}(\Omega))}}^{2}\\& +  B(\epsilon,
\beta, \Omega,T,d)(1+\frac{\beta}{2\epsilon})\frac{J^{d+2}}{nm}\log{(1+C_{2}nm)}+ \frac{1+s}{nm} .
\end{align}
and we obtain in expectation
\begin{align}
\E||\hat{u}_{J}- u^{*}||_{H^{1}_{t}(0,T;H^{1}_{x}(\Omega))}^2 &\lesssim  \bigg[1+(1+\frac{\beta}{2\epsilon}) B(\epsilon,
\beta, \Omega,T,d)\bigg]||u_{J}-u^{*}||_{{H^{1}_{t}(0,T;H^{1}_{x}(\Omega))}}^{2}\\& +  B(\epsilon,
\beta, \Omega,T,d)(1+\frac{\beta}{2\epsilon})\frac{J^{d+2}}{nm}\log{(1+C_{2}nm)}+ \frac{1}{nm} .
\end{align}
Next, we want to control the approximation error $\underset{u_{J}\in \U_{J}}{\inf}\,\,||u_{J} - u^{*}||_{H^{1}_{t}(0,T;H^{1}_{x}(\Omega))}^{2}$ that appeared in Proposition \ref{localization}. Since our solution $u^{*}$ is defined in compact domain space-time domain $\Omega\times[0,T]$, it is sufficient to control $\bar{u}_{J}:=\prod_{J} ext\,u^{*}$, which is a projection of $u^{*}$ over $J$ truncated local basis extended to entire space-time domain $\R^{d}$ and $\R$. In light of extension Lemma~\ref{Sobolev_extension_theorem}, ext $u^{*}$ preserves Sobolev norm up to some constant depending on $\Omega$ and $[0,T]$, and together with assumption A3 in Appendix \ref{sec:main_assumptions}, we can show under Sobolev norm $||\cdot ||_{H^{1}_{t}(0,T;H^{1}_{x}(\Omega))}$, approximation errors are controllable.
\begin{lemma} Given $p,q \in \mathbb{Z}^{+}$ and $J,J^{\prime} \in\mathbb{Z}^{+}$. For any function $f(x,t)\in H^{p}_{t}(\R;H^{q}_{x}(\R^{d})) $, there exists Spatial-Temporal approximation $f_{J, J^{\prime}}(x,t)$ in the space $\mathcal{F}_{J,J^{\prime}}\cap H^{p}_{t}(\R;H^{q}_{x}(\R^{d})) $ such that the following approximation inequalities hold for $0\leq  p_{1} \leq p, 0\leq  q_{1} \leq q$,
\label{lemma_E.4}
\begin{equation}
||f-f_{J,J^{\prime}}||_{H^{p_{1}}_{t}(\R;H^{q_{1}}_{x}(\R^{d}))}^{2} \lesssim J^{-2(q-q_{1})}{J^{\prime}}^{-2(p-p_{1})}||f||_{H^{p}_{t}(\R;H^{q}_{x}(\R^{d}))}^{2}.
\end{equation}
In particular, if $J^{\prime} = J^{2}$, 
\begin{equation}
||f-f_{J,J^{\prime}}||_{H^{p_{1}}_{t}(\R;H^{q_{1}}_{x}(\R^{d}))}^{2} \lesssim J^{-2(q+2p-q_{1}-2p_{1})}||f||_{H^{p}_{t}(\R;H^{q}_{x}(\R^{d}))}^{2},
\end{equation}
where 
\begin{equation}
\mathcal{F}_{J,J^{\prime}}:=\bigg\{f \,\,\bigg{|} \,\, f \in V_{J,J^{\prime}}\cap H^{p}_{t}(\R;H^{q}_{x}(\R^{d})) \bigg\}.
\end{equation}
\end{lemma}
\begin{proof}[Proof of Lemma~\ref{lemma_E.4}]
For any $f\in H^{p}_{t}(\R;H^{q}_{x}(\R^{d}))$, there exists Spatial-temporal  representation in terms of $L_2$ basis in space-time domain $\Omega_{T}$,

\begin{equation}
f(x,t) = f_{J,J^{\prime}}(x,t) +  f_{J,>J^{\prime}}(x,t) +  f_{>J,J^{\prime}}(x,t) +  f_{>J,>J^{\prime}}(x,t).
\end{equation}
Where $f_{J,J^{\prime}}(x,t)$ is the truncated term and the rest are tail terms and $\phi_{\xi,\tau}(x,t)$ is basis of $L^2$ of product space of space and time 
\begin{equation}
f_{J,J^{\prime}}(x,t):= \sum_{|\xi|_{\infty}\leq J} \sum_{|\tau|\leq J^{\prime}} a_{\xi,\tau} \phi_{\tau}^{\xi}(t)\phi_{\xi}(x)
,\quad f_{J,>J^{\prime}}(x,t):= \sum_{|\xi|_{\infty}\leq J} \sum_{|\tau|> J^{\prime}} a_{\xi,\tau} \phi_{\tau}^{\xi}(t)\phi_{\xi}(x),
\end{equation}
\begin{equation}
f_{>J,J^{\prime}}(x,t):= \sum_{|\xi|_{\infty}> J} \sum_{|\tau|\leq J^{\prime}} a_{\xi,\tau} \phi_{\tau}^{\xi}(t)\phi_{\xi}(x), \quad 
f_{>J,>J^{\prime}}(x,t):= \sum_{|\xi|_{\infty}> J} \sum_{|\tau|> J^{\prime}} a_{\xi,\tau} \phi_{\tau}^{\xi}(t)\phi_{\xi}(x).
\end{equation}
According to the definition of Sobolev norm (\ref{sobolev_norm_in_fourier_transform}) and definition of time-dependent Sobolev space in Appendix \ref{time_dependent_space}, we can rewrite the tailed error of Fourier approximation of $f(x,t)$ as the following

\begin{align}
||f(\cdot,t)-f_{J,J^{\prime}}(\cdot,t)||_{H^{q_{1}}_{x}(\R^{d})}^2 &= ||f_{>J,J^{\prime}}(\cdot,t) + f_{J,>J^{\prime}}(\cdot,t) + f_{>J,>J^{\prime}}(\cdot,t)||_{H^{q_{1}}_{x}(\R^{d})}^2 \\
& \lesssim\Bigl(\sum_{|\xi|_{\infty}> J} \sum_{|\tau|> J^{\prime}} + \sum_{|\xi|_{\infty}> J} \sum_{|\tau|\leq J^{\prime}}  + \sum_{|\xi|_{\infty}\leq J} \sum_{|\tau|> J^{\prime}} \Bigl(1+|\xi|^{q_{1}}\Bigr)^{2}  a_{\xi,\tau}^2 \phi_{\tau}^{2}(t) \Bigr).
\end{align}\\
We then calculate $H^{p_{1}}_{t}(\R;H^{q_{1}}_{x}(\R^{d}))$ norm and obtain
\begin{equation}
||f-f_{J,J^{\prime}}||_{H^{p_{1}}_{t}(\R;H^{q_{1}}_{x}(\R^d))}^2 \lesssim\Bigl(\sum_{|\xi|_{\infty}> J} \sum_{|\tau|> J^{\prime}} + \sum_{|\xi|_{\infty}> J} \sum_{|\tau|\leq J^{\prime}}  + \sum_{|\xi|_{\infty}\leq J} \sum_{|\tau|> J^{\prime}} \Bigl(1+\tau^{p_{1}} \Bigr)^{2}\Bigl(1+\xi^{q_{1}}\Bigr)^{2}  a_{\xi,\tau}^2 \Bigr). 
\end{equation}
Using the highest order term $\tau^{2p_{1}}$ and $\xi^{2q_{1}}$ to bound above equality, we obtain 

\begin{align}
||f -f_{J,J^{\prime}}||_{{H^{p_{1}}_{t}(\R;H^{q_{1}}_{x}(\R^{d}))}}^{2} &\lesssim \bigg(\sum_{|\xi|_{\infty}> J} \sum_{|\tau|> J^{\prime}} + \sum_{|\xi|_{\infty}> J} \sum_{|\tau|\leq J^{\prime}} + \sum_{|\xi|_{\infty}\leq J} \sum_{|\tau|> J^{\prime}} \tau^{2p_{1}}\xi^{2q_{1}} a_{\xi,\tau}^2  \bigg)\\
&= \bigg(\sum_{|\xi|_{\infty}> J} \sum_{|\tau|> J^{\prime}} + \sum_{|\xi|_{\infty}> J} \sum_{|\tau|\leq J^{\prime}} + \sum_{|\xi|_{\infty}\leq J} \sum_{|\tau|> J^{\prime}} \tau^{-2(p-p_{1})}\xi^{-2(q-q_{1})} \tau^{2p}\xi^{2q}a_{\xi,\tau}^2  \bigg)\\
&\lesssim J^{-2(p-p_{1})}{J^{\prime}}^{-2(q-q_{1})}|| f ||_{H^{p_{1}}_{t}(\R;H^{q_{1}}_{x}(\R^{d}))}^{2}.
\end{align}

Finally, we obtain the desired result
\begin{equation}
||f -f_{J,J^{\prime}}||_{H^{p_{1}}_{t}(\R;H^{q_{1}}_{x}(\R^{d}))}^{2} \lesssim J^{-2(p-p_{1})}{J^{\prime}}^{-2(q-q_{1})}|| f ||_{H^{p_{1}}_{t}(\R;H^{q_{1}}_{x}(\R^{d}))}^{2}.
\end{equation}
\end{proof}
If we can show $\bar{u}_{J}\in \mathcal{U}_{J}$, then it immediately follows from Lemma \ref{lemma_E.4} and space-time extension Lemma~\ref{space_time_extension}
\begin{align}
\underset{u_{J}\in \U_{J}}{\inf}\,\,||u_{J} - u^{*}||_{H^{1}_{t}(0,T;H^{1}_{x}(\Omega))}^{2}&\leq ||\bar{u}_{J} - ext\,u^{*}||_{H^{1}_{t}(0,T;H^{1}_{x}(\Omega))}^{2}\\
&\leq  ||\bar{u}_{J} - ext \,u^{*}||_{H^{1}_{t}(\R;H^{1}_{x}(\R^{d}))}^{2}\lesssim R^{2}J^{-2(q+2p-3)}.
\end{align}
It remains to show $\bar{u}_{J}\in \U_{J}$. We note extension operator can preserve time-dependent Sobolev norm from Theorem \ref{space_time_extension} up to some constant $C(\Omega,T)$. Therefore, $\bar{u}_{J}$ satisfies boundedness assumption of time-dependent Sobolev norm in assumption A3 in Appendix \ref{sec:main_assumptions}, and $\bar{u}_{J}\in \U_{J}$.

Finally, we incorporate approximation errors into bound of excessive risk in expectation, leading to bias-variance trade-off.\\\\
On one hand, when $q\geq \frac{d}{2}$,
\begin{align}
\E||\hat{u}_{J}- u^{*}||_{H^{1}_{t}(0,T;H^{1}_{x}(\Omega))}^2 &\lesssim_{R} \bigg[1+(1+\frac{\beta}{2\epsilon}) B(\epsilon,
\beta, \Omega,T,d)\bigg]J^{-2(2p+q-3)}\label{bias-variance-tradeoff_1.1}\\ &+  B(\epsilon,
\beta, \Omega,T,d)(1+\frac{\beta}{2\epsilon})\frac{J^{d+2}}{nm}\log{(1+C_{2}nm)}+ \frac{1}{nm}.
\label{bias-variance-tradeoff_1.2}\end{align}
On other hand, when $q< \frac{d}{2}$,
\begin{align}
\E||\hat{u}_{J}- u^{*}||_{H^{1}_{t}(0,T;H^{1}_{x}(\Omega))}^2 &\lesssim_{R}\bigg[1+(1+\frac{\beta}{2\epsilon}) B(\epsilon,
\beta, \Omega,T,d)\bigg]J^{-2(2p+q-3)}\label{bias-variance-tradeoff_2.1}\\& + (1+ \frac{\beta}{2\epsilon} )B^{2}(\epsilon,
\beta, \Omega,T,d)\frac{J^{2(d+2)}}{nm}\log{(1+C_{2}nm)}\\&
+\frac{1}{B(\epsilon,
\beta, \Omega,T,d)nm}
\label{bias-variance-tradeoff_2.2}.\end{align}
Bias decreases as approximation accuracy controlled by $J$ increases, while variance becomes larger when introducing high-frequency statistics information. Therefore, we need to find trade-off between bias and variance to balance variance and bias together with regularization parameter.

Similarly, we also need to balance regularization $B(\epsilon, \beta, \Omega, T)$ and $1+\frac{\beta}{\epsilon}$ on either variance or bias, and we derive regularization threshold $\beta^{*} \approx \log (1+C_{2}N)^{\frac{b}{a+2b}}N^{\frac{a}{a+2b}} \beta_{0}$ in \eqref{optimal_beta_q<d/2} of Appendix \ref{E.3}. We notice $\beta^{*}$ is a small constant and therefore regularization falls into the regime of bias regularization for the most of times. If $0< \beta< \beta^{*}$, $\beta$ is in the regime of regularizing variance; if $\beta>\beta^{*}$, $\beta$ is in the regime of regularizing bias. 

Now, we're ready to prove upper bound on estimation error.

\subsection*{Proof of Theorem~\ref{upper_bound}}
\label{appendix_L}
\label{section_L}
\hfill\\
In this section, we discuss the bias-variance trade-off for two different regularity condition: $q\leq \frac{d}{2}$ and $q>\frac{d}{2}$. First, we focus on the regularization regime when $q\leq \frac{d}{2}$.
According to Lemma \ref{lemma_E.4} with $J^{\prime} = J^{2}$ and assumption A1 of Appendix \ref{sec:main_assumptions}, we have 
\begin{equation}
\,\,\underset{u_{J}}{\inf}\,\,||u_{J} - u^{*}||_{H^{1}_{t}(0,T;H^{1}_{x}(\Omega))}^{2} \lesssim J^{-2(2p+q-3)}R^{2}.
\label{approximation_error}
\end{equation}
Substituting approximation error (\ref{approximation_error}) into Proposition \ref{high_probability_bound_on_excess_risk}, we obtain the following\\\\
When $q<\frac{d}{2}$,
\begin{align}
\E||\hat{u}_{J}- u^{*}||_{H^{1}_{t}(0,T;H^{1}_{x}(\Omega))}^2 \lesssim_{R}  &\underbrace{(1+ \frac{\beta}{2\epsilon} )B^{2}(\epsilon,
\beta, \Omega,T,d)\frac{J^{2(d+2)}}{nm} \log{(1+C_{2}nm)}}_{\text{Variance}}\\
& + \underbrace{\bigg[1+(1+\frac{\beta}{2\epsilon}) B(\epsilon,
\beta, \Omega,T,d)\bigg]J^{-2(2p+q-3)}}_{\text{Bias}} + \frac{1}{B(\epsilon,
\beta, \Omega,T,d)nm}
\label{excess_risk_bound_unbalanced}.
\end{align}
When $q\geq\frac{d}{2}$,
\begin{align}
\E||\hat{u}_{J}- u^{*}||_{H^{1}_{t}(0,T;H^{1}_{x}(\Omega))}^2 & \lesssim  \underbrace{B(\epsilon,
\beta, \Omega,T,d)(1+\frac{\beta}{2\epsilon})\frac{J^{d+2}}{nm}\log{(1+C_{2}nm)}}_{\text{bias}} \\ &+ \underbrace{\bigg[1+(1+\frac{\beta}{2\epsilon}) B(\epsilon,
\beta, \Omega,T,d)\bigg]J^{-2(2p+q-3)}R^{2}}_{\text{variance}}+ \frac{1}{nm}.
\end{align}
\subsection{Regime of regularization ($q< \frac{d}{2}$)}\label{E.3}\label{regime_regularization_1}
Here we want to find the regimes for the regularization parameter $\beta$ with respect to variance scaling $J^{2(d+2)}$.
Let $N:=nm$ denote the total sample size and the excess risk upper bound by . Consider the above excess-risk bound \eqref{excess_risk_bound_unbalanced}, we set $a:=2(d+2)$ and $b:=2(2p+q-3)$,
\begin{align}
\label{eq:err_beta_1}
&\mathcal{E}(J,\beta)
\;\lesssim_{R}
(1+ \frac{\beta}{2\epsilon} )B^{2}(\epsilon,
\beta, \Omega,T,d)\,\frac{J^{a}}{N}\log (1+C_{2}N)\\
&\quad\quad\quad\quad+\;
\bigg[1+(1+\frac{\beta}{2\epsilon}) B(\epsilon,
\beta, \Omega,T,d)\bigg]\,J^{-b}\;+\; \frac{1}{B(\epsilon,
\beta, \Omega,T,d)nm},
\end{align}
where the $\beta$-dependent variance term is
\begin{equation}
\label{optimal_beta_q<d/2}
B(\epsilon,
\beta, \Omega,T,d)
=
\Bigl(\frac13+\frac{\beta e^{-T/\varepsilon}}{C^2(d,\varepsilon,\Omega)}\Bigr)^{-1}
=
\Bigl(\frac13+\frac{\beta}{\beta_0}\Bigr)^{-1},
\qquad
\beta_0:=C^2(d,\varepsilon,\Omega)\,e^{T/\varepsilon}.
\end{equation}
Thus $\beta$ decreases the variance term $B(\beta)$, but increases the bias term
$\bigl(1+\beta/(2\varepsilon)\bigr)$.
\subsubsection{Optimal truncation level for a given $\beta$}
Balancing variance and bias in (\ref{eq:err_beta_1})
 yields
\begin{equation}
\label{J_*1}
J_{1,*}(\beta)
\;\asymp\;
\Biggl[
\frac{N}{\log (1+C_{2}N)}
\cdot
\frac{1}{(1+ \frac{\beta}{2\epsilon} )B^{2}(\epsilon,
\beta, \Omega,T,d)}\cdot \bigg(1+(1+\frac{\beta}{2\epsilon})B(\epsilon,
\beta, \Omega,T,d)\bigg)\Biggr]^{\frac{1}{a+b}}
\end{equation}
and the corresponding optimized bound becomes
\begin{equation}
\label{eq:ErrStar}
\E||\hat{u}_{J}- u^{*}||_{H^{1}_{t}(0,T;H^{1}_{x}(\Omega))}^2 \lesssim_{R} \Bigl(1+\frac{\beta_{0}}{2\epsilon}\Bigr)\cdot
\Bigl(\frac{\log (1+C_{2}N)}{N}\Bigr)^{\frac{b}{a+b}}
\cdot
(\frac{\beta_{0}}{\beta_{0}+\beta})^{\frac{b}{a+b}}
 + (1+\frac{\beta}{\beta_{0}}) \frac{1}{N}.
\end{equation}
\subsubsection{Case I (From Unregularized to Mildly Regularized Regime): $ 0\leq \beta\leq\beta_{0} $}
In this regime where energy regularization can only improve the rate up to some constant, we obtain sub-optimal rate 
\begin{equation}
\E||\hat{u}_{J}- u^{*}||_{H^{1}_{t}(0,T;H^{1}_{x}(\Omega))}^2 \lesssim_{R}  \Bigl(1+\frac{\beta_{0}}{2\epsilon}\Bigr)\cdot\Bigl(\frac{\log (1+C_{2}N)}{N}\Bigr)^{\frac{b}{a+b}}.
\end{equation}

\subsubsection{Case II (Optimal Rate with fine-tuned $\beta$): $\beta=\beta_{*}
$} when we balance $\beta$, we obtain optimal rate\begin{equation}\beta^{*} \approx \log (1+C_{2}N)^{\frac{b}{a+2b}}N^{\frac{a}{a+2b}} \beta_{0}\label{beta_threshold}\end{equation} and plugging $\beta^{*}$ back, we obtain the following rate
\begin{equation}
\E||\hat{u}_{J}- u^{*}||_{H^{1}_{t}(0,T;H^{1}_{x}(\Omega))}^2 \lesssim_{R}  \Bigr(\frac{\sqrt{\log (1+C_{2}N)}}{N} \Bigr)^{\frac{2b}{a+2b}}.
\end{equation}
In this regime, regularization reduces the variance with mildly inflating the bias term. 
\subsubsection{Case III (Variance-improving Regime): $\beta_{0}<\beta<\beta^{*}$}
In this scenario, variance is reduced to improve the rate
\begin{equation}
\E||\hat{u}_{J}- u^{*}||_{H^{1}_{t}(0,T;H^{1}_{x}(\Omega))}^2 \lesssim_{R} \Bigl(1+\frac{\beta_{0}}{2\epsilon}\Bigr)\cdot\Bigl(\frac{\log (1+C_{2}N)}{N}\Bigr)^{\frac{b}{a+b}}
\cdot
(\frac{\beta_{0}}{\beta_{0}+\beta})^{\frac{b}{a+b}}.
\end{equation}

\subsubsection{Case V (Bias-inflating Regime): $\beta^{*}<\beta<\infty$}
In this scenario, variance is inflated to make the rate worse
\begin{equation}
\E||\hat{u}_{J}- u^{*}||_{H^{1}_{t}(0,T;H^{1}_{x}(\Omega))}^2 \lesssim_{R}  (1+\frac{\beta}{\beta_{0}}) \frac{1}{N}.
\end{equation}
Next we consider the regularization regime where $q>\frac{d}{2}$. We take similar step to set variance equal to bias
\begin{align}
\E||\hat{u}_{J}- u^{*}||_{H^{1}_{t}(0,T;H^{1}_{x}(\Omega))}^2 &\lesssim  \bigg[1+(1+\frac{\beta}{2\epsilon}) B(\epsilon,
\beta, \Omega,T,d)\bigg]J^{-2(2p+q-3)}\\& +  B(\epsilon,
\beta, \Omega,T,d)(1+\frac{\beta}{2\epsilon})\frac{J^{d+2}}{nm}\log{(1+C_{2}nm)}+ \frac{1}{nm} .
\end{align}
\subsection{Regime of regularization ($q\geq\frac{d}{2}$)}
\label{regime_regularization_2}when $q>\frac{d}{2}$, variance scaling becomes $J^{2(d+2)}$. With the same notation as before, we set $N:=nm$ the total sample size,  $\frac{a}{2}:=(d+2)$, and $b:=2(2p+q-3)$,
\begin{align}
\label{eq:err_beta_1_q>d/2}
&\mathcal{E}(J,\beta)
\;\lesssim\;
B(\epsilon,
\beta, \Omega,T,d)(1+ \frac{\beta}{2\epsilon})\,\frac{J^{\frac{a}{2}}}{N}\log (1+C_{2}N)+
\bigg[1+(1+\frac{\beta}{2\epsilon}) B(\epsilon,
\beta, \Omega,T,d)\bigg]\,J^{-b}\;+\; \frac{1}{nm}.
\end{align}
\subsubsection{Optimal truncation level for a given $\beta$}
Balancing variance and bias in \eqref{eq:err_beta_1_q>d/2} yields
\begin{equation}
\label{J_*2}
J_{2,*}(\beta)
\;\asymp\;
\Biggl[
\frac{N}{\log (1+C_{2}N)}
\cdot
\frac{1+\frac{\beta}{\beta_{0}}}{1+\frac{\beta}{2\epsilon}}\cdot \bigg(1+\frac{1+\frac{\beta}{2\epsilon}}{1+\frac{\beta}{\beta_{0}}}\bigg)\Biggr]^{\frac{2}{a+2b}}
\end{equation}
and the corresponding optimized bound becomes
\begin{equation}
\E||\hat{u}_{J}- u^{*}||_{H^{1}_{t}(0,T;H^{1}_{x}(\Omega))}^2 \lesssim_{R}  
\Bigl(\frac{\log (1+C_{2}N)}{N}\Bigr)^{\frac{2b}{a+2b}}
\cdot
(1+\frac{1+\frac{\beta}{2\epsilon}}{1+\frac{\beta}{\beta_{0}}})^{\frac{2b}{a+2b}}
\cdot
\Bigl(\frac{1+\frac{\beta}{2\epsilon}}{1+\frac{\beta}{\beta_{0}}}\Bigr)^{\frac{a}{a+2b}} +  \frac{1}{N}.
\end{equation}
For any $\beta>0$, the rate remains optimal up to some constant
\begin{equation}
\E||\hat{u}_{J}- u^{*}||_{H^{1}_{t}(0,T;H^{1}_{x}(\Omega))}^2 \lesssim_{R}  
\Bigl(\frac{\log (1+C_{2}N)}{N}\Bigr)^{\frac{2b}{a+2b}}
\cdot
(1+\frac{\beta_{0}}{2\epsilon})^{\frac{a}{a+2b}}
\cdot
\Bigl(\frac{\beta_{0}}{2\epsilon}\Bigr)^{\frac{2b}{a+2b}}.
\end{equation}
\\
Finally, we wrap up our proof of convergence rates in \eqref{regime_regularization_1} and \eqref{regime_regularization_2} by using scaling relation $\beta =\lambda \cdot nm$ between scaled regularization parameter $\beta$ and unscaled regularization parameter $\lambda$ in Section \ref{lambda_beta_scale} to change $\beta$ back to $\lambda$.
\end{proof}
\subsection{Sampling Balance}
Next, we will show it is necessary to balance space-time sample complexity using parabolic anisotropic balancing as reflected in multi-level truncation $J$ for space and $J^{2}$ for time.

\begin{proposition}[Sampling Balance]
\label{prop:parabolic-balance}
Given functional space $V_{J,J^{2}}$ defined in \ref{functional_space}, the evaluation map of our empirical loss $\hat{L}_{\epsilon,\beta}$ is injective only if 
\begin{equation}\label{eq:n-lower}
n\geq J^2.
\end{equation}
Furthermore, it is necessary to require 
\begin{equation}
n\gtrsim m^{2/d} \label{eq:m-upper}
\end{equation}
up to logarithmic factor for theorem 3.1 hold.
\end{proposition}

\begin{proof} 
\emph{Lower bound on the temporal sample size.}
We first show lower bound on temporal sample complexity $n$. Given temporal samples $t_1,\dots,t_n$, it is enough to look at the evaluation map $ev:\ u\in V_{J,J^{2}}\;\longmapsto\;
\big(\partial_t^kD_x^\alpha u(X_{ij},t_j)\big)_{i\leqslant m,\
j\leqslant n}$. It is clear the evaluation mapping $ev$ is linear on the finite-dimensional space $V_{J,J^2}$, so it fails to be injective if its kernel is nontrivial.We will show  the contrapositive is true.

We choose $\psi=\sum_{\tau=1}^{K}c_\tau\phi_\tau\in\Psi_K$ with
its coefficient vector $c\in\mathbb{R}^K$, with $
\psi^{(k)}(t_j)=0$ and $\psi(0)=0$, for $1\leq j\leq n, 0\leq k\leq r$. It forms a homogeneous linear system $Ac=0$ with evaluation matrix
\[
A=
\begin{pmatrix}
\phi_1(t_1) & \phi_2(t_1) & \cdots & \phi_K(t_1)\\
\phi_1'(t_1) & \phi_2'(t_1) & \cdots & \phi_K'(t_1)\\
\vdots & \vdots & & \vdots\\
\phi_1(t_n) & \phi_2(t_n) & \cdots & \phi_K(t_n)\\
\phi_1'(t_n) & \phi_2'(t_n) & \cdots & \phi_K'(t_n)\\
\phi_1(0) & \phi_2(0) & \cdots & \phi_K(0)
\end{pmatrix}
\in\mathbb{R}^{\big(2n+1\big)\times K}
\]. The matrix has $2n+1$ rows, so
$\rank A\leq 2n+1<K$, and the rank--nullity theorem gives
\[
\dim\ker A\;=\;K-\rank A\;\geq\;K-2n-1\;\geq\;1
\]
The corresponding
$\psi\in\Psi_K\setminus\{0\}$ vanishes, together with its first $r$
derivatives, at every temporal position $t_{j}$. Pick any spatial basis element
$\phi_\xi$ with $|\xi|_\infty\leqslant J$ and set
$v:=\psi\,\phi_\xi\in V_{J,J^2}\setminus\{0\}$; by the tensor structure,
for all $0\leqslant k\leqslant r$ and $\alpha\in\mathbb{N}^d$,
\[
\partial_t^kD_x^\alpha v\,(X_{ij},t_j)
=\psi^{(k)}(t_j)\,D_x^\alpha\phi_\xi(X_{ij})=0 ,
\]
so $v\in\ker\mathrm{ev}$ and $\mathrm{ev}$ is not injective. Since
$K\asymp J^2$, \eqref{eq:n-lower} follows.

Next, we show upper bound on the spatial sample size $m$.
Using $n\geqslant J^2$, we fix
$\phi_\tau$ at the temporal level $2^\ell\asymp J^2$ and $\phi_\xi$
with $|\xi|_\infty\leqslant J$, $supp\,\phi_\xi\subset\Omega^\circ$, and
consider the squared basis direction
$h_J:=a_J^2\phi_\xi^2\phi_\tau^2$ in the loss class, with admissible
amplitude $a_J$. We deonote $P_X$, $P_T$ for the uniform laws on $\Omega$
and $[0,T]$, and $\widehat P$, $P$ as the empirical and population
means on $\Omega_T$, decompose
\begin{align}
&(\widehat P-P)h_J=W_J+B_J,\\
&W_J:=\frac1n\sum_{j=1}^n\Big[\frac1m\sum_{i=1}^m h_J(X_{ij},t_j)
-P_Xh_J(\cdot,t_j)\Big],\\
&B_J:=\frac1n\sum_{j=1}^nP_Xh_J(\cdot,t_j)-Ph_J.
\end{align}

Define the events
\[
E_1:=\Big\{|B_J|\geqslant c_1a_J^2\sqrt{J^2/n}\Big\},
\qquad
E_2:=\Big\{|W_J|<\tfrac{c_1}{2}\,a_J^2\sqrt{J^2/n}\Big\}.
\]
By lemma \ref{lem:moments_of_wavelets}, 
\[
\mathbb{E} W_J^2
\lesssim\;a_J^4\,\frac{J^{d+2}}{nm}, \quad\quad \mathbb{P}(E_1)\geqslant c_2
\]
using $\int\phi_\xi^4\asymp J^d$ and $\int\phi_\tau^4\asymp J^2$ and
Chebyshev's inequality,
\[
\mathbb{P}(E_2^c)
\;\leq\;\frac{4\,\mathbb{E} W_J^2}{c_1^2\,a_J^4\,J^2/n}
\;\lesssim\;\frac{J^d}{m}
\;\leq\;\frac{c_2}{2}
\qquad\text{once } m\geqslant C_0J^d
\]
for a sufficiently large constant $C_0(c_{1},c_{2})$. Hence
\[
\mathbb{P}(E_1\cap E_2)\;\geq\;\mathbb{P}(E_1)-\mathbb{P}(E_2^c)
\;\geqslant\;c_2-\frac{c_2}{2}\;=\;\frac{c_2}{2},
\]
and on $E_1\cap E_2$ the reverse triangle inequality gives
\[
\big|(\widehat P-P)h_J\big|
\;=\;|B_J+W_J|
\;\geqslant\;|B_J|-|W_J|
\;>\;c_1a_J^2\sqrt{J^2/n}-\tfrac{c_1}{2}\,a_J^2\sqrt{J^2/n}
\;=\;\tfrac{c_1}{2}\,a_J^2\sqrt{J^2/n}\,.
\]

The direction $h_J$ lies in the localized fixed-$J$ class over which our inequality bounds the empirical process by
$C\,a_J^2\sqrt{J^{d+2}\log(1+C_2N)/N}$ if 
\[
\sqrt{\frac{J^2}{n}}\;\lesssim\;\sqrt{\frac{J^{d+2}\log(1+C_2nm)}{nm}},
\qquad
\]
This implies $
m\;\lesssim\;J^{d}\log(1+C_2nm)$, and furthermore  $n\;\gtrsim\;J^2\;\gtrsim\;\Big(\frac{m}{\log(1+C_2nm)}\Big)^{2/d},
$\end{proof}

\subsection{Lower bound on estimation error}

\begin{proof}[Proof of Theorem~\ref{minimax_lower_bound}]
The main idea of our proof directly refers to standard construction tools from \cite{Tsybakov:1315296} with additional tool from PDE theory Theorem~\ref{Duhamel_principle}. By definition of $\U$ and $u^{*}$ in (\ref{functional_space}), for any ground truth solution $u^{*}\in \U$, we can always find data $f = Lu^{*} \in \mathcal{F}$ with which we can construct estimator for $\hat{u}$. The construction is following:
We follow our convention to treat time domain as another dimension so we write $d+1$ hypercube as $\mathcal{D}_{T} = [0,1]^{d}\times [0,1]$. According to regularity theory (Theorem~\ref{higher_regulairty_thm}), given forcing function $f\in \mathcal{F}$ and initial condition $u_{0}^{*}\in H^{2s+1}([0,1]^{d})$, there exists solution $h\in \U$ defined in (\ref{functional_space}) corresponding to this $f$. Our $h$ here is not necessarily a smooth function in the regular sense, but it is smooth in the Sobolev sense and its smoothness is determined by the functional space of forcing $\mathcal{F}$ and the initial data $u_{0}^{*}$. In particular, $h$ has the following smoothness property

 \begin{equation}
    \frac{d^{p}h}{dt^{p}}\in L^{2}(0,1;H^{2s+2-2p}([0,1]^{d})),  0\leq p\leq s,\, s\geq 1, s,q\in \Z.
\end{equation}

This implies
$\frac{dh}{dt}$ and $\triangle h$ exists in $L^{2}_{t,x}([0,1]^{d}\times[0,1])$ sense and $h\in H^1(0,T;H^{1}([0,1]^{d})):= H^{1}_{t}H^{1}_{x}$. Moreover, since $f$ and is non-zero, then $h$ can't be identitically zero for any $(x,t)\in [0,1]^{d}\times [0,1]$. For $f_{0}$, there exists unique solution $u_{f_{0}}:= h$.We also note the difference between $f_{0}$ and other $f_{k}$ is smooth perturbation, because its convolution with heat kernel, by Duhamel's Principle \ref{Duhamel_principle}, is smooth. Therefore, we then pick some function $g(x,t)$ for this perturbation such that $\frac{\partial g}{\partial t}$ and $\nabla g$ are not identically zero.

Let $r$ be the number of grid size on hypercube $[0,1]^{d}$ to be determined later. According to scaling invariance property Corollary~\ref{scaling_invariant} of $h$, $g(x,t)$ is a solution to heat equation if only if $g(\frac{x}{r}, \frac{t}{r^{2}})$ is a solution to heat equation. Therefore, we should have $r^{2}$ number of grid on unit time horizon $[0,1]$ while $r$ number of grid on spacial hypercube $[0,1]^{d}$ so that temporal scaling by $\frac{1}{r^2}$ can match with spacial scaling $\frac{1}{r}$ in upcoming construction of hypothesis test.

We define our hypothesis as the following
\begin{align}
u_{f_{k}}(x,t) &= h(x,t) +\sum_{j\in[r^{2}]} \sum_{i\in[r]^{d}}c\,\tau_{i,j}^{(k)}\, w^{2p+q} g(\frac{x-x^{i}}{w}, \frac{t-t^{j}}{w^{2}}),\, w=\frac{1}{2r},\,k=1,...,2^{\frac{r^{d+2}}{8}},\\
u_{f_0}(x,t) &= h(x,t),
\end{align} 
where $c,\lambda$ are some constants.

It is readily to check the collection of functions $\{g(\frac{x-x^{i}}{w},\frac{t-t^{j}}{w^{2}})\}_{i,j}$ has disjoint compact support for pair $(i,j)$ and $h(\frac{x-x^{i}}{w},\frac{t-t^{j}}{w^{2}})$ satisfies heat equation given initial condition $u_{0}^{*}$ and $f$. Therefore, $u_{k}$ satisfies the heat equation as well.

We discretize the space $[0,1]^{d}$ by defining a collection of vectors $\{\,x^{i}: x^{i} =(x^{i}_{1}, ..., x^{i}_{d}), i \in [r]^{d}\,\}\subset [0,1]^{d}$ to have coordinates $x_{l}^{i}: =\frac{i_{l}-1}{r}$, $i_{l} \in [r]$, $l=1,...,d$. Again, we similarly define another collection of numbers $\{ t^{j}: t^{j} = \frac{j-1}{r^{2}},\,\,j\in [r^{2}] \}\subset[0,1]$. According to the Varshamov-Gilbert lemma, there exists $\frac{r^{d+2}}{8}$ many $0$-$1$ sequences $\tau^{(1)}\,,...,\,\tau^{(2^{r^{d+2}/8})} \in \{0,1\}^{r^{d+2}}$ such that   $||\,\tau^{k} - \tau^{k^{\prime}}\,||^{2}\geq \frac{{r^{d+2}}}{8}\,,k\neq k^{\prime}$. Then we cast solving parabolic PDEs Problem into high-dimensional multiple hypothesis testing on hypercube $[0,1]^{d}\times [0,1]$ which considers all mappings $\Phi$ from $nm$ sampled data on domain $\Omega_{T}$ to the hypothesis 
\begin{equation}
\Phi: (x,t)\in(\mathbb{R})^{{d+1}^{\otimes{nm}}} \times \mathbb{R}^{\otimes nm} \rightarrow U:= \big\{u_{k}(x,t) \mid k=1,2,...,2^{\frac{r^{d+2}}{8}} \big\}.
\end{equation}
We recall Fano's inequality
\begin{equation}
P(\hat{u} \neq u) \geq 1 - \frac{I(U;X) + \log{2}}{ \log(|U|)}, \quad\quad \text{for any estimator } \hat{u}.
\end{equation}
We obtain a constant lower bound of $P(u_{k}^{\prime}\neq u_{k})$ by bounding KL-divergence from above 
\begin{equation}
I(U;X) \leq \frac{1}{|U|} \sum_{(x,t)} \sum_{u_{k}\neq u_{0}
} D_{KL}(\mathbb{P}_{u_{k}}||\mathbb{P}_{u_{0}}).
\end{equation}
where $P_{u_{k}}$ is a joint distribution of the sampled data $((x,t),f_{k}(x,t))$ for hypothesis $u_{k}$ corresponding to noisy observation of force function $f_{k}$, and $(x,t)\sim P_{0}$, $P_{0}$ is a product of Lebesgue measure on $[0,1]^{d}$ and Lebesgue measure on $ [0,1]$; $ y_{k}=f_{k}(x,t)+ \epsilon$ where $\epsilon$ is an i.i.d Gaussian noise for each location point $(x,t)$, and so $y_{k} \sim Q_{k}$, $Q_{k}$ is probability measure of standard normal random variable with mean $f_{k}$. Therefore, we can write $W_{f_{k}}:= P_{0}\otimes Q_{k}$, where $P_{0}$ is independent from $Q_{k}$ across tensor product\\
Note $\triangle u_{k} - \frac{\partial u}{\partial{t}} = f_{k}$ and $\triangle u_0 - \frac{\partial u_{0}}{\partial t} = \triangle h - \frac{\partial h}{\partial t} = f_{0}$. 
\begin{align}
D_{KL}(W_{f_{k}}||W_{f_{0}}) &= D_{KL}(P_{0}\otimes Q_{k} || P_{0}\otimes Q_{0}) \\
& = \E \log(\frac{dQ_{k}}{dQ_{{0}}})\\
&= ||f_{k} - f_{0}||_{L^{2}_{t,x}([0,1]\times[0,1]^{d})}^{2}\\
& = ||(\triangle u_{f_k} - \frac{\partial u_{f_k}}{\partial{t}}) - (\triangle u_{f_0} - \frac{\partial u_{f_0}}{\partial{t}}) ||_{L^{2}_{t,x}([0,1]\times[0,1]^{d})}^{2}\\
&= \bigg|\bigg|\sum_{i\in[r^{d+1}]}c\,\tau_{i}^{(k)}\, w^{2p+q-1} \Big( g_{t}(\dfrac{x-x^{i}}{w}, \frac{t-t^{j}}{w^{2}}) - \triangle g(\dfrac{x-x^{i}}{w}, \frac{t-t^{j}}{w^{2}}) \Big) \bigg|\bigg|_{L^{2}_{t,x}([0,1]\times[0,1]^{d})}^{2} \\
&\leq  c^2 w^{2(2p+q-3)}  \big|\big|g_{t} - \triangle g\big|\big|_{L^{2}_{t,x}([0,1]\times[0,1]^{d})}^{2} \\
&\lesssim \frac{c^2}{r^{2(2p+q-3)}}.
\end{align}
Next, we want to use energy estimate from Theorem 5 in chapter 7.1 of \cite{evans}, which upper bounds $u_{k}$ using $f_{k}$. Note $u(0,x) =0$ in our setting, so we have 
\begin{equation}
||u_{k} - u_{0}||^{2}_{L^{2}([0,1];L^{2}([0,1]^{d}))} \leq C ||f_{k} - f_{0}||^{2}_{L^{2}([0,1];L^{2}([0,1]^{d}))} \lesssim\frac{c^2}{r^{2(2p+q-3)}}.
\end{equation}

With $r = [(nm)^{\frac{1}{2(2p+q-3)+d+2}}]$, 
\begin{equation}
\mathbb{P}(u_{k} \neq u^{\prime}_{k} )\geq 1-\dfrac{nmD_{KL}(\mathbb{P}_{u}||\mathbb{P}_{u^{\prime}})}{\log(|U|)}\geq 1-\dfrac{\frac{8c^2}{r^{2(2p+q-3)}}}{r^{d+2}\log 2} =1-\dfrac{8c^2 }{r^{2{(2p+q-3)}+d+2} \log{2}}.
\end{equation}
We then choose large $c$ and $\lambda$ such that $\mathbb{P}(u_{k} \neq u_{k}^{\prime} )\geq \dfrac{1}{2}$.Finally, we check separation condition under the evaluation metric $H^{1}_{t}H^{1}_{x}$ for hypothesis $u_{k}$ and $u_{k}^{\prime}$
\begin{equation}
||u_{k}- u^{\prime}_{k}||_{H^{1}_{t}H^{1}_{x}}^{2} = H(\tau^{k},\tau^{k^{\prime}}) c^{2} w^{2(2p+q-3)}||\,g (\frac{x-x^{i}}{w},\frac{t-t^{j}}{w^{2}}))\,||_{L^2_{t,x}([0,1]^{d}\times [0,1])}^{2} w^{d+2} \gtrsim  r^{-2(2p+q-3)}.
\end{equation}
Plugging in $r= [(nm)^{\frac{1}{2(2p+q-3)+d+2}}]$ and $2p+q=2s+2$, we obtain
\begin{equation}
\underset{\hat{u}}{\inf}\,\,\underset{u^{*}\in H^{s}(\Omega)}{\sup}\big|\big|\, \hat{u} - u^{*}\,\big|\big|_{H^{1}_{t}(0,1;H^{1}_{x}([0,1]^{d})}^{2} \gtrsim {(nm)}^{-\frac{2(2p+q-3)}{2(2p+q-3)+d+2}} ={(nm)}^{-\frac{4s-2}{4s+d}}.
\end{equation}
\end{proof}

\subsection{Excess Risk Bound}

\begin{proof}[Proof of Theorem~\ref{thm_3.3}]
     We fix $\epsilon>0$,\,\,$s\geq 1$ and $2p+q=2s+2$ for $0\leq p \leq s$. According to Theorem~\ref{thm_3.1}, regardless of the relation between $p$ and $d$, as long as we choose
\begin{equation}\lambda^{*}:=\log (1+C_{2}nm)^{\frac{4s-2}{4s+d}}(nm)^{-\frac{2+d}{4s+d}} e^{T/\epsilon}.
\end{equation}
Our estimation rate becomes optimal as follows
\begin{equation}
\underset{}{\mathbb{E}}||\hat{u}_{J}- u^{*}||_{H^{1}_{t}(0,T;H^{1}_{x}(\Omega))}^2 \lesssim \Bigl(\frac{\log (1+C_{2}nm)}{nm}\Bigr)^{\frac{4s-2}{4s+d}}.
\end{equation}
Thanks to the stability estimate in Proposition~\ref{prop:stability}, excess risk is upper bounded by estimation error
\begin{equation}
\underset{}{\mathbb{E}}\,\, L_{\epsilon,\beta}(\hat{u}_{J})-L_{\epsilon,\beta}(u^{*}) \lesssim  \frac{e^{T/\epsilon}}{\epsilon}\Bigl(\frac{\log (1+C_{2}nm)}{nm}\Bigr)^{\frac{4s-2}{4s+d}}.
\end{equation}
\end{proof}

\subsection{Generalization Error Bound}
To prove generalization error bound,  We define functional classes that are closely associated with our loss function $L_{\epsilon,\beta}(u)$
{\footnotesize
\begin{equation}
\begin{aligned}
\mathcal{G}_{\mathrm{data}}
&:= \Bigl\{ g \,\Big|\, g = u-u^{*} + \nabla(u-u^{*}) + \partial_t(u-u^{*})+ \nabla\partial_t(u-u_t^{*}) \Bigr\}, \\
\mathcal{G}_{\mathrm{energy}}
&:= \Bigl\{ g \,\Big|\, g = \frac{e^{-t/\epsilon}}{2}|u_t|^2
+ \frac{1}{2\epsilon}|\nabla u|^2
- \frac{e^{-t/\epsilon}}{\epsilon}fu \Bigr\}, \\
\mathcal{G}_{\mathrm{correction}}
&:= \Bigl\{ g \,\Big|\, g = e^{-T/\epsilon}u_t^{*}(x,T)u(x,T)
+ e^{-t/\epsilon}u_{tt}^{*}u \Bigr\}.
\end{aligned}
\end{equation}
}
where $u^{*}, u\in \mathcal{U}$  and $u^{*}_{t}(x,T)$ and $u(x,T)$ are evaluated at $t=T$
\\\\
We relate the above functional classes  $\mathcal{G}_{data}$, $\mathcal{G}_{energy}$ and $\mathcal{G}_{correction}$ to their Rademacher Complexities $\mathcal{R}$
\begin{lemma}
\label{lemma_E.5} Given the functional spaces defined above, the generalization error bound is upper bound by their Rademacher Complexities as follows

\begin{equation}
\E\Delta \mathcal{E}_{S,gen} \leq 2[\, \mathcal{R}_{N}(\mathcal{G}_{data}) + \mathcal{R}_{N}(\mathcal{G}_{energy}) + \mathcal{R}_{N}(\mathcal{G}_{correction})\,].
\end{equation}

\begin{proof}[Proof of Lemma~\ref{lemma_E.5}]
    Fix $\epsilon>0$ and $\beta = \beta_{reg}$ and denote $N=nm$.
\begin{align}
\E\Delta\mathcal{E}_{S,gen}\,&\leq \mathbb{E} \underset{u\in \mathcal{U}}{\sup} |L_{\epsilon,\beta}(\hat{u}_{S}) - \hat{L}_{\epsilon,\beta}(\hat{u}_{S})|\\
&\leq \mathbb{E} \underset{u\in \mathcal{U}}{\sup} |\mathcal{E}_{data} - \mathcal{E}_{data,n}| + \mathbb{E} \underset{u\in \mathcal{U}}{\sup} |\mathcal{E}_{energy} - \mathcal{E}_{energy,n}| + \mathbb{E} \underset{u\in \mathcal{U}}{\sup} |\mathcal{E}_{data} - \mathcal{E}_{data,n}|\\
&\leq \mathbb{E}\,\bigg|\frac{1}{nm} \sum_{i=1,j=1}^{nm}g_{data}(x_{ij},t_{j}) - \mathbb{E} [g_{data}]\,\bigg|\\
&\quad+ \mathbb{E}\,\bigg|\frac{1}{nm} \sum_{i=1,j=1}^{nm}g_{correction}(x_{ij},t_{j}) - \mathbb{E} [g_{correction}]\,\bigg|\\
&\quad\ + \mathbb{E}\,\bigg|\frac{1}{nm} \sum_{i=1,j=1}^{nm}g_{energy}(x_{ij},t_{j}) - \mathbb{E} [g_{energy}]\,\bigg|.
\end{align}
Thanks to Lemma ~$\ref{lemma_F.2}$, we have 
\begin{align}
\E\Delta\mathcal{E}_{S,gen}\,&\leq 2\mathcal{R}_{N}(\mathcal{G}_{data}) + 2\mathcal{R}_{N}(\mathcal{G}_{energy}) + 2\mathcal{R}_{N}(\mathcal{G}_{correction}).
\end{align}
Next we apply Dudley's theorem (cf. Theorem~\ref{theorem_A.3}) to bound the $\epsilon_{1}$-covering number of the $\mathcal{R}_{N}(\mathcal{G}_{data})$, $\mathcal{R}_{N}(\mathcal{G}_{energy})$, and $\mathcal{R}_{N}(\mathcal{G}_{correction})$.
\end{proof}

According to the Lemma \ref{lemma_E.5}, it is essential to bound Rademacher complexities to control the generalization error bound $\mathcal{E}_{S,gen}$. The idea is to bound the Rademacher complexities in terms of the metric entropy by using the well-known Dudley theorem \cite{Dudley1967TheSO}.
\end{lemma}

\begin{lemma}
\label{lemma_E.6}
Let $\mathcal{N}(\mathcal{F}, \rho ,\epsilon_{1})$ be $\epsilon_{1}$-covering number for functional space $\mathcal{F}$ under evaluation metric $\rho$.  When $\rho:=||\cdot||_{L^{2}_{t}(0,T;L^{2}_{x}(\Omega))}$, then
\begin{align}
&\log N(\mathcal{G}_{data}, \rho,\epsilon_{1}) \lesssim  J^{d+2}\bigg[\log{\frac{1}{\epsilon_{1}}} + 2d \log{J} \bigg], \\
&\log N(\mathcal{G}_{energy}, \rho,\epsilon_{1}) \lesssim_{R,F}  J^{d+2}\bigg[\log{\frac{\beta}{\epsilon\cdot\epsilon_{1}}} + 4dq\log{J} \bigg],\\
&\log N(\mathcal{G}_{correction}, \rho,\epsilon_{1}) \lesssim_{R,T}  J^{d+2}\bigg[ \log{\frac{\beta}{\epsilon \cdot \epsilon_{1}}} + 4dq \log{J}\bigg].
\end{align}
\end{lemma}
Next we use covering numbers in Lemma \ref{lemma_E.6} along with Dudley theorem to prove the following theorem.

\begin{proof}[Proof of Lemma~\ref{lemma_E.6}]
Our estimator is defined as $\hat{u} :=\hat{u}_{\theta}=\hat{u}_{J}$, which is discussed in \eqref{empirical_loss_function} and \eqref{parametrized_function_class}). Before we dive into our proof, we want to define the following relevant functional spaces:
\begin{align}
\mathcal{A}&: = \{\,a \mid a:= a_{\xi,\tau}, ||a_{\xi,\tau}||_{2}\leq R\,\},\\
\mathcal{G}_{data,1}&: = \{\,g \mid g:=\nabla u_{t} -\nabla u_{t}^{*}, u^{*},u\in \mathcal{U}\,\},\\
\mathcal{G}_{energy,1}&: = \{\,g \mid g:= \frac{\beta \cdot e^{-\frac{t}{\epsilon}}}{2} |u_{t}|^{2}, u\in \mathcal{U}_{J}\,\},\\
\mathcal{G}_{energy,2}&: = \{\,g \mid g:= \frac{\beta\cdot e^{-\frac{t}{\epsilon}}}{2\epsilon} |\nabla u|^{2}, u\in \mathcal{U}_{J}\,\},\\
\mathcal{G}_{energy,2}&: = \{\,g \mid g:= -\frac{\beta\cdot e^{-\frac{t}{\epsilon}}}{\epsilon} fu, u\in \mathcal{U}_{J}, \sup_{t,x}|f|\le F\,\},\\
\mathcal{G}_{correction,1}&: = \{\,g \mid g:= \beta e^{-\frac{T}{\epsilon}}u^{*}_{t}(x,T)u(x,T),\,\, u^{*},u\in \mathcal{U}_{J}\,\},\\
\mathcal{G}_{correction,2}&: = \{\,g \mid g:= \beta e^{-\frac{t}{\epsilon}}u^{*}_{tt}u,\,\, u^{*},u\in \mathcal{U}_{J}\,\}.
\end{align}

We first obtain covering number for $\mathcal{G}_{data}$. It is sufficient to focus on $\nabla \hat{u}_{t} -\nabla u_{t}^{*}$ in data fitting term because it contains largest metric entropy compared to other terms due to loss of space-time derivatives.
\begin{equation}
\bigg|\,(\nabla \hat{u}_{t}^{1} -  \nabla u_{t}^{*}) - (\nabla \hat{u}_{t}^{2} -  \nabla u_{t}^{*})\,\bigg|=\bigg|\,\nabla \hat{u}_{t}^{1} - \nabla \hat{u}_{t}^{2}\,\bigg| \leq  \sum_{1\leq|\xi|,\tau \leq J} |a_{\xi,\tau}^{1} - a_{\xi,\tau}^{2} |\cdot|\xi|\cdot|\tau |\cdot\phi_{\xi,\tau}(x,t).
\end{equation}
According to orthogonality property of local basis $\phi_{\xi,\tau}(x,t)$ , we have 
\begin{align}
||\,\nabla \hat{u}_{t}^{1} - \nabla \hat{u}_{t}^{2}\,||_{L^{2}_{t}(\mathbb{R}; L^{2}_{x}(\mathbb{R}^{d}))}^{2}& \leq  \sum_{1\leq|\xi|_{\infty}\leq J,1\leq \tau \leq J^{2}} |a_{\xi,\tau}^{1} - a_{\xi,\tau}^{2} |^{2}\cdot|\xi|^{2}\cdot|\tau |^{2 } \\& \leq J^{2(d+2)} \sum_{1\leq|\xi|_{\infty}\leq J,1\leq \tau \leq J^{2}} |a_{\xi,\tau}^{1} - a_{\xi,\tau}^{2} |^{2}.
\end{align}

The inequality above suggests a control of finitely many spatial-temporal cofficients $a_{\xi,\tau}\in \mathcal{A}$ in $L^{2}$ ball of coefficient space is sufficient to construct an $\epsilon_{1}$-net for $g\in \mathcal{G}_{data,1}$.Let $\delta:= \frac{\epsilon_{1}}{J^{d+1}}$, we want to construct $L^{2}$ packing number for the space of  finitely many Fourier coefficients $a_{\xi,\tau}$. Using standard $L^{2}$-packing for $L^{2}$ balls argument in Lemma ~\ref{lemma_F.6} and note the dimension of coefficients of local basis representation are at most $J^{d+1}$, we have

\begin{equation}
\log N(\mathcal{A}, ||\cdot||_2,\delta)
\lesssim  \log(\frac{J^{d+2}}{\epsilon_{1}})^{J^{d+2}}= J^{d+2}\bigg[\log{\frac{1}{\epsilon_{1}}} + (d+2)\log{J}\bigg].
\end{equation}

In this way, we obtain $\delta$-covering number for the functional space $\mathcal{G}_{data,1}$
\begin{equation}
\log N(\mathcal{G}_{data,1}, ||\cdot||_{L^{2}_{t}(0,T;L^{2}_{x}(\Omega))},\epsilon_{1}) \lesssim
\log N(\mathcal{A}, ||\cdot||_{2},\delta) \lesssim  J^{d+2}\bigg[\log{\frac{1}{\epsilon_{1}}} + (d+2)\log{J}\bigg].
\end{equation}
Next we omit the same procedure applied to $\mathcal{G}_{energy}$ and $\mathcal{G}_{correction}$. In similar fashion, we have the following elementary estimates 
\begin{align}
|g_{energy,1}|&=\bigg|\frac{\beta \cdot e^{-\frac{t}{\epsilon}}}{2} |\hat{u}_{t}^{1}|^{2} - \frac{\beta \cdot e^{-\frac{t}{\epsilon}}}{2} |\hat{u}_{t}^{2}|^{2}\bigg|^{2}\leq \frac{\beta^2}{4} \bigg|\hat{u}_{t}^{1} - \hat{u}_{t}^{2}\bigg|^{2} \bigg(\frac{3}{2}|\hat{u}_{t}^{1}|^{2} +\frac{3}{2} | \hat{u}_{t}^{2}|^{2}\bigg),\\
|g_{energy,2}|&=\bigg|\frac{\beta \cdot e^{-\frac{t}{\epsilon}}}{2\epsilon} |\nabla \hat{u}^{1}|^{2} - \frac{\beta \cdot e^{-\frac{t}{\epsilon}}}{2\epsilon} |\nabla \hat{u}_{t}^{2}|^{2}\bigg|^{2}\leq \frac{\beta^2}{4\epsilon^{2}} \bigg|\nabla \hat{u}^{1} - \nabla \hat{u}^{2}\bigg|^{2} \bigg(\frac{3}{2}|\nabla \hat{u}^{1}|^{2} +\frac{3}{2} |\nabla \hat{u}^{2}|^{2}\bigg),\\
|g_{energy,3}|&=\bigg|\frac{\beta\cdot e^{-\frac{t}{\epsilon}}}{\epsilon} fu^{1} - \frac{\beta\cdot e^{-\frac{t}{\epsilon}}}{\epsilon} fu^{2}  \bigg|^{2}\leq \frac{\beta^2}{\epsilon^{2}} \cdot f\cdot \bigg|\hat{u}^{1} - \hat{u}^{2}\bigg|^{2} .
\end{align}
Thanks to Sobolev Multiplication theorem Theorem~\ref{Thm.A5}, when $q\geq \frac{d}{2} +1$, we can obtain for any $t$,
\begin{align}
||g_{energy,1}(\cdot,t)||_{L^{2}_{x}(\mathbb{R}^{d})}\leq \frac{3\beta}{4}||\hat{u}_{t}^{1}(\cdot,t) - \hat{u}_{t}^{2}(\cdot,t)||_{H^{q}_{x}(\mathbb{R}^{d})}(||\hat{u}_{t}^{1}(\cdot,t)||_{H^{q}_{x}(\mathbb{R}^{d})} +||\hat{u}_{t}^{2}(\cdot,t)||_{H^{q}_{x}(\mathbb{R}^{d})}),\\
||g_{energy,2}(\cdot,t)||_{L^{2}_{x}(\mathbb{R}^{d})}\leq \frac{3\beta}{4\epsilon}||\hat{u}^{1}(\cdot,t) - \hat{u}^{2}(\cdot,t)||_{H^{q}_{x}(\mathbb{R}^{d})}(||\hat{u}^{1}(\cdot,t)||_{H^{q}_{x}(\mathbb{R}^{d})} +||\hat{u}^{2}(\cdot,t)||_{H^{q}_{x}(\mathbb{R}^{d})}).
\end{align}
Applying uniform control of Sobolev norm in time $t$ and orthogonality property of local basis $\phi_{\xi,\tau}(x,t)$, we have
\begin{align}
||g_{energy,1}||_{L^{2}_{t}(\mathbb{R};L^{2}_{x}(\mathbb{R}^{d}))}^{2}&\leq \frac{9\cdot R^{2}\cdot\beta^{2}}{16}||\hat{u}^{1} - \hat{u}^{2}||_{H^{1}_{t}(\mathbb{R};H^{q}_{x}(\mathbb{R}^{d}))}^{2}\\
&\leq \frac{9\cdot R^{2}\cdot\beta^{2}}{16}J^{2+4dq} \sum_{1\leq |\xi|_{\infty},\tau\leq J} |a_{\xi,\tau}^{1} - a_{\xi,\tau}^{2}|^{2},
\end{align}
\begin{align}
||g_{energy,2}||_{L^{2}_{t}(\mathbb{R};L^{2}_{x}(\mathbb{R}^{d}))}^{2}&\leq \frac{9\cdot R^{2}\cdot \beta^{2}}{16\epsilon^{2}}||\hat{u}^{1} - \hat{u}^{2}||_{L^{2}_{t}(\mathbb{R};H^{q}_{x}(\mathbb{R}^{d}))}^{2}\\
&\leq \frac{9\cdot R^{2}\cdot \beta^{2}}{16\epsilon^{2}}J^{4dq} \sum_{1\leq |\xi|_{\infty},\tau\leq J} |a_{\xi,\tau}^{1} - a_{\xi,\tau}^{2}|^{2}\\,
\end{align}
\begin{align}
||g_{energy,3}||_{L^{2}_{t}(\mathbb{R};L^{2}_{x}(\mathbb{R}^{d}))}^{2}&\leq \frac{\beta^{2}}{\epsilon^{2}}\cdot F^{2}  ||\hat{u}^{1} - \hat{u}^{2}||_{L^{2}_{t}(\mathbb{R};L^{2}_{x}(\mathbb{R}^{d}))}^{2}\\
&\leq \frac{\beta^{2}}{\epsilon^{2}}\cdot F^{2}J^{2d+4} \sum_{1\leq |\xi|_{\infty},\tau\leq J} |a_{\xi,\tau}^{1} - a_{\xi,\tau}^{2}|^{2}.
\end{align}
It follows as before
\begin{align}
\log N(\mathcal{G}_{energy,1}, ||\cdot||_{L^{2}_{t}(0,T;L^{2}_{x}(\Omega))},\epsilon_{1})  &\lesssim J^{d+2}\bigg[\log{\frac{3R\cdot\beta}{4\epsilon_{1}}} + (2+dq)\log{J}\bigg],  \\
\log N(\mathcal{G}_{energy,2}, ||\cdot||_{L^{2}_{t}(0,T;L^{2}_{x}(\Omega))},\epsilon_{1}) &\lesssim J^{d+2}\bigg[\log{\frac{3R\cdot\beta}{4\epsilon\cdot \epsilon_{1}}} +dq\log{J}\bigg] ,\\
\log N(\mathcal{G}_{energy,3}, ||\cdot||_{L^{2}_{t}(0,T;L^{2}_{x}(\Omega))},\epsilon_{1}) &\lesssim J^{d+1}\bigg[\log{\frac{F\cdot \beta}{\epsilon\cdot\epsilon_{1}}} + (d+2)\log{J}\bigg].
\end{align}
we then bound $\mathcal{G}_{correction}$
\begin{align}
|g_{correction, 1}|&=\bigg| e^{-\frac{T}{\epsilon}} u_{t}^{*}(x,T) \hat{u}^{1}(x,T) -e^{-\frac{T}{\epsilon}} u_{t}^{*}(x,T) \hat{u}^{2}(x,T) \bigg|^{2}\\
&\leq  \bigg|\hat{u}_{t}^{1}(x,T) - \hat{u}_{t}^{2}(x,T)\bigg|^{2}\cdot \bigg|u^{*}_{t}(x,T)\bigg|^{2},\\
|g_{correction, 2}|&=\bigg|e^{-\frac{t}{\epsilon}}u^{*}_{tt}\hat{u}^{1} - e^{-\frac{t}{\epsilon}}u^{*}_{tt}\hat{u}^{2}\bigg|^{2}\leq e^{-\frac{2t}{\epsilon}} \bigg|\hat{u}^{1} -   \hat{u}^{2}\bigg|^{2} \cdot \bigg|u_{tt}^{*}\bigg|^{2}.
\end{align}
Again, we apply the Sobolev multiplication theorem (cf. Theorem~\ref{thm:sobolev_multiplication}) and use upper bound on Sobolev norm uniformly in time $t$ to get
\begin{align}
||\,g_{correction,1}\,||_{L^{2}_{t}(0,T;L^{2}_{x}(\mathbb{R}^{d})}&\leq \beta\cdot T ||\,\hat{u}^{1}_{t}(\cdot,T) - \hat{u}^{2}_{t}(\cdot,T)\,||_{ H^{q}_{x}(\mathbb{R}^{d})}||\,u_{t}^{*}(\cdot,T)\,||_{ H^{q}_{x}(\mathbb{R}^{d})}\\
&\leq \beta\cdot R\cdot T||\,\hat{u}^{1}_{t}(\cdot,T) - \hat{u}^{2}_{t}(\cdot,T)\,||_{ H^{q}_{x}(\mathbb{R}^{d})},\\
||g_{correction, 2}||_{L^{2}_{t}(0,T;L^{2}_{x}(\mathbb{R}^{d})}&\leq   \beta\cdot ||\hat{u}^{1} -   \hat{u}^{2}||_{L^{2}_{t}(0,T;H^{q}_{x}(\mathbb{R}^{d}))} \cdot ||u_{tt}^{*}||_{L^{2}_{t}(0,T;H^{q}_{x}(\mathbb{R}^{d}))}\\
&\leq R||\hat{u}^{1} -   \hat{u}^{2}||_{L^{2}_{t}(0,T;H^{q}_{x}(\mathbb{R}^{d}))}.
\end{align}
Here we use property \ref{basis_property_2} of our local basis $\phi_{\tau}(t)$ to get $\sup_{t} \phi_{\tau}\leq C(T)$, where $C(T)$ is some constant depending on $T$
\begin{align}
||g_{correction, 1}||_{L^{2}_{t}(0,T;L^{2}_{x})}^{2} &\leq  R^{2} T^{2} C(T)^{2} \beta^{2}J^{2+2dq} \sum_{1\leq |\xi|_{\infty},\tau\leq J} |a_{\xi,\tau}^{1} - a_{\xi,\tau}^{2}|^{2},\\
||g_{correction, 2}||_{L^{2}_{t}(0,T;L^{2}_{x})}^{2} &\leq R^{2} \beta^{2}J^{2dq} \sum_{1\leq |\xi|_{\infty},\tau\leq J} |a_{\xi,\tau}^{1} - a_{\xi,\tau}^{2}|^{2}.
\end{align}
It follows
\begin{align}\log N(\mathcal{G}_{correction,1}, ||\cdot||_{L^{2}_{t}(0,T;L^{2}_{x}(\Omega))},\epsilon_{1}) &\lesssim J^{d+2}\bigg[\log{\frac{R\cdot T\cdot C(T)\cdot \beta}{4\epsilon\cdot \epsilon_{1}}} + (2+dq)\log{J}\bigg], \\
\log N(\mathcal{G}_{correction,2}, ||\cdot||_{L^{2}_{t}(0,T;L^{2}_{x}(\Omega))},\epsilon_{1})  &\lesssim J^{d+2}\bigg[\log{\frac{R\cdot \beta}{4\epsilon_{1}}} + dq\log{J}\bigg].
\end{align}
Finally, we obtain the following covering numbers for $\mathcal{G}_{data}$, $\mathcal{G}_{energy}$, and $\mathcal{G}_{correction}$ respectively
\begin{align}
\log N(\mathcal{G}_{data}, ||\cdot||_{L^{2}_{t}(0,T;L^{2}_{x}(\Omega))},\epsilon_{1}) & \lesssim  J^{d+2}\bigg[\log{\frac{1}{\epsilon_{1}}} + (1+d) \log{J} \bigg], \\
\log N(\mathcal{G}_{energy}, ||\cdot||_{L^{2}_{t}(0,T;L^{2}_{x}(\Omega))},\epsilon_{1}) & \lesssim  C_{1}(F)J^{d+2}\bigg[\log{\frac{\beta}{\epsilon_{1}}} + (1+dq)\log{J}\bigg],  \\
\log N(\mathcal{G}_{correction}, ||\cdot||_{L^{2}_{t}(0,T;L^{2}_{x}(\Omega))},\epsilon_{1}) &\lesssim C_{2}(R,T) J^{d+2}\bigg[\log{\frac{\beta}{\epsilon\cdot \epsilon_{1}}} + (1+dq)\log{J}\bigg],
\end{align}
where $C_{1}$ is some constant depending on $F$ and $R$; $C_{2}$ is some constant depending on $R$ and $T$.
\end{proof}
Finally, we present our proof to Theorem ~\ref{thm_3.4}.
\begin{proof}[Proof of Theorem~\ref{thm_3.4}]
First, we define population risk and its empirical parts:
\begin{align}
\mathcal{E}_{data}&:= {\mathbb{E}}\bigg[u(x,t) - u^{*}(x,t) + \nabla (u(x,t) - u^{*}(x,t)) \\&\quad\quad\quad\quad\quad+  (u_{t}(x,t) - u^{*}_{t}(x,t))+ \nabla u_{t}(x,t) -  \nabla u^{*}_{t}(x,t))\bigg], \\
\mathcal{E}_{data,n}&:= \frac{1}{nm} \sum_{i=1,j=1}^{n,m} \bigg[u(x_{ij},t_{j}) - u^{*}(x_{ij},t_{j}) + \nabla (u(x_{ij},t_{j}) - u^{*}(x_{ij},t_{j}))\\
&\quad\quad\quad\quad\quad\quad+  (u_{t}(x_{ij},t_{ij})- u^{*}_{t}(x_{ij},t_{j}))+ \nabla( u_{t}(x_{ij},t_{j})-  u^{*}_{t}(x,t))\bigg],\\
\mathcal{E}_{energy}&:= {\mathbb{E}}\bigg[\dfrac{\beta\cdot e^{-\frac{t}{\epsilon}}}{2}\bigg( |u_{t}|^{2}(x,t) +  \frac{1}{2\epsilon} |\nabla u|^{2}(x,t)- f(x,t)u(x,t) \bigg) \bigg], \\
\mathcal{E}_{energy,n}&:= \frac{\beta}{nm} \sum_{i=1,j=1}^{n,m}\cdot\bigg[\dfrac{e^{-\frac{t}{\epsilon}}}{2} \bigg(|u_{t}|^{2}(x_{ij},t_{j}) +  \frac{1}{2\epsilon} |\nabla u|^{2}(x_{ij},t_{j})- f(x_{ij},t_{j})u(x_{ij},t_{j}) \bigg) \bigg],\\
\mathcal{E}_{correction}&:= \beta\cdot{\mathbb{E}} \bigg[e^{-\frac{T}{\epsilon}}u^{*}_{t}(x,t_{m})u(x,t_{m})+  e^{-\frac{t}{\epsilon}}u^{*}_{tt}(x,t)u(x,t) \bigg],\\
\mathcal{E}_{correction,n}&:= \frac{\beta}{nm} \sum_{i=1,j=1}^{n,m} \bigg[e^{-\frac{T}{\epsilon}}u^{*}_{t}(x_{ij},t_{m})u(x_{ij},t_{m})+  e^{-\frac{t}{\epsilon}}u^{*}_{tt}(x_{ij},t_{j})u(x_{ij},t_{j}) \bigg].
\end{align}
It is essential to prove Lemma~\ref{lemma_E.5} and Lemma~\ref{lemma_E.6} before proving Theorem \ref{thm_3.4}.

With $L^{2}$-covering numbers at our hand, we can directly apply Dudley theorem to bound Rademacher complexities. Before applying Lemma \ref{lemma_F.8}, we need to find $\sup_{f\in \mathcal{F}} \mathbb{E}(f^{2})$. For $\mathcal{G}_{data}$, we have
\begin{equation}
\mathbb{E}[g_{data,1}^{2}]\lesssim \rho.
\end{equation}
For $\mathcal{G}_{energy}$, since $||f||_{L^{\infty}} \leq F$ and when $q\geq \frac{d}{2}+1$, we can apply embedding Theorem \ref{lemma_F.8} to spatial variable and obtain $H^{q}(\mathbb{R}^{d}) \subset W^{1,4}(\mathbb{R}^{d})$. Furthermore, when $p\geq2$, we have $u_{t}(t,\cdot) \in L^{4}([0,T])$ by standard Sobolev embedding theorem as well, and it follows that
\begin{align}
\mathbb{E}[g_{energy}^{2}]&\lesssim \int_{0}^{T} \int_{\mathbb{R}^{d}} (1+\frac{e^{\frac{-2t}{\epsilon}}}{4} )|u_{t}^{4}| + (1+\frac{1}{4\epsilon^{2}} )|\nabla u|^{4} + (1+\frac{e^{\frac{-2t}{\epsilon}}}{\epsilon^{2}}) f^{4} + (1+\frac{e^{\frac{-2t}{\epsilon}}}{\epsilon^{2}})u^{4}dxdt \\
&\lesssim \frac{5}{4}||u_{t}||_{L^{4}_{t}(0,T;L^{4}_{x}(\mathbb{R}^{d}))}^{4} +  (1+\frac{5}{4\epsilon^{2}})||u||_{L^{4}_{t}(0,T;W^{1,4}_{x}(\mathbb{R}^{d}))}^{4} + (1+\frac{1}{\epsilon^{2}})F^{4}\\
&\lesssim (1+\frac{5}{4\epsilon^{2}})(R^{4} + F^{4}).
\end{align}
For $\mathcal{G}_{correction}^{2}$, it follows that
\begin{align}
\mathbb{E}[\mathcal{G}_{correction}^{2}]&\leq \int_{0}^{T}\int_{\Omega} \frac{3}{4}|u_{tt}|^4 +\frac{3}{2} |u|^{4} dxdt + T\int_{\Omega} u_{t}^{4}(x,T) dx \\
&\lesssim \frac{3}{4}||u_{tt}||^{4}_{L^{4}_{t}(0,T;L^{4}_{x}(\mathbb{R}^{d}))} + \frac{3}{2}||u||^{4}_{L^{4}_{t}(0,T;L^{4}_{x}(\mathbb{R}^{d}))} + T||u_{t}(\cdot,T)||^{4}_{L^{4}_{x}(\mathbb{R}^{d})}\\
&\lesssim \frac{3}{2}R^{4} + T \max_{t\in[0,T]} ||u_{t}(\cdot,t)||_{L^{4}_{x}(\mathbb{R}^{d})}^{4} = \max\{\frac{3}{2},T\} (R^{4}+M^{4}(T)).
\end{align}
Finally, we can calculate with $C_{1}(R,F):={(1+\frac{5}{4\epsilon^{2}})(R^{4} + F^{4})} $ and $C_{2}(R,F) = (1+\frac{5}{4\epsilon^{2}})(R^{4} + F^{4})$, 
\begin{equation}
\mathcal{R}_n(\mathcal{G}_{data}) \le\inf_{0 \le \delta}\left(
4\delta + \frac{12}{\sqrt{nm}} \int_{\delta}^{R^{2}}\sqrt{J^{d+1}\bigg[\log{\frac{1}{\epsilon_{1}}} + (2+d) \log{J} \bigg]}\,d\varepsilon_{1}\right),
\end{equation}
\begin{equation}
\mathcal{R}_n(\mathcal{G}_{energy}) \le\inf_{0 \le \delta}\left(
4\delta + \frac{12}{\sqrt{nm}} \int_{\delta}^{C_{1}(R,F)}\sqrt{J^{d+1}\bigg[\log{\frac{\beta}{\epsilon\cdot\epsilon_{1}}} + (2+dq)\log{J} \bigg]}\,d\varepsilon_{1}\right),
\end{equation}
\begin{equation}
\mathcal{R}_n(\mathcal{G}_{correction}) \le\inf_{0 \le \delta}\left(
4\delta+ \frac{12}{\sqrt{nm}} \int_{\delta}^{C_{2}(R,F)}\sqrt{J^{d+1}\bigg[\log{\frac{\beta}{\epsilon\cdot\epsilon_{1}}} + (2+dq)\log{J} \bigg]}\,d\varepsilon_{1}\right).
\end{equation}
It follows that
\begin{align}
\mathcal{R}_{n}(\mathcal{G}_{data})&\lesssim  \sqrt{\frac{\rho}{nm}} J^{\frac{d+1}{2}}  \bigg[\sqrt{\log(\frac{1}{\rho}+1)} + \sqrt{(2+d)\log{J}}\bigg] \\
&\lesssim C_{3}(R) \sqrt{\frac{\rho}{nm}} J^{\frac{d+1}{2}} \sqrt{\log J \log (1+\frac{1}{\rho})},\\ 
\mathcal{R}_{n}(\mathcal{G}_{data})&\lesssim \sqrt{\frac{\beta}{nm}} J^{\frac{d+2}{2}} \sqrt{(1+\frac{5}{4})(R^{4}+F^{4})}\bigg[\sqrt{\log(\frac{1}{{(1+\frac{5}{4})(R^{4}+F^{4})\epsilon}}+1)}\\
&\quad+ \sqrt{(2+dq)\log{J}}\bigg]\\
&\lesssim C_{4}(\epsilon,d,q,R,F)\sqrt{\frac{\beta}{nm}} J^{\frac{d+2}{2}}\sqrt{\log{J}},  \\
\mathcal{R}_{n}(\mathcal{G}_{data})&\lesssim\sqrt{\frac{\beta}{nm}} J^{\frac{d+1}{2}} \sqrt{\max\{\frac{3}{2},T\}(R^{4}+M^{4}(T))} \bigg[\sqrt{\log(\frac{1}{{\max\{\frac{3}{2},T\}(R^{4}+M^{4}(T))\epsilon} }+1)}\\
&\quad\quad+ \sqrt{(2+dq)\log{J}}\bigg]\\
&\lesssim C_{5}(\epsilon, d,q,R,T)\sqrt{\frac{\beta}{nm}} J^{\frac{d+2}{2}} \sqrt{\log{J}}.
\end{align}
We need to appropriately balance the radius of localization ball $\rho$ and regularizing parameter $\beta$. We pick $\rho \approx \frac{1}{nm}$ and $\beta \approx \frac{\log{nm}}{nm} \log J$, and then
\begin{equation}
\mathcal{R}_{n}(\mathcal{G}) \lesssim C(\epsilon, d,q, R,T, F) J^{\frac{d+2} {2}}\log{J}\frac{\log{nm}}{nm}. 
\end{equation}
\end{proof}

\section{Useful Technical Tools}

\begin{theorem}[Bernstein Inequality]
\label{thm:bernstein_ineq}
Let $X_{i}$,...,$X_{N}$ be i.i.d random variables. Suppose that $\mathbb{E}X_{i} =0$, $|X_{i}|\leq M$ almost surely, and $\frac{1}{N}\sum_{i=1}^{n}Var(X_{i}) \leq \sigma^{2}$. For any $t\geq0$, 
\begin{equation}
\mathbb{P}\left( \frac{1}{n}\sum_{i=1}^{n} X_{i}\geq t \right)\leq \exp\left(-\frac{nt^{2}}{2\sigma^2 + \frac{2}{3} Mt}\right).
\end{equation}
\end{theorem}

\begin{proof}[Proof of Theorem~\ref{thm:bernstein_ineq}]
    \citet[Lemma~2.2.9]{vanderVaart1996}.
\end{proof}

\begin{lemma}[Maximal inequality of empirical process]
\label{lemma_F.2}
Let $\Omega$ be the domain and $\mathcal{F}$ be a set of functions $f$ and $\mathcal{P}_{\Omega}$ is probability measure from which $X$ is sampled. Then 
\begin{equation}
\mathbb{E}\, \underset{f\in \mathcal{F}}{\sup}\,\bigg|\frac{1}{n} \sum_{i=1}^{n} f(X_{i}) - \mathbb{E} f(X_i) \bigg| \leq 2 \mathcal{R}_{n,\mathcal{F}}.
\end{equation}
\end{lemma}

\begin{proof}[Proof of Lemma~\ref{lemma_F.2}]
    See Proposition 4.11 in~\citep{wainwright2019high}.
\end{proof}

\begin{theorem}[Dudley's Theorem]
\label{theorem_A.3}
Let $\mathcal{F}$ be a function class such that $\sup_{f \in \mathcal{F}} \|f\|_{L_2(P_{n})} \le M $.
Then the Rademacher complexity $\mathcal{R}_n(\mathcal{F})$ satisfies
\begin{equation}
\mathcal{R}_n(\mathcal{F}) \le\inf_{0 \le \delta}\left\{
4\delta + \frac{12}{\sqrt{n}} \int_{\delta}^{\sup_{f\in\mathcal{F}}\mathbb{E}(f^{2})}\sqrt{\log N\!\left(\varepsilon,\,\mathcal{F},\,\|\cdot\|_{L^{2}(P_{n})}\right)}\,d\varepsilon\right\},
\end{equation}
where $N(\varepsilon,\mathcal{F},\|\cdot\|_{\infty})$ denotes the
$\varepsilon$-covering number of $\mathcal{F}$ with respect to the
$L^{\infty}$-norm.
\end{theorem}

\begin{proof}[Proof of Theorem~\ref{theorem_A.3}]
    See Theorem 4.28 in~\citep{ma2020statistical}.
\end{proof}

\begin{theorem}[Bernstein Chaining]
Let $\mathcal{F}$ be a class of measurable functions such that 
$\E[f^2]<\sigma^2$ and $||f||_{\infty}<L$ for all $f\in \mathcal{F}$. Then, 
\begin{equation}
\E\bigg[\underset{f\in \mathcal{F}}{\sup} |\sqrt{n}\bigl(\hat{P} - P\bigr)f | \bigg] 
\lesssim \D_{[\,\,]}(\sigma, \mathcal{F},L_{2}(P)) \bigg(\D_{[\,\,]}(\sigma, \mathcal{F},L_{2}(P)) + \frac{\D_{[\,\,]}(\sigma, \mathcal{F},L_{2}(P))}{\sigma^2\sqrt{n}} L \bigg).
\end{equation}
\label{Bernstein Chaining}
\begin{proof}
Apply \citet[Lemma~3.4.2]{vanderVaart1996}
to $\mathcal F/L$ and rescale.
\end{proof}
\end{theorem}

\begin{proof}[Proof of Theorem~\ref{Bernstein Chaining}]
    See Lemma 4.3.2 in~\citep{74e220ef6fe34e04a1f9009a5c17c6cb}.
\end{proof}

\begin{theorem}[Talagrand's Inequality]
Let $\mathcal{F}$ be a class of measurable functions such that $\E[f^2]<\sigma^2$ and $||f||_{\infty}<L$ for all $f\in \mathcal{F}$. Then
\begin{equation}
\mathbb{P}\bigg( \underset{f\in\mathcal{F}}{\sup} \sqrt{n} |(\hat{P}-P)f|\geq 2\E[\underset{f\in\mathcal{F}}{\sup} |(\hat{P} - P)f| + \sigma\sqrt{2s} + \frac{L}{\sqrt{n}}s  \bigg)\leq e^{-s}.
\end{equation} \label{Talagrand Inequality}
\end{theorem}

\begin{proof}[Proof of Theorem~\ref{Talagrand Inequality}]
    See  \citet[theorem 2.3]{BOUSQUET2002495}.
\end{proof}

\begin{lemma} (Covering numbers for $L_2$ balls) Let $p\in \mathbb{N}$ and $B_{2}(r)$ be the $L_{2}$ balls of $\mathbb{R}^{d}$ with radius $r$ such that $\epsilon \leq r$. Then 
\begin{equation}N(B_{2}(r),||\cdot||_{2},\epsilon) \leq \left(\frac{3r}{\epsilon}\right)^{d}.
\end{equation}
\label{lemma_F.6}
\end{lemma}

\begin{proof}[Proof of Lemma~\ref{lemma_F.6}]
    This lemma follows directly from Corollary 4.2.13 in~\citep{Vershynin_2018} by rescaling.
\end{proof}

\begin{theorem}[Sobolev Multiplication Theorem]
\label{thm:sobolev_multiplication}
Let $\Omega \subseteq \mathbb{R}^d$ be a Lipschitz domain, and let $s > \frac{d}{p}$ with $1 \leq p \leq \infty$. For any $u, v \in W^{s,p}(\Omega)$,  there exists a constant $C = C(n, s, p, \Omega)$ such that:
\label{Thm.A5}
\begin{equation}
    \|uv\|_{W^{s,p}(\Omega)} \leq C \|u\|_{W^{s,p}(\Omega)} \|v\|_{W^{s,p}(\Omega)}.
\end{equation}
\end{theorem}

\begin{proof}[Proof of Theorem~\ref{thm:sobolev_multiplication}]
    This theorem follows from Theorem 2.8.3 in~\citep{triebel2006theory}.
\end{proof}

\begin{lemma}[Morrey's embedding]
\label{lem:morrey_embedding}
Let $\Omega\subset\mathbb R^d$ be a bounded domain
with $C^1$ boundary. Let $1\le p<\infty$ and
$k,m\in\mathbb N_0$ satisfy $m>d/p$.
Then the embedding
\[
W^{k+m,p}(\Omega)\hookrightarrow C^k(\overline\Omega)
\]
is continuous and compact. Then
\begin{equation}
W^{k+m,p} \hookrightarrow C^{k}(\bar{\Omega}).
\end{equation}
In particular, the embedding is compact if $m-\frac{d}{p}>k$.
\label{Sobolev_embedding}
\end{lemma}

\begin{proof}[Proof of Lemma~\ref{lem:morrey_embedding}]
    This lemma is a consequence of Section 5.6.2 in \citep{evans} and Theorem 4.12 in \citep{MR0450957}.
\end{proof}

\begin{lemma}[Weighted Poincar\'e Inequality]
\label{lem:weighted-poincare}
Let $\Omega$ be bounded Lipschitz domain and $w(x)$ be non-negative weight on $\Omega$. There exist constant $C := C(d,\Omega, w) > 0$ such that for any $f\in H^{1}(\Omega)$, we have
\begin{equation}
\|(f- f_{w,\Omega})w(\cdot)\|_{L^{2}(\Omega)} \leq C \|w(\cdot)\nabla f \|_{L^{2}(\Omega)},
\end{equation}
where $f_{w,\Omega}:= \frac{\int_{\Omega}w(x)f(x)dx}{\int_{\Omega}w(x)dx}$.
\end{lemma}

\begin{proof}[Proof of Lemma~\ref{lem:weighted-poincare}]
The lemma can be found in \citep{AIF_1954__5__305_0} and Section 3 of \citep{Turesson}. This variant of Poincare Inequality is also called the Deny-Lions Lemma. Detailed proof for the case $p=q=2$ can be found in \citep{Rathmair}.
\end{proof}

\begin{lemma}\label{lemma_F.8} Given $p,q\geq 0$. Assume $f\in H^{p}_{t}(\R;H^{q}_{x}(\R^{d}))\cap \mathcal{F}_{J,J}$ and $f$ admits spatial-temporal Fourier series with smooth basis dense in $L^{2}$ on compact space-time domain $\Omega_{T} = \Omega\times [0,T]$, then the following estimates hold \\
\begin{align}
||P_{J}f||_{L^{2}_{t}(\R;L^{2}_{x}(\R^{d}))}^{2} &\leq C(\Omega,T)||f||_{H^{1}_{t}([0,T];H^{1}_{x}(\Omega))}^{2}, \\
||P_{J}f||_{\infty}&\leq C(\Omega)R, \hspace{1.7cm}\text{if }\,\, q>\frac{d}{2},\\
||P_{J}f||_{\infty}&\leq C(\Omega,T)J^{d+2}, \hspace{0.8cm}\text{if }\,\, q\leq\frac{d}{2},
\end{align}
where $C_{1}(\Omega, T)$ is a constant depending on $\Omega_{T}$
\label{J_truncation_estimates}.
\end{lemma}

\begin{proof}[Proof of Lemma~\ref{lemma_F.8}]
To prove (F.1), we note 
\begin{align}
||P_{J}f||_{L^{2}_{t}(\R;L^{2}_{x}(\R^{d}))}^{2}  &\leq ||f||_{H^{1}_{t}(\R;H^{1}_{x} (\R^{d}))}^{2}\\
&\leq C(\Omega,T) ||f||_{H^{1}_{t}([0,T];H^{1}_{x} (\Omega))}^{2} \hspace{0.5cm} - \text{ Lemma~\ref{space_time_extension} }.
\end{align}
To prove (70), we only need to rewrite 
\begin{equation}
P_{J}f = \sum_{1\leq \tau,|\xi|\leq J} |\xi|^{p} \tau^{q} \frac{1}{|\xi|^{p}\tau^{q}} a_{\xi,\tau} \phi_{\xi,\tau}(x,t).
\end{equation}
From Morrey's embedding in case of $q>\frac{d}{2}$ and the Sobolev norm of $f$ is bounded by $R$ and the extension Theorem \ref{Sobolev_extension_theorem},
\begin{align}
||P_{J}f(\cdot,t)||_{\infty} \leq C(\Omega)||P_{J}f(\cdot,t)||_{H^{q}(\Omega)}\leq C(\Omega)R,\quad \forall t\in[0,T].
\end{align}
Otherwise, if $q\leq \frac{d}{2}$, we note there are finitely many coefficients $a_{\xi,\tau}$ and each $\phi_{\xi,\tau}(x,t)$ has local compact support
\begin{equation}
P_{J}f = \sum_{1\leq \tau\leq J^{2},1\leq |\xi|\leq J}  a_{\xi,\tau} \phi_{\xi,\tau}(x,t) \leq ||a_{\xi,\tau}||_{\infty}  \sum_{1\leq \tau\leq J^{2},1\leq |\xi|\leq J}  \phi_{\xi,\tau}(x,t) \leq C(\Omega, T) J^{d+2}.
\end{equation}
\end{proof}

\begin{lemma}\label{lem:moments_of_wavelets}
Let $\phi_\tau(t)=2^{\ell/2}\varphi(2^\ell t-k_0)$ be an $L^2$-normalized
compactly supported wavelet at level $2^\ell\asymp J^2$, let
$T_1,\dots,T_n$ be i.i.d.\ uniform on $[0,T]$, and set
$Y_j:=\phi_\tau(T_j)^2$,
$B:=\frac1n\sum_{j=1}^n\big(Y_j-\mathbb{E} Y_1\big)$. Then
\begin{equation}\label{eq:moments}
\Var(Y_1)\asymp J^2,\qquad
\mathbb{E}\big|Y_1-\mathbb{E} Y_1\big|^4\lesssim J^6,
\end{equation}
and there exist constants $c_1,c_2>0$ such that, whenever
$n\geq J^2$,
\begin{equation}\label{eq:anticonc}
\mathbb{P}\Big(|B|\geq c_1\sqrt{J^2/n}\Big)\;\geq\;c_2.
\end{equation}
\end{lemma}

\begin{proof}
By the change of variables $u=2^\ell t-k_0$,
$\int\phi_\tau^{2q}=2^{\ell(q-1)}\|\varphi\|_{L^{2q}}^{2q}$ for
$q\in\{1,2,4\}$, whence $\mathbb{E} Y_1=T^{-1}$,
$\mathbb{E} Y_1^2\asymp2^\ell\asymp J^2$, and
$\mathbb{E} Y_1^4\asymp2^{3\ell}\asymp J^6$; \eqref{eq:moments} follows. For
i.i.d.\ centered variables,
\[
\mathbb{E} B^2=\frac{\Var(Y_1)}{n},
\qquad
\mathbb{E} B^4=\frac{\mathbb{E}(Y_1-\mathbb{E} Y_1)^4}{n^3}
+\frac{3(n-1)}{n^3}\Var(Y_1)^2
\lesssim\frac{J^6}{n^3}+\frac{J^4}{n^2}.
\]
If $n\geq J^2$, then $J^6/n^3\leq J^4/n^2$, so
$\mathbb{E} B^4\leq\kappa\,(\mathbb{E} B^2)^2$ for an absolute constant
$\kappa$. The Paley--Zygmund inequality applied to $B^2$ gives, for
$\theta\in(0,1)$,
\[
\mathbb{P}\big(B^2\geqslant\theta\,\mathbb{E} B^2\big)
\geq(1-\theta)^2\frac{(\mathbb{E} B^2)^2}{\mathbb{E} B^4}
\geq\frac{(1-\theta)^2}{\kappa};
\]
taking $\theta=\tfrac12$ and using $\mathbb{E} B^2\asymp J^2/n$ yields
\eqref{eq:anticonc}.
\end{proof}

\section{More Experiments}
\label{A.7}
\subsubsection{Implementation Details}
The spatial domain $\Omega$ is set as $[-1,1]^{d}$ and the finite time horizon is $[0,1]$. All simulations are implemented in Python 3.11, Pytorch 2.8.0,  Numpy 2.3.0, and Scipy 1.16.2. Numerical derivatives are computed using findiff library 0.12.1 with an accuracy of  order 6. In low dimensions, numerical integration is performed using Scipy's integration tools in combination with vectorized Einstein summation through NumPy's einsum function. In higher dimensions, where traditional Numerical integration becomes  computationally expensive, so we employ a Monte Carlo method by randomly sampling 4096 points to approximate the integral.

\subsubsection{Data Preparation} 
Our estimator defined in \eqref{empirical_loss_function} assumes the access to training data on space-time cloud points $S:=\{(X_{ij},t_{j},u^{*}_{ij},\nabla u^{*}_{ij}, \partial_{t}u^{*}_{ij},\nabla \partial_{t} u_{ij}^{*},f_{ij})\}_{i=1,j=1}^{n,m}$. Following the sampling procedure described in (\ref{regression_model}), we uniformly sample $N_{t}$ time locations from interval $[0,1]$ and at each time location $t_{j}$, we then uniformly sample $N_{x}$ space location $x_{ij}$ from domain $\Omega$. To compute the data fidelity term under the $H^{1}_{t}H^{1}_{x}$ norm, we require not only ground truth solution $u^{*}$ as our primary label but also derivative labels $u_{t}^{*}$, $u_{x}^{*}$ and $\nabla u_{t}^{*}$ through numerical derivative computation. Training models with additional derivative information is commonly refered to Sobolev Training \citep{10.5555/3294996.3295182, lu2022sobolev, pmlr-v89-mroueh19a}. These derivative labels are generated using numerical solver of PDEs and the findiff library. In our experiments, accurate computation of derivative labels is critical. We notice imprecise derivatives can significantly degrade performance and even destabilize the learning process. To obtain high fidelity derivative datas, we propose two common sources to obtain derivatives: (1) numerically solved PDEs solutions (2) known analytical solutions. We test data from both sources. In Section ~\ref{numerical_experiments}, our training data is obtained from numerically solved PDEs solutions, with derivatives calculated by findiff, and our examples all have known analytical solutions, with derivatives calculated by autograd. For each random seed, we sample 50 time coordinates uniformly from
$[0,1]$. At every sampled time, we independently draw 80 spatial locations
uniformly from $[-1,1]$, giving 4,000 co-located training points.

\subsubsection{Training} Our loss function is the empirical counterpart of the population loss function defined in (\ref{loss_function}). In practice, we observe model correction term has negligible impact on training outcomes. To accommodate for solving different PDEs, one might need to modify energy regularization term accordingly.
The model architecture is based on a standard multilayer perceptron (MLP), following design choices commonly adopted in recent state-of-the-art deep learning approaches for physics learning tasks. We believe such architecture is not optimal architecture suitable for our framework, but we stick with MLP for fair experiments.  In this work, we adopt MLP structure with fully connected layers and Tanh activations. Optimization is carried out using the Adam optimizer \cite{Kingma2014AdamAM} as implemented in torch.optim with a fixed learning rate $1\text{e}^{-3}$. We optimize the empirical loss using Adam with a fixed learning rate of $1\text{e}^{-3}$, which we empirically observe it achieves unstable convergence with fixed learning rate Figure~\ref{fig_13}. Adaptive or smaller learning rate will make convergence stable Figure~\ref{fig_15}. Understanding the convergence behavior and robustness of different optimizers in solving PDEs remains an important question, which leave as a direction for future investigation.

\subsubsection{Model Validation} Our dataset consists of $N_t \times N_x$ spatio-temporal samples, each associated with corresponding ground-truth labels. For evaluation purposes, we use mainly relative $L_2$ norm for evaluation metric. The accuracy metric is the relative space-time $L^2$ error
\[
\operatorname{RelErr}_{L^2}(\widehat u,u^{*})
=
\frac{
\|\widehat u-u^{*}\|_{L^2(\Omega\times[0,T])}
}{
\|u^{*}\|_{L^2(\Omega\times[0,T])}
}.
\]
For these benchmarks, all reported scalar values are medians over 10
independent random seeds.However, to assess the effectiveness of our proposed method, we compare its performance against a vanilla Physics-Informed Neural Network (PINN) baseline. Both models are evaluated on the test set using the mean squared error (MSE) between predicted and ground-truth solution values as the primary performance metric. It has been reported PINNs usually have low MSE error but high $L_2$ relative error, which is also verified in table \ref{table_3}.

\subsection{Parabolic PDEs}
\label{more_experiments}
\begin{table}[htbp]
\centering
\caption{
Overview of the additional PDEs experiments .
}
\label{tab:additional-pde-overview}
\small
\setlength{\tabcolsep}{5pt}
\begin{tabular}{llll}
\toprule
Benchmark
& PDE type
& Dimension
\\
\midrule
Heat equation
& Linear parabolic
& $d=1,3$
\\
Sine--Gordon equation
& Nonlinear hyperbolic
& $d=1,3$
\\
LP1--LP3
& Variable-coefficient parabolic
& $d=3$
\\
Viscous HJ equation
& Nonlinear parabolic
& $d=1,3,10,20,50$
\\
\bottomrule
\end{tabular}
\end{table}
We are motivated by a toy example solving the Cauchy problem for the 1D heat equation. The equation has a uniquely determined solution $u(x,t) = e^{-t}\sin(x)$ for $u_t - \Delta u = 0$ with $u_0 = \sin(x)$. We compare several methods against the vanilla PINN baseline before and after introducing new labeled data.

\begin{figure}[htbp]
    \centering
    \begin{minipage}{0.7\textwidth}
        \centering
     \includegraphics[width=\textwidth]{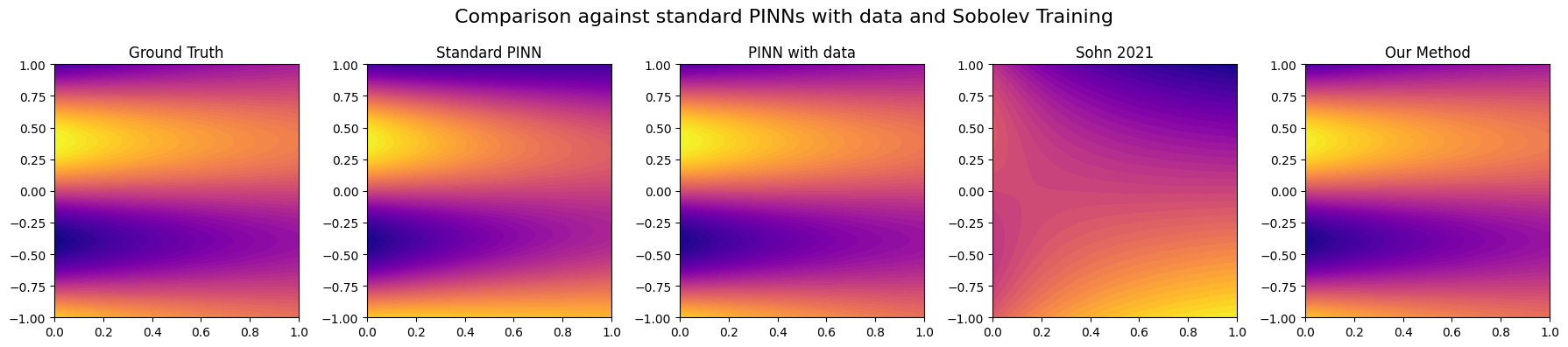}
        \includegraphics[width=\textwidth]{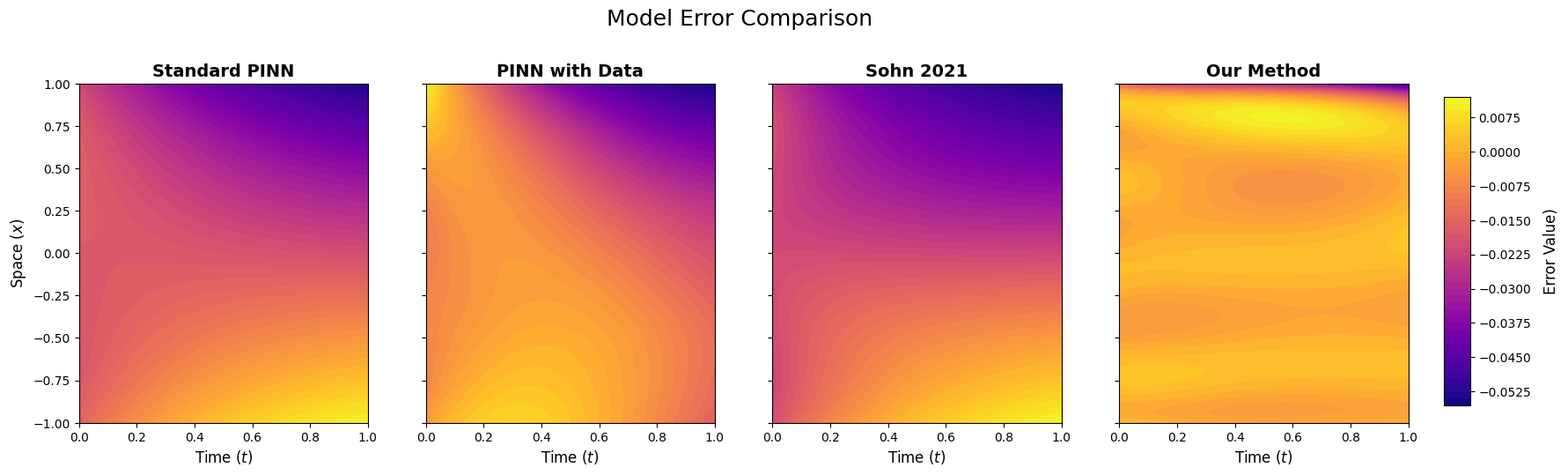}
        \caption{Error plot.}
        \label{fig:sobolev_comparison}
    \end{minipage}
\end{figure}
In standard PINNs \cite{PINN_Karniadakis2021}, we choose weight $w=1.0$ for the residuals of the initial conditions $\mathcal{R}_{\text{ic}}$. Based on the standard PINN, introducing a data-fitting term alongside $\mathcal{R}_{\text{pde}}$ and $\mathcal{R}_{\text{ic}}$ leads to significant improvement. Meanwhile, Sobolev training \cite{10.5555/3294996.3295182} uses derivative labels to provide extra information. While \cite{son2021sobolevtrainingphysicsinformed} proposes incorporating higher derivatives in the PDE loss and derivative labels in boundary/initial conditions, their experiments primarily show improvement via supervised components. We investigate whether Sobolev training is applicable to the \textit{unsupervised} component of PINNs. Our results suggest it is not directly applicable here, as forcing higher-order derivatives in $\mathcal{R}_{\text{pde}}$ creates a solution search class that is too large due to the infinite number of solutions satisfying higher-order constraints.

Since our approach aims to provide data-driven alternative to solve for non-smooth solutions, we also compare our method against other variants of PINNs that improve accuracy to learn weak solutions for PDEs. We pick \textbf{W-PINNs}\cite{doi:10.1137/22M1522504}, \textbf{hp-vPINNs} \cite{kharazmi2021hp}, \textbf{SA-PINNs} \cite{McClenny2023SelfAdaptivePINN}, and \textbf{adversarial training} for PINNs \cite{wang2022is} as these methods stand out among many invariants of PINNs that address discontinuity of shocks waves and variational forms of PDEs to solve PDEs. We briefly walk through the key-features of these methods: \textbf{hp-vPINNs} utilize domain decomposition to physically isolate sharp gradients and smooth out using test functions, preventing error propagation across the solution domain. Alternatively, \textbf{W-PINNs} and \textbf{SA-PINNs} fall under the umbrella of adversarial frameworks to force network attention onto high-error regions dynamically by continuously optimizing trainable residual weights for weak-form residuals on collocation points. Adversarial training for PINNs \cite{wang2022is} extend this adversarial philosophy by abandoning fixed points entirely, utilizing gradient ascent to physically shift collocation coordinates into worst-case regions to evaluate loss in $L^{\infty}$ norm.

\begin{figure}[htbp]
    \centering
    \includegraphics[width=0.5\textwidth]{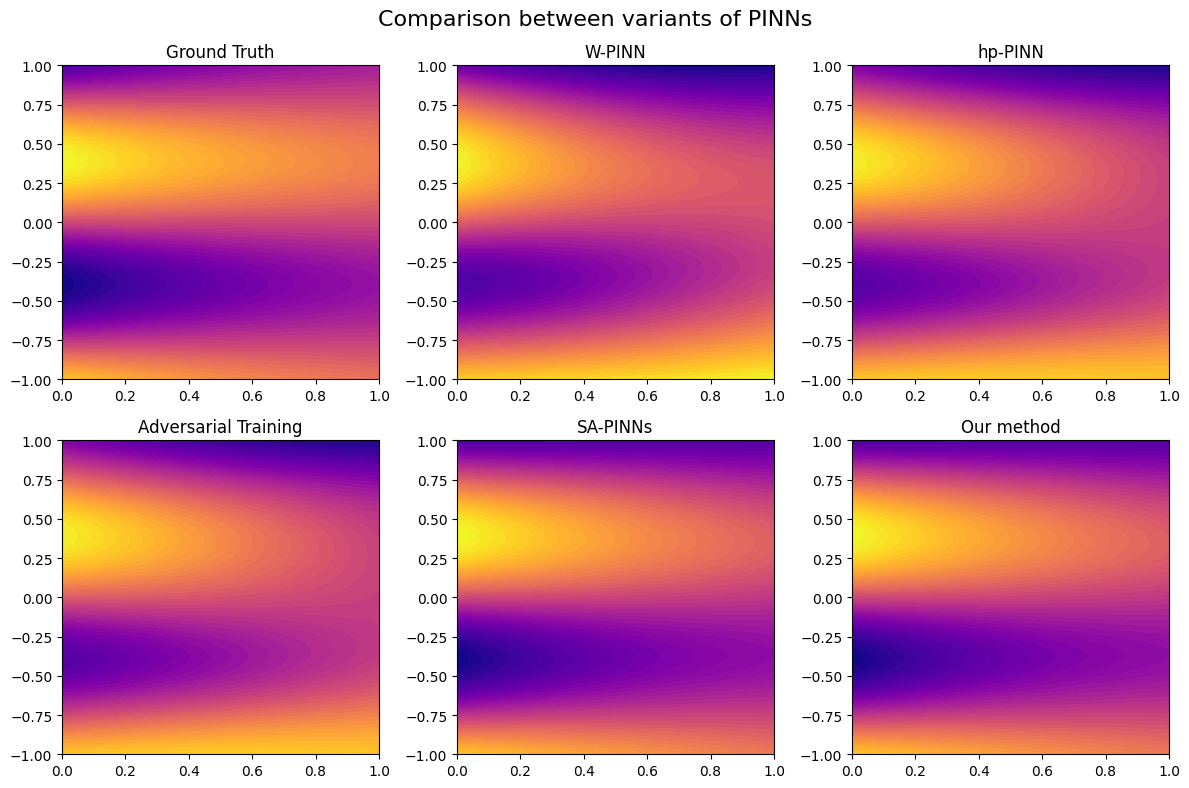}
    \caption{Learned solutions for PINN variants for 1d heat equation.}
    \includegraphics[width=0.8\textwidth]{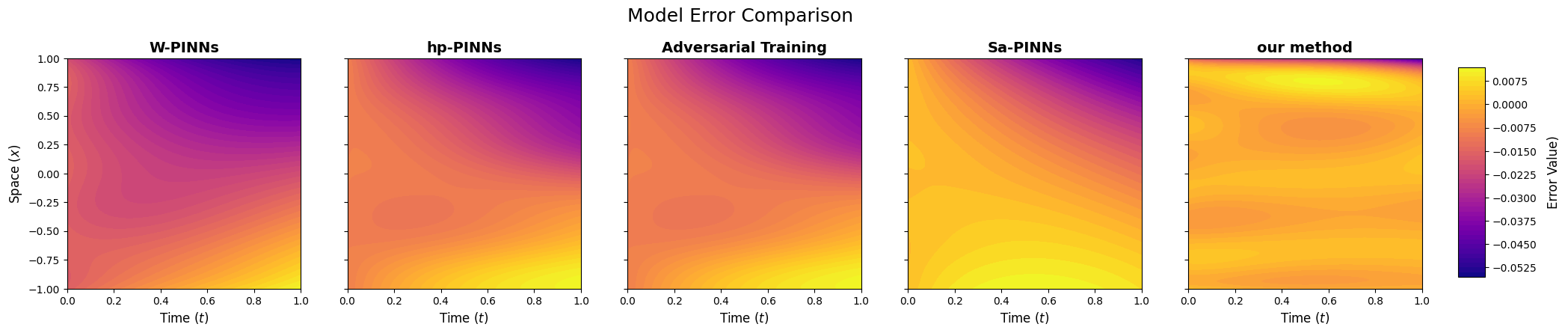}
    \caption{Error plot for PINN variants.}
    \label{fig:variants_comparison}
\end{figure}

\paragraph{Variable-coefficient and nonlinear parabolic equations.}

We next consider three linear variable-coefficient parabolic equations and
one nonlinear viscous Hamilton--Jacobi equation in 3d dimensional setting. These experiments examine
whether the implementation can accommodate diffusion, reaction, and drift
coefficients that are more general than those of the heat equation. The four equations are defined as follows.

\paragraph{LP1: Variable diffusion--reaction equation.}
\[
u_t-\partial_x\!\left(a_1(x)u_x\right)+c_1(x)u=f,
\]
where
\[
a_1(x)=1+\frac12\sin(\pi x),
\qquad
c_1(x)=\frac12+x^2.
\]

\paragraph{LP2: Space--time-dependent diffusion--reaction equation.}
\[
u_t-\partial_x\!\left(a_2(x,t)u_x\right)+c_2(x,t)u=f,
\]
where
\[
a_2(x,t)
=
1+\frac14\sin(\pi x)\cos(2\pi t),
\]
and
\[
c_2(x,t)
=
\frac12+\frac14(1+t)\cos^2(\pi x).
\]

\paragraph{LP3: Reversible variable-drift equation.}
\[
u_t
-
\mu(x)^{-1}
\partial_x\!\left(\mu(x)a_3(x)u_x\right)
+
c_3(x)u
=
f,
\]
where
\[
a_3(x)=1+\frac{x^2}{4},
\qquad
\mu(x)=e^{-x^2/4},
\qquad
c_3(x)=\frac14(1+x^2).
\]

\paragraph{HJ: Viscous quadratic Hamilton--Jacobi equation.}
\[
u_t-0.1u_{xx}+\frac12u_x^2=f.
\]

For LP1--LP3, we use the common manufactured solution
\[
u^{*}(x,t)=G(t)P(x),
\]
where
\[
G(t)
=
e^{-t}\left[1+\frac14\sin(2\pi t)\right],
\qquad
P(x)
=
(1-x^2)\left[1+\frac15\sin(\pi x)\right].
\]
The required derivatives are
\[
G'(t)
=
e^{-t}
\left[
\frac{\pi}{2}\cos(2\pi t)
-1
-\frac14\sin(2\pi t)
\right],
\]
\[
P'(x)
=
-2x\left[1+\frac15\sin(\pi x)\right]
+
\frac{\pi}{5}(1-x^2)\cos(\pi x),
\]
and
\[
P''(x)
=
-2\left[1+\frac15\sin(\pi x)\right]
-\frac{4\pi x}{5}\cos(\pi x)
-\frac{\pi^2}{5}(1-x^2)\sin(\pi x).
\]

For the one-dimensional Hamilton--Jacobi equation, we use
\[
u_{\mathrm{HJ}}^{*}(x,t)
=
\frac12(1-e^{-t})(1-x^2).
\]
The initial and boundary conditions are prescribed consistently with the
corresponding solutions.

Substituting the manufactured solutions into the equations gives
\[
\begin{aligned}
f_{\mathrm{LP1}}(x,t)
={}&
G'(t)P(x)
-
G(t)
\left[
\left(1+\frac12\sin(\pi x)\right)P''(x)
+
\frac{\pi}{2}\cos(\pi x)P'(x)
\right]
\\
&+
\left(\frac12+x^2\right)G(t)P(x),
\\[1.5mm]
f_{\mathrm{LP2}}(x,t)
={}&
G'(t)P(x)
-
G(t)
\left[
\left(
1+\frac14\sin(\pi x)\cos(2\pi t)
\right)P''(x)
\right.
\\
&\left.
\hspace{31mm}
+
\frac{\pi}{4}
\cos(\pi x)\cos(2\pi t)P'(x)
\right]
\\
&+
\left[
\frac12+\frac14(1+t)\cos^2(\pi x)
\right]G(t)P(x),
\\[1.5mm]
f_{\mathrm{LP3}}(x,t)
={}&
G'(t)P(x)
-
G(t)
\left[
\left(1+\frac{x^2}{4}\right)P''(x)
-
\frac{x^3}{8}P'(x)
\right]
\\
&+
\frac14(1+x^2)G(t)P(x),
\\[1.5mm]
f_{\mathrm{HJ}}(x,t)
={}&
\frac12e^{-t}(1-x^2)
+
0.1(1-e^{-t})
+
\frac12(1-e^{-t})^2x^2.
\end{aligned}
\]

\paragraph{High-dimensional viscous Hamilton--Jacobi equations.}

To study \textbf{dimensional scalability} on a controlled nonlinear family, we consider
\[
u_t-0.1\Delta u+\frac12\|\nabla u\|_2^2=f
\]
on
\[
\Omega_d\times[0,1]=[-1,1]^d\times[0,1],
\qquad
d\in\{3,10,20,50\}.
\]
We use the manufactured solution
\[
u_d^{*}(x,t)
=
\frac12(1-e^{-t})
\left(
d-\|x\|_2^2
\right),
\]
which gives
\[
f_d(x,t)
=
\frac12e^{-t}
\left(
d-\|x\|_2^2
\right)
+
0.1d(1-e^{-t})
+
\frac12(1-e^{-t})^2\|x\|_2^2.
\]
The initial condition is zero, and the boundary condition is prescribed by
the trace of $u_d^{*}$ on $\partial\Omega_d$.

\subsection{Comparison with Other Methods on LP1--LP3 and HJB}
\label{sec:lp-method-comparison}
\label{sec:pde-method-comparison}

For comparison baselines, we consider \textbf{W-PINNs} \citep{doi:10.1137/22M1522504}, \textbf{hp-vPINNs} \citep{kharazmi2021hp}, \textbf{SA-PINNs} \citep{McClenny2023SelfAdaptivePINN}, and \textbf{adversarial training} for PINNs \citep{wang2022is} as these methods stand out among many invariants of PINNs that address discontinuity of shocks waves and variational forms to solve PDEs. Figures~\ref{fig:lp1-method-comparison}--\ref{fig:lp3-method-comparison} compare our method with five PINN-based baselines on 3d LP1--LP3 benchmarks. Our method achieves the lowest mean error in both metrics across parabolic equations. Notably, on LP3, it reduces test MSE and relative $L^2$ error by 27.2\% and 16.5\%, respectively, compared with the strongest baseline, Adversarial PINN. 

\begin{figure}[!htbp]
    \centering
    \includegraphics[width=0.95\linewidth]{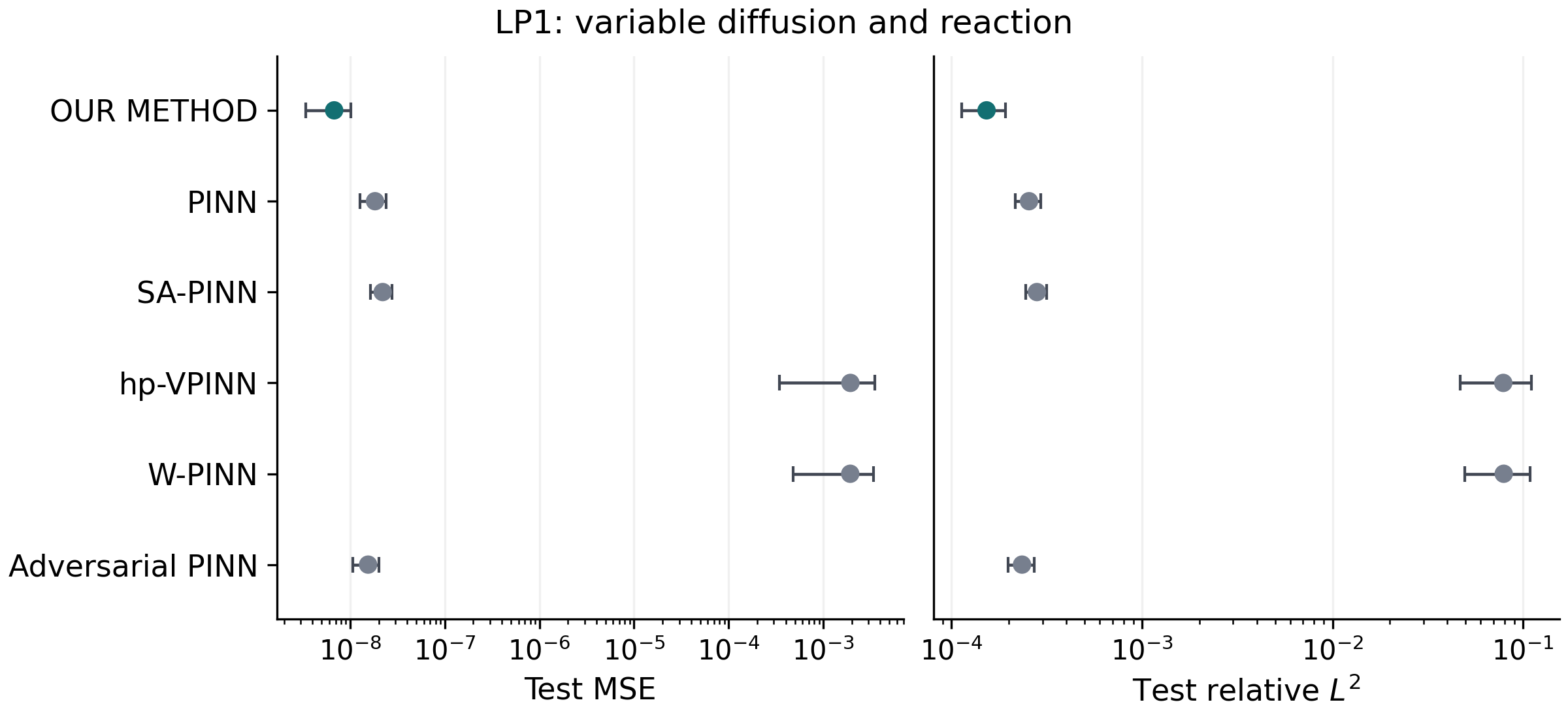}
    \caption{Comparison with other methods on LP1. Points show mean test errors and bars indicate one sample standard deviation over five seeds.}
    \label{fig:lp1-method-comparison}
\end{figure}

\begin{figure}[!htbp]
    \centering
    \includegraphics[width=0.95\linewidth]{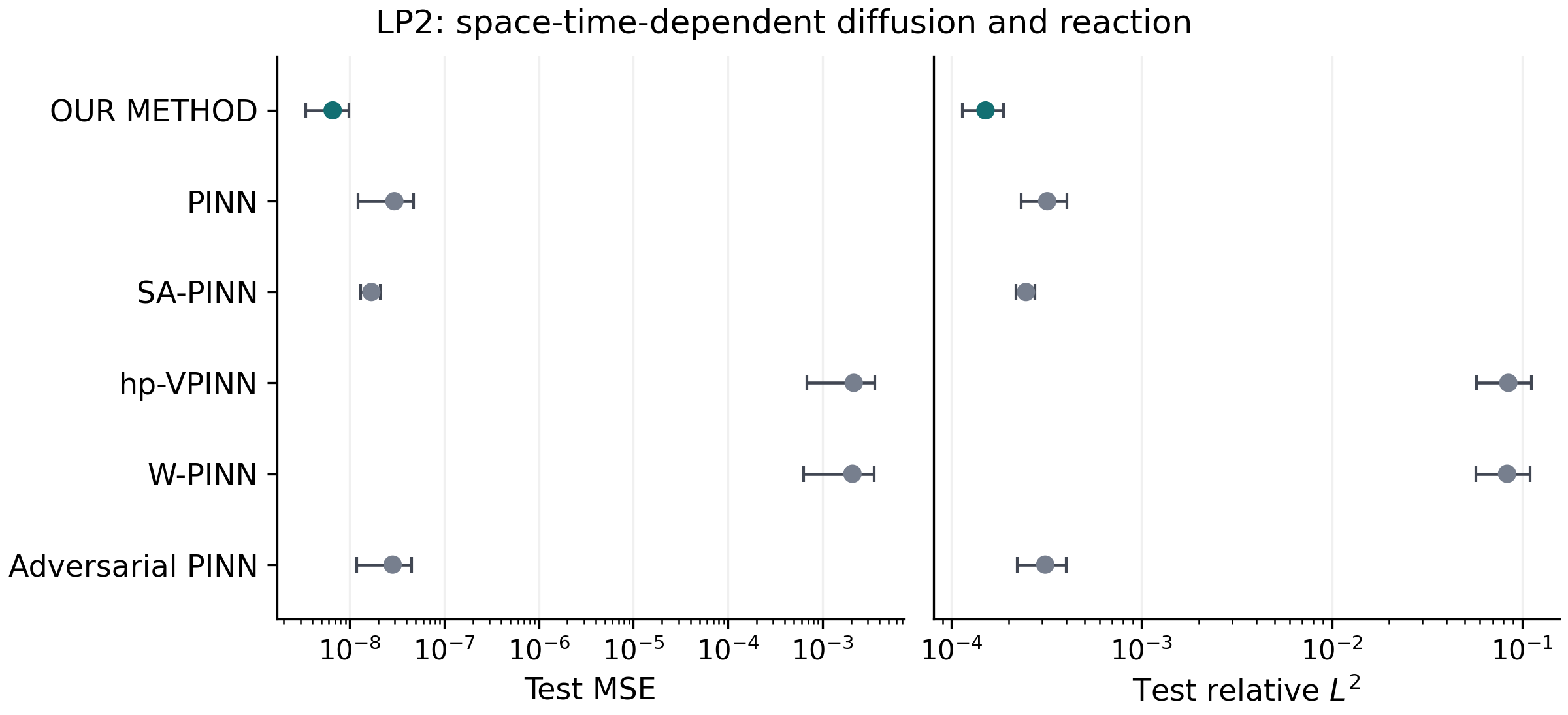}
    \caption{Comparison with other methods on LP2. Points show mean test errors and bars indicate one sample standard deviation over five seeds.}
    \label{fig:lp2-method-comparison}
\end{figure}

\begin{figure}[!htbp]
    \centering
    \includegraphics[width=0.95\linewidth]{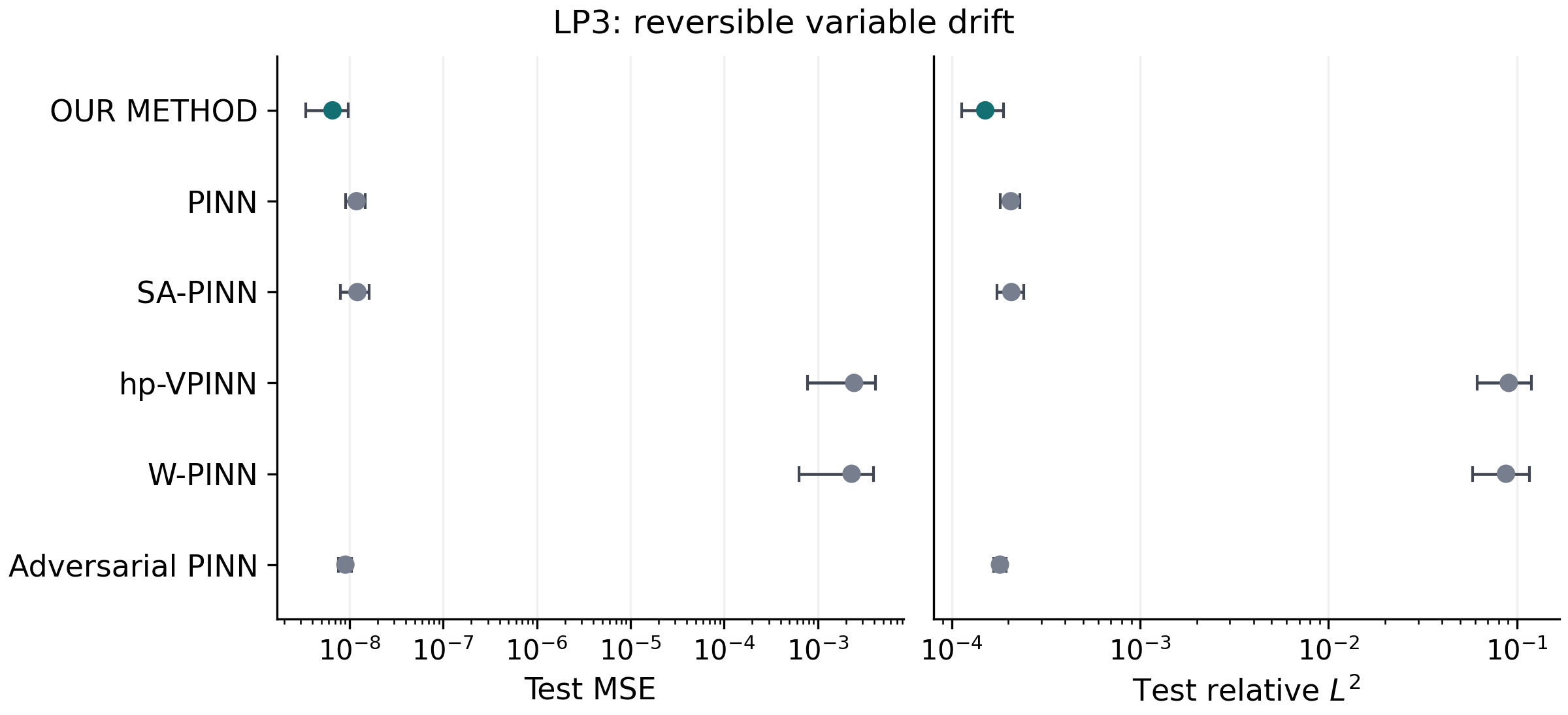}
    \caption{Comparison with other methods on LP3. Points show mean test errors and bars indicate one sample standard deviation over five seeds.}
    \label{fig:lp3-method-comparison}
\end{figure}
\newpage
\begingroup
\raggedbottom
\paragraph{ Viscous Hamilton Jacobi Equation in various dimensions.}
Figures~\ref{fig:hjb-d1-method-comparison}--\ref{fig:hjb-d50-method-comparison}
show the same six methods for the viscous Hamilton--Jacobi family
(HJB, denoted HJ above) in $d\in\{1,3,10,20,50\}$. 

Figures~\ref{fig:lp3-method-comparison}--\ref{fig:hjb-d50-method-comparison}
compare our method with five PINN-based baselines using test MSE and
relative $L^2$ error. Results are averaged over ten random seeds, with error bars indicating one sample standard deviation. Our method achieves the lowest mean error in all five HJB. Across $d\in\{1,3,10,20,50\}$, our method reduces test MSE by 79.5\%--89.4\% and relative $L^2$ error by 55.3\%--67.4\% compared with the best competing method in each configuration. At $d=50$, it achieves a relative $L^2$ error of 0.964\%, compared with 2.959\% for SA-PINN, while reducing test MSE by 89.4\%. Although its errors increase from $d=3$ to $d=50$, our method maintains an accuracy advantage throughout the dimensions. Besides, our results show all PINNs baselines and our method suffer from the curse of dimensionality.

\begin{center}
\begin{minipage}{0.95\textwidth}
\centering
\includegraphics[width=\linewidth]{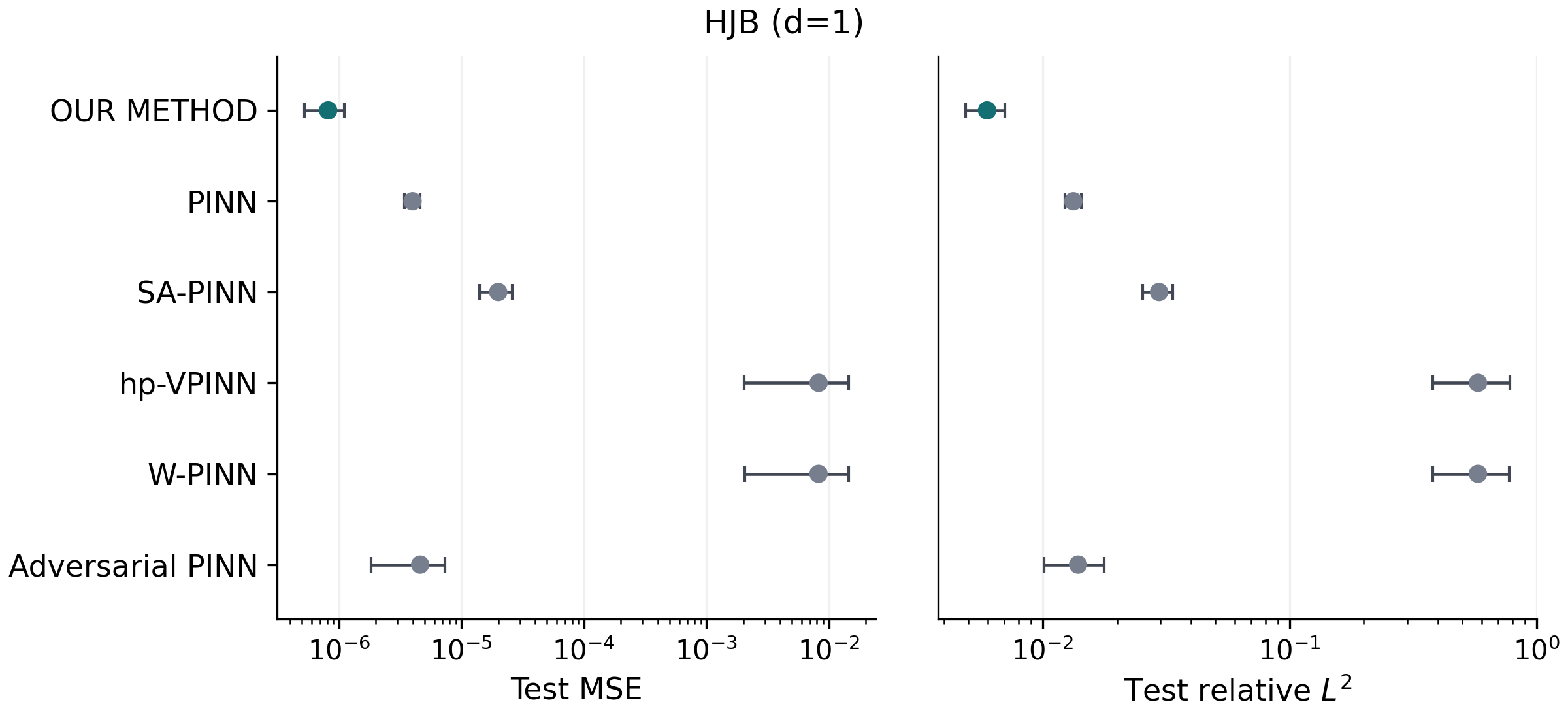}
\captionof{figure}{Comparison with other methods on HJB in $d=1$. Points show mean test errors and bars indicate one sample standard deviation over ten seeds.}
\label{fig:hjb-d1-method-comparison}
\end{minipage}
\end{center}

\begin{center}
\begin{minipage}{0.95\textwidth}
\centering
\includegraphics[width=\linewidth]{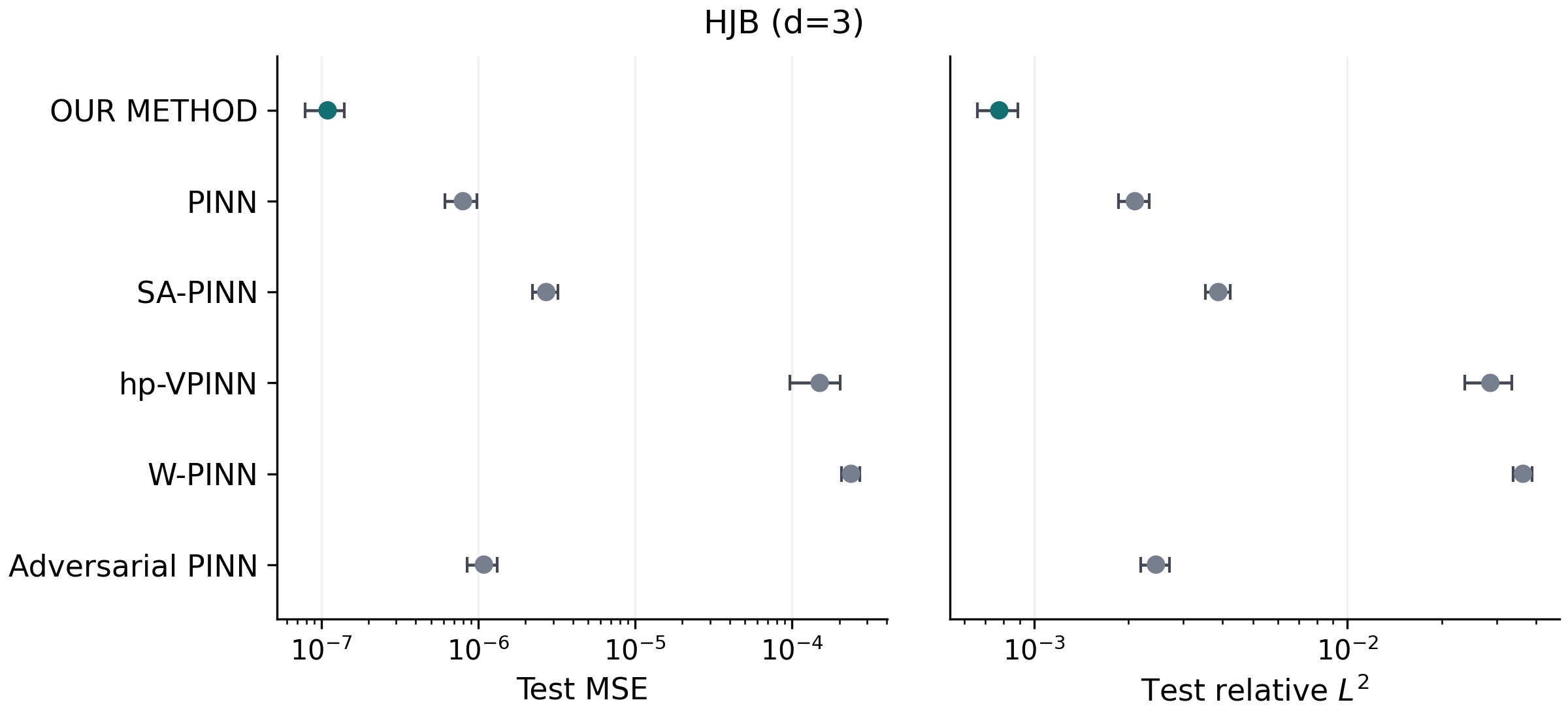}
\captionof{figure}{Comparison with other methods on HJB in $d=3$. Points show mean test errors and bars indicate one sample standard deviation over ten seeds.}
\label{fig:hjb-d3-method-comparison}
\end{minipage}
\end{center}

\begin{center}
\begin{minipage}{0.95\textwidth}
\centering
\includegraphics[width=\linewidth]{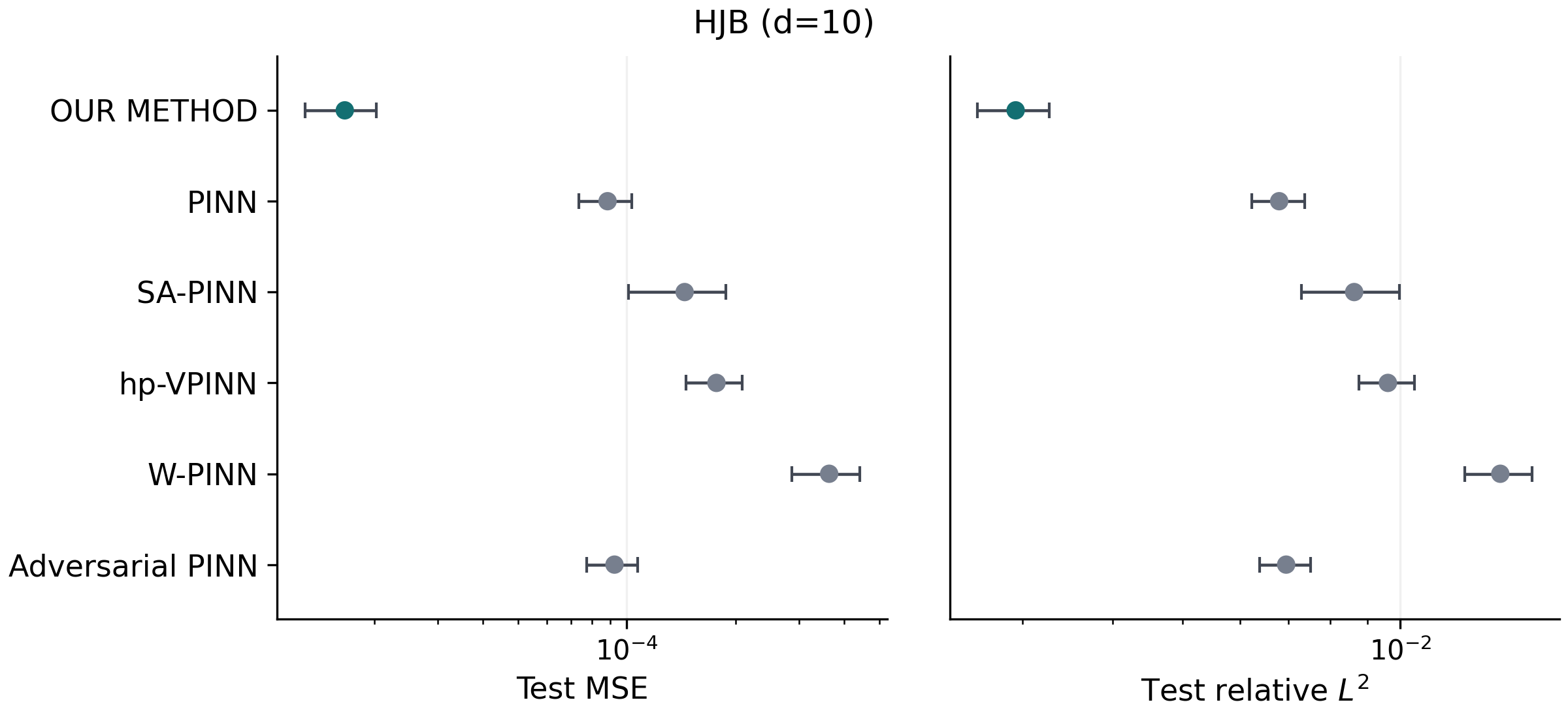}
\captionof{figure}{Comparison with other methods on HJB in $d=10$. Points show mean test errors and bars indicate one sample standard deviation over ten seeds.}
\label{fig:hjb-d10-method-comparison}
\end{minipage}
\end{center}

\begin{center}
\begin{minipage}{0.95\textwidth}
\centering
\includegraphics[width=\linewidth]{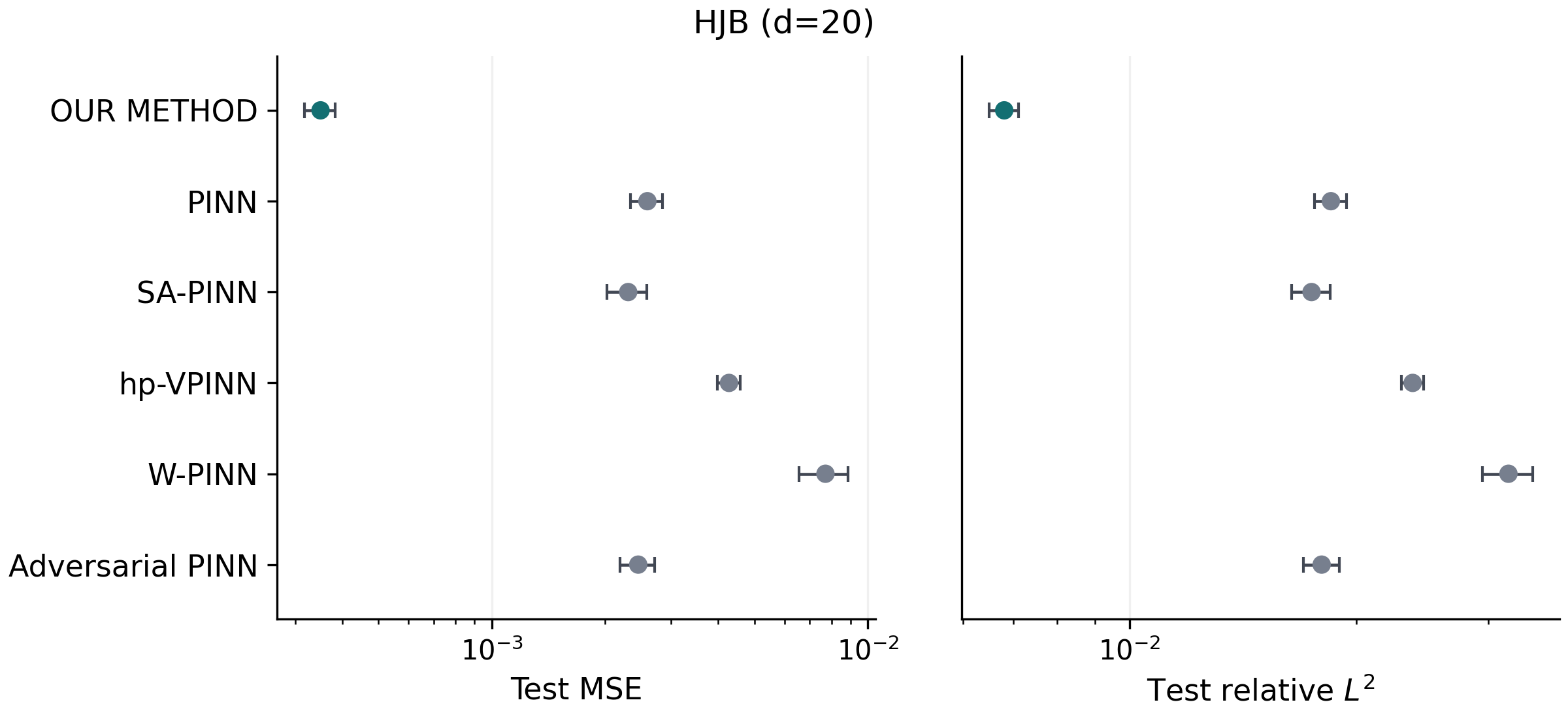}
\captionof{figure}{Comparison with other methods on HJB in $d=20$. Points show mean test errors and bars indicate one sample standard deviation over ten seeds.}
\label{fig:hjb-d20-method-comparison}
\end{minipage}
\end{center}

\begin{center}
\begin{minipage}{0.95\textwidth}
\centering
\includegraphics[width=\linewidth]{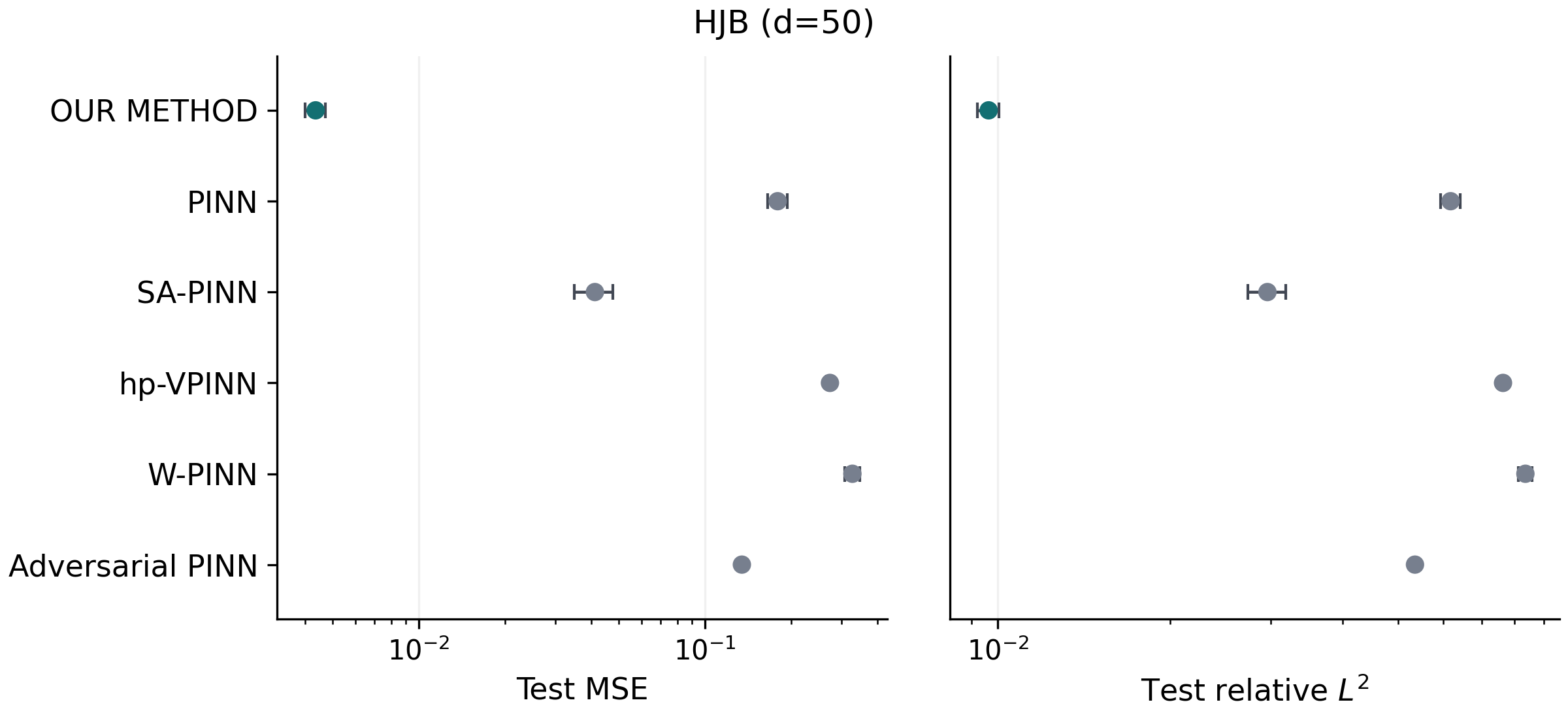}
\captionof{figure}{Comparison with other methods on HJB in $d=50$. Points show mean test errors and bars indicate one sample standard deviation over ten seeds.}
\label{fig:hjb-d50-method-comparison}
\end{minipage}
\end{center}

\begingroup
\raggedbottom
\subsection{Regularization Improvement for Some Parabolic PDEs}
\label{sec:controlled-parabolic-regularization}

We examine regularization effects on heat equations and other parabolic equations LP1-LP3. We compare the purely Sobolev fitting term $\widehat D_{H_t^1H_x^1}$ with the regularized term $\widehat D_{H_t^1H_x^1}+\beta(I_\epsilon+M_\epsilon)$, with the optimal selected epsilon-beta pair. Our first example includes fixed-size comparison using fixed regularization $\beta$, and then we will exam how unregularized estimator versus regularized estimator using fixed regularization and adaptive regularization across various sample size.
\paragraph{Fixed-size paired comparison.}
We compare the two regularized and unregularized methods using $n=128$ and $m=128$ space time sampling locations and forcing function with $5\%$ noise level.
\begin{center}
\begin{minipage}{\linewidth}
\centering
\small
\setlength{\tabcolsep}{3.3pt}
\begin{tabular}{lrrr}
\toprule
Example
& \shortstack{Sobolev $H^1$\\MSE}
& \shortstack{Regularized Sobolev $H^1$\\MSE}
& \shortstack{MSE gain\\(\%)} \\
\midrule
Heat & $(4.44\pm3.65)\times10^{-6}$
     & $(9.45\pm7.28)\times10^{-7}$ & 78.7 \\
LP1  & $(1.89\pm2.09)\times10^{-5}$
     & $(1.79\pm1.22)\times10^{-6}$ & 90.5 \\
LP2  & $(3.21\pm2.97)\times10^{-7}$
     & $(3.32\pm1.93)\times10^{-8}$ & 89.7 \\
LP3  & $(1.40\pm1.13)\times10^{-7}$
     & $(3.21\pm1.72)\times10^{-9}$ & 97.7 \\
\bottomrule
\end{tabular}
\end{minipage}
\end{center}
\paragraph{Fixed-size comparison.}
At a fixed sample size of $n=128$ and $m=128$ and a noise level of $5\%$,
we compare the mean MSE over 100 runs.
As shown in Figure~\ref{fig:positive-fixed-comparison},
regularization reduces the mean MSE by $78.7\%$, $90.5\%$,
$89.7\%$, and $97.7\%$ for Heat, LP1, LP2, and LP3,
respectively, relative to the unregularized Sobolev baseline.

\begin{figure}[htbp]
    \centering
    \includegraphics[width=0.7\linewidth]
    {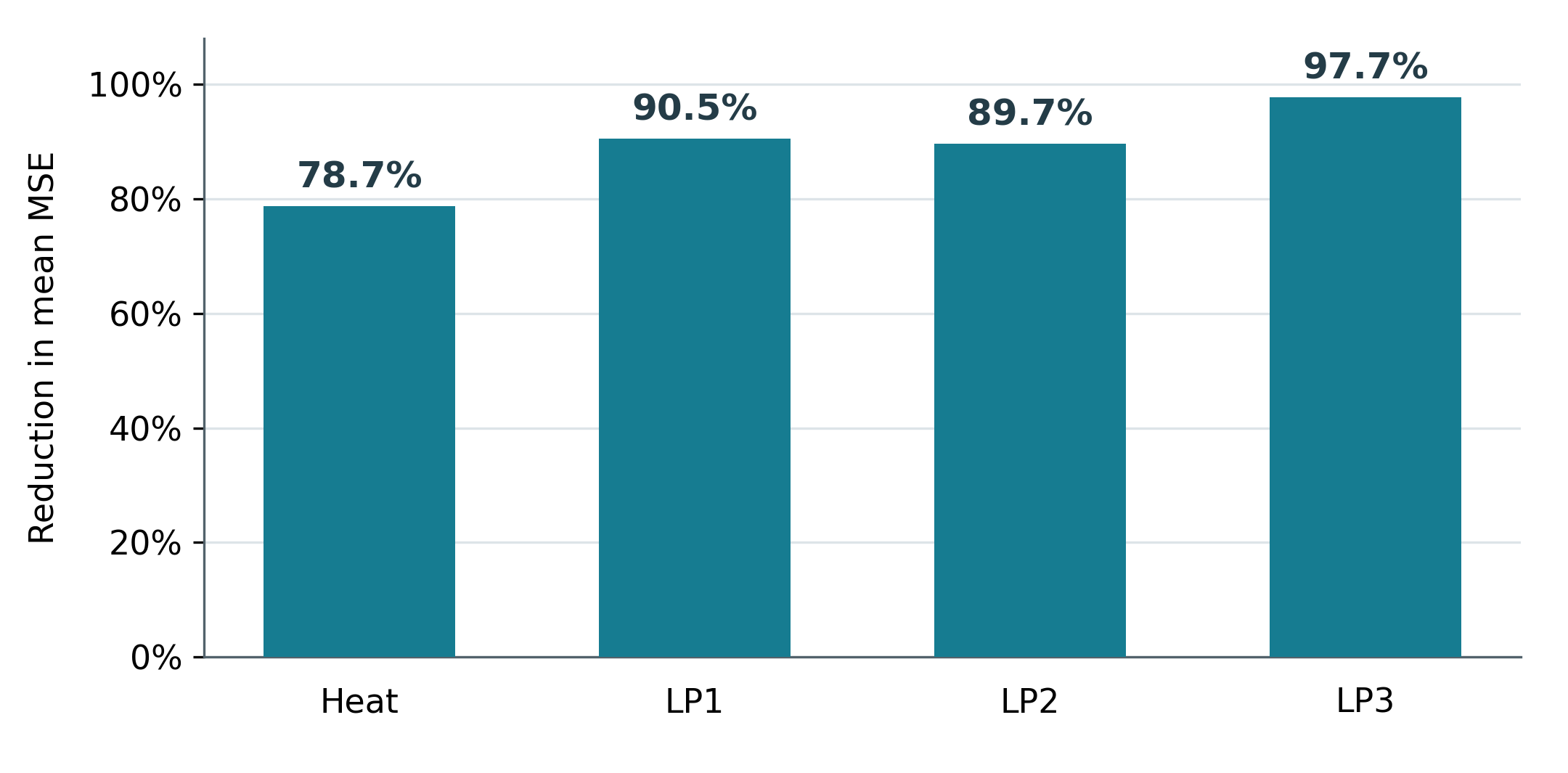}
    \caption{MSE reduction from regularization for parabolic
    equations at $n,m=128$ and $5\%$ noise.}
    \label{fig:positive-fixed-comparison}
\end{figure}

We next evaluate generalization gap over various space-temporal temporal size $N=64,128,256,512,1024,2048,4096,16{,}384,65{,}536,262{,}144,10^6$.

\paragraph{Generalization Error}
We define $E_{\rm test}$ be testing error and
$E_{\rm train}$ training error .
We report absolute generalization gap as 
\[
 G_{generalization}=|E_{\rm test}-E_{\rm train}|
\]
Figures~\ref{fig:positive-sample-test-error}
and~\ref{fig:positive-sample-clean-gap} compare the test MSE
and the mean absolute clean-reference generalization gap,
respectively, as the sample size increases. In particular, we implement unregularized baseline, regularization using fixed $\beta$ and adaptive $\beta$ proposed in our framework. For heat and LP1--LP3, regularization consistently reduces both test mse and generalization gap relative to the unregularized Sobolev baseline across various validation sample sizes. What we found is that constant $\beta$ only improves the rate by only constant, while adaptive regularization as proposed in our scheme results in change of rate order.

\paragraph{Relation to the theoretical parameter choice.}
Under the manuscript's rescaling $\lambda=\beta/N$, our choice for $\beta:= (1/N)^{\frac{1}{2}}$ differs only a bit from
the sample-dependent optimal choice in the supercritical regime
of Theorem~\ref{thm_3.1}, which gives
\[
\beta^*(N)=N\lambda^*(N)
=
e^{T/\epsilon}
N^{\frac{d+2}{4s+d}}
\bigl[\log(1+C_2N)\bigr]^{\frac{2s-1}{4s+d}}.
\]
We admit that our experiments do not exactly verify theoretically optimal regularization, and we notice sample-size dependent tuning regularization parameter in general improve the rate of convergence.

\begin{center}
\begin{minipage}{.8\textwidth}
\includegraphics[width=\linewidth]{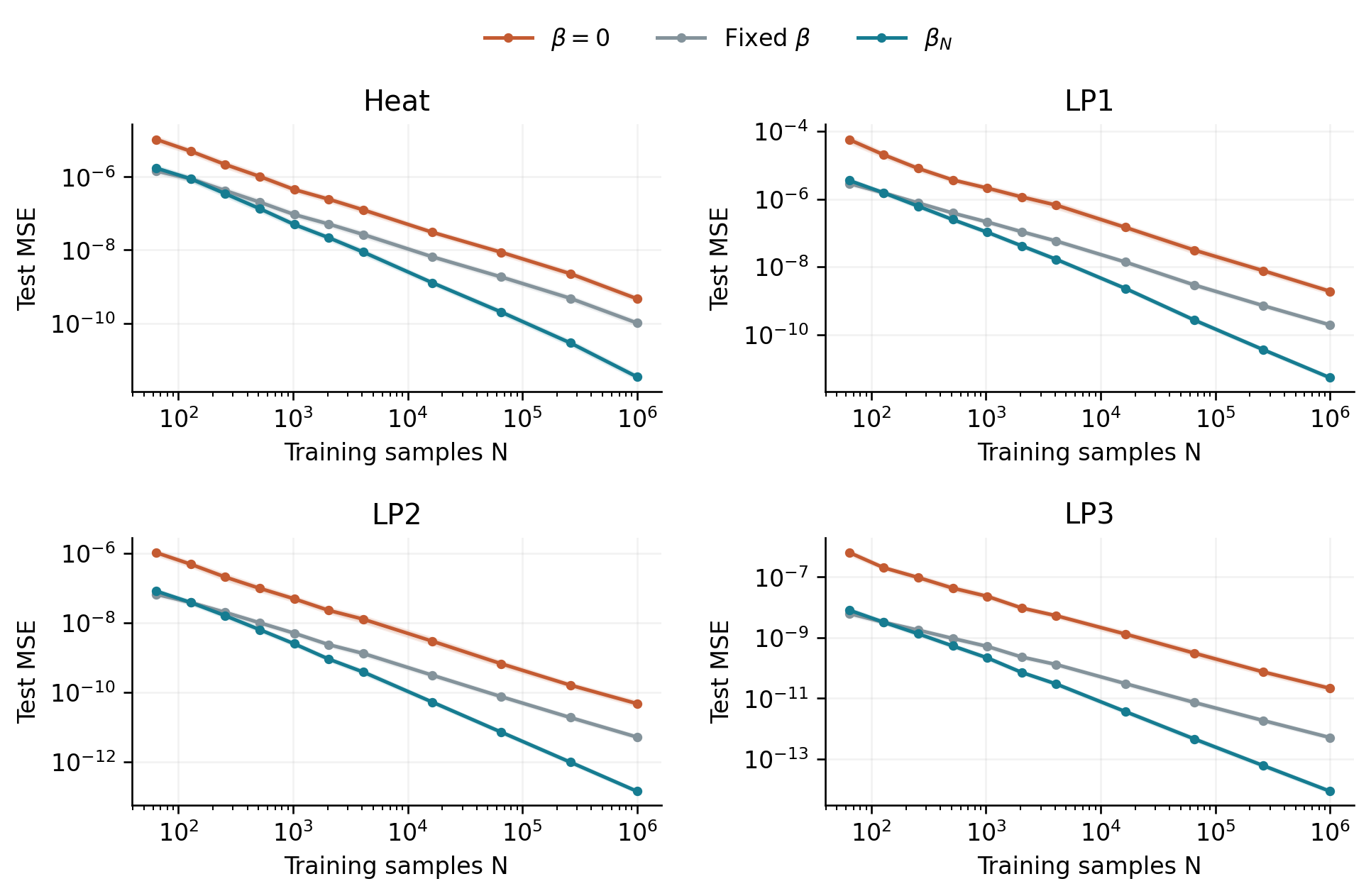}
\captionof{figure}{Test MSE versus sample size for parabolic equations with and without regularization.}
\label{fig:positive-sample-test-error}
\end{minipage}
\end{center}
\begin{center}
\begin{minipage}{.8\textwidth}
\includegraphics[width=\linewidth]{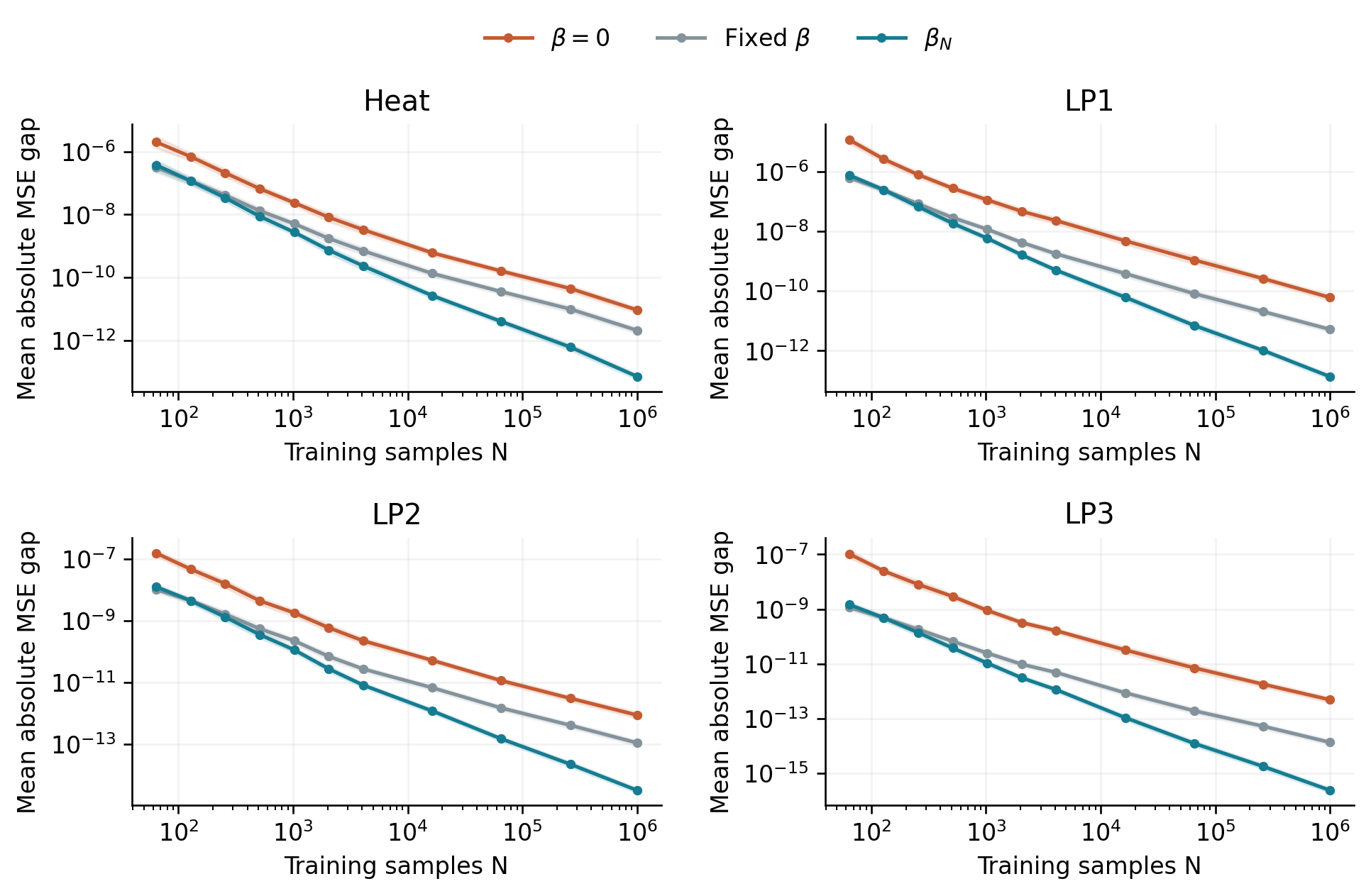}
\captionof{figure}{Mean absolute generalization gap versus sample size for parabolic equations with and without regularization.}
\label{fig:positive-sample-clean-gap}
\end{minipage}
\end{center}

\subsubsection{Noise ablation studies}
\label{A.8.2}

\paragraph{Noise in the solution and forcing labels.}
\label{app:noise-ablation}

We now keep the parabolic equations in previous experiments, network architecture, sampling
procedure, and optimization settings fixed, and vary only the observation
noise. We distinguish two observation models:

\begin{enumerate}
    \item additive noise in the solution and forcing labels, while keeping
    the derivative labels exact; and
    \item additive perturbations applied directly to the derivative labels,
    while keeping the solution labels exact.
\end{enumerate}

The second experiment is not equivalent to differentiating noisy solution
observations. This distinction is discussed explicitly below.

\begin{table}[htbp]
\centering
\caption{Median test relative $L^2$ error over ten paired seeds for the
heatmap-selected mixed-$H^1$ estimator under additive solution and forcing
noise. Derivative labels and validation/test values are exact.}
\label{tab:solution-forcing-noise}
\small
\setlength{\tabcolsep}{6pt}
\begin{tabular}{ccccc}
\toprule
Equation & $\rho_u$ & $\rho_f=0$ & $\rho_f=0.05$ & $\rho_f=0.10$ \\
\midrule
LP1 & $0$ & $1.616\times10^{-4}$ & $1.609\times10^{-4}$ & $1.608\times10^{-4}$ \\
LP1 & $0.05$ & $1.644\times10^{-4}$ & $1.641\times10^{-4}$ & $1.636\times10^{-4}$ \\
LP1 & $0.1$ & $2.070\times10^{-4}$ & $2.073\times10^{-4}$ & $2.072\times10^{-4}$ \\
\midrule
LP2 & $0$ & $1.604\times10^{-4}$ & $1.593\times10^{-4}$ & $1.597\times10^{-4}$ \\
LP2 & $0.05$ & $1.659\times10^{-4}$ & $1.657\times10^{-4}$ & $1.654\times10^{-4}$ \\
LP2 & $0.1$ & $2.091\times10^{-4}$ & $2.074\times10^{-4}$ & $2.067\times10^{-4}$ \\
\midrule
LP3 & $0$ & $1.593\times10^{-4}$ & $1.591\times10^{-4}$ & $1.595\times10^{-4}$ \\
LP3 & $0.05$ & $1.660\times10^{-4}$ & $1.666\times10^{-4}$ & $1.663\times10^{-4}$ \\
LP3 & $0.1$ & $2.097\times10^{-4}$ & $2.094\times10^{-4}$ & $2.090\times10^{-4}$ \\
\midrule
HJ & $0$ & $5.569\times10^{-3}$ & $6.127\times10^{-3}$ & $5.839\times10^{-3}$ \\
HJ & $0.05$ & $6.434\times10^{-3}$ & $6.075\times10^{-3}$ & $5.797\times10^{-3}$ \\
HJ & $0.1$ & $8.666\times10^{-3}$ & $1.015\times10^{-2}$ & $9.574\times10^{-3}$ \\
\bottomrule
\end{tabular}
\end{table}

Table~\ref{tab:joint-noise} extracts the diagonal entries of
Table~\ref{tab:solution-forcing-noise}. For median error $E_{\rho_u,\rho_f}$,
the change relative to the clean result is
$\Delta_\rho=100(E_{\rho,\rho}/E_{0,0}-1)\%$.

\begin{table}[htbp]
\centering
\caption{Fresh reruns of the clean and jointly noisy conditions with
ten paired seeds. These values are the corresponding diagonal entries of
Table~\ref{tab:solution-forcing-noise}; percentage changes use unrounded medians.}
\label{tab:joint-noise}
\small
\setlength{\tabcolsep}{4pt}
\begin{tabular}{cccccc}
\toprule
Equation & $(0,0)$ & $(0.05,0.05)$ & $(0.10,0.10)$ & $\Delta_{0.05}$ & $\Delta_{0.10}$ \\
\midrule
LP1 & $1.616\times10^{-4}$ & $1.641\times10^{-4}$ & $2.072\times10^{-4}$ & $+1.50\%$ & $+28.17\%$ \\
LP2 & $1.604\times10^{-4}$ & $1.657\times10^{-4}$ & $2.067\times10^{-4}$ & $+3.28\%$ & $+28.85\%$ \\
LP3 & $1.593\times10^{-4}$ & $1.666\times10^{-4}$ & $2.090\times10^{-4}$ & $+4.59\%$ & $+31.22\%$ \\
HJ & $5.569\times10^{-3}$ & $6.075\times10^{-3}$ & $9.574\times10^{-3}$ & $+9.09\%$ & $+71.90\%$ \\
\bottomrule
\end{tabular}
\end{table}

Across all nine noise settings, the median relative error for LP1--LP3
ranges from $0.0159\%$ to $0.0210\%$.
Across all nine noise settings, the median relative error for HJ
ranges from $0.5569\%$ to $1.0154\%$.
The largest signed relative change along the jointly noisy diagonal is
$+71.90\%$ for HJ at $\rho_u=\rho_f=0.10$.
The absolute errors and relative changes should be considered together,
because a small clean error can produce a large percentage increase.
\clearpage
\section{Parameter Sensitivity Analysis}
\label{sec:parameter-sensitivity-analysis}
In this section, we present careful analysis on the relation between regularization parameter $\beta$ and PDE regularization parameter $\epsilon$ and total sample size $N$ to examine how $N$, $\epsilon$, and $\beta$ jointly affect the benefit of  energy regularization over mixed-$H^1$
data fitting term. Figures~\ref{fig:psa-lp1-accuracy}, \ref{fig:psa-lp2-accuracy},
\ref{fig:psa-lp3-accuracy}, and~\ref{fig:psa-hjb3-accuracy} show relative
$L^2$ error and test MSE; the companion figures report the signed
generalization gap and the test-MSE ratio to the fitting-only baseline.
The latter ratio makes improvements over $\beta=0$ directly visible.

\paragraph{Interaction between $\epsilon$ and $\beta$.}
In the second heatmaps of Figures~\ref{fig:psa-lp2-generalization}
and~\ref{fig:psa-lp3-generalization}, the rows with very small $\beta$
remain close to a ratio of one, while intermediate weights produce
localized improvements. For example, in the upper noise row at
$N=100$, $\epsilon=0.5$, and $\beta=10^{-2}$, the LP2 and LP3 ratios
are approximately $0.63$ and $0.64$, corresponding to test-MSE reductions
of $37.3\%$ and $36.2\%$. The corresponding relative-$L^2$ panels in
Figures~\ref{fig:psa-lp2-accuracy} and~\ref{fig:psa-lp3-accuracy}
also improve, by $17.1\%$ and $16.8\%$.
At the same $(N,\beta)$, the LP2 ratio is approximately one when
$\epsilon=0.01$, whereas the $\beta=1$ rows show substantial error
increases throughout both LP ratio heatmaps. Thus, the benefit depends
on the combination of $\epsilon$ and $\beta$; increasing the energy weight alone does not ensure improvement. This matches with our upper bound where $\beta$ depends on both $epsilon$ and $N$.

\paragraph{Dependence on sample size and noise location.}
Reading the $\beta=0$ rows from left to right in
Figures~\ref{fig:psa-lp1-accuracy}, \ref{fig:psa-lp2-accuracy},
\ref{fig:psa-lp3-accuracy}, and~\ref{fig:psa-hjb3-accuracy}
shows substantially smaller baseline errors at the largest sample sizes.
The orange outlines at $N=10^6$ in
Figures~\ref{fig:psa-lp2-generalization} and~\ref{fig:psa-lp3-generalization}
select $\beta=0$ for every displayed $\epsilon$ and both noise settings,
indicating  additional improvement for these LP examples at large $N$.
For HJB-3D, the upper noise row of the ratio heatmap in
Figure~\ref{fig:psa-hjb3-generalization} retains a ratio of about $0.92$
at $N=10^6$, $\epsilon=0.5$, and $\beta=0.1$.
The two noise rows also differ: at $N=100$ with the same
$(\epsilon,\beta)=(0.5,0.1)$, the HJB ratio is $0.87$ for noisy $f$
but $1.98$ for noisy $u$. The corresponding accuracy panels in
Figure~\ref{fig:psa-hjb3-accuracy} also shows this reversal relation.
These observations support adapting the parameters to the equation,
sample size, and noise location.

\paragraph{Prediction accuracy and generalization gap.}
The first heatmap in Figure~\ref{fig:psa-lp1-generalization} shows why
the gap must be read together with the MSE ratio. In its upper noise
row at $N=100$ and $\epsilon=0.2$, increasing $\beta$ from $0$ to $0.1$
changes the mean signed gap from approximately $1.8\times10^{-8}$
to $-9.4\times10^{-7}$, while the companion ratio rises to about $24.2$.
A more negative gap therefore need not indicate better prediction or
a smaller train--test discrepancy. In contrast, the upper noise row
of Figure~\ref{fig:psa-hjb3-generalization} at $N=100$,
$\epsilon=0.5$, and $\beta=0.1$ shows both a lower test-MSE ratio
and a positive mean gap closer to zero than the baseline.

\paragraph{Consistency with the theoretical discussion.}
The improved regions in Figures~\ref{fig:psa-lp2-generalization}
and~\ref{fig:psa-lp3-generalization} provide empirical support for
finite-sample improvement from fine-tuned weighted energy guidance. Their limited gains at the largest $N$ are qualitatively compatible with the subcritical discussion following Theorem~\ref{thm_3.1}, where energy guidance need not improve the leading-order rate for sufficiently smooth targets. The HJB observations in
Figure~\ref{fig:psa-hjb3-generalization} illustrate the empirical
extension beyond the heat-equation setting of that theorem.
This agreement remains qualitative: the heatmaps measure value-based
errors and a clean value-MSE gap for neural approximations, whereas
the theory concerns its specified estimator and the full regularized
loss.

\clearpage
\begin{figure}[p]
\centering
\includegraphics[width=\linewidth]{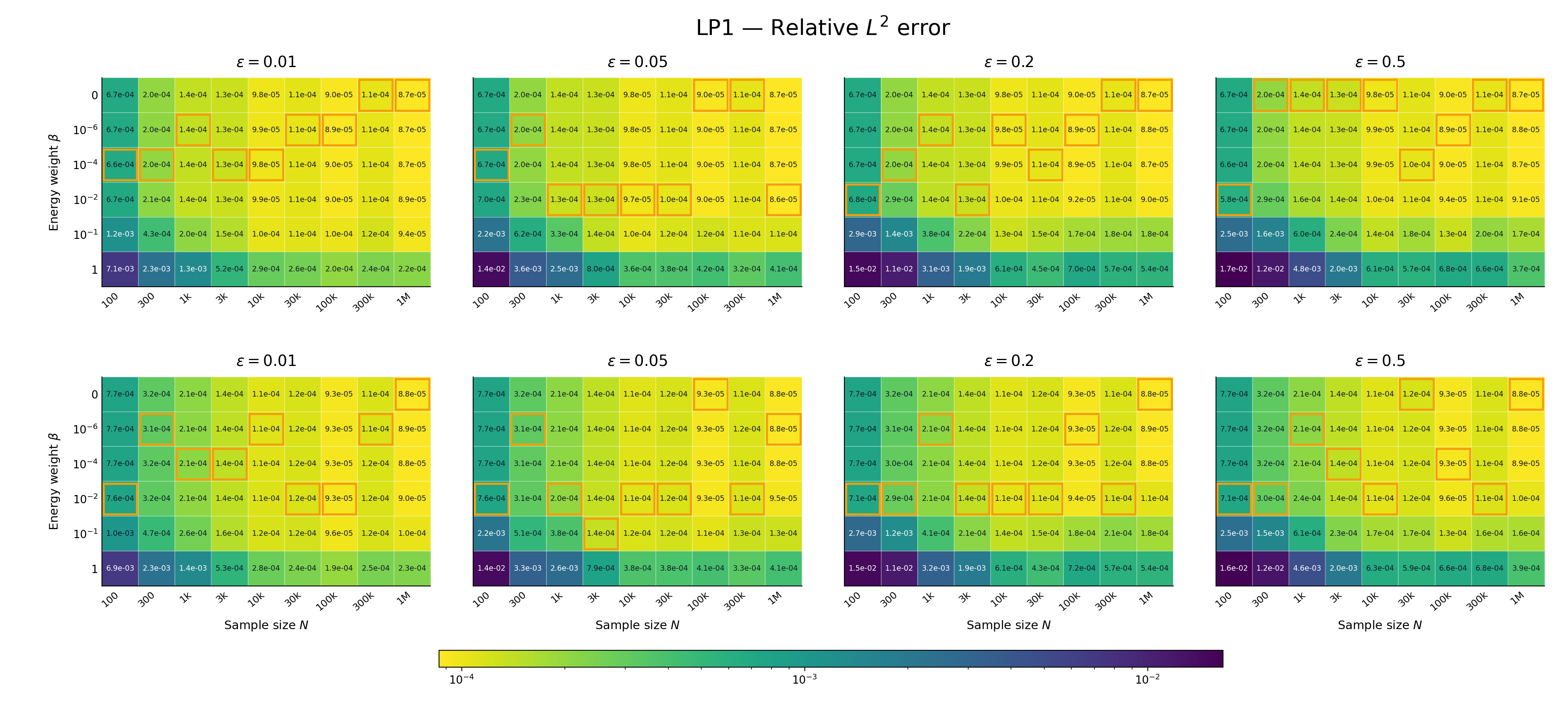}
\par\vspace{0.9em}
\includegraphics[width=\linewidth]{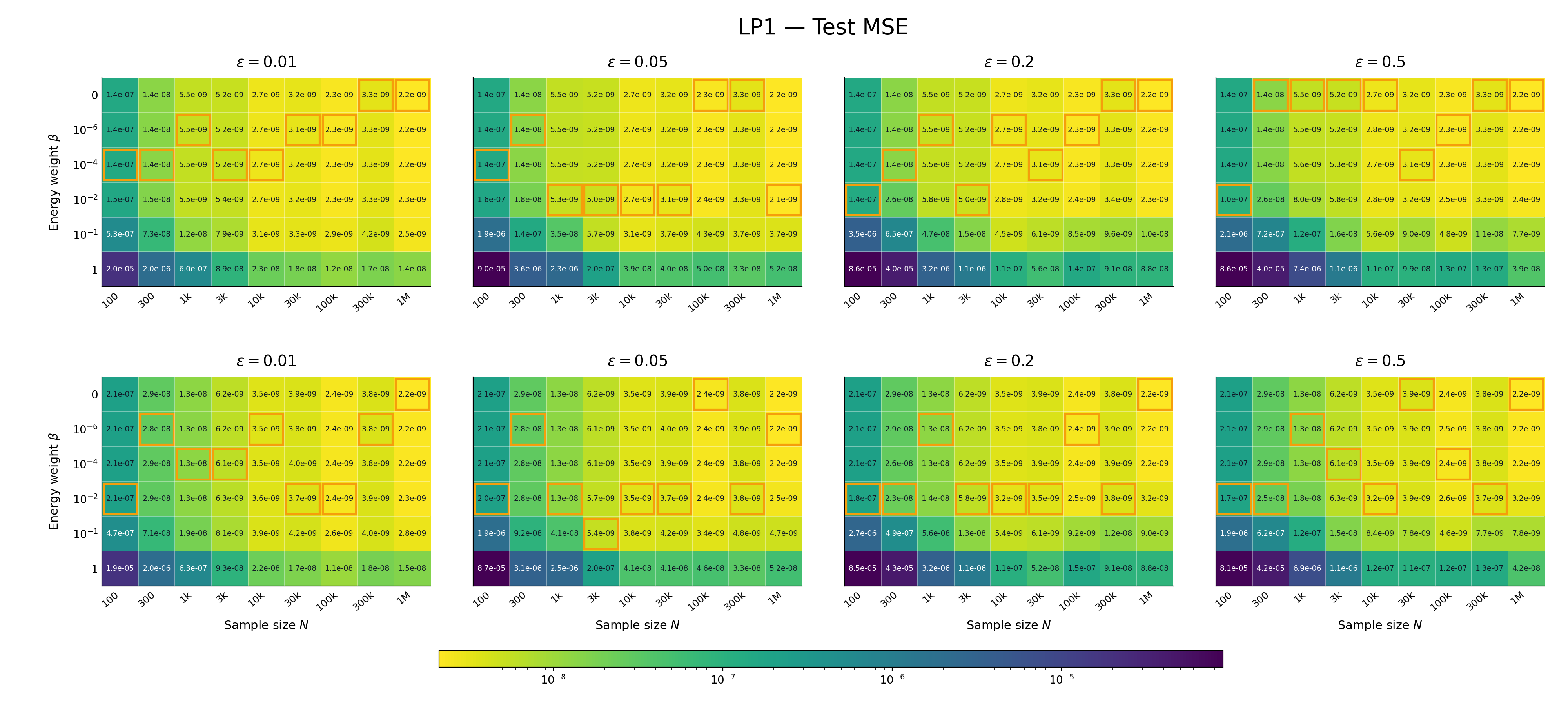}
\caption{LP1: relative $L^2$ error (first heatmap) and test MSE (second heatmap). Within each heatmap, the upper row has noisy $f$ and clean $u$; the lower row has clean $f$ and noisy $u$. Orange outlines mark the validation-selected $\beta$ for each fixed $(\epsilon,N)$.}
\label{fig:psa-lp1-accuracy}
\end{figure}

\begin{figure}[p]
\centering
\includegraphics[width=\linewidth]{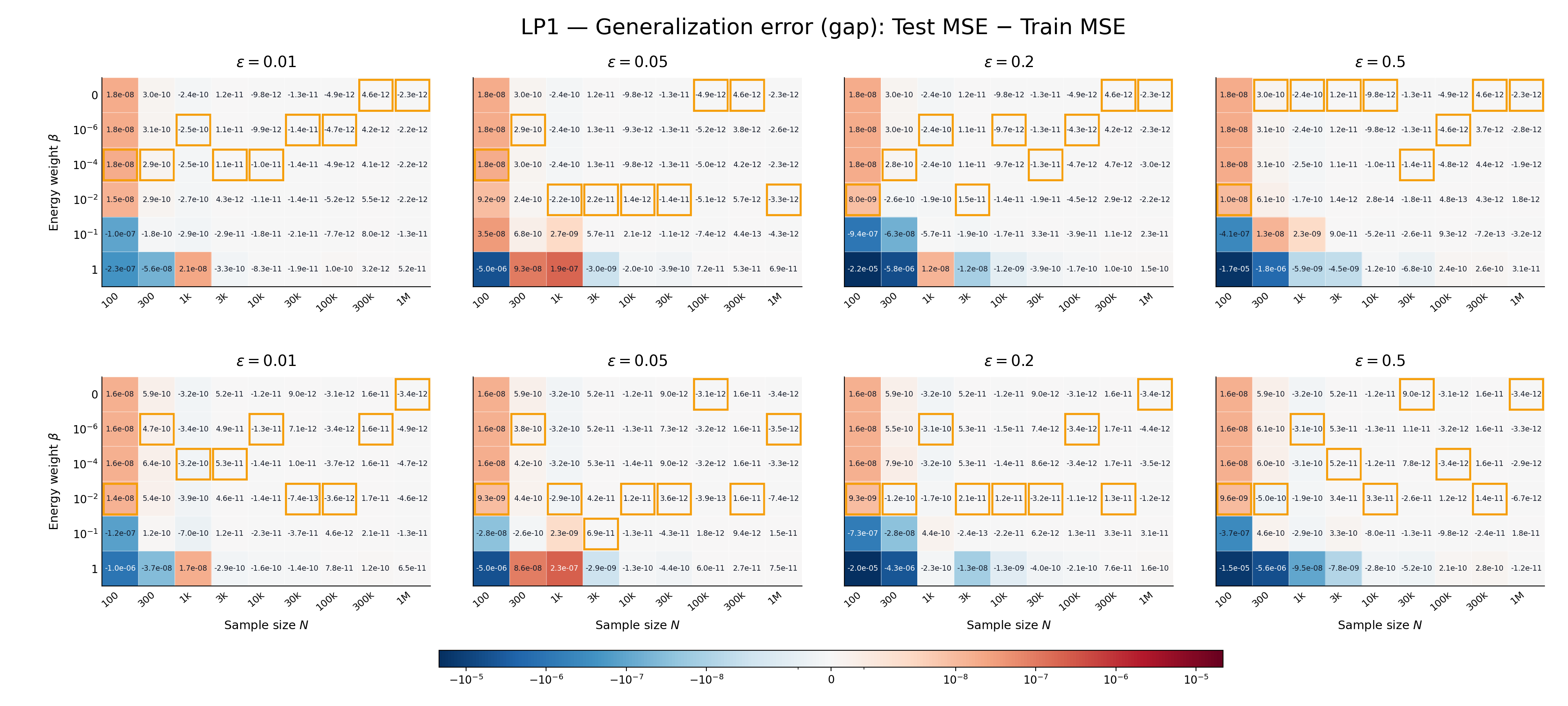}
\par\vspace{0.9em}
\includegraphics[width=\linewidth]{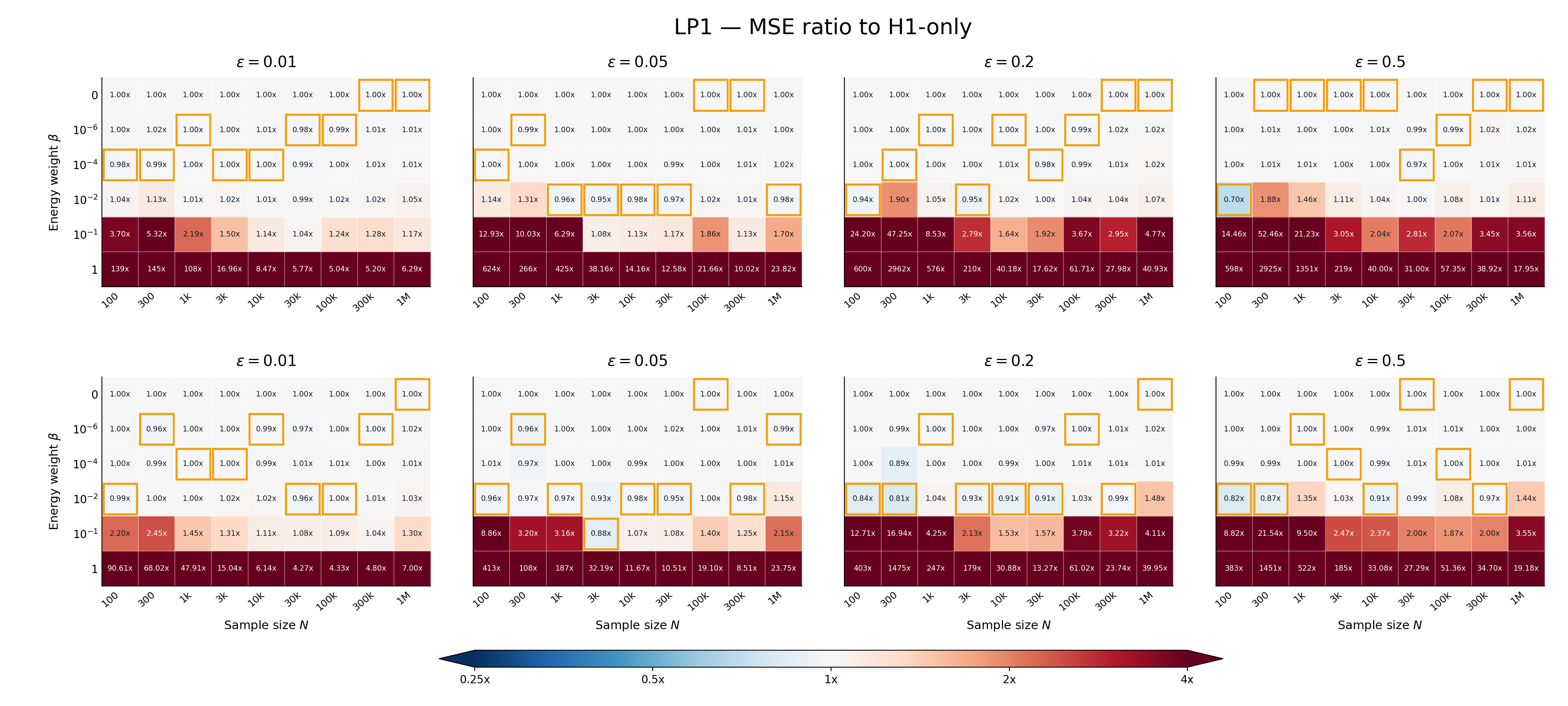}
\caption{LP1: signed value-MSE generalization gap, $G=E_{\mathrm{test}}-E_{\mathrm{train}}$ (first heatmap), and test-MSE ratio to the paired mixed-$H^1$ fitting baseline (second heatmap). Gap cells are means of signed gaps across seeds; ratios compare mean test MSEs. Ratios below one indicate improvement. Noise rows and orange outlines follow Figure~\ref{fig:psa-lp1-accuracy}.}
\label{fig:psa-lp1-generalization}
\end{figure}

\begin{figure}[p]
\centering
\includegraphics[width=\linewidth]{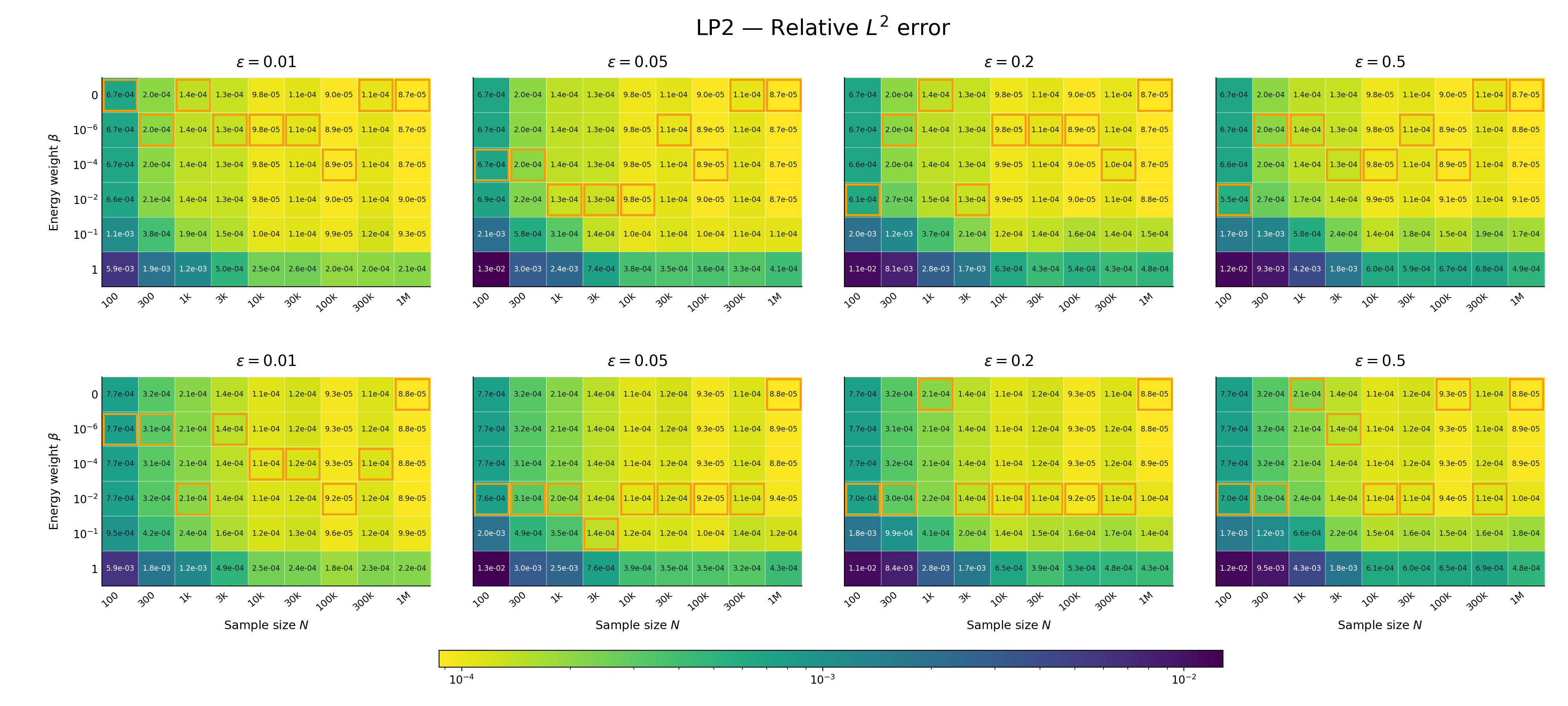}
\par\vspace{0.9em}
\includegraphics[width=\linewidth]{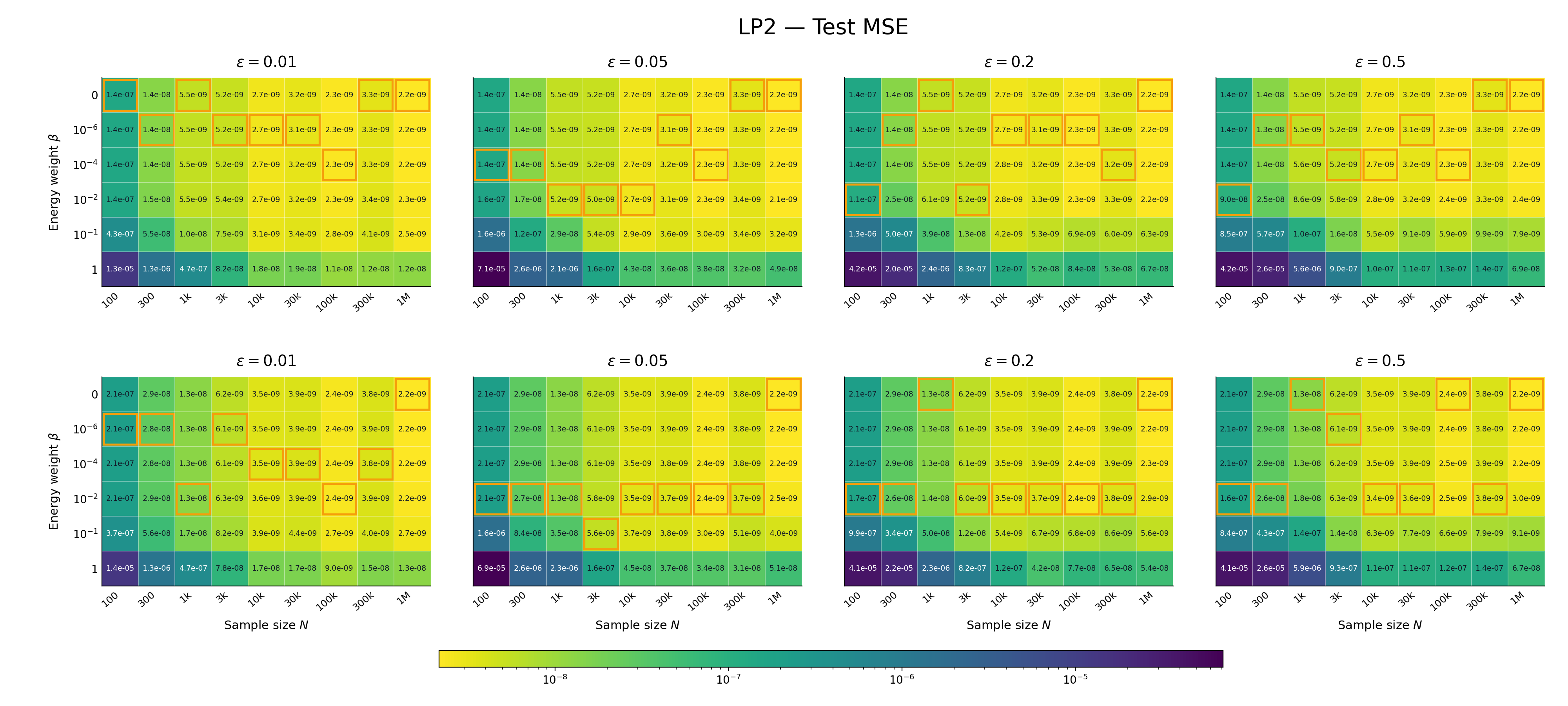}
\caption{LP2: relative $L^2$ error (first heatmap) and test MSE (second heatmap). Within each heatmap, the upper row has noisy $f$ and clean $u$; the lower row has clean $f$ and noisy $u$. Orange outlines mark the validation-selected $\beta$ for each fixed $(\epsilon,N)$.}
\label{fig:psa-lp2-accuracy}
\end{figure}

\begin{figure}[p]
\centering
\includegraphics[width=\linewidth]{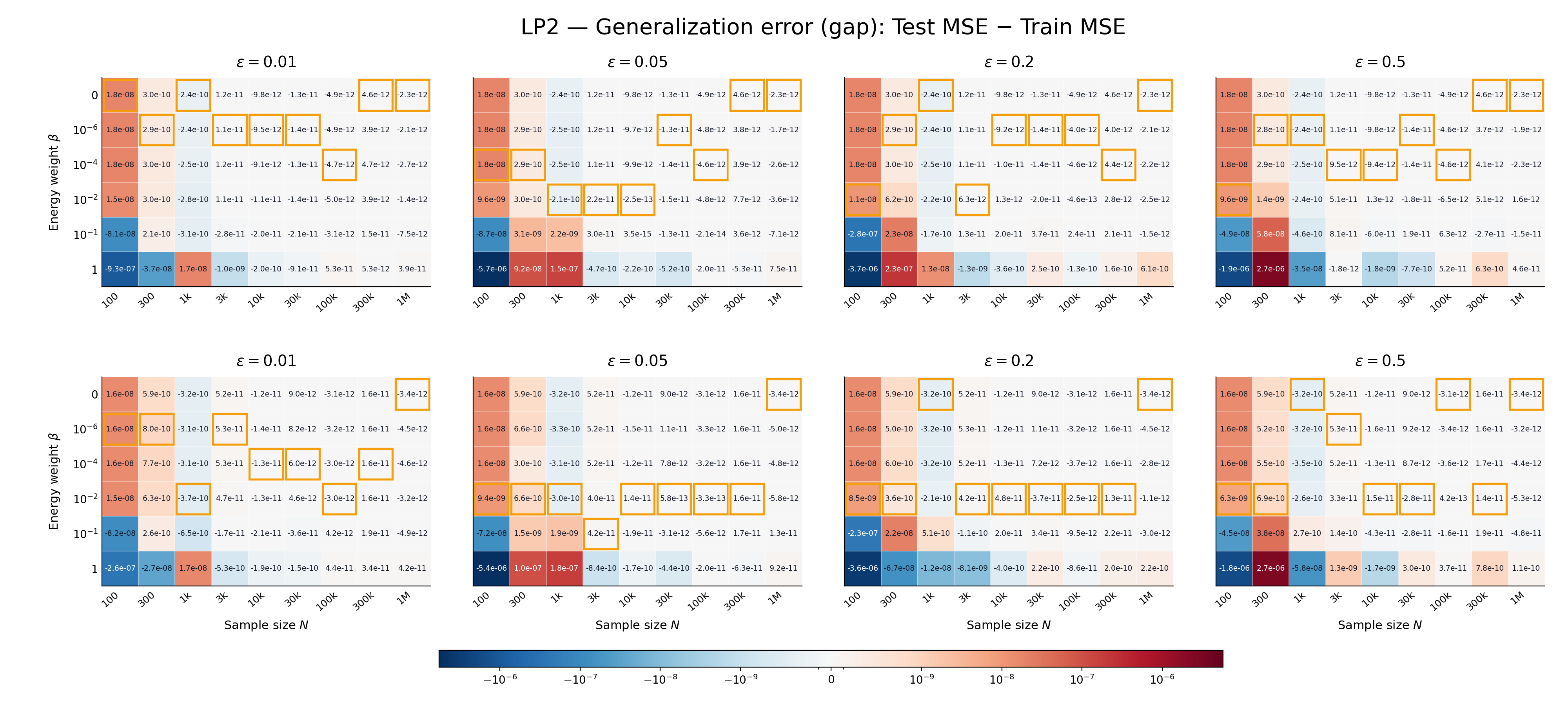}
\par\vspace{0.9em}
\includegraphics[width=\linewidth]{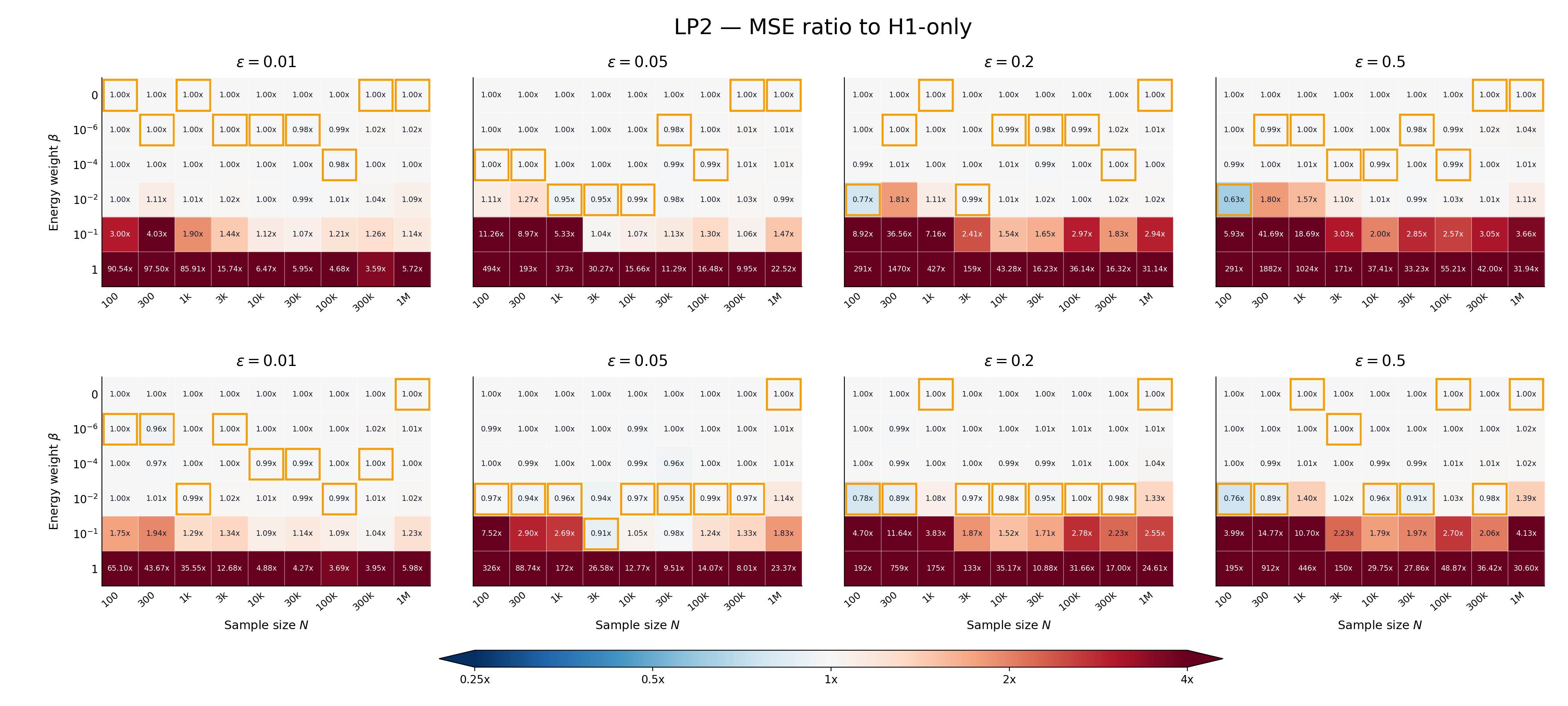}
\caption{LP2: signed value-MSE generalization gap, $G=R_{\mathrm{test}}-R_{\mathrm{train}}$ (first heatmap), and test-MSE ratio to the paired mixed-$H^1$ fitting baseline (second heatmap). Gap cells are means of signed gaps across seeds; ratios compare mean test MSEs. Ratios below one indicate improvement. Noise rows and orange outlines follow Figure~\ref{fig:psa-lp2-accuracy}.}
\label{fig:psa-lp2-generalization}
\end{figure}

\begin{figure}[p]
\centering
\includegraphics[width=\linewidth]{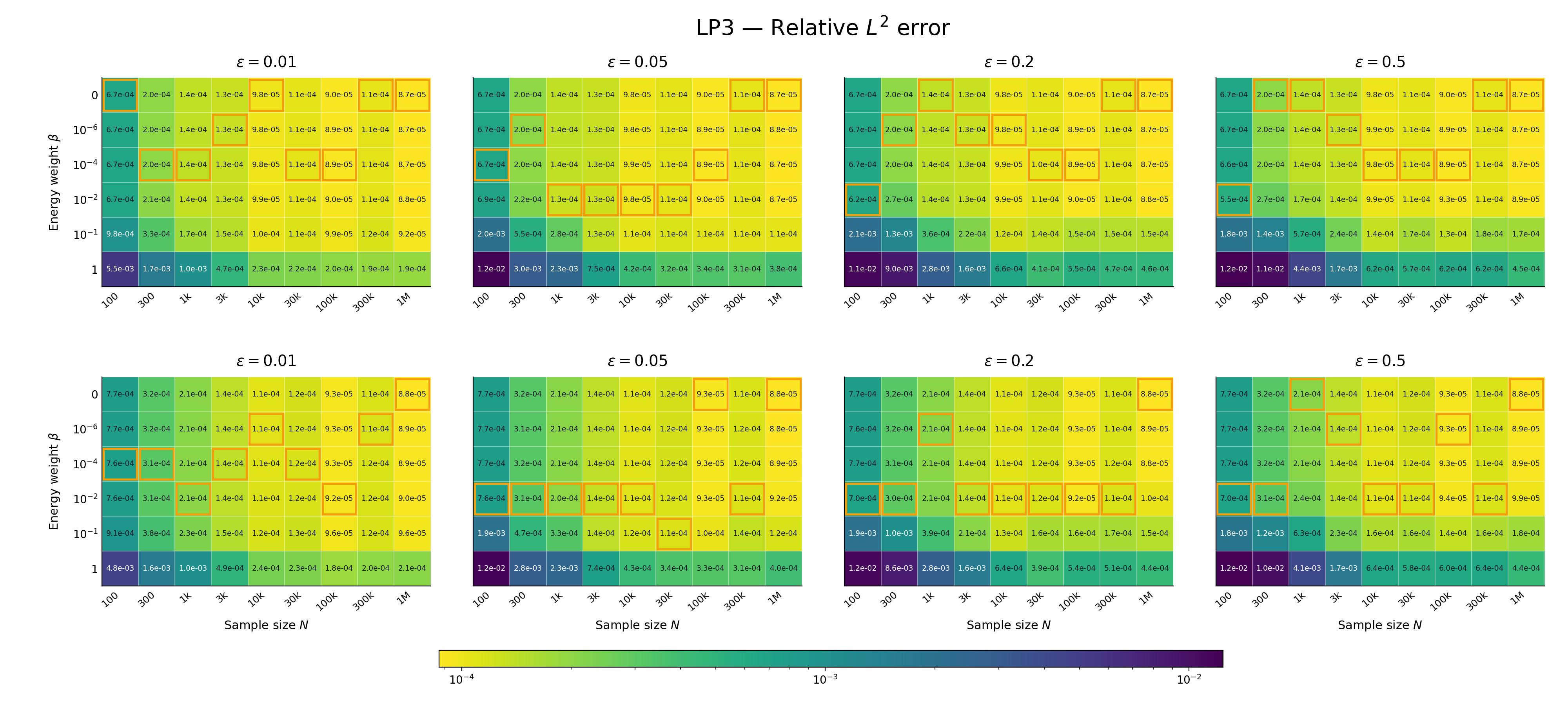}
\par\vspace{0.9em}
\includegraphics[width=\linewidth]{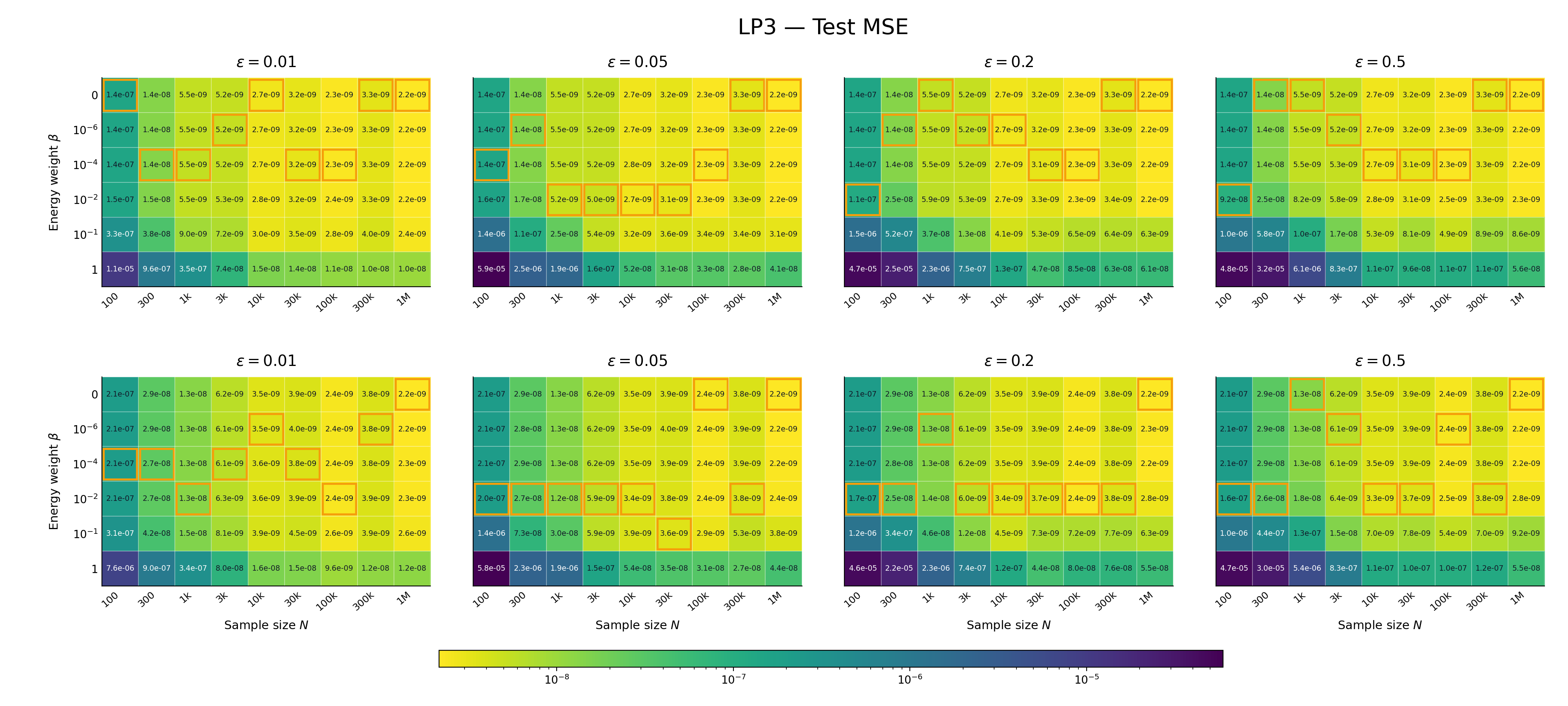}
\caption{LP3: relative $L^2$ error (first heatmap) and test MSE (second heatmap). Within each heatmap, the upper row has noisy $f$ and clean $u$; the lower row has clean $f$ and noisy $u$. Orange outlines mark the validation-selected $\beta$ for each fixed $(\epsilon,N)$.}
\label{fig:psa-lp3-accuracy}
\end{figure}

\begin{figure}[p]
\centering
\includegraphics[width=\linewidth]{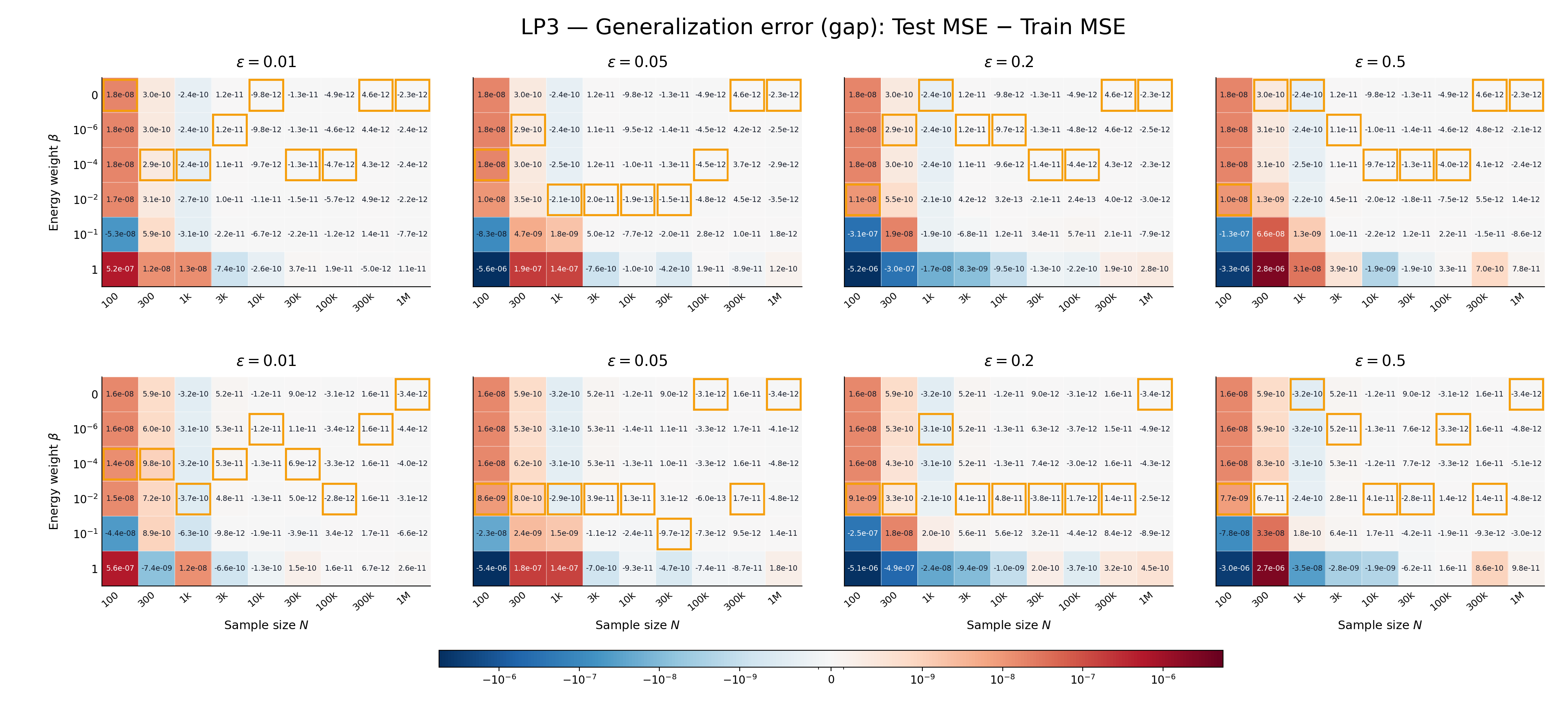}
\par\vspace{0.9em}
\includegraphics[width=\linewidth]{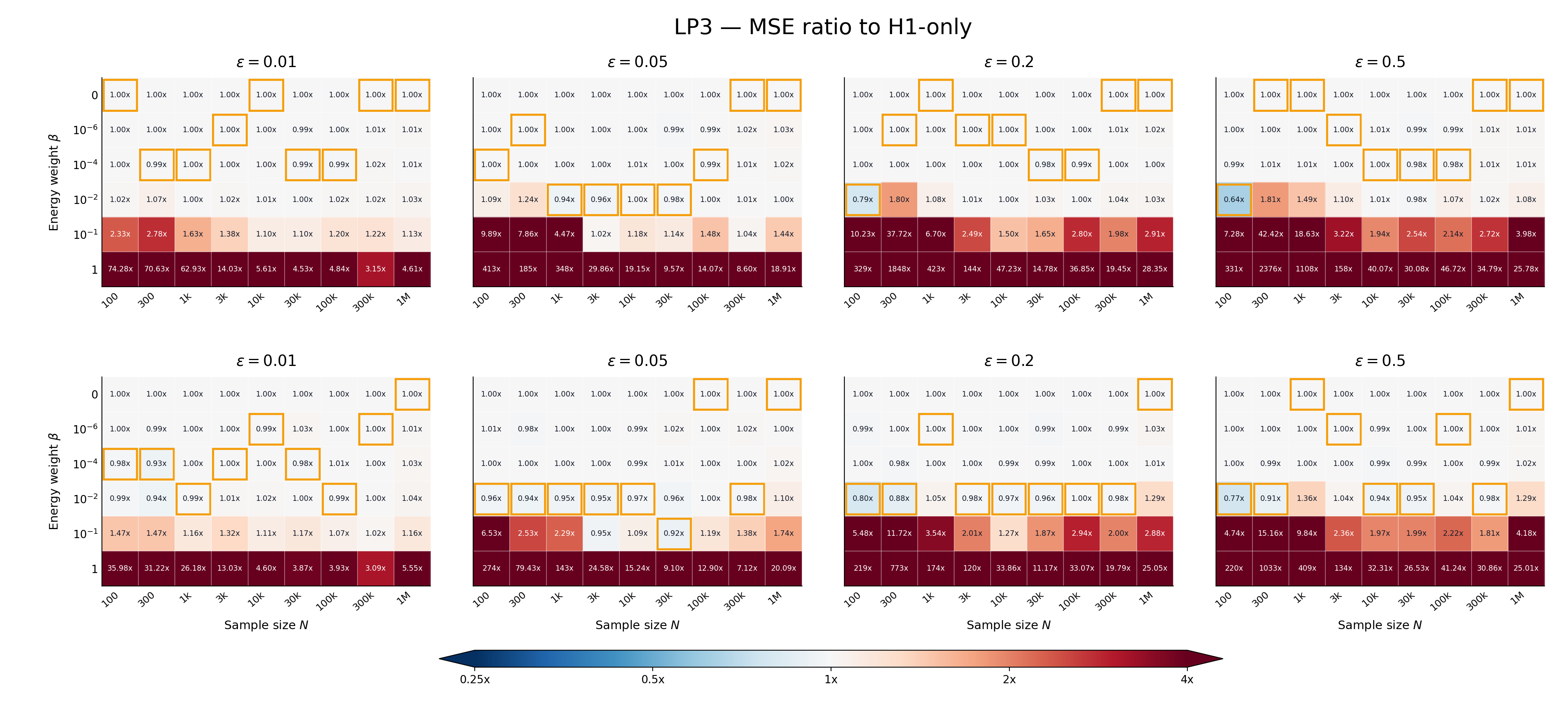}
\caption{LP3: signed value-MSE generalization gap, $G=R_{\mathrm{test}}-R_{\mathrm{train}}$ (first heatmap), and test-MSE ratio to the paired mixed-$H^1$ fitting baseline (second heatmap). Gap cells are means of signed gaps across seeds; ratios compare mean test MSEs. Ratios below one indicate improvement. Noise rows and orange outlines follow Figure~\ref{fig:psa-lp3-accuracy}.}
\label{fig:psa-lp3-generalization}
\end{figure}

\begin{figure}[p]
\centering
\includegraphics[width=\linewidth]{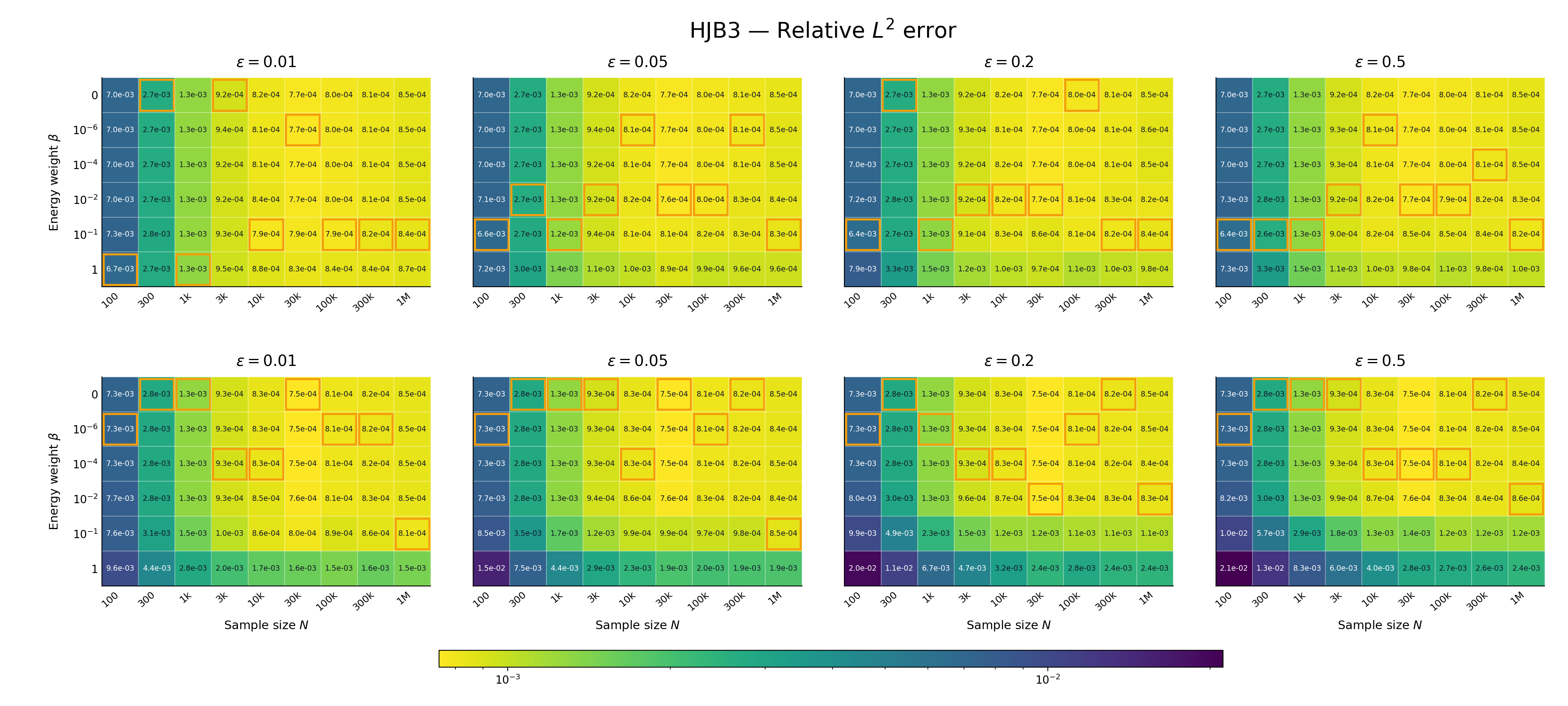}
\par\vspace{0.9em}
\includegraphics[width=\linewidth]{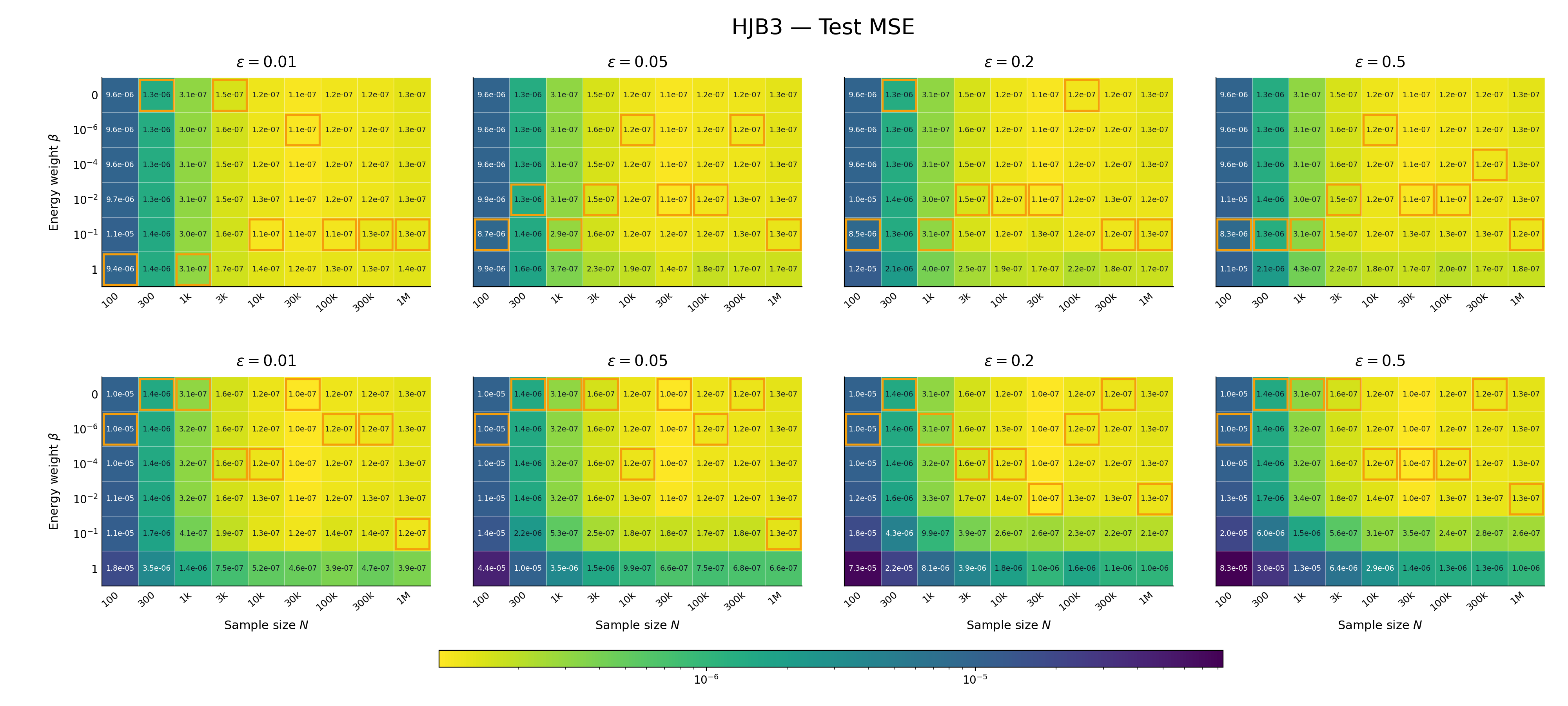}
\caption{HJB-3D: relative $L^2$ error (first heatmap) and test MSE (second heatmap). Within each heatmap, the upper row has noisy $f$ and clean $u$; the lower row has clean $f$ and noisy $u$. Orange outlines mark the validation-selected $\beta$ for each fixed $(\epsilon,N)$.}
\label{fig:psa-hjb3-accuracy}
\end{figure}

\begin{figure}[p]
\centering
\includegraphics[width=\linewidth]{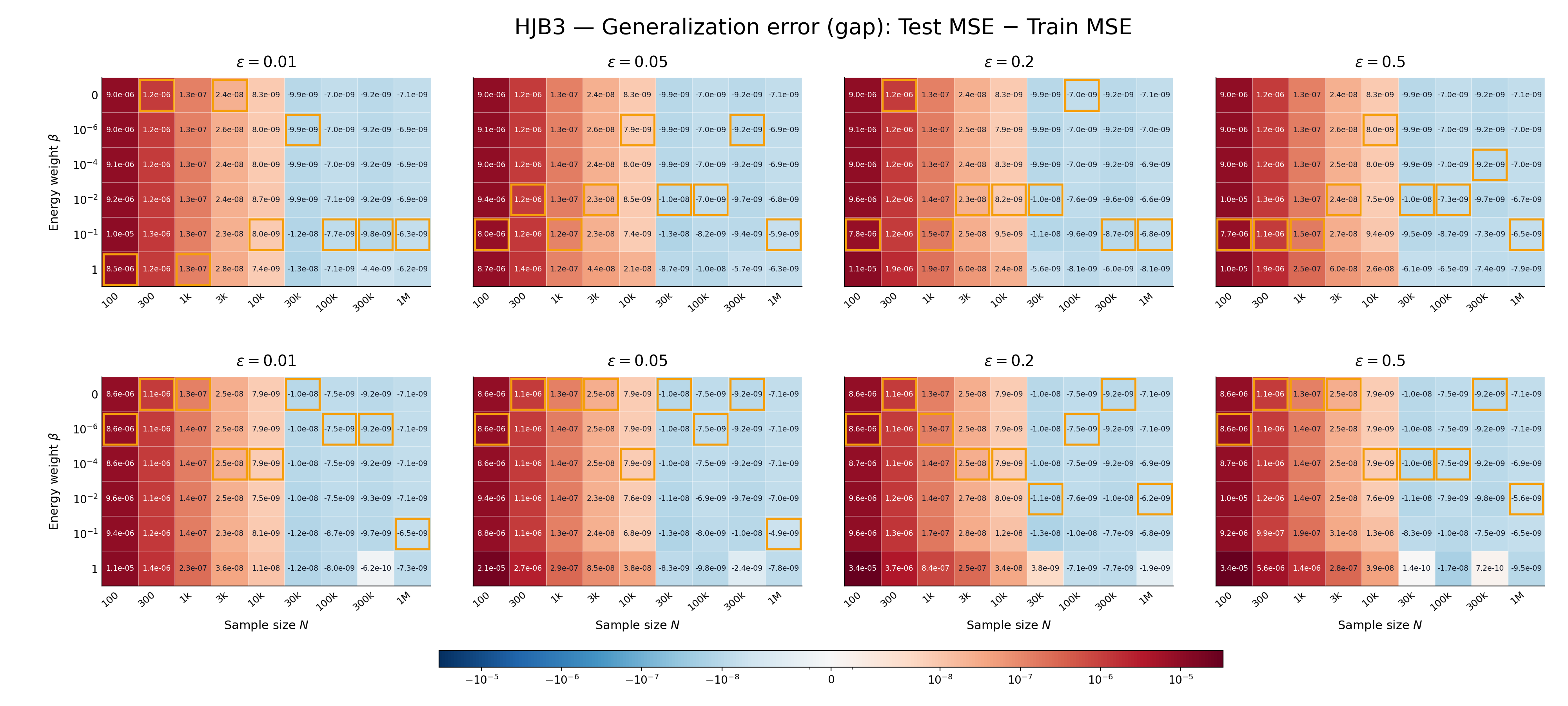}
\par\vspace{0.9em}
\includegraphics[width=\linewidth]{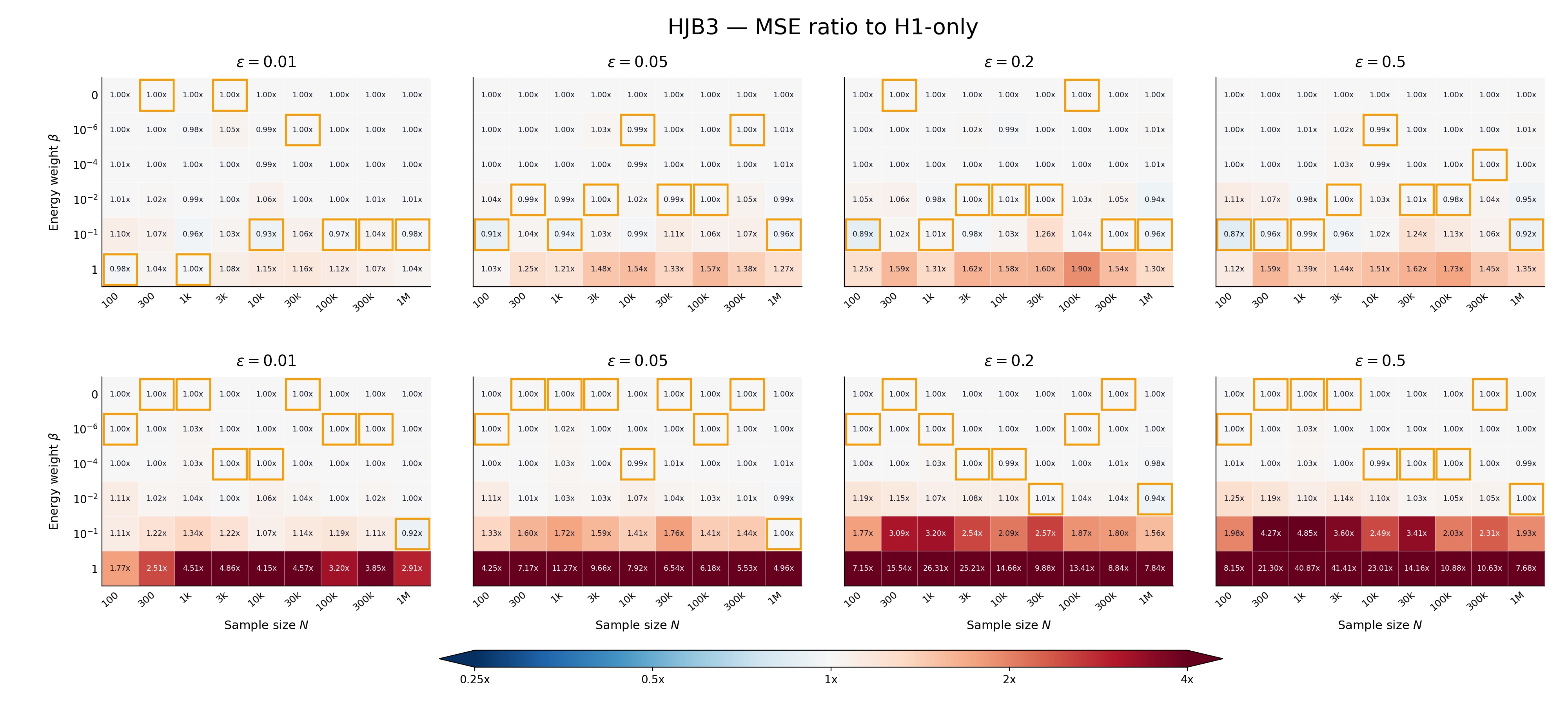}
\caption{HJB-3D: signed value-MSE generalization gap, $G=R_{\mathrm{test}}-R_{\mathrm{train}}$ (first heatmap), and test-MSE ratio to the paired mixed-$H^1$ fitting baseline (second heatmap). Gap cells are means of signed gaps across seeds; ratios compare mean test MSEs. Ratios below one indicate improvement. Noise rows and orange outlines follow Figure~\ref{fig:psa-hjb3-accuracy}.}
\label{fig:psa-hjb3-generalization}
\end{figure}

\clearpage
\section{Extension to Hyperbolic Equations}
\label{extension_to_hyperbolic}
We next use one family of hyperbolic equations to examine the extendability of our framework to see whether the weighted-energy construction can be adapted computationally to a different class of evolutionary PDEs. Further investigation of our method's extendability is worth to investigate in the future.
\paragraph{Nonlinear hyperbolic equations: Sine--Gordon.}
 As a nonlinear hyperbolic example, we consider the Sine--Gordon equation. Variational
approaches based on weighted inertia--energy functionals have been developed
for nonlinear wave equations; following this principle
\cite{akagi2025weighted}, we use
\begin{equation}
I_{\epsilon}^{\mathrm{hyp}}(u)
:=
\int_{0}^{T}\int_{\Omega}
e^{-t/\epsilon}
\left(
\frac{\epsilon^{2}\rho}{2}|u_{tt}|^{2}
+
\frac12|\nabla u|^{2}
+
F(u)
\right)
\,dx\,dt.
\label{hyperbolic_enregy_functional}
\end{equation}
The corresponding limiting wave equation has the form
\begin{equation}
\begin{cases}
\rho u_{tt}-\Delta u+F'(u)=0,
& (x,t)\in\Omega\times(0,T),\\
u(x,0)=u_0(x),
& x\in\Omega,\\
u_t(x,0)=u_1(x),
& x\in\Omega.
\end{cases}
\label{eq:nonlinear-wave}
\end{equation}

For the Sine--Gordon equation, we set
\[
\rho=1,
\qquad
F(u)=1-\cos(u),
\qquad
F'(u)=\sin(u),
\]
so that
\[
u_{tt}-\Delta u+\sin(u)=0.
\]
Let
\[
s_d(x,t)
=
\frac{
d^{-1/2}\sum_{j=1}^{d}x_j-vt
}{
\sqrt{1-v^2}
},
\qquad
v=0.3.
\]
A planar traveling-wave solution is
\begin{equation}
u_d^{*}(x,t)
=
4\arctan\!\left(\exp\!\left(s_d(x,t)\right)\right).
\label{eq:sine-gordon-ground-truth}
\end{equation}
The initial displacement and velocity are prescribed from this exact
solution:
\[
u_0(x)
=
4\arctan\!\left(\exp\!\left(s_d(x,0)\right)\right),
\]
and
\[
u_1(x)
=
-\frac{2v}{\sqrt{1-v^2}}
\operatorname{sech}\!\left(s_d(x,0)\right).
\]

We use this experiment only as a computational extension of the
energy-guided construction. The statistical theory in the main text remains
specific to the parabolic setting.

\paragraph{One-dimensional Sine--Gordon equation.}

Figure~\ref{fig:sine-gordon-solution} compares the learned solution produced
by our method with the manufactured traveling-wave solution.

\begin{figure}[htbp]
    \centering
    \includegraphics[width=0.58\textwidth]
    {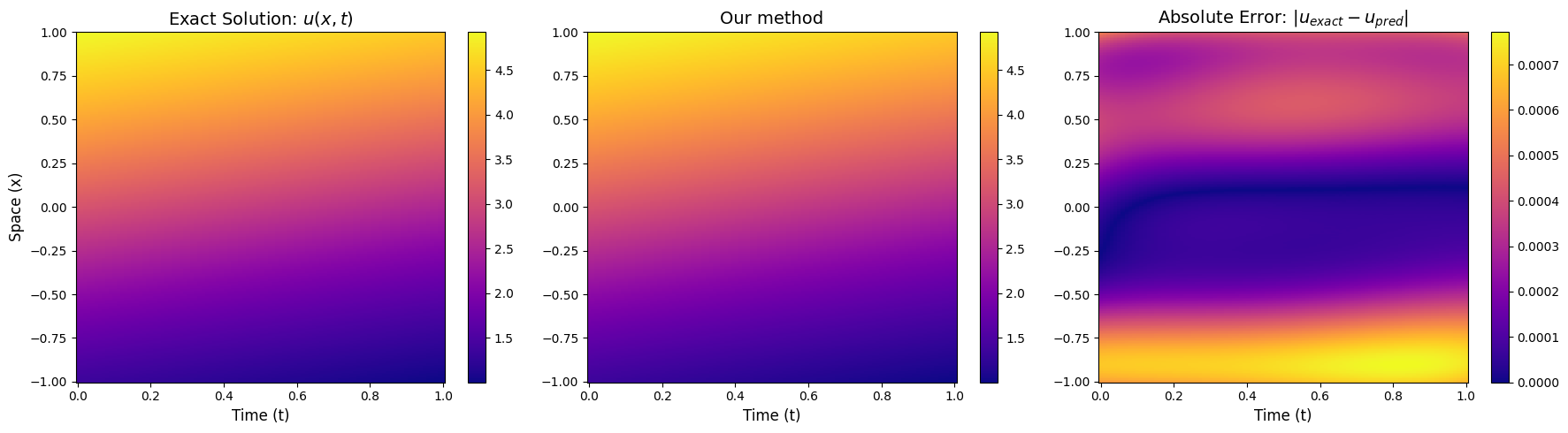}
    \caption{
    Learned solution from our method and the exact traveling-wave solution
    for the one-dimensional Sine--Gordon equation.
    }
    \label{fig:sine-gordon-solution}
\end{figure}

We use $\epsilon=10^{-5}$ in
\eqref{hyperbolic_enregy_functional}. We compare three basic objectives:
the vanilla PINN loss, the weighted-energy functional without labeled
solution data, and our complete objective combining data fitting with
energy guidance. Figure~\ref{fig_16} shows that the weighted-energy
functional alone is not sufficient to identify an accurate approximation
under this training configuration. Its combination with labeled data
produces a substantially more accurate solution.

\begin{figure}[htbp]
    \centering
    \includegraphics[width=0.62\textwidth]
    {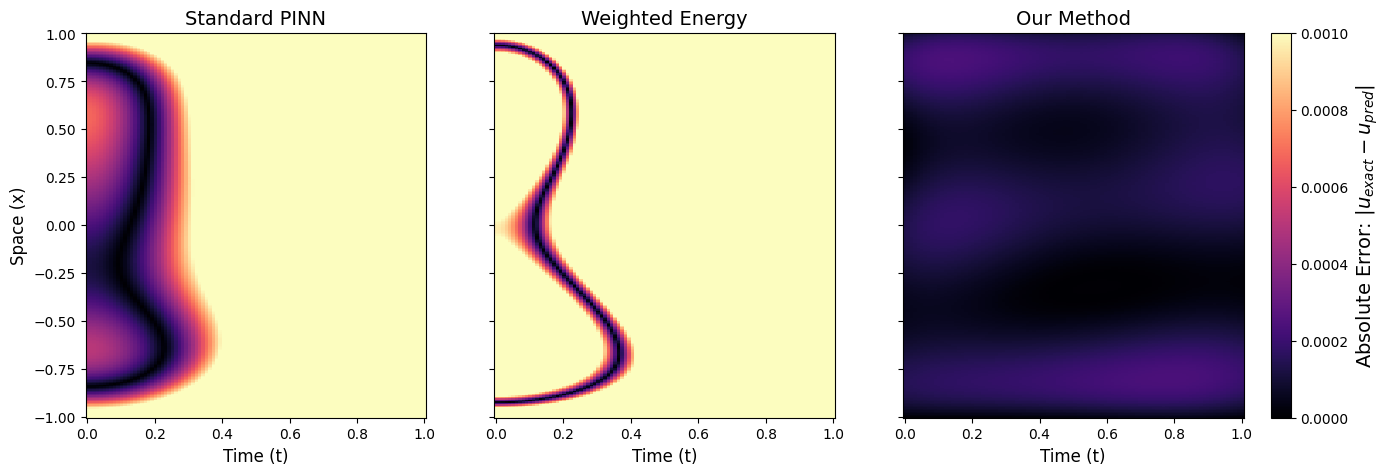}
    \caption{
    Error comparison for the one-dimensional Sine--Gordon equation:
    vanilla PINN, the weighted-energy objective without solution labels,
    and our complete energy-guided estimator.
    }
    \label{fig:sine-gordon-three-methods}
\end{figure}

The same conclusion is visible in the relative $L^2$ error curves in
Figure~\ref{fig_13}. With 5,000 training samples, combining the
data-fitting and weighted-energy terms gives the smallest final relative
error among the three objectives.

\begin{figure}[htbp]
    \centering
    \includegraphics[width=0.62\textwidth]
    {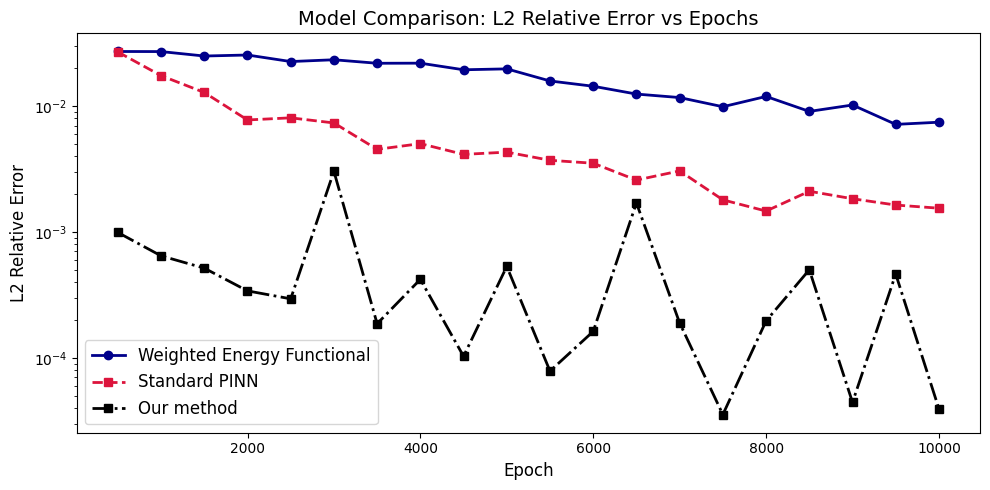}
    \caption{
    Relative $L^2$ error over training for the three objectives on the
    one-dimensional Sine--Gordon equation with 5,000 training samples.
    }
    \label{fig_13}
    \label{fig:sine-gordon-training-curves}
\end{figure}

The training curves occasionally contain sharp error spikes. These spikes
are observed in both the one- and three-dimensional experiments and are
primarily associated with optimization instability in the multi-term
objective. In particular, the data-fitting, derivative-label, and
energy-guidance terms can generate gradients with substantially different
scales. Consequently, a parameter update that decreases the training
objective can temporarily increase the test error.

Figure~\ref{fig_12} examines the effects of batch size and learning rate.
Smaller batches can initially reduce the test error more rapidly than
full-batch training, but they also introduce larger stochastic
fluctuations. Reducing the learning rate from $10^{-3}$ to $10^{-4}$
substantially stabilizes the training trajectory. Appropriate weighting of
the derivative and energy terms provides another possible stabilization
mechanism, although it introduces the loss-balancing problem commonly
encountered in PINN training
\cite{WangTengPerdikaris2021_PINN-weak}.

\begin{figure}[htbp]
    \centering
    \includegraphics[width=0.49\textwidth]
    {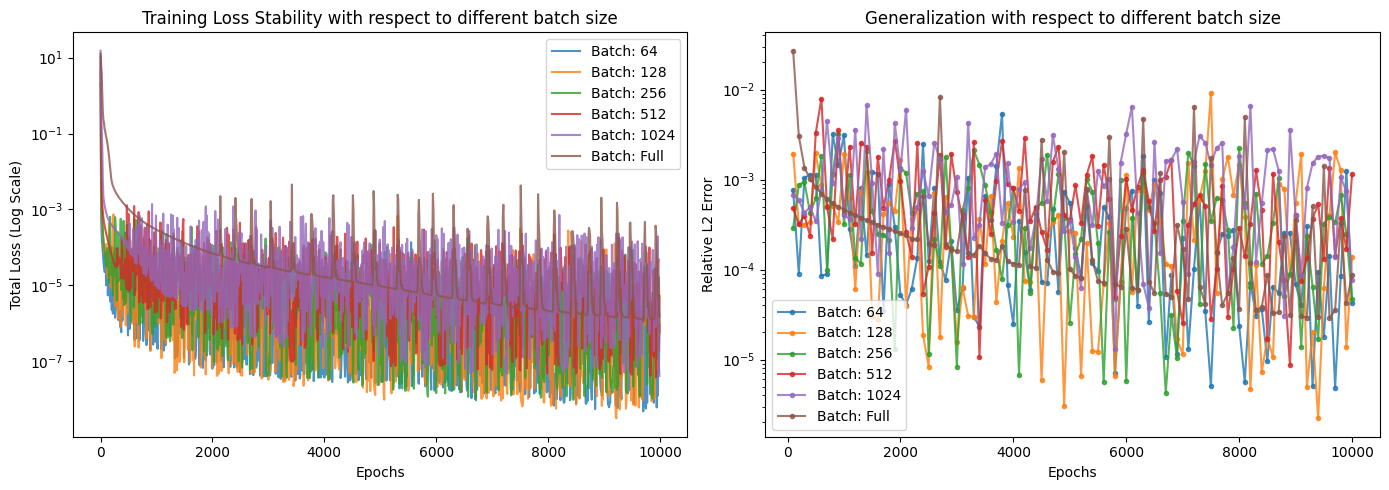}
    \includegraphics[width=0.49\textwidth]
    {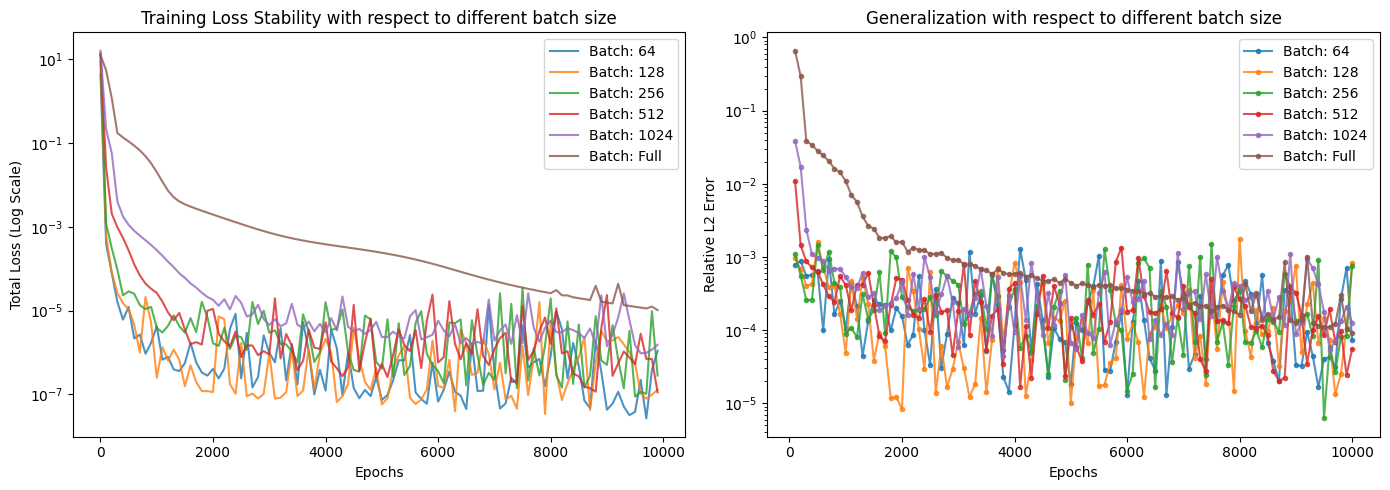}
    \caption{
    Batch-size and learning-rate ablations for the one-dimensional
    Sine--Gordon equation. Left: batch-size comparison with learning rate
    $10^{-3}$. Right: training with the reduced learning rate $10^{-4}$.
    }
    \label{fig_12}
    \label{fig:sine-gordon-optimization}
\end{figure}

After reducing the learning rate, the comparison in
Figure~\ref{fig_15} shows more stable convergence of the proposed estimator
relative to the baseline methods.

\begin{figure}[htbp]
    \centering
    \includegraphics[width=0.62\textwidth]
    {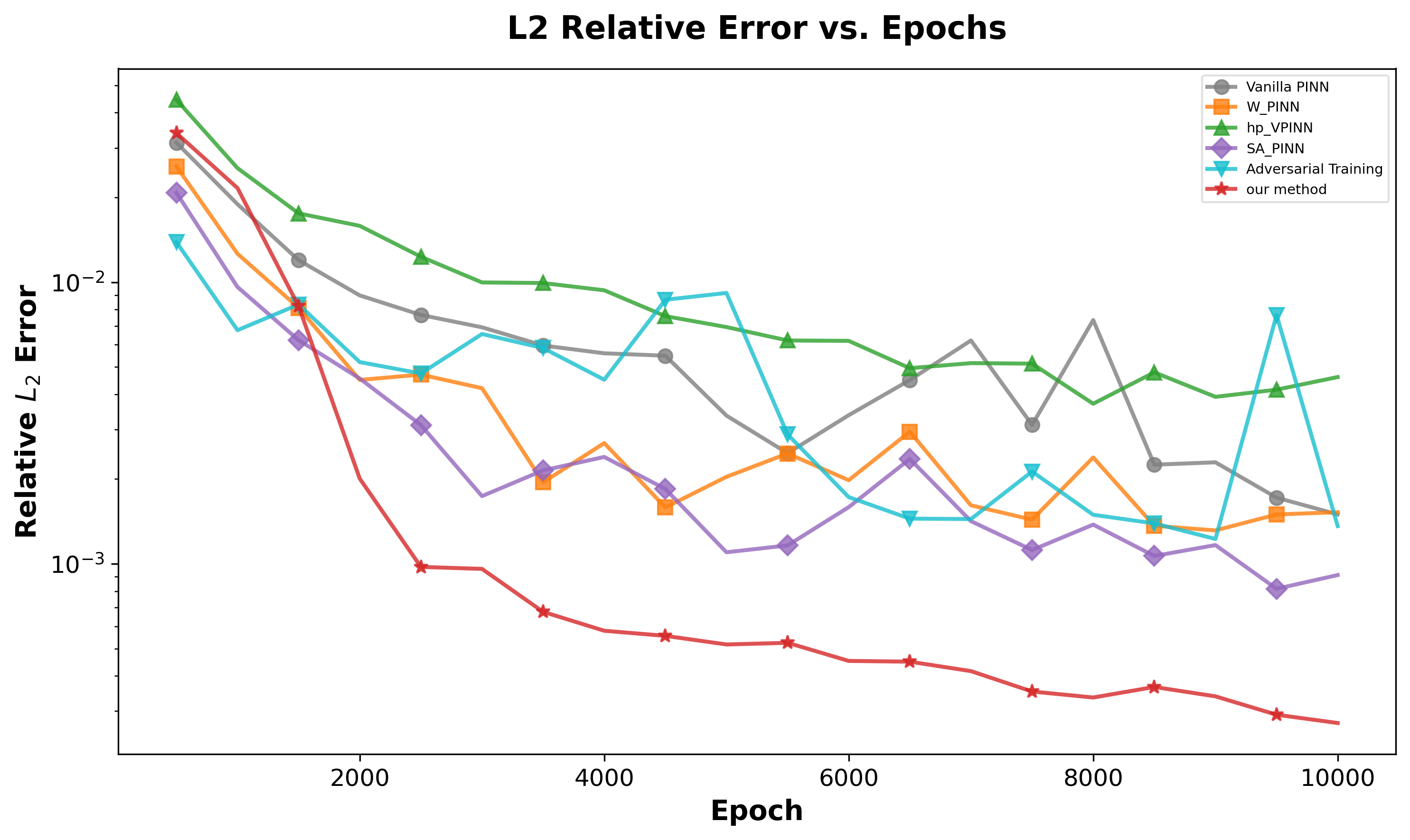}
    \caption{
    Method comparison for the one-dimensional Sine--Gordon equation with
    learning rate $10^{-4}$.
    }
    \label{fig_15}
    \label{fig:sine-gordon-method-comparison}
\end{figure}

\paragraph{Three-dimensional Sine--Gordon equation.}

We also evaluate the planar traveling-wave solution
\eqref{eq:sine-gordon-ground-truth} in three spatial dimensions.
Figure~\ref{fig_17} presents representative temporal and spatial slices of
the learned solution and its error relative to the manufactured solution.

\begin{figure}[htbp]
    \centering
    \includegraphics[width=0.49\textwidth]
    {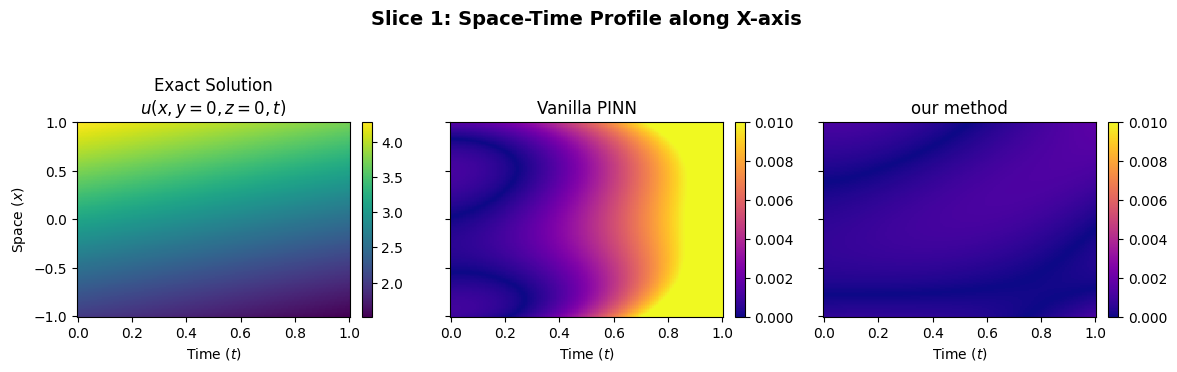}
    \includegraphics[width=0.49\textwidth]
    {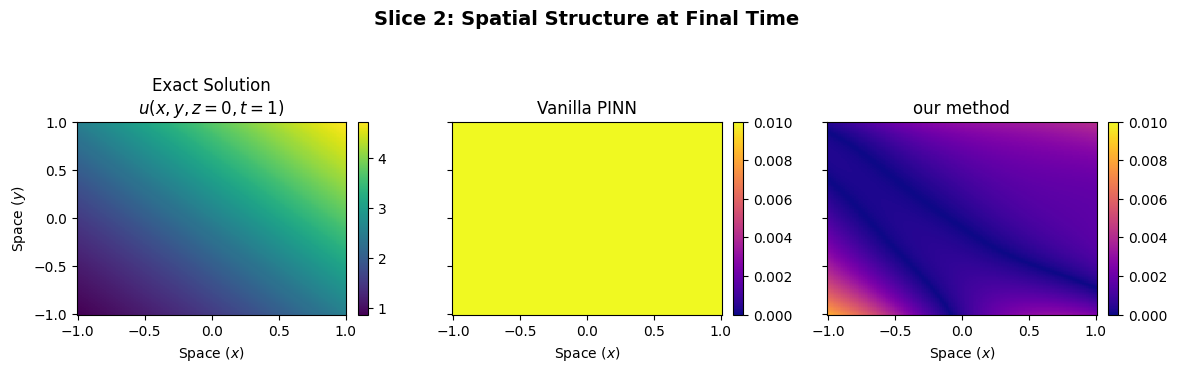}
    \caption{
    Representative slices for the three-dimensional Sine--Gordon equation.
    The ground-truth slices are shown together with the corresponding errors
    of the learned solutions.
    }
    \label{fig_17}
    \label{fig:sine-gordon-3d-slices}
\end{figure}

The one- and three-dimensional experiments indicate that the
energy-guided objective can be adapted computationally to this nonlinear
hyperbolic example. They also show that the physical energy term alone is
not sufficient under the present data and optimization configuration; the
improvement is obtained by combining the energy structure with labeled
solution information. It is clear from Table~\ref{table_3} that our method achieves the lowest mean squared error (MSE) and L2 relative error against various PINNs methods that struggles to handle singularity. 
Moreover, comparisons between our learned solution and ground-truth for temporal evolution snapshots when $z=0$ are shown in Figure~\ref{3D_Sine_Gordon_EXAMPLE}.
\begin{figure}[htbp]
\centering

\begin{minipage}[t]{0.6\linewidth}
\centering
\vspace{0pt}
\footnotesize
\setlength{\tabcolsep}{3pt}
\renewcommand{\arraystretch}{1.2}

\begin{tabular*}{\linewidth}
{@{\extracolsep{\fill}}p{0.52\linewidth}cc@{}}
\toprule
\textbf{Method} & \textbf{Rel. $L^2$ error} & \textbf{MSE} \\
\midrule
Vanilla PINNs \cite{PINN_Karniadakis2021}
& $4.01\times10^{-3}$ & $1.51\times10^{-4}$ \\

W-PINNs \cite{WangTengPerdikaris2021_PINN-weak}
& $5.32\times10^{-3}$ & $2.65\times10^{-4}$ \\

SA-PINNs \cite{McClenny2023SelfAdaptivePINN}
& $3.85\times10^{-3}$ & $1.39\times10^{-4}$ \\

Adversarial Training \cite{wang2022is}
& $8.25\times10^{-2}$ & $6.37\times10^{-2}$ \\

\textbf{Our method}
& $\mathbf{8.89\times10^{-4}}$
& $\mathbf{7.39\times10^{-6}}$ \\
\bottomrule
\end{tabular*}

\captionof{table}{Test error comparison for the 3D Sine--Gordon}
\label{table_3}
\end{minipage}
\hfill
\begin{minipage}[t]{0.38\linewidth}
\centering
\vspace{0pt}

\includegraphics[width=\linewidth]
{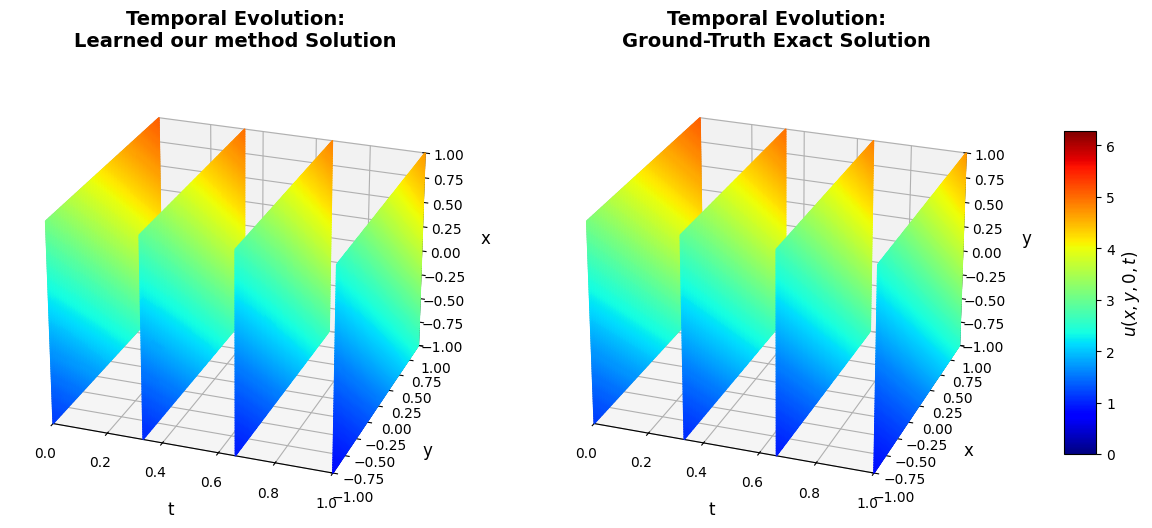}

\caption{Solution snapshots for the 3D Sine--Gordon equation at $z=0$.}
\label{3D_Sine_Gordon_EXAMPLE}
\end{minipage}

\end{figure}

\end{document}